\documentclass{amsart}
\usepackage{amssymb}
\usepackage{amscd}
\usepackage{verbatim}
\usepackage{epsfig}
\DeclareMathOperator{\Imag}{\im}
\DeclareMathOperator{\Realp}{\re}

\begin{document}
\newcommand\Mand{\ \text{and}\ }
\newcommand\Mor{\ \text{or}\ }
\newcommand\Mfor{\ \text{for}\ }
\newcommand\Real{\mathbb{R}}
\newcommand\RR{\mathbb{R}}
\newcommand\im{\operatorname{Im}}
\newcommand\re{\operatorname{Re}}
\newcommand\sign{\operatorname{sign}}
\newcommand\sphere{\mathbb{S}}
\newcommand\BB{\mathbb{B}}
\newcommand\HH{\mathbb{H}}
\newcommand\ZZ{\mathbb{Z}}
\newcommand\codim{\operatorname{codim}}
\newcommand\Sym{\operatorname{Sym}}
\newcommand\End{\operatorname{End}}
\newcommand\Span{\operatorname{span}}
\newcommand\Ran{\operatorname{Ran}}
\newcommand\Ker{\operatorname{Ker}}
\newcommand\ep{\epsilon}
\newcommand\Cinf{\cC^\infty}
\newcommand\dCinf{\dot \cC^\infty}
\newcommand\CI{\cC^\infty}
\newcommand\dCI{\dot \cC^\infty}
\newcommand\Cx{\mathbb{C}}
\newcommand\Nat{\mathbb{N}}
\newcommand\dist{\cC^{-\infty}}
\newcommand\ddist{\dot \cC^{-\infty}}
\newcommand\pa{\partial}
\newcommand\Card{\mathrm{Card}}
\renewcommand\Box{{\square}}
\newcommand\WF{\mathrm{WF}}
\newcommand\WFh{\mathrm{WF}_h}
\newcommand\WFb{\mathrm{WF}_\bl}
\newcommand\Vf{\mathcal{V}}
\newcommand\Vb{\mathcal{V}_\bl}
\newcommand\Vz{\mathcal{V}_0}
\newcommand\Hb{H_{\bl}}
\newcommand\Hom{\mathrm{Hom}}
\newcommand\Id{\mathrm{Id}}
\newcommand\sgn{\operatorname{sgn}}
\newcommand\ff{\mathrm{ff}}
\newcommand\tf{\mathrm{tf}}
\newcommand\supp{\operatorname{supp}}
\newcommand\vol{\mathrm{vol}}
\newcommand\Diff{\mathrm{Diff}}
\newcommand\Diffd{\mathrm{Diff}_{\dagger}}
\newcommand\Diffs{\mathrm{Diff}_{\sharp}}
\newcommand\Diffb{\mathrm{Diff}_\bl}
\newcommand\DiffbI{\mathrm{Diff}_{\bl,I}}
\newcommand\Diffbeven{\mathrm{Diff}_{\bl,\even}}
\newcommand\Diffz{\mathrm{Diff}_0}
\newcommand\Psih{\Psi_{\semi}}
\newcommand\Psihcl{\Psi_{\semi,\cl}}
\newcommand\Psib{\Psi_\bl}
\newcommand\Psibc{\Psi_{\mathrm{bc}}}
\newcommand\TbC{{}^{\bl,\Cx} T}
\newcommand\Tb{{}^{\bl} T}
\newcommand\Sb{{}^{\bl} S}
\newcommand\Lambdab{{}^{\bl} \Lambda}
\newcommand\zT{{}^{0} T}
\newcommand\Tz{{}^{0} T}
\newcommand\zS{{}^{0} S}
\newcommand\dom{\mathcal{D}}
\newcommand\cA{\mathcal{A}}
\newcommand\cB{\mathcal{B}}
\newcommand\cD{\mathcal{D}}
\newcommand\cE{\mathcal{E}}
\newcommand\cG{\mathcal{G}}
\newcommand\cH{\mathcal{H}}
\newcommand\cU{\mathcal{U}}
\newcommand\cO{\mathcal{O}}
\newcommand\cF{\mathcal{F}}
\newcommand\cM{\mathcal{M}}
\newcommand\cQ{\mathcal{Q}}
\newcommand\cR{\mathcal{R}}
\newcommand\cI{\mathcal{I}}
\newcommand\cL{\mathcal{L}}
\newcommand\cK{\mathcal{K}}
\newcommand\cC{\mathcal{C}}
\newcommand\cX{\mathcal{X}}
\newcommand\cY{\mathcal{Y}}
\newcommand\cP{\mathcal{P}}
\newcommand\cS{\mathcal{S}}
\newcommand\Ptil{\tilde P}
\newcommand\ptil{\tilde p}
\newcommand\chit{\tilde \chi}
\newcommand\yt{\tilde y}
\newcommand\zetat{\tilde \zeta}
\newcommand\xit{\tilde \xi}
\newcommand\taut{{\tilde \tau}}
\newcommand\phit{{\tilde \phi}}
\newcommand\mut{{\tilde \mu}}
\newcommand\sigmah{\hat\sigma}
\newcommand\zetah{\hat\zeta}
\newcommand\etah{\hat\eta}
\newcommand\loc{\mathrm{loc}}
\newcommand\compl{\mathrm{comp}}
\newcommand\reg{\mathrm{reg}}
\newcommand\GBB{\textsf{GBB}}
\newcommand\GBBsp{\textsf{GBB}\ }
\newcommand\bl{{\mathrm b}}
\newcommand{\sH}{\mathsf{H}}
\newcommand{\cte}{\digamma}
\newcommand\cl{\operatorname{cl}}
\newcommand\hsf{\mathcal{S}}
\newcommand\Div{\operatorname{div}}
\newcommand\hilbert{\mathfrak{X}}

\newcommand\bM{\bar M}
\newcommand\bdiff{{}^{\bl}d}

\newcommand\xib{{\underline{\xi}}}
\newcommand\etab{{\underline{\eta}}}
\newcommand\zetab{{\underline{\zeta}}}

\newcommand\xibh{{\underline{\hat \xi}}}
\newcommand\etabh{{\underline{\hat \eta}}}
\newcommand\zetabh{{\underline{\hat \zeta}}}

\newcommand\zn{z}
\newcommand\sigman{\sigma}
\newcommand\psit{\tilde\psi}
\newcommand\rhot{{\tilde\rho}}

\newcommand\hM{\hat M}

\newcommand\Op{\operatorname{Op}}
\newcommand\Oph{\operatorname{Op_{\semi}}}

\newcommand\innr{{\mathrm{inner}}}
\newcommand\outr{{\mathrm{outer}}}
\newcommand\full{{\mathrm{full}}}
\newcommand\semi{\hbar}

\newcommand\elliptic{\mathrm{ell}}
\newcommand\difford{k}
\newcommand\even{\mathrm{even}}
\newcommand\waveopen{\cO}

\setcounter{secnumdepth}{3}
\newtheorem{lemma}{Lemma}[section]
\newtheorem{prop}[lemma]{Proposition}
\newtheorem{thm}[lemma]{Theorem}
\newtheorem{cor}[lemma]{Corollary}
\newtheorem{result}[lemma]{Result}
\newtheorem*{thm*}{Theorem}
\newtheorem*{prop*}{Proposition}
\newtheorem*{cor*}{Corollary}
\newtheorem*{conj*}{Conjecture}
\numberwithin{equation}{section}
\theoremstyle{remark}
\newtheorem{rem}[lemma]{Remark}
\newtheorem*{rem*}{Remark}
\theoremstyle{definition}
\newtheorem{Def}[lemma]{Definition}
\newtheorem*{Def*}{Definition}

\newcommand{\mar}[1]{{\marginpar{\sffamily{\scriptsize #1}}}}
\newcommand\av[1]{\mar{AV:#1}}

\renewcommand{\theenumi}{\roman{enumi}}
\renewcommand{\labelenumi}{(\theenumi)}

\title[Microlocal asymptotically hyperbolic and Kerr-de Sitter]{Microlocal analysis of
asymptotically hyperbolic and Kerr-de Sitter spaces}
\author[Andras Vasy]{Andr\'as Vasy\\ \\{\smaller with an appendix by} Semyon Dyatlov}
\address{Department of Mathematics, Stanford University, CA
  94305-2125, USA}
\thanks{A.V.'s contact information:  Department of Mathematics, Stanford University, CA
  94305-2125, USA. Fax: 1-650-725-4066. Tel: 1-650-723-2226. E-mail: \texttt{andras@math.stanford.edu}}

\email{andras@math.stanford.edu}


\date{Revised: November 1, 2011. Original version: December 20, 2010}

\begin{abstract}
In this paper we develop a general, systematic, microlocal framework for the
Fredholm analysis of non-elliptic problems, including high energy (or
semiclassical) estimates, which is stable under perturbations.
This framework, described in
Section~\ref{sec:microlocal},
resides on a compact manifold without boundary, hence in the
standard setting of microlocal analysis.

Many natural
applications arise in the setting of non-Riemannian b-metrics in
the context of Melrose's b-structures. These include
asymptotically de Sitter-type metrics on a blow-up
of the natural compactification, Kerr-de Sitter-type metrics,
as well as asymptotically
Minkowski metrics.

The simplest application is a new approach to
analysis on Riemannian or Lorentzian (or indeed, possibly of other
signature)
conformally compact spaces (such as asymptotically hyperbolic or de
Sitter spaces), including a new construction
of the meromorphic extension of the resolvent of the Laplacian in the
Riemannian case, as well as high energy estimates for the spectral
parameter in strips of the
complex plane. These results are also available in a follow-up paper
which is more expository in nature, \cite{Vasy:Microlocal-AH}.

The appendix written by Dyatlov relates his analysis of resonances
on exact Kerr-de Sitter space (which then was used to analyze the wave equation
in that setting) to the more general method described here.
\end{abstract}

\maketitle

\section{Introduction}
In this paper we develop a general microlocal framework
which in particular allows us to analyze
the asymptotic behavior of solutions of the wave equation
on asymptotically
Kerr-de Sitter and Minkowski space-times, as well as the behavior of the analytic
continuation of the resolvent of the Laplacian on so-called conformally compact spaces.
This framework is non-perturbative,
and works, in particular, for black holes,
for relatively large angular momenta (the restrictions come {\em purely} from
dynamics, and not from methods of analysis of PDE), and also for perturbations of
Kerr-de Sitter space, where
`perturbation' is only relevant to the extent that it guarantees that the relevant
structures are preserved. In the context of analysis on conformally
compact spaces, our framework establishes a Riemannian-Lorentzian
duality; in this duality the spaces of different signature are smooth continuations of each
other across a boundary at which the differential operator we study
has some radial points in the sense of microlocal analysis.

Since it is particularly easy to state, and only involves Riemannian geometry,
we start by giving a result on manifolds with
{\em even} conformally compact metrics.
These are Riemannian metrics $g_0$
on the interior of a compact manifold with boundary
$X_0$ such that near the boundary $Y$,
with a product decomposition nearby and
a defining function $x$, they are
of the form
$$
g_0=\frac{dx^2+h}{x^2},
$$
where $h$ is a family of metrics on $\pa X_0$ depending on $x$ in an even manner,
i.e.\ only even powers of $x$ show up in the Taylor series. (There is a much more
natural way to phrase the evenness condition, see
\cite[Definition~1.2]{Guillarmou:Meromorphic}.) We also write $X_{0,\even}$ for
the manifold $X_0$ when the smooth structure has been changed so that
$x^2$ is a boundary defining function; thus, a smooth function on $X_0$ is
even if and only if it is smooth when regarded as a function on $X_{0,\even}$.
The analytic continuation of the resolvent in this category (but without the
evenness condition) was obtained
by Mazzeo and Melrose \cite{Mazzeo-Melrose:Meromorphic}, with possibly
some essential singularities at pure imaginary half-integers as noticed by Borthwick and Perry
\cite{Borthwick-Perry:Scattering}.
Using methods of Graham and Zworski
\cite{Graham-Zworski:Scattering}, Guillarmou
\cite{Guillarmou:Meromorphic} showed that for even metrics the latter do not
exist, but generically they do exist for non-even metrics. Further, if the manifold
is actually asymptotic to hyperbolic space (note that hyperbolic space is of this form
in view of the Poincar\'e model),
Melrose, S\'a Barreto and Vasy \cite{Melrose-SaBarreto-Vasy:Semiclassical} showed
high energy resolvent estimates in strips around the real axis via a parametrix
construction; these are exactly
the estimates that allow expansions for solutions of the wave equation in terms
of resonances. Estimates just on the real axis were obtained by
Cardoso and Vodev for more general conformal infinities
\cite{Cardoso-Vodev:Uniform, Vodev:Local}.
One implication of our methods is a generalization of these results.

Below $\dCI(X_0)$ denotes `Schwartz functions' on $X_0$, i.e.\ $\CI$ functions
vanishing with all derivatives at $\pa X_0$, and $\dist(X_0)$ is the dual space
of `tempered distributions' (these spaces are naturally identified for
$X_0$ and $X_{0,\even}$), while $H^s(X_{0,\even})$ is the standard Sobolev
space on $X_{0,\even}$ (corresponding
to extension across the boundary, see e.g.\ \cite[Appendix~B]{Hor}, where these
are denoted by $\bar H^s(X_{0,\even}^\circ)$)
and $H^s_h(X_{0,\even})$ is the standard
semiclassical Sobolev space, so for $h>0$ fixed this is
the same as $H^s(X_{0,\even})$;
see \cite{Dimassi-Sjostrand:Spectral, Evans-Zworski:Semiclassical}.

\begin{thm*}(See Theorem~\ref{thm:conf-compact-high} for the full statement.)
Suppose that $X_0$ is an $(n-1)$-dimensional
manifold with boundary $Y$ with
an even Riemannian conformally compact metric $g_0$. Then
the inverse of
$$
\Delta_{g_0}-\left(\frac{n-2}{2}\right)^2-\sigma^2,
$$
written as $\cR(\sigma):L^2\to L^2$,
has a meromorphic continuation from
$\im\sigma\gg0$ to $\Cx$,
$$
\cR(\sigma):\dCI(X_0)\to\dist(X_0),
$$
with poles with finite rank residues. If in addition
$(X_0,g_0)$ is non-trapping, then
non-trapping estimates hold in every strip $|\im\sigma|<C$, $|\re\sigma|\gg 0$:
for
$s>\frac{1}{2}+C$,
\begin{equation}\label{eq:intro-nontrap}
\|x^{-(n-2)/2+\imath\sigma} \cR(\sigma)f\|_{H^s_{|\sigma|^{-1}}(X_{0,\even})}
\leq \tilde C|\sigma|^{-1}\|x^{-(n+2)/2+\imath\sigma}f\|_{H^{s-1}_{|\sigma|^{-1}}(X_{0,\even})}.
\end{equation}
If $f$ has compact support in $X_0^\circ$,
the $s-1$ norm on $f$ can be replaced by the $s-2$ norm.
\end{thm*}

Further, as stated in Theorem~\ref{thm:conf-compact-high}, the resolvent
is {\em semiclassically outgoing} with a loss of $h^{-1}$, in the sense
of recent results of Datchev and Vasy \cite{Datchev-Vasy:Gluing-prop}
and \cite{Datchev-Vasy:Trapped}. This means that for mild trapping (where,
in a strip near the spectrum,
one has polynomially bounded resolvent for a compactly localized version of
the trapped model) one obtains resolvent bounds of the same kind as for the
above-mentioned trapped models, and lossless estimates microlocally
away from the trapping. In particular, one obtains logarithmic losses compared
to non-trapping on the spectrum for hyperbolic trapping in
the sense of \cite[Section~1.2]{Wunsch-Zworski:Resolvent}, and
polynomial losses in strips, since
for the compactly localized model this was
recently shown by Wunsch and Zworski \cite{Wunsch-Zworski:Resolvent}.

For conformally compact spaces, without using wave propagation as motivation,
our method is to change the smooth structure, replacing $x$ by $\mu=x^2$,
conjugate the operator by an appropriate weight as well as remove a vanishing
factor of $\mu$, and show that the new operator continues smoothly and
non-degenerately (in an appropriate sense) across $\mu=0$, i.e.\ $Y$,
to a (non-elliptic)
problem which we can analyze utilizing
by now almost standard tools of microlocal analysis. These steps are reflected in
the form of the estimate \eqref{eq:intro-nontrap}; $\mu$ shows up in the evenness,
conjugation due to the presence of $x^{-n/2+\imath\sigma}$, and the two halves
of the vanishing factor of $\mu$ being removed in $x^{\pm 1}$ on the left and
right hand sides. This approach is explained in full detail in the
more expository and self-contained follow-up article,
\cite{Vasy:Microlocal-AH}.

However, it is useful to think of a wave equation motivation --- then
$(n-1)$-dimensional hyperbolic
space shows up (essentially) as a model at infinity inside a backward light cone
from a fixed point $q_+$ at future infinity
on $n$-dimensional de Sitter space $\hM$, see \cite[Section~7]{Vasy:De-Sitter},
where this was
used to construct the Poisson operator.
More precisely, the light cone
is singular at $q_+$, so to desingularize it, consider $[\hM;\{q_+\}]$.
After a Mellin transform in the defining function of the front face;
the model continues smoothly
across the light cone $Y$ inside the front face of $[\hM;\{q_+\}]$. The inside of
the light cone corresponds to $(n-1)$-dimensional
hyperbolic space (after conjugation, etc.) while
the exterior is (essentially) $(n-1)$-dimensional
de Sitter space; $Y$ is the `boundary' separating them.
Here $Y$ should be thought of as the event horizon in black hole
terms (there is nothing more to event horizons in terms of local geometry!).

The resulting operator $P_\sigma$ has radial points at the
conormal bundle $N^*Y\setminus o$
of $Y$
in the sense of microlocal analysis, i.e.\ the Hamilton vector field is radial
at these points, i.e.\ is
a multiple of the generator of dilations of the fibers of the cotangent bundle
there. However, tools exist to deal with these, going back to Melrose's
geometric treatment of scattering theory on asymptotically Euclidean
spaces \cite{RBMSpec}. Note that $N^*Y\setminus o$ consists of two components,
$\Lambda_+$, resp.\ $\Lambda_-$, and in $S^*X=(T^*X\setminus o)/\RR^+$
the images, $L_+$, resp.\ $L_-$, of these
are sinks, resp.\ sources, for the Hamilton flow. At $L_\pm$ one has choices regarding
the direction one wants to propagate estimates (into or out of the radial points),
which directly correspond to working with strong or weak Sobolev spaces.
For the present problem, the relevant choice is propagating estimates {\em away from}
the radial points, thus working with the `good' Sobolev spaces (which can be
taken to have as positive order as one wishes; there is a minimum amount of
regularity imposed by our choice of propagation direction, cf.\ the requirement
$s>\frac{1}{2}+C$ above \eqref{eq:intro-nontrap}).
All other points are either elliptic, or real principal type.
It remains to either deal with the non-compactness of the `far end' of the
$(n-1)$-dimensional
de Sitter space --- or instead, as is indeed more convenient when one wants to
deal with more singular geometries, adding complex absorbing potentials,
in the spirit of works of Nonnenmacher and Zworski
\cite{Nonnenmacher-Zworski:Quantum} and Wunsch and Zworski
\cite{Wunsch-Zworski:Resolvent}. In fact, the complex absorption could be
replaced by adding a space-like boundary, see Remark~\ref{rem:add-bdy},
but for many microlocal purposes
complex absorption is more desirable, hence we follow the latter method.
However, crucially, these complex absorbing techniques (or the addition
of a space-like boundary) already enter
in the non-semiclassical problem in our case, as we are in a non-elliptic setting.

One can reverse the direction of the
argument and analyze the wave equation on an $(n-1)$-dimensional even
asymptotically
de Sitter space $X_0'$ by extending it across the boundary, much like the
the Riemannian conformally compact
space $X_0$ is extended in this approach. Then, performing microlocal propagation
in the opposite direction, which amounts to working with the adjoint operators
that we already need in order to prove existence of solutions for the Riemannian
spaces\footnote{This adjoint analysis also shows up for Minkowski space-time as
the `original' problem.},
we obtain existence, uniqueness and structure results for asymptotically
de Sitter spaces, recovering a large
part\footnote{Though not the parametrix construction for the Poisson operator, or
for the forward fundamental solution of Baskin \cite{Baskin:Parametrix};
for these we would need a parametrix construction in the present compact
boundaryless, but analytically non-trivial (for this purpose), setting.}
of the results of \cite{Vasy:De-Sitter}. Here we only briefly indicate this method
of analysis in Remark~\ref{rem:asymp-dS}.

In other words, we establish a Riemannian-Lorentzian duality, that will have
counterparts both in the pseudo-Riemannian setting of higher signature and in
higher rank symmetric spaces, though in the latter the analysis might become
more complicated. Note that asymptotically hyperbolic and de Sitter spaces
are not connected by a `complex rotation' (in the sense of an actual deformation);
they are smooth continuations of each other in the sense we just discussed.

To emphasize the simplicity of our method, we list all of the microlocal techniques
(which are relevant both in the classical and in the semiclassical setting)
that we use on a {\em compact manifold without boundary}; in all cases {\em only
microlocal Sobolev estimates} matter (not parametrices, etc.):
\begin{enumerate}
\item
Microlocal elliptic regularity.
\item
Real principal type propagation of singularities.
\item
{\em Rough} analysis at a Lagrangian invariant under the Hamilton flow
which roughly behaves like a collection of
radial points, though the internal structure does not matter, in the spirit
of \cite[Section~9]{RBMSpec}.
\item
Complex absorbing `potentials' in the spirit of
\cite{Nonnenmacher-Zworski:Quantum} and
\cite{Wunsch-Zworski:Resolvent}.
\end{enumerate}
These are almost `off the shelf' in terms of modern microlocal analysis, and thus
our approach, from a microlocal perspective, is quite simple. We use these
to show that on the continuation across the boundary of the conformally compact
space we have a Fredholm problem, on a perhaps slightly exotic function space,
which however is (perhaps apart from the complex absorption)
the simplest possible coisotropic function space based on
a Sobolev space, with order dictated by the radial points. Also, we propagate
the estimates along bicharacteristics in different directions depending
on the component $\Sigma_\pm$
of the characteristic set under consideration; correspondingly
the sign of the complex absorbing `potential' will vary with $\Sigma_\pm$, which
is perhaps slightly unusual. However, this is completely parallel to solving the
standard Cauchy, or forward, problem for the wave equation, where one propagates
estimates in {\em opposite} directions relative to the Hamilton vector field in the
two components.

The complex absorption we use modifies the operator $P_\sigma$ outside
$X_{0,\even}$. However, while $(P_\sigma-\imath Q_\sigma)^{-1}$ depends on $Q_\sigma$,
its behavior on $X_{0,\even}$, and even near $X_{0,\even}$, is independent of this
choice; see the proof of Proposition~\ref{prop:hyp-resolvent-construct}
for a detailed explanation. In particular, although $(P_\sigma-\imath Q_\sigma)^{-1}$
may have resonances
other than those of $\cR(\sigma)$, the resonant states of
these additional resonances are supported
outside $X_{0,\even}$, hence do not affect the singular behavior of the resolvent in
$X_{0,\even}$. In the setting of Kerr-de Sitter space an analogous
role is played by semiclassical versions of the standard energy
estimate; this is stated in Subsection~\ref{subsec:local-wave}.

While the results are stated for the scalar equation, analogous results hold
for operators on natural vector bundles, such as the Laplacian on differential
forms. This is so because the results work if the principal symbol of the extended
problem is scalar
with the demanded properties, and the imaginary part
of the subprincipal symbol is either scalar at the
`radial sets', or instead satisfies appropriate estimates (as an endomorphism of
the pull-back of the vector bundle to the cotangent bundle) at this location;
see Remark~\ref{rem:bundles}.
The only change in terms of results on asymptotically hyperbolic spaces
is that the threshold $(n-2)^2/4$ is shifted; in
terms of the explicit conjugation of Subsection~\ref{subsec:conf-comp-results}
this is so because of the change in the first order term in \eqref{eq:conf-comp-Lap-form}.

While here we mostly
consider conformally compact Riemannian or Lorentzian spaces
(such as hyperbolic space and de Sitter space)
as appropriate boundary values (Mellin transform)
of a blow-up of de Sitter space of one higher
dimension, they also show up as a boundary value of Minkowski space. This
is related to Wang's work on b-regularity \cite{Wang:Thesis}, though Wang
worked on a blown up version of Minkowski space-time; she also obtained
her results for the (non-linear) Einstein equation there. It is also related
to the work of Fefferman and Graham \cite{Fefferman-Graham:Conformal} on
conformal invariants by extending an asymptotically hyperbolic manifold to
Minkowski-type spaces of one higher dimension.
We discuss asymptotically Minkowski spaces briefly
in Section~\ref{sec:Minkowski}.

Apart from trapping --- which is well away from the event horizons for
black holes that do not rotate too fast --- the microlocal
structure on de Sitter space is {\em exactly} the same as on Kerr-de Sitter space,
or indeed Kerr space
near the event horizon. (Kerr space has a Minkowski-type end as well; although
Minkowski space also fits into our framework, it does so a different way than Kerr
at the event horizon, so the result there is not immediate; see the comments below.)
This is to be understood as follows: from the perspective we present here
(as opposed to the perspective of \cite{Vasy:De-Sitter}), the tools that go into
the analysis of de Sitter space-time suffice also for Kerr-de Sitter space, and indeed
a much wider class, apart from the need to deal with trapping.
The trapping itself was analyzed
by Wunsch and Zworski \cite{Wunsch-Zworski:Resolvent}; their work
fits immediately with our microlocal methods. Phenomena such as the ergosphere
are mere shadows of dynamics in the phase space which is barely changed, but
whose projection to the base space (physical space) undergoes serious changes.
It is thus of great value to work microlocally, although it is certainly possible
that for some non-linear purposes it is convenient to rely on physical space
to the maximum possible extent, as was done in the recent (linear) works of
Dafermos and Rodnianski \cite{Dafermos-Rodnianski:Black,
Dafermos-Rodnianski:Axi}.

Below we state theorems for Kerr-de Sitter space time. However, it is important
to note that all of these theorems have analogues in the general microlocal
framework discussed in Section~\ref{sec:microlocal}. In particular, analogous
theorems hold on conjugated, re-weighted, and even versions of Laplacians
on conformally compact spaces (of which one example was stated above as
a theorem), and similar results apply on `asymptotically Minkowski' spaces,
with the slight twist that it is adjoints of operators considered here that
play the direct role there.

We now turn to Kerr-de Sitter space-time and give some history.
In exact Kerr-de Sitter space and for small angular momentum,
Dyatlov \cite{Dyatlov:Quasi-normal, Dyatlov:Exponential}
has shown exponential decay to constants, even across the event horizon.
This followed earlier work of
Melrose, S\'a Barreto and Vasy \cite{Melrose-SaBarreto-Vasy:Asymptotics}, where
this was shown up to the event horizon in de Sitter-Schwarzschild space-times
or spaces strongly asymptotic to these (in particular, no rotation of the black hole
is allowed), and
of Dafermos and
Rodnianski in \cite{Dafermos-Rodnianski:Sch-dS} who had shown polynomial
decay in this setting.
These in turn followed up pioneering work of S\'a Barreto and Zworski
\cite{Sa-Barreto-Zworski:Distribution} and Bony and H\"afner \cite{Bony-Haefner:Decay}
who studied resonances and decay away from the event horizon in these settings.
(One can solve the wave equation explicitly on de Sitter space
using special functions, see \cite{Polarski:Hawking} and \cite{Yagdjian-Galstian:De-Sitter};
on asymptotically de Sitter spaces
the forward fundamental solution was constructed as an appropriate Lagrangian
distribution by Baskin \cite{Baskin:Parametrix}.) 

Also, polynomial decay on Kerr space
was shown recently by Tataru and Tohaneanu
\cite{Tataru-Tohaneanu:Local, Tataru:Local}
and Dafermos and Rodnianski \cite{Dafermos-Rodnianski:Black,
Dafermos-Rodnianski:Axi}, after pioneering work of Kay and Wald
in \cite{Kay-Wald:Linear}
and \cite{Wald:Stability} in the Schwarzschild setting.
(There was also recent work by Marzuola, Metcalf, Tataru and
Tohaneanu \cite{Marzuola-Metcalf-Tataru-Tohaneanu:Strichartz}
on Strichartz estimates, and by Donninger, Schlag and Soffer
\cite{Donninger-Schlag-Soffer:Price} on $L^\infty$
estimates on Schwarzschild black holes,
following $L^\infty$ estimates of Dafermos and Rodnianski
\cite{Dafermos-Rodnianski:Red-shift, Dafermos-Rodnianski:Price},
of Blue and Soffer \cite{Blue-Soffer:Phase}
on non-rotating charged black holes giving
$L^6$ estimates, and Finster, Kamran, Smoller and
Yau \cite{Finster-Kamran-Smoller-Yau:Decay,
Finster-Kamran-Smoller-Yau:Linear} on Dirac waves on Kerr.)
While some of these papers employ microlocal methods
at the trapped set, they are mostly based on physical space where the phenomena
are less clear than in phase space (unstable tools, such as
separation of variables, are often used in phase space though).
We remark that Kerr space is less amenable to
immediate
microlocal analysis to attack the decay of solutions of the wave equation due
to the singular/degenerate behavior at zero frequency, which will be explained below
briefly. This is closely related to the behavior of solutions of the wave equation
on Minkowski space-times. Although our methods also deal with Minkowski
space-times, this holds in a slightly different way than for de Sitter (or Kerr-de Sitter)
type spaces at infinity, and combining the two
ingredients requires some additional work.
On perturbations of Minkowski space itself, the full non-linear analysis
was done in the path-breaking work of Christodoulou and Klainerman
\cite{Christodoulou-Klainerman:Global}, and Lindblad and Rodnianski
simplified the analysis \cite{Lindblad-Rodnianski:Global-existence,
Lindblad-Rodnianski:Global-Stability},
Bieri \cite{Bieri:Extensions, Bieri-Zipser:Extensions}
succeeded in relaxing the
decay conditions, while Wang \cite{Wang:Thesis}
obtained additional, b-type, regularity as already mentioned. Here we only
give a linear result, but hopefully its simplicity will also shed new light
on the non-linear problem.

As already mentioned,
a microlocal study of the trapping in Kerr or Kerr-de Sitter was performed by
Wunsch and Zworski in \cite{Wunsch-Zworski:Resolvent}. This is particularly
important to us, as this is the only part of the phase space which does not
fit directly into a relatively simple microlocal framework. Our general method
is to use microlocal analysis to understand the rest of the phase space (with
localization away from trapping realized via a complex absorbing potential),
then use the gluing result of Datchev and Vasy \cite{Datchev-Vasy:Gluing-prop}
to obtain the full result.

Slightly more concretely, in the appropriate (partial) compactification of space-time,
near the boundary of which space-time has the form $X_\delta\times[0,\tau_0)_\tau$,
where $X_\delta$ denotes an extension of the space-time across the event horizon.
Thus, there is a manifold with boundary $X_0$, whose boundary $Y$
is the event horizon,
such that $X_0$ is embedded into $X_\delta$, a (non-compact)
manifold without boundary. We write
$X_+=X_0^\circ$ for `our side' of the event horizon and $X_-=X_\delta\setminus X_0$
for the `far side'. Then
the Kerr or Kerr-de Sitter d'Alembertians are b-operators in the sense
of Melrose \cite{Melrose:Atiyah} that extend smoothly
across the event horizon $Y$. Recall that in the Riemannian setting, b-operators
are usually called `cylindrical ends', see \cite{Melrose:Atiyah} for a general
description; here the form at the boundary
(i.e.\ `infinity') is similar, modulo ellipticity (which is lost). Our results
hold for small smooth perturbations of Kerr-de Sitter space in this b-sense. Here the
role of `perturbations' is simply to ensure that the microlocal picture, in particular
the dynamics, has not changed drastically. Although b-analysis is the
right conceptual framework, we mostly work with the Mellin transform, hence
on manifolds without boundary, so the reader need not be concerned
about the lack of familiarity with b-methods. However, we briefly discuss
the basics in Section~\ref{sec:Mellin-Lorentz}.

We {\em immediately} Mellin transform in the defining function
of the boundary (which is temporal infinity, though is not space-like everywhere)
--- in Kerr and Kerr-de Sitter spaces this is operation is `exact', corresponding to
$\tau\pa_\tau$ being a Killing vector field, i.e.\ is not merely
at the level of normal operators, but this makes little difference (i.e. the
general case is similarly treatable). After this transform we get a family of operators
that e.g.\ in de Sitter space is elliptic on $X_+$, but in Kerr space ellipticity
is lost there. We consider the event horizon as a completely artificial boundary
even in the de Sitter setting, i.e.\ work on a manifold that includes a neighborhood
of $X_0=\overline{X_+}$, hence a neighborhood of the
event horizon $Y$.

As already mentioned, one feature of these space-times is some
relatively mild trapping in $X_+$;
this only plays a role in high energy (in the Mellin parameter, $\sigma$),
or equivalently semiclassical (in $h=|\sigma|^{-1}$)
estimates. We ignore a (semiclassical) microlocal
neighborhood of the trapping for a moment; we place an absorbing
`potential' there. Another important feature of the space-times is that
they are not naturally compact on the `far side' of the event horizon
(inside the black hole), i.e.\ $X_-$,
and bicharacteristics from the event horizon
(classical or semiclassical) propagate into this region. However,
we place an absorbing `potential' (a second order operator) there to annihilate
such phenomena which do not affect what happens on `our side' of the event
horizon, $X_+$, in view of the characteristic nature of the latter. This absorbing
`potential' could
{\em easily} be replaced by a space-like boundary, in the spirit of introducing
a boundary $t=t_1$, where $t_1>t_0$,
when one solves the Cauchy problem from $t_0$
for the standard wave equation; note that such a boundary does not affect
the solution of the equation in $[t_0,t_1]_t$. Alternatively, if $X_-$
has a well-behaved infinity, such as in de Sitter space, the analysis could be
carried out more globally. However, as we wish to emphasize the microlocal
simplicity of the problem, we do not touch on these issues.

All of our results are in a general setting of microlocal analysis explained
in Section~\ref{sec:microlocal}, with the Mellin transform and
Lorentzian connection explained in Section~\ref{sec:Mellin-Lorentz}. However,
for the convenience of the reader here we state the results for perturbations
of Kerr-de Sitter spaces.
We refer to Section~\ref{sec:Kerr} for
details.
First, the general assumption is that
\begin{quote}
$P_\sigma$, $\sigma\in\Cx$,
is either the Mellin transform of the d'Alembertian $\Box_g$
for a Kerr-de Sitter spacetime, or more generally the Mellin transform of the
normal operator of the
d'Alembertian $\Box_g$
for a small perturbation, in the sense of b-metrics,
of such a Kerr-de Sitter space-time;
\end{quote}
see Section~\ref{sec:Mellin-Lorentz} for
an explanation of these concepts. Note that for such perturbations
the usual `time' Killing vector field (denoted by $\pa_{\tilde t}$ in
Section~\ref{sec:Kerr}; this is indeed time-like
in $X_+\times[0,\ep)_{\tilde t}$
sufficiently far from $\pa X_+$) is no longer Killing. Our results on these space-times
are proved by showing that the hypotheses of Section~\ref{sec:microlocal}
are satisfied. We show this in general (under the conditions \eqref{eq:digamma-def},
which corresponds to $0<\frac{9}{4}\Lambda r_s^2<1$ in de Sitter-Schwarzschild
spaces, and \eqref{eq:alpha-restrict}, which corresponds to the lack of classical
trapping in $X_+$;
see Section~\ref{sec:Kerr}), except where semiclassical dynamics matters.
As in the analysis of Riemannian conformally compact spaces, we use
a complex absorbing operator $Q_\sigma$; this means that its principal symbol
in the relevant (classical, or semiclassical) sense has the correct sign on
the characteristic set; see Section~\ref{sec:microlocal}.

When semiclassical dynamics does matter, the {\em non-trapping assumption} with
an absorbing operator $Q_\sigma$, $\sigma=h^{-1}z$, is
\begin{quote}
in both the forward and backward directions,
the bicharacteristics from any point in the semiclassical characteristic set
of $P_\sigma$
either enter the semiclassical elliptic set of $Q_\sigma$
at some finite
time, or tend to $L_\pm$;
\end{quote}
see Definition~\ref{Def:non-trap}. Here, as in the discussion above, $L_\pm$
are two components of the image of $N^*Y\setminus o$ in $S^*X$.
(As $L_+$ is a sink while $L_-$ is a source, even
semiclassically, outside $L_\pm$ the `tending' can only happen in the forward,
resp.\ backward, directions.) Note that the semiclassical non-trapping assumption
(in the precise sense used below) implies a classical non-trapping assumption, i.e.\ the
analogous statement for classical bicharacteristics, i.e.\ those in $S^*X$.
It is important to keep in mind that the classical non-trapping assumption can
always be satisfied with $Q_\sigma$ supported in $X_-$, far from $Y$.

In our first result in the Kerr-de Sitter type setting,
to keep things simple, we ignore semiclassical
trapping via the use of $Q_\sigma$; this means that $Q_\sigma$ will have support
in $X_+$. However, in $X_+$, $Q_\sigma$
only
matters in the semiclassical, or high energy, regime, and only for (almost) real
$\sigma$. If the black hole is rotating relatively slowly, e.g.\
$\alpha$ satisfies the bound \eqref{eq:semicl-alpha-limit},
the (semiclassical)
trapping is always far from the event horizon, and one can make $Q_\sigma$
supported away from there. Also, the Klein-Gordon
parameter $\lambda$ below is `free' in the sense that it does not affect any of
the relevant information in the analysis\footnote{It does affect the
  {\em location} of the poles and corresponding resonant states
of $(P_\sigma-\imath Q_\sigma)^{-1}$,
  hence the constant in Theorem~\ref{thm:exp-decay} has to be replaced by the
  appropriate resonant state and exponential growth/decay, as in the
  second part of that theorem.} (principal and subprincipal symbol; see
below). {\em Thus, we drop it in the following theorems for simplicity.}

\begin{thm}\label{thm:complete-absorb}
Let $Q_\sigma$ be an absorbing
formally self-adjoint operator such that the semiclassical non-trapping assumption
holds. Let $\sigma_0\in\Cx$, and
\begin{equation*}\begin{split}
&\cX^s=\{u\in H^s:\ (P_{\sigma_0}-\imath Q_{\sigma_0})u\in H^{s-1}\},\ \cY^s=H^{s-1},\\
&\qquad\|u\|_{\cX^s}^2=\|u\|_{H^s}^2+\|(P_{\sigma_0}-\imath Q_{\sigma_0})u\|^2_{H^{s-1}}.
\end{split}\end{equation*}
Let $\beta_\pm>0$ be given by the geometry at conormal bundle of the
black hole ($-$), resp.\ de Sitter ($+$) event horizons,
see Subsection~\ref{subsec:Kerr-geo}, and in particular \eqref{eq:subpr-Kerr}.
For $s\in\RR$, let\footnote{This means that we require the stronger
of $\im\sigma>\beta_\pm^{-1}(1-2s)$ to hold in \eqref{eq:half-space-intro}.
If we perturb Kerr-de Sitter space time, we need to increase the requirement
on $\im\sigma$ slightly, i.e.\ the size of the half space has to be slightly
reduced.}
$\beta=\max(\beta_+,\beta_-)$ if $s\geq 1/2$,
$\beta=\min(\beta_+,\beta_-)$ if $s<1/2$.
Then, for $\lambda\in\Cx$,
$$
P_\sigma-\imath Q_\sigma-\lambda:\cX^s\to\cY^s
$$
is an analytic family of Fredholm operators on
\begin{equation}\label{eq:half-space-intro}
\Cx_s=\{\sigma\in\Cx:\ \im\sigma>\beta^{-1}(1-2s)\}
\end{equation}
and
has a meromorphic inverse,
$$
R(\sigma)=(P_\sigma-\imath Q_\sigma-\lambda)^{-1},
$$
which is holomorphic in an upper half plane, $\im\sigma>C$.
Moreover,
given any $C'>0$, there are only finitely many poles in $\im\sigma>-C'$, and
the resolvent satisfies non-trapping estimates there, which e.g.\ with $s=1$ (which
might need a reduction in $C'>0$) take the form
$$
\|R(\sigma)f\|_{L^2}^2+|\sigma|^{-2}\|dR(\sigma)\|_{L^2}^2
\leq C''|\sigma|^{-2}\|f\|_{L^2}^2.
$$
\end{thm}

The analogous result also holds on Kerr space-time
if we suppress the Euclidean end by a complex absorption.

Dropping the semiclassical absorption in $X_+$,
i.e.\ if
we make $Q_\sigma$ supported only in $X_-$,
we have\footnote{Since we are not making a statement for almost real $\sigma$,
semiclassical trapping, discussed in the previous paragraph, does not matter.}

\begin{thm}\label{thm:spatial-absorb}
Let $P_\sigma$, $\beta$, $\Cx_s$ be as in Theorem~\ref{thm:complete-absorb},
and let $Q_\sigma$ be an absorbing
formally self-adjoint operator supported in $X_-$ which is classically non-trapping.
Let $\sigma_0\in\Cx$, and
$$
\cX^s=\{u\in H^s:\ (P_{\sigma_0}-\imath Q_{\sigma_0})u\in H^{s-1}\},\ \cY^s=H^{s-1},
$$
with
$$
\|u\|_{\cX^s}^2=\|u\|_{H^s}^2+\|\tilde Pu\|^2_{H^{s-1}}.
$$
Then,
$$
P_\sigma-\imath Q_\sigma:\cX^s\to\cY^s
$$
is an analytic family of Fredholm operators on $\Cx_s$, and
has a meromorphic inverse,
$$
R(\sigma)=(P_\sigma-\imath Q_\sigma)^{-1},
$$
which for any $\ep>0$ is holomorphic in a translated
sector in the upper half plane, $\im\sigma>C+\ep |\re \sigma|$.
The poles of the resolvent
are called {\em resonances}.
In addition, taking $s=1$ for instance, $R(\sigma)$ satisfies
non-trapping estimates, e.g.\ with $s=1$,
$$
\|R(\sigma)f\|_{L^2}^2+|\sigma|^{-2}\|dR(\sigma)\|_{L^2}^2
\leq C'|\sigma|^{-2}\|f\|_{L^2}^2
$$
in such a translated sector.
\end{thm}

It is in this setting that $Q_\sigma$ could be replaced by working on a manifold
with boundary, with the boundary being space-like, essentially as a time level set
mentioned above, since it is supported in $X_-$.

Now we make the assumption that {\em the only semiclassical
trapping is due to hyperbolic
trapping with trapped set $\Gamma_z$, $\sigma=h^{-1}z$}, with
hyperbolicity
understood as in the `Dynamical Hypotheses' part of \cite[Section~1.2]{Wunsch-Zworski:Resolvent}, i.e.
\begin{quote}
in both the forward and backward directions,
the bicharacteristics from any point in the semiclassical characteristic set
of $P_\sigma$
either enter the semiclassical elliptic set of $Q_\sigma$
at some finite
time, or tend to $L_\pm\cup \Gamma_z$.
\end{quote}
We remark that just hyperbolicity of the trapped set suffices for the
results of \cite{Wunsch-Zworski:Resolvent}, see Section~1.2 of that
paper; however, if one wants stability of the results under
perturbations, one needs to assume that $\Gamma_z$ is {\em normally
hyperbolic}. We refer to \cite[Section~1.2]{Wunsch-Zworski:Resolvent}
for a discussion of these concepts.
We show in Section~\ref{sec:Kerr} that for 
black holes satisfying \eqref{eq:semicl-alpha-limit}
(so the angular momentum can be comparable to the mass)
the operators $Q_\sigma$ can be chosen so that they
are supported in $X_-$ (even quite far from $Y$)
and the hyperbolicity requirement is satisfied. Further, we also show
that for slowly
rotating black holes the trapping is normally hyperbolic.
Moreover, the (normally) hyperbolic trapping statement is
purely in Hamiltonian dynamics, not regarding PDEs. It might be known
for an even
larger range of rotation speeds, but the author is not aware of this.

Under this assumption, one can
combine Theorem~\ref{thm:complete-absorb} with
the results of Wunsch and Zworski
\cite{Wunsch-Zworski:Resolvent}
about hyperbolic trapping and the gluing results of Datchev and
Vasy \cite{Datchev-Vasy:Gluing-prop} to obtain a better result for
the merely spatially localized problem, Theorem~\ref{thm:spatial-absorb}:

\begin{thm}\label{thm:glued}
Let $P_\sigma$, $Q_\sigma$, $\beta$, $\Cx_s$, $\cX^s$ and $\cY^s$
be as in Theorem~\ref{thm:spatial-absorb}, and assume that
the only semiclassical trapping is due to hyperbolic
trapping.
Then,
$$
P_\sigma-\imath Q_\sigma:\cX^s\to\cY^s
$$
is an analytic family of Fredholm operators on $\Cx_s$, and
has a meromorphic inverse,
$$
R(\sigma)=(P_\sigma-\imath Q_\sigma)^{-1},
$$
which is holomorphic in an upper half plane, $\im\sigma>C$. Moreover, there
exists $C'>0$ such that there are only finitely many poles in $\im\sigma>-C'$,
and
the resolvent satisfies polynomial estimates there as $|\sigma|\to\infty$,
$|\sigma|^\varkappa$, for some $\varkappa>0$, compared to
the non-trapping case,
with merely a logarithmic
loss compared to non-trapping for real $\sigma$, e.g.\ with $s=1$:
$$
\|R(\sigma)f\|_{L^2}^2+|\sigma|^{-2}\|dR(\sigma)\|_{L^2}^2
\leq C''|\sigma|^{-2}(\log|\sigma|)^2\|f\|_{L^2}^2.
$$
\end{thm}

Farther, there are approximate lattices of poles generated by the
trapping, as studied by S\'a Barreto and Zworski in
\cite{Sa-Barreto-Zworski:Distribution}, and further
by Bony and H\"afner in \cite{Bony-Haefner:Decay}, in the exact De Sitter-Schwarzschild
and Schwarzschild settings, and in ongoing work by Dyatlov in the
exact Kerr-de Sitter setting.

Theorem~\ref{thm:glued} immediately and directly gives the asymptotic
behavior of solutions of the wave equation across the event horizon. Namely,
the asymptotics of the wave equation depends on the finite number of resonances;
their precise behavior depends on specifics of the space-time, i.e.\ on these
resonances. This is true even in arbitrarily regular b-Sobolev spaces -- in fact,
the more decay we want
to show, the higher Sobolev spaces we need to work in. Thus, a forteriori,
this gives $L^\infty$ estimates. We state this formally as a theorem
in the simplest case of slow rotation; in the general case one needs
to analyze the (finite!) set of resonances along the reals to obtain
such a conclusion, and for the perturbation part also to show normal
hyperbolicity (which we only show for slow rotation):

\begin{thm}\label{thm:exp-decay}
Let $M_\delta$ be the partial compactification of Kerr-de Sitter space
as in Section~\ref{sec:Kerr}, with
$\tau$ the boundary defining function.
Suppose that $g$ is either a slowly rotating Kerr-de Sitter metric,
or a small perturbation as a symmetric bilinear form on
$\Tb M_\delta$. Then there exist
$C'>0$, $\varkappa>0$
such that for $0<\ep<C'$ and $s>(1+\beta \ep)/2$
solutions of $\Box_gu=f$
with $f\in \tau^\ep H^{s-1+\varkappa}_{\bl}(M_\delta)$ vanishing in $\tau>\tau_0$, and
with $u$ vanishing in $\tau>\tau_0$,
satisfy that for some constant $c_0$,
$$
u-c_0\in \tau^\ep H^s_{\bl,\loc}(M_\delta).
$$
Here the norms of both $c_0$ in $\Cx$ and $u-c_0$ in $\tau^\ep
H^s_{\bl,\loc}(M_{\delta'})$ are bounded by that of $f$ in $\tau^\ep
H^{s-1+\varkappa}_{\bl}(M_\delta)$ for $\delta'<\delta$.

More generally, if $g$ is a Kerr-de Sitter metric with hyperbolic
trapping\footnote{This is shown in Section~\ref{sec:Kerr} when
  \eqref{eq:semicl-alpha-limit} is satisfied.}, then there exist $C'>0$, $\varkappa>0$ such that
for $0<\ell<C'$ and $s>(1+\beta \ell)/2$
solutions of $\Box_gu=f$
with $f\in \tau^\ell H^{s-1+\varkappa}_{\bl}(M_\delta)$ vanishing in $\tau>\tau_0$, and
with $u$ vanishing in $\tau>\tau_0$,
satisfy that for some $a_{j\kappa}\in\CI(X_\delta)$ (which are
resonant states) and $\sigma_j\in\Cx$ (which are the resonances),
\begin{equation*}
u'=u-\sum_j\sum_{\kappa\leq m_j} \tau^{\imath\sigma_j}(\log
|\tau|)^\kappa a_{j\kappa}
\in \tau^\ell H^s_{\bl,\loc}(M_\delta).
\end{equation*}
Here the (semi)norms of both $a_{j\kappa}$ in $\CI(X_{\delta'})$ and $u'$ in $\tau^\ell
H^s_{\bl,\loc}(M_{\delta'})$ are bounded by that of $f$ in $\tau^\ell
H^{s-1+\varkappa}_{\bl}(M_\delta)$ for $\delta'<\delta$.
The same conclusion holds for sufficiently small perturbations of the
metric as a symmetric bilinear form on
$\Tb M_\delta$ provided the trapping is normally hyperbolic.
\end{thm}

In special
geometries (without the ability to add perturbations)
such decay has been described by delicate separation of variables
techniques, again
see Bony-H\"afner \cite{Bony-Haefner:Decay} in the De Sitter-Schwarzschild
and Schwarzschild settings, but only away from the event horizons,
and by Dyatlov \cite{Dyatlov:Quasi-normal, Dyatlov:Exponential} in the
Kerr-de Sitter setting. Thus, in these settings, we recover in a direct manner
Dyatlov's result across the event horizon \cite{Dyatlov:Exponential}, modulo
a knowledge of resonances near the origin contained in \cite{Dyatlov:Quasi-normal}.
In fact, for small angular momenta one can use the results from de Sitter-Schwarzschild
space directly to describe these finitely many resonances,
as exposed in the works of S\'a Barreto and Zworski
\cite{Sa-Barreto-Zworski:Distribution}, Bony and H\"afner \cite{Bony-Haefner:Decay}
and Melrose, S\'a Barreto and Vasy \cite{Melrose-SaBarreto-Vasy:Asymptotics},
since $0$ is an isolated resonance with multiplicity $1$ and eigenfunction $1$;
this persists under small deformations, i.e.\ for small angular momenta. Thus,
exponential decay to constants, Theorem~\ref{thm:exp-decay}, follows immediately.

One can also work with Kerr space-time,
apart from issues of analytic continuation. By using
weighted spaces and Melrose's results from \cite{RBMSpec} as well
as those of Vasy and Zworski in the semiclassical setting
\cite{Vasy-Zworski:Semiclassical}, one easily gets
an analogue of Theorem~\ref{thm:spatial-absorb} in $\im\sigma>0$, with
smoothness and the almost non-trapping estimates corresponding to
those of Wunsch and Zworski \cite{Wunsch-Zworski:Resolvent}
down to $\im\sigma=0$ for $|\re\sigma|$ large.
Since a proper treatment of this
would exceed the bounds of this paper, we refrain from this here. Unfortunately,
even if this analysis were carried out, low energy problems would still remain,
so the result is not strong enough to deduce the wave expansion. As already alluded
to, Kerr space-time
has features of both Minkowski and de Sitter space-times; though both
of these fit into our framework, they do so in different ways,
so a better way of dealing with the Kerr space-time, namely adapting our
methods to it, requires additional work.

While de Sitter-Schwarzschild space (the special case of Kerr-de Sitter space with
vanishing rotation), via the same methods as those on de Sitter space which
give rise to the hyperbolic Laplacian and its continuation across infinity, gives
rise essentially to the Laplacian of a conformally compact metric, with
similar structure but different curvature at the two ends (this was used by
Melrose, S\'a Barreto and Vasy \cite{Melrose-SaBarreto-Vasy:Asymptotics} to
do analysis up to the event horizon there), the analogous problem for Kerr-de Sitter
is of edge-type in the sense of Mazzeo's edge calculus \cite{Mazzeo:Edge}
apart from a degeneracy at the poles corresponding to the axis of rotation, though
it is not Riemannian. Note that edge operators have global properties in the fibers;
in this case these fibers are the orbits of rotation.
A reasonable interpretation of the appearance of this class of operators
is that the global properties in the fibers capture non-constant (or non-radial)
bicharacteristics (in the
classical sense) in the conormal bundle
of the event horizon, and also possibly the (classical) bicharacteristics entering $X_+$.
This
suggests that the methods of
Melrose, S\'a Barreto and Vasy \cite{Melrose-SaBarreto-Vasy:Asymptotics}
would be much harder to apply in the presence of rotation.

It is important to point out that the results of this paper are stable under small
$\CI$ perturbations\footnote{Certain kinds of perturbations conormal to the boundary,
in particular polyhomogeneous ones, would only change the analysis and the conclusions slightly.} of the Lorentzian metric on the b-cotangent bundle at
the cost of changing the function spaces slightly; this follows from the
estimates being stable in these circumstances. Note that the function spaces
depend on the principal symbol of the operator under consideration, and
the range of $\sigma$ depends on the subprincipal symbol at the conormal
bundle of the event horizon; under general small smooth perturbations, defining
the spaces exactly as before, the results remain valid if the range of $\sigma$
is slightly restricted.

In addition, the method is stable under gluing: already Kerr-de Sitter space behaves
as two separate black holes (the Kerr and the de Sitter end),
connected by semiclassical dynamics; since only one component
(say $\Sigma_{\semi,+}$)
of
the semiclassical characteristic set moves far into $X_+$,
one can easily add as many Kerr black holes as one wishes by gluing beyond the reach
of the other component, $\Sigma_{\semi,-}$.
Theorems~\ref{thm:complete-absorb} and
\ref{thm:spatial-absorb} automatically remain valid (for the semiclassical
characteristic set is then irrelevant), while Theorem~\ref{thm:glued} remains
valid provided that the resulting dynamics only exhibits mild trapping (so that
compactly localized models have at most polynomial resolvent growth),
such as normal hyperbolicity, found in Kerr-de Sitter space.

Since the specifics of Kerr-de Sitter space-time are, as already mentioned, irrelevant
in the microlocal approach we take, we start with
the abstract microlocal discussion in Section~\ref{sec:microlocal},
which is translated into the setting of the wave equation on manifolds with
a Lorentzian b-metric in Section~\ref{sec:Mellin-Lorentz},
followed by
the description of de Sitter, Minkowski
and Kerr-de Sitter space-times in Sections~\ref{sec:dS}, \ref{sec:Minkowski}
and \ref{sec:Kerr}. Theorems~\ref{thm:complete-absorb}-\ref{thm:exp-decay} are
proved in Section~\ref{sec:Kerr} by showing that they fit into
the abstract framework of Section~\ref{sec:microlocal}; the approach is completely
analogous to de Sitter and Minkowski spaces, where the fact that they
fit into the abstract framework is shown in Sections~\ref{sec:dS}
and \ref{sec:Minkowski}.
As another option, we encourage the reader to
read the discussion of de Sitter space first, which also includes the discussion
of conformally compact spaces, presented in Section~\ref{sec:dS}, as well
as Minkowski space-time presented in the section afterwards,
to gain some geometric insight, then the general microlocal machinery, and
finally the Kerr discussion to see how that space-time fits into our setting.
Finally, if the reader is interested how conformally compact metrics fit into
the framework and wants to jump to the relevant
calculation, a reasonable place to start
is Subsection~\ref{subsec:conf-comp-results}. We emphasize
that  for the conformally compact results, only
Section~\ref{sec:microlocal} and
Section~\ref{subsec:dS-Mellin}-\ref{subsec:conf-comp-results},
starting with the paragraph of \eqref{eq:p-sig-symbol-dS}, are
strictly needed.

\section{Microlocal framework}\label{sec:microlocal}
We now develop a setting which includes the geometry of the
`spatial' model of de Sitter space near its `event horizon',
as well as the model of Kerr and Kerr-de Sitter settings near
the event horizon, and the model at infinity for Minkowski space-time near the
light cone (corresponding to the adjoint of the problem described below in the last
case). As a general reference for microlocal analysis, we refer to
\cite{Hor}, while for semiclassical analysis, we refer to
\cite{Dimassi-Sjostrand:Spectral, Evans-Zworski:Semiclassical}; see also
\cite{Shubin:Pseudo} for the high-energy (or large parameter) point of view.

\subsection{Notation}\label{subsec:notation}
We recall the basic conversion between these frameworks. First,
$S^\difford(\RR^p;\RR^\ell)$ is
the set of $\CI$ functions on $\RR^p_z\times\RR^\ell_\zeta$
satisfying uniform
bounds
$$
|D_z^\alpha D_\zeta^\beta a|\leq C_{\alpha\beta}\langle \zeta\rangle^{\difford-|\beta|},
\ \alpha\in\Nat^p,\ \beta\in\Nat^\ell.
$$
If $O\subset\RR^p$ and $\Gamma\subset\RR^\ell_\zeta$ are open,
we define $S^\difford(O;\Gamma)$ by requiring\footnote{Another possibility would
be to require uniform estimates on compact subsets; this makes no difference
here.} these estimates to hold
only for $z\in O$ and $\zeta\in\Gamma$.
The class of classical (or one-step polyhomogeneous) symbols is
the subset $S_{\cl}^\difford(\RR^p;\RR^\ell)$ of $S^\difford(\RR^p;\RR^\ell)$
consisting of symbols possessing an asymptotic expansion
$$
a(z,r\omega)\sim \sum a_j(z,\omega) r^{\difford-j},
$$
where $a_j\in\CI(\RR^p\times\sphere^{\ell-1})$.
Then on $\RR^n_z$, pseudodifferential
operators $A\in\Psi^{\difford}(\RR^n)$ are of the form
\begin{equation*}\begin{split}
A=\Op(a);\ &\Op(a)u(z)
=(2\pi)^{-n}\int_{\RR^n} e^{i(z-z')\cdot\zeta} a(z,\zeta)\,u(z')\,d\zeta\,dz',\\
&\qquad u\in\cS(\RR^n),\ a\in S^\difford(\RR^n;\RR^n);
\end{split}\end{equation*}
understood as an oscillatory integral. Classical pseudodifferential operators,
$A\in\Psi_{\cl}^\difford(\RR^n)$, form the subset where $a$ is a classical symbol.
The principal symbol $\sigma_{\difford}(A)$
of $A\in\Psi^\difford(\RR^n)$ is the equivalence class
$[a]$ of $a$ in $S^\difford(\RR^n;\RR^n)/S^{\difford-1}(\RR^n;\RR^n)$.
For classical $a$, one can instead consider $a_0(z,\omega) r^\difford$ as the
principal symbol; it is a $\CI$ function on $\RR^n\times(\RR^n\setminus \{0\})$,
which is homogeneous of degree $\difford$ with respect to the $\RR^+$-action
given by dilations in the second factor, $\RR^n\setminus \{0\}$.

Differential operators on $\RR^n$ form the subset of $\Psi(\RR^n)$
in which $a$ is polynomial in the second factor, $\RR^n_\zeta$, so locally
$$
A=\sum_{|\alpha|\leq\difford} a_\alpha(z)D_z^\alpha,\qquad
\sigma_{\difford}(A)=\sum_{|\alpha|=\difford} a_\alpha(z)\zeta^\alpha.
$$

If $X$ is a manifold,
one can transfer these definitions to $X$ by localization and requiring that
the Schwartz kernels are $\CI$ densities away from the diagonal in $X^2=X\times X$;
then $\sigma_{\difford}(A)$ is in $S^\difford(T^*X)/S^{\difford-1}(T^*X)$,
resp.\ $S^\difford_{\hom}(T^*X\setminus o)$ when
$A\in\Psi^\difford(X)$, resp.\ $A\in\Psi^{\difford}_{\cl}(X)$; here $o$ is the zero section,
and $\hom$ stands for symbols homogeneous with respect to the $\RR^+$ action.
If $A$ is a differential operator, then the classical (i.e.\ homogeneous) version
of the principal symbol is a homogeneous polynomial in the fibers of the
cotangent bundle of degree $\difford$.
We can also work with operators depending on a parameter $\lambda\in O$ by
replacing $a\in S^\difford(\RR^n;\RR^n)$ by $a\in S^\difford(\RR^n\times O;\RR^n)$,
with $\Op(a_\lambda)\in\Psi^\difford(\RR^n)$ smoothly dependent on $\lambda\in O$.
In the case of differential operators, $a_\alpha$ would simply depend smoothly
on the parameter $\lambda$.

The large parameter, or high energy,
version of this, with the large parameter denoted by $\sigma$,
is that 
\begin{equation*}\begin{split}
A^{(\sigma)}=\Op^{(\sigma)}(a),
\ &\Op^{(\sigma)}(a) u(z)=(2\pi)^{-n}\int_{\RR^n} e^{i(z-z')\cdot\zeta} a(z,\zeta,\sigma)\,u(z')\,d\zeta\,dz',\\
&\qquad u\in\cS(\RR^n),\ a\in S^\difford(\RR^n;\RR^n_\zeta\times\Omega_\sigma),
\end{split}\end{equation*}
where $\Omega\subset\Cx$, with $\Cx$ identified with $\RR^2$; thus there
are joint symbol estimates in $\zeta$ and $\sigma$. The high energy
principal symbol now
should be thought of as an equivalence class of functions
on $\RR^n_z\times\RR^n_\zeta\times\Omega_\sigma$, or invariantly on
$T^*X\times\Omega$. Differential operators with polynomial dependence
on $\sigma$ now take the form
\begin{equation}\label{eq:large-param-diff}
A^{(\sigma)}=\sum_{|\alpha|+j \leq\difford} a_{\alpha,j}(z)\sigma^j D_z^\alpha,
\qquad \sigma^{(\sigma)}_{\difford}(A)=\sum_{|\alpha|+j =\difford} a_{\alpha,j}(z)\sigma^j \zeta^\alpha.
\end{equation}
Note that the principal symbol includes terms that would be subprincipal
with $A^{(\sigma)}$ considered
as a differential operator for a fixed value of $\sigma$.

The semiclassical operator algebra\footnote{We adopt the convention that $\semi$
denotes semiclassical objects, while $h$ is the actual semiclassical
parameter.}, $\Psih(\RR^n)$, is given by
\begin{equation*}\begin{split}
A_h=\Oph(a);\ &\Oph(a)u(z)=(2\pi h)^{-n}\int_{\RR^n} e^{i(z-z')\cdot\zeta/h} a(z,\zeta,h)\,u(z')\,d\zeta\,dz',\\
&\qquad u\in\cS(\RR^n),\ a\in \CI([0,1)_h;
S^\difford(\RR^n;\RR^n_\zeta));
\end{split}\end{equation*}
its classical subalgebra, $\Psihcl(\RR^n)$ corresponds to
$a\in \CI([0,1)_h;
S_{\cl}^\difford(\RR^n;\RR^n_\zeta))$.
The semiclassical principal symbol is now $\sigma_{\semi,\difford}(A)=a|_{h=0}\in
S^\difford(T^*X)$. We can again add an extra parameter $\lambda\in O$, so
$a\in \CI([0,1)_h;S^\difford(\RR^n\times O;\RR^n_\zeta))$; then in the invariant setting
the principal symbol is $a|_{h=0}\in
S^\difford(T^*X\times O)$.
Note that if $A^{(\sigma)}=\Op^{(\sigma)}(a)$ is a classical
operator with a large parameter, then
for $\lambda\in O\subset\Cx$, $\overline{O}$ compact, $0\notin \overline{O}$,
$$
h^\difford \Op^{(h^{-1}\lambda)}(a)
=\Oph(\tilde a),\ \tilde a(z,\zeta,h)=h^\difford a(z,h^{-1}\zeta,h^{-1}\lambda),
$$
and $\tilde a\in\CI([0,1)_h;S^\difford_{\cl}(\RR^n\times O_\lambda;\RR^n_\zeta))$.
The converse is not quite true: roughly speaking, the semiclassical algebra
is a blow-up of the large parameter algebra; to obtain an equivalence,
we would need to demand in the definition of the large parameter
algebra merely that
$a\in S^\difford(\RR^n;[\overline{\RR^n_\zeta\times\Omega_\sigma};
\pa\overline{\RR^n_\zeta\times\{0\}}])$, so in particular for bounded $\sigma$,
$a$ is merely a family of symbols depending smoothly on $\sigma$ (not jointly
symbolic); we do not discuss this here further. Note, however, that it
is the (smaller, i.e.\ stronger) large parameter algebra that arises naturally
when one Mellin transforms in the b-setting, see Subsection~\ref{subsec:Mellin}.

Differential operators now take the form
\begin{equation}\label{eq:semicl-diff}
A_{h,\lambda}=\sum_{|\alpha|\leq\difford} a_{\alpha}(z,\lambda;h)(hD_z)^\alpha.
\end{equation}
Such a family has two principal symbols, the standard one (but taking into
account the semiclassical degeneration, i.e.\ based on $(hD_z)^\alpha$ rather than
$D_z^\alpha$), which depends on $h$ and is homogeneous,
and the semiclassical one, which is at $h=0$, and is not homogeneous:
\begin{equation*}\begin{split}
&\sigma_{\difford}(A_{h,\lambda})=\sum_{|\alpha|=\difford} a_{\alpha}(z,\lambda;h)
\zeta^\alpha,\\
&\sigma_{\semi}(A_{h,\lambda})
=\sum_{|\alpha|\leq \difford} a_{\alpha}(z,\lambda;0)\zeta^\alpha.
\end{split}\end{equation*}
However, the restriction of $\sigma_{\difford}(A_{h,\lambda})$ to $h=0$ is the
principal part of $\sigma_{\semi}(A_{h,\lambda})$. In the special case in which
$\sigma_{\difford}(A_{h,\lambda})$ is independent of $h$ (which is true in
the setting considered below), one can simply regard the usual principal
symbol as the principal part of the semiclassical symbol.
Note that for $A^{(\sigma)}$ as in \eqref{eq:large-param-diff},
$$
h^\difford A^{(h^{-1}\lambda)}=\sum_{|\alpha|+j \leq\difford} h^{\difford-j-|\alpha|}a_{\alpha,j}(z)\lambda^j (hD_z)^\alpha,
$$
which is indeed of the form \eqref{eq:semicl-diff}, with polynomial dependence
on both $h$ and $\lambda$. Note that in this case the standard principal
symbol is independent of $h$ and $\lambda$.

\subsection{General assumptions}\label{subsec:micro-gen}
Let $X$ be a compact manifold and $\nu$ a smooth non-vanishing
density on it; thus $L^2(X)$ is well-defined as a Hilbert space (and not
only up to equivalence). We consider
operators
$P_\sigma\in\Psi_{\cl}^\difford(X)$
on $X$ depending on
a complex parameter $\sigma$, with the dependence being analytic
(thus, if $P_\sigma$ is a differential operator, the coefficients depend analytically on $\sigma$).
We also consider a complex absorbing `potential',
$Q_\sigma\in\Psi_{\cl}^\difford(X)$, defined for $\sigma\in \Omega$,
$\Omega\subset\Cx$
is open. It can be convenient to take
$Q_\sigma$ formally self-adjoint, which is possible when $Q_\sigma$ is
independent of $\sigma$, but this is inconvenient when one wants to
study the large $\sigma$ (i.e.\ semiclassical) behavior.
The operators we study
are
$P_\sigma-\imath Q_\sigma$ and $P_\sigma^*+\imath Q_\sigma^*$;
$P_\sigma^*$ depends on the choice of the density $\nu$.

Typically we shall be interested in $P_\sigma$ on an open subset $U$ of $X$,
and have $Q_\sigma$
supported in the complement of $U$, such that over some subset $K$ of
$X\setminus U$,
$Q_\sigma$ is elliptic on the characteristic set of $P_\sigma$. In the Kerr-de Sitter
setting, we would have $\overline{X_+}\subset U$. However,
{\em this is not part of the general set-up}.

It is often convenient to work with the fiber-radial compactification
$\overline{T}^*X$ of $T^*X$, in particular when discussing semiclassical analysis;
see for instance \cite[Sections~1 and 5]{RBMSpec}.
Thus, $S^*X$ should
be considered as the boundary of $\overline{T}^*X$. When one
is working with homogeneous objects, as is the case in classical microlocal
analysis, one can think of $S^*X$ as
$(\overline{T}^*X\setminus o)/\RR^+$, but this is not a useful point of view
in semiclassical analysis\footnote{In fact, even in classical microlocal
analysis it is better to keep at least a `shadow' of the interior of $S^*X$
by working with $T^*X\setminus o$ considered as a half-line bundle
over $S^*X$ with homogeneous objects on it; this keeps the action of the
Hamilton vector field on
the fiber-radial variable, i.e.\ the defining function of $S^*X$ in
$\overline{T}^*X$, non-trivial, which is important at radial points.}.
Thus, if $\rhot$ is a non-vanishing homogeneous
degree $-1$ function on $T^*X\setminus o$, it is
a defining function of $S^*X$ in $\overline{T}^*X\setminus o$; if the homogeneity
requirement is dropped it can be modified near the zero section to make
it a defining function of $S^*X$ in $\overline{T}^*X$. The principal
symbols $p,q$ of $P_\sigma,Q_\sigma$ are homogeneous degree $\difford$
functions on $T^*X\setminus o$, so $\rhot^\difford p,\rhot^\difford q$
are homogeneous degree $0$ there, thus are functions\footnote{This depends \label{footnote:factor}
on choices unless $\difford=0$; they are naturally sections of a line bundle that
encodes the differential of the boundary defining function at $S^*X$. However,
the only relevant notion here is ellipticity, and later the Hamilton vector
field up to multiplication by a positive function, which is independent of choices.
In fact, we emphasize that all the requirements listed for $p$, $q$ and later $p_{\semi,z}$
and $q_{\semi,z}$, except possibly \eqref{eq:subpr-form}-\eqref{eq:tilde-beta-form},
are also fulfilled if $P_\sigma-\imath Q_\sigma$
is replaced by {\em any} smooth positive multiple, so one may factor out
positive factors at will. This is useful in the Kerr-de Sitter space discussion.
For \eqref{eq:subpr-form}-\eqref{eq:tilde-beta-form}, see Footnote~\ref{footnote:subpr}.}
on $\overline{T}^*X$ near
its boundary, $S^*X$, and in particular on $S^*X$. Moreover,
$\sH_p$ is homogeneous degree $\difford-1$ on $T^*X\setminus o$, thus
$\rhot^{\difford-1}\sH_p$
a smooth vector field tangent to the boundary on $\overline{T}^*X$ (defined near
the boundary), and in particular induces a smooth vector field on $S^*X$.

We assume that the principal symbol $p$, resp.\ $q$, of $P_\sigma$,  resp.\ $Q_\sigma$,
are real, are independent of $\sigma$, $p=0$ implies $dp\neq 0$. We assume that
the characteristic set of $P_\sigma$ is of the form
$$
\Sigma=\Sigma_+\cup\Sigma_-,\ \Sigma_+\cap\Sigma_-=\emptyset,
$$
$\Sigma_\pm$ are relatively open\footnote{Thus, they are connected components
in the extended sense that they may be empty.} in $\Sigma$, and
$$
\mp q\geq 0\ \text{near}\ \Sigma_\pm.
$$
We assume that there are conic
submanifolds $\Lambda_\pm\subset\Sigma_\pm$
of $T^*X\setminus o$, outside which the Hamilton
vector field $\sH_p$ is not
radial, and to which the Hamilton vector field $\sH_p$ is tangent.
Here $\Lambda_\pm$ are typically Lagrangian, but this is not needed\footnote{An extreme example would be $\Lambda_\pm=\Sigma_\pm$. Another extreme
is if one or both are empty.}.
The properties we want at $\Lambda_\pm$ are (probably) not 
stable under general smooth perturbations; the perturbations need to have
certain properties at $\Lambda_\pm$. However, the estimates we then derive {\em are
stable} under such perturbations. First, we want that for a homogeneous
degree $-1$ defining function
$\tilde\rho$ of $S^*X$ near $L_\pm$, the image of $\Lambda_\pm$
in $S^*X$,
\begin{equation}\label{eq:weight-definite}
\rhot^{\difford-2}
\sH_p\tilde\rho|_{L_\pm}=\mp\beta_0,\ \beta_0\in\CI(L_\pm),\ \beta_0>0.
\end{equation}
Next, we require the existence
of a non-negative homogeneous degree zero
quadratic defining function $\rho_0$, of $\Lambda_\pm$
(i.e.\ it vanishes quadratically at $\Lambda_\pm$, and is non-degenerate)
and $\beta_1>0$ such that
\begin{equation}\label{eq:rho-0-property}
\mp\tilde\rho^{\difford-1}\sH_p\rho_0-\beta_1\rho_0
\end{equation}
is $\geq 0$ modulo cubic vanishing terms at $\Lambda_\pm$.
(The precise behavior of $\mp\tilde\rho^{\difford-1}\sH_p\rho_0$, or
of linear defining functions, is irrelevant, because we only need a relatively
weak estimate.
It would be relevant if one wanted to prove Lagrangian regularity.)
Under these assumptions, $L_-$ is a source and $L_+$ is a sink for
the $\sH_p$-dynamics in the sense
that nearby bicharacteristics tend to $L_\pm$ as the parameter along the
bicharacteristic goes to $\pm\infty$.
Finally, we assume that
the imaginary part of the subprincipal symbol at $\Lambda_\pm$, which is
the symbol of $\frac{1}{2\imath}(P_\sigma-P_\sigma^*)\in\Psi_{\cl}^{\difford-1}(X)$ as $p$ is
real, is\footnote{If $\sH_p$ is radial at $L_\pm$,
this is independent of the choice of the
density $\nu$. Indeed, with respect to $f\nu$,
the adjoint of $P_\sigma$ is $f^{-1}P_\sigma^* f$, with $P_\sigma^*$ denoting the
adjoint with respect to $\nu$. This is $P_\sigma^*+f^{-1}[P_\sigma^*,f]$, and
the principal symbol of
$f^{-1}[P_\sigma^*,f]\in\Psi_{\cl}^{\difford-1}(X)$ vanishes
at $L_\pm$ as $\sH_p f=0$. In general, we can only change the density by factors
$f$ with $\sH_p f|_{L_\pm}=0$, which in Kerr-de Sitter space-times would mean
factors independent of
$\phi$ at the event horizon. A similar argument shows the independence
of the condition from the choice of $f$
when one replaces $P_\sigma$ by $fP_\sigma$, under the same
conditions: either radiality, or just $\sH_p f|_{L_\pm}=0$.\label{footnote:subpr}}
\begin{equation}\label{eq:subpr-form}
\pm\tilde\beta\beta_0(\im \sigma)\tilde\rho^{-\difford+1},\ \tilde\beta\in\CI(L_\pm),
\end{equation}
$\tilde\beta$ is positive
along $L_\pm$, and write
\begin{equation}\label{eq:tilde-beta-form}
\beta_{\sup}=\sup\tilde\beta,\ \beta_{\inf}=\inf\tilde\beta>0.
\end{equation}
If $\tilde\beta$ is a constant, we may write
\begin{equation}\label{eq:beta-def}
\beta=\beta_{\inf}=\beta_{\sup}.
\end{equation}
The results take a little nicer form in this case since depending on various signs,
sometimes $\beta_{\inf}$ and sometimes $\beta_{\sup}$ is the relevant quantity.

\begin{figure}[ht]
\begin{center}
\mbox{\epsfig{file=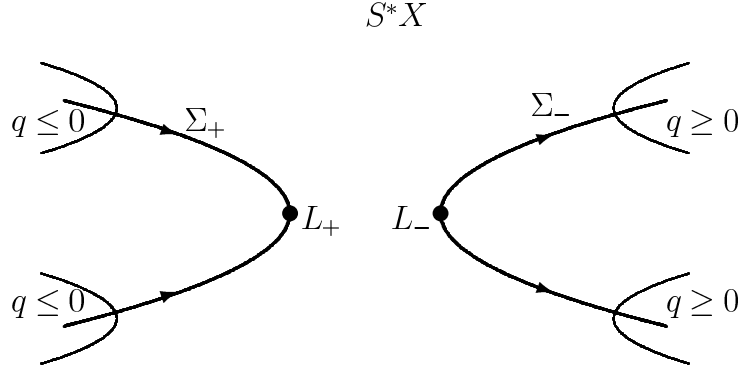}}
\end{center}
\caption{The components $\Sigma_\pm$ of the characteristic
set in the cosphere bundle $S^*X$. The submanifolds $L_\pm$ are points here,
with $L_-$ a source, $L_+$ a sink. The thin lined parabolic regions near the edges
show the absorbing region, i.e.\ the support of $q$.
For $P_\sigma-\imath Q_\sigma$,the estimates are
always propagated away from $L_\pm$ towards the support of $q$, so in the
direction of the Hamilton flow in $\Sigma_-$, and in the direction
opposite of the Hamilton flow in $\Sigma_+$; for $P_\sigma^*+\imath Q_\sigma^*$,
the directions
are reversed.}
\label{fig:microloc-damp1}
\end{figure}

We make the following {\em non-trapping} assumption.
For $\alpha\in S^*X$, let $\gamma_+(\alpha)$, resp.\ $\gamma_-(\alpha)$
denote the image of the forward, resp.\ backward, half-bicharacteristic
from $\alpha$. We write $\gamma_\pm(\alpha)\to L_\pm$
(and say $\gamma_\pm(\alpha)$ tends to $L_\pm$) if given any neighborhood $O$
of $L_\pm$, $\gamma_\pm(\alpha)\cap O\neq\emptyset$; by the source/sink property
this implies that the points on the curve are in $O$ for sufficiently large (in absolute
value) parameter values. We assume that, with $\elliptic(Q_\sigma)$
denoting the elliptic set of $Q_\sigma$,
\begin{multline}\label{eq:non-trapping-impl}
\alpha\in\Sigma_-\setminus L_-\Rightarrow \big(\gamma_-(\alpha)\to L_-\Mor
\gamma_-(\alpha)\cap\elliptic(Q_\sigma)\neq\emptyset\big)\Mand
\gamma_+(\alpha)\cap\elliptic(Q_\sigma)\neq\emptyset,\\
\alpha\in\Sigma_+\setminus L_+\Rightarrow \big(\gamma_+(\alpha)\to L_+\Mor
\gamma_+(\alpha)\cap\elliptic(Q_\sigma)\neq\emptyset\big) \Mand
\gamma_-(\alpha)\cap\elliptic(Q_\sigma)\neq\emptyset.\\
\end{multline}
That is, all forward and backward half-(null)bicharacteristics of
$P_\sigma$ either enter the elliptic set of $Q_\sigma$, or go to $\Lambda_\pm$,
i.e.\ $L_\pm$ in $S^*X$.
The point of the assumptions
regarding $Q_\sigma$ and the flow is that we are able to propagate estimates
forward near where $q\geq 0$, backward near where $q\leq 0$, so by our hypotheses
we can always propagate estimates for $P_\sigma-\imath Q_\sigma$
from $\Lambda_\pm$ towards the elliptic set of
$Q_\sigma$, and also if both ends of a bicharacteristic go to the elliptic set of $Q_\sigma$
then we can propagate the estimates from one of the directions. On the
other hand, for $P_\sigma^*+\imath Q_\sigma^*$, we can propagate estimates
from the elliptic set of $Q_\sigma$ towards $\Lambda_\pm$, and
again if both ends of a bicharacteristic go to the elliptic set of $Q_\sigma$
then we can propagate the estimates from one of the directions. This behavior
of $P_\sigma-\imath Q_\sigma$ vs.\ $P_\sigma^*+\imath Q_\sigma^*$ is important
for duality reasons.

\begin{rem}\label{rem:bundles}
For simplicity of notation we have not considered vector bundles on $X$. However,
if $E$ is a vector bundle on $X$ with a positive definite inner product on the fibers
and $P_\sigma,Q_\sigma\in\Psi_{\cl}^\difford(X;E)$ with scalar
principal symbol $p$, and in case of $P_\sigma$ the imaginary part
of the subprincipal symbol is of the form \eqref{eq:subpr-form} with
$\tilde\beta$ a bundle-endomorphism
satisfying an inequality in \eqref{eq:tilde-beta-form}
as a bundle endomorphism, the arguments we present go through.
\end{rem}

\subsection{Elliptic and real principal type points}\label{subsec:ellip-hyp}
We now turn to analysis.
First, by the usual elliptic theory, on the elliptic set of $P_\sigma-\imath Q_\sigma$,
so both on the elliptic set of $P_\sigma$ and on the elliptic set of $Q_\sigma$,
one has elliptic estimates\footnote{Our convention in estimates such as
\eqref{eq:micro-elliptic} and \eqref{eq:micro-hyperbolic} is that
if one assumes that all the quantities on the right hand side are in the
function spaces indicated by the norms
then so is the quantity on the left hand side, and the estimate holds.
As we see below, at $\Lambda_\pm$ not all relevant function space statements
appear in the estimate, so we need to be more explicit there.}:
for all $s$ and $N$, and for all $B,G\in\Psi^0(X)$ with
$G$ elliptic on $\WF'(B)$,
\begin{equation}\label{eq:micro-elliptic}
\|Bu\|_{H^s}\leq C(\|G(P_\sigma-\imath Q_\sigma)u\|_{H^{s-\difford}}+\|u\|_{H^{-N}}),
\end{equation}
with the estimate also holding for $P_\sigma^*+\imath Q_\sigma$.
By propagation of singularities,
in $\Sigma\setminus(\WF'(Q_\sigma)\cup L_+\cup L_-)$,
one can propagate regularity estimates either
forward or backward along bicharacteristics, i.e.\ for all $s$ and $N$,
and for all $A,B,G\in\Psi^0(X)$ such that $\WF'(G)\cap\WF'(Q_\sigma)=\emptyset$,
and forward (or backward) bicharacteristics
from $\WF'(B)$ reach the elliptic set of $A$, while remaining in the elliptic set of $G$,
one has estimates
\begin{equation}\label{eq:micro-hyperbolic}
\|Bu\|_{H^s}\leq C(\|GP_\sigma u\|_{H^{s-\difford+1}}+\|Au\|_{H^s}+\|u\|_{H^{-N}}).
\end{equation}
Here $P_\sigma$ can be replaced by $P_\sigma-\imath Q_\sigma$ or
$P_\sigma^*+\imath Q_\sigma^*$ by the
condition on $\WF'(G)$; namely $GQ_\sigma\in\Psi^{-\infty}(X)$, and can thus
be absorbed into the $\|u\|_{H^{-N}}$ term. As usual, there is a loss of one
derivative compared to the elliptic estimate, i.e.\ the assumption on $P_\sigma u$
is $H^{s-\difford+1}$, not $H^{s-\difford}$, and one needs to make $H^s$ assumptions on $Au$,
i.e.\ regularity propagates.

\subsection{Analysis near $\Lambda_\pm$}
At $\Lambda_\pm$, for $s\geq m>(\difford-1-\beta\im\sigma)/2$, $\beta$ given by
the subprincipal symbol at $\Lambda_\pm$, we can propagate estimates
{\em away} from $\Lambda_\pm$:

\begin{prop}\label{prop:micro-out}
For $\im\sigma\geq 0$, let\footnote{Note that this is consistent with
\eqref{eq:beta-def}.} $\beta=\beta_{\inf}$, for $\im\sigma<0$, let
$\beta=\beta_{\sup}$.
For all $N$, for $s\geq m>(\difford-1-\beta\im\sigma)/2$,
and for all $A,B,G\in\Psi^0(X)$ such that
$\WF'(G)\cap\WF'(Q_\sigma)=\emptyset$, $A$ elliptic at $\Lambda_\pm$,
and forward (or backward) bicharacteristics
from $\WF'(B)$ tend to $\Lambda_\pm$,
with closure in the elliptic set of $G$, one has estimates
\begin{equation}\label{eq:radial-out}
\|Bu\|_{H^s}\leq C(\|GP_\sigma u\|_{H^{s-\difford+1}}+\|Au\|_{H^{m}}+\|u\|_{H^{-N}}),
\end{equation} 
in the sense that if $u\in H^{-N}$, $Au\in H^m$ and $GP_\sigma u\in H^{s-\difford+1}$,
then $Bu\in H^s$, and \eqref{eq:radial-out} holds. In fact, $Au$ can be dropped
from the right hand side (but one must assume $Au\in H^m$):
\begin{equation}\label{eq:radial-out-p}
Au\in H^m\Rightarrow \|Bu\|_{H^s}\leq C(\|GP_\sigma u\|_{H^{s-\difford+1}}+\|u\|_{H^{-N}}),
\end{equation}
where $u\in H^{-N}$ and $GP_\sigma u\in H^{s-\difford+1}$ is considered implied by the
right hand side. Note that $Au$ does not appear on the right hand
side, hence the display before
the estimate.
\end{prop}

This is completely analogous to Melrose's estimates in asymptotically Euclidean
scattering theory at the radial sets \cite[Section~9]{RBMSpec}.
Note that the $H^s$ regularity of $Bu$
is `free' in the
sense that we do not need to impose $H^s$ assumptions on $u$ anywhere; merely
$H^m$ at $\Lambda_\pm$ does the job; of course, on $P_\sigma u$ one must make the
$H^{s-\difford+1}$ assumption, i.e.\ the loss of one derivative compared to the elliptic setting.
At the cost of changing regularity, one can
propagate estimate {\em towards} $\Lambda_\pm$. Keeping in mind that
for $P_\sigma^*$ the subprincipal symbol becomes $\beta\overline{\sigma}$,
we have the following:

\begin{prop}\label{prop:micro-in}
For $\im\sigma>0$, let\footnote{Note the switch compared to
Proposition~\ref{prop:micro-out}! Also, $\beta$ does not matter when $\im\sigma=0$;
we define it here so that the two Propositions are consistent via
dualization, which reverses the sign of the imaginary part.}
$\beta=\beta_{\sup}$, for $\im\sigma\leq 0$, let
$\beta=\beta_{\inf}$.
For $s<(\difford-1-\beta\im\sigma)/2$, for all $N$,
and for all $A,B,G\in\Psi^0(X)$ such that
$\WF'(G)\cap\WF'(Q_\sigma)=\emptyset$, $B,G$ elliptic at $\Lambda_\pm$,
and forward (or backward) bicharacteristics
from $\WF'(B)\setminus\Lambda_\pm$ reach $\WF'(A)$,
while remaining in the elliptic set of $G$, one has estimates
\begin{equation}\label{eq:radial-in}
\|Bu\|_{H^s}\leq C(\|GP_\sigma^* u\|_{H^{s-\difford+1}}+\|Au\|_{H^{s}}+\|u\|_{H^{-N}}).
\end{equation}
\end{prop}

\begin{proof}[Proof of Propositions~\ref{prop:micro-out}-\ref{prop:micro-in}.]
It suffices to prove that there exist $O_j$ open with $L_\pm\subset O_{j+1}\subset O_j$,
$\cap_{j=1}^\infty O_j=L_\pm$, and
$A_j,B_j,G_j$ with $\WF'$ in $O_j$, $B_j$ elliptic on $L_\pm$, such that
the statements of the propositions hold. Indeed, in case of
Proposition~\ref{prop:micro-out} the general case follows by taking
$j$ such that $A,G$ are elliptic on $O_j$, use the estimate for $A_j,B_j,G_j$, where
the right hand side then can be estimated by $A$ and $G$, and then use microlocal
ellipticity, propagation of singularities and a covering argument to prove the
proposition. In case of Proposition~\ref{prop:micro-in}, the general case follows
by taking $j$ such that $G,B$ are elliptic on $O_j$, so all forward (or backward)
bicharacteristics
from $O_j\setminus\Lambda_\pm$ reach $\WF'(A)$, thus microlocal
ellipticity, propagation of singularities and a covering argument proves
$\|A_j u\|_{H^s}\leq C(\|GP_\sigma^* u\|_{H^{s-\difford+1}}+\|Au\|_{H^{s}}+\|u\|_{H^{-N}})$,
and then the special case of the proposition for this $O_j$ gives an estimate
for $\|B_j u\|_{H^s}$ in terms of the same quantities. The full estimate for
$\|B u\|_{H^s}$ is then again a straightforward consequence of
microlocal
ellipticity, propagation of singularities and a covering argument.

We now consider
commutants $C_\ep^*C_\ep$ with $C_\ep\in\Psi^{s-(\difford-1)/2-\delta}(X)$ for $\ep>0$,
uniformly bounded in $\Psi^{s-(\difford-1)/2}(X)$ as $\ep\to 0$; with the
$\ep$-dependence used to regularize the argument. More precisely, let
$$
c=\phi(\rho_0)\tilde\rho^{-s+(\difford-1)/2},\qquad c_\ep=c(1+\ep \tilde\rho^{-1})^{-\delta},
$$
where $\phi\in\CI_c(\RR)$ is identically $1$ near $0$, $\phi'\leq 0$ and
$\phi$ is supported sufficiently close to $0$ so that
\begin{equation}\label{eq:rad-localizer-est}
\rho_0\in\supp d\phi\Rightarrow \mp\tilde\rho^{\difford-1}\sH_p\rho_0>0;
\end{equation}
such $\phi$ exists by \eqref{eq:rho-0-property}.
Note that the sign of $\sH_p \tilde\rho^{-s+(\difford-1)/2}$ depends on the sign of
$-s+(\difford-1)/2$ which explains the difference between $s>(\difford-1)/2$
and $s<(\difford-1)/2$
in Propositions~\ref{prop:micro-out}-\ref{prop:micro-in} when
there are no other contributions to the threshold value of $s$. The contribution of
the subprincipal symbol, however, shifts the critical value $(\difford-1)/2$.

Now let $C\in\Psi^{s-(\difford-1)/2}(X)$
have principal symbol $c$, and have $\WF'(C)\subset \supp \phi\circ\rho_0$, and
let $C_\ep=C S_\ep$, $S_\ep\in\Psi^{-\delta}(X)$ uniformly bounded in
$\Psi^0(X)$ for $\ep>0$, converging
to $\Id$ in $\Psi^{\delta'}(X)$ for $\delta'>0$ as $\ep\to 0$, with principal symbol
$(1+\ep \tilde\rho^{-1})^{-\delta}$. Thus, the principal symbol of $C_\ep$ is $c_\ep$.

First, consider \eqref{eq:radial-out}. 
Then
\begin{equation}\begin{split}\label{eq:rad-comm-expand-1}
&\sigma_{2s}(\imath(P^*_\sigma C_\ep^*C_\ep-C_\ep^*C_\ep P_\sigma))
=\sigma_{\difford-1}(\imath(P^*_\sigma-P_\sigma)) c_\ep^2
+2c_\ep\sH_p c_\ep\\
&=\mp 2\left(-\tilde\beta\im\sigma\beta_0\phi
+\beta_0\left(-s+\frac{\difford-1}{2}\right)\phi\mp(\tilde\rho^{\difford-1}\sH_p\rho_0)\phi'
+\delta\beta_0 \frac{\ep}{\tilde\rho+\ep}\phi\right)\\
&\hspace{9cm}
\phi\tilde\rho^{-2s}(1+\ep \tilde\rho^{-1})^{-\delta},
\end{split}\end{equation}
so
\begin{equation}\begin{split}\label{eq:rad-comm-expand}
\mp\sigma_{2s}&(\imath(P^*_\sigma C_\ep^*C_\ep-C_\ep^*C_\ep P_\sigma))\\
&\leq
-2\beta_0\left(s-\frac{\difford-1}{2}+\tilde\beta\im\sigma-\delta\right) \tilde\rho^{-2s}(1+\ep \tilde\rho^{-1})^{-\delta}\phi^2\\
&\qquad\qquad+2(\mp\tilde\rho^{\difford-1}\sH_p\rho_0)
\tilde\rho^{-2s}(1+\ep \tilde\rho^{-1})^{-\delta}\phi'\phi
.
\end{split}\end{equation}
Here the first term on the right hand side
is negative if $s-(\difford-1)/2+\beta\im\sigma-\delta>0$
(since $\tilde\beta\im\sigma\geq\beta\im\sigma$ by our definition of $\beta$),
and this is the same sign as that of $\phi'$ term;
the presence of $\delta$ (needed for the regularization) is the reason for the
appearance of $m$ in the estimate.
To avoid using the sharp G{\aa}rding inequality,
we choose $\phi$ so that $\sqrt{-\phi\phi'}$
is $\CI$, and then
\begin{equation*}
\imath(P^*_\sigma C_\ep^*C_\ep-C_\ep^*C_\ep P_\sigma)=-S_\ep^*(B^*B+ B_1^*B_1+B_{2,\ep}^* B_{2,\ep})S_\ep+F_\ep,
\end{equation*}
with $B,B_1,B_{2,\ep}\in\Psi^s(X)$, $B_{2,\ep}$ uniformly bounded in $\Psi^s(X)$ as
$\ep\to 0$,
$F_\ep$ uniformly bounded in
$\Psi^{2s-1}(X)$, and $\sigma_s(B)$ an elliptic multiple of $\phi(\rho_0)\tilde\rho^{-s}$.
Computing the pairing, using an extra regularization (insert a
regularizer $\Lambda_r\in\Psi^{-1}(X)$, uniformly bounded in
$\Psi^0(X)$, converging to $\Id$ in $\Psi^\delta(X)$ to justify
integration by parts, and use that $[\Lambda_r,P^*_\sigma]$ is
uniformly bounded in $\Psi^{1}(X)$, converging to $0$ strongly,
cf.\ \cite[Lemma~17.1]{Vasy:Propagation-2} and its use in \cite[Lemma~17.2]{Vasy:Propagation-2}) yields
$$
\langle \imath(P^*_\sigma C_\ep^*C_\ep-C_\ep^*C_\ep P_\sigma)u,u\rangle=
\langle \imath C_\ep^*C_\ep u,P_\sigma u\rangle-\langle\imath P_\sigma,C_\ep^*C_\ep u\rangle.
$$
Using Cauchy-Schwartz on the right hand side, a standard functional analytic argument
(see, for instance, Melrose \cite[Proof of Proposition~7 and Section~9]{RBMSpec})
gives an estimate for $Bu$, showing $u$ is in $H^s$ on the elliptic set of $B$,
provided $u$ is microlocally in $H^{s-\delta}$. A standard
inductive argument, starting with $s-\delta=m$ and improving regularity by $\leq 1/2$
in each step proves \eqref{eq:radial-out}.

For \eqref{eq:radial-in}, the argument
is similar, but we want to change the sign of the first term on the right hand side of
\eqref{eq:rad-comm-expand}, i.e.\ we want it to be positive. This
is satisfied if
$s-(\difford-1)/2+\beta\im\sigma-\delta<0$
(since $\tilde\beta\im\sigma\leq\beta\im\sigma$ by our definition of $\beta$
in Proposition~\ref{prop:micro-in}), hence (as $\delta>0$)
if $s-(\difford-1)/2+\beta\im\sigma<0$, so regularization is not an issue.
On the other hand, $\phi'$ now has
the wrong sign, so one needs to make an assumption on $\supp d\phi$, which
is the $Au$ term in \eqref{eq:radial-in}. Since the details are standard,
see \cite[Section~9]{RBMSpec}, we leave these to the reader.
\end{proof}

\begin{rem}\label{rem:stable}
Fixing a $\phi$, it follows from the proof that the same $\phi$ works for (small)
smooth
perturbations of $P_\sigma$ with real principal
symbol\footnote{Reality is needed to ensure that \eqref{eq:rad-comm-expand-1} holds.},
even if those perturbations do not preserve
the event horizon, namely even if \eqref{eq:rho-0-property} does not hold any more:
only its implication, \eqref{eq:rad-localizer-est},
on $\supp d\phi$ matters, which {\em is} stable under
perturbations. Moreover, as the rescaled Hamilton vector field $\rhot^{\difford-1}\sH_p$
is a smooth vector field tangent to the boundary of the fiber-compactified
cotangent bundle, i.e.\ a b-vector field, and as such depends smoothly on
the principal symbol, and it is {\em non-degenerate} radially by
\eqref{eq:weight-definite}, the weight,
which provides the positivity at the radial points in the
proof above, still gives a positive Hamilton derivative for small perturbations.
Since this proposition thus holds for $\CI$ perturbations
of $P_\sigma$ with real principal symbol,
and this proposition is the only delicate estimate we use,
and it is only marginally so, we deduce that
all the other results below also hold in this generality.
\end{rem}

\subsection{Complex absorption}\label{subsec:complex-absorb}
Finally, one has propagation
estimates for complex absorbing operators, requiring a sign condition.
We refer to, for instance, \cite{Nonnenmacher-Zworski:Quantum}
and \cite[Lemma~5.1]{Datchev-Vasy:Gluing-prop} in the semiclassical setting;
the changes are minor in the `classical' setting. We also give a
sketch of the main `commutator' calculation below.

First, one can propagate regularity to $\WF'(Q_\sigma)$ (of course, in the elliptic
set of $Q_\sigma$ one has a priori regularity). Namely, for all $s$ and $N$,
and for all $A,B,G\in\Psi^0(X)$ such that $q\leq 0$, resp.\ $q\geq 0$, on $\WF'(G)$,
and forward, resp.\ backward, bicharacteristics of $P_\sigma$
from $\WF'(B)$ reach the elliptic set of $A$, while remaining in the elliptic set of $G$,
one has the usual propagation estimates
$$
\|Bu\|_{H^s}\leq C(\|G(P_\sigma-\imath Q_\sigma) u\|_{H^{s-\difford+1}}+\|Au\|_{H^s}+\|u\|_{H^{-N}}).
$$
Thus, for $q\geq 0$ one can propagate regularity in the forward direction along
the Hamilton flow, while for $q\leq 0$ one can do so in the backward direction.

On the other hand, one can propagate regularity away from the elliptic set
of $Q_\sigma$. Namely, for all $s$ and $N$,
and for all $B,G\in\Psi^0(X)$ such that $q\leq 0$, resp.\ $q\geq 0$, on $\WF'(G)$,
and forward, resp.\ backward, bicharacteristics of $P_\sigma$
from $\WF'(B)$ reach the elliptic set of $Q_\sigma$,
while remaining in the elliptic set of $G$,
one has the usual propagation estimates
$$
\|Bu\|_{H^s}\leq C(\|G(P_\sigma-\imath Q_\sigma) u\|_{H^{s-\difford+1}}+\|u\|_{H^{-N}}).
$$
Again, for $q\geq 0$ one can propagate regularity in the forward direction along
the Hamilton flow, while for $q\leq 0$ one can do so in the backward direction.
At the cost of reversing the signs of $q$, this also gives that
for all $s$ and $N$,
and for all $B,G\in\Psi^0(X)$ such that $q\geq 0$, resp.\ $q\leq 0$, on $\WF'(G)$,
and forward, resp.\ backward, bicharacteristics of $P_\sigma$
from $\WF'(B)$ reach the elliptic set of $Q_\sigma$,
while remaining in the elliptic set of $G$,
one has the usual propagation estimates
$$
\|Bu\|_{H^s}\leq C(\|G(P_\sigma^*+\imath Q_\sigma^*) u\|_{H^{s-\difford+1}}+\|u\|_{H^{-N}}).
$$

We remark that again, these estimates are stable under small
perturbations in $\Psi^\difford(X)$ of $P_\sigma$ and $Q_\sigma$
provided the perturbed operators still have real principal symbols,
and in the case of $Q_\sigma$, satisfy $q\geq 0$. This can be easily
seen by following the proof of
\cite[Lemma~5.1]{Datchev-Vasy:Gluing-prop}; the role of the absorbing
potential $W\geq 0$ there is played by the formally
self-adjoint operator
$\tilde Q_\sigma=\frac{1}{2}(Q_\sigma+Q_\sigma^*)$ with principal symbol $q$
here. Although there $W$ is a function on $X$
(rather than a general pseudodifferential operator), the only
properties that matter in the present notation are that
the principal symbols are real and $q\geq 0$. Indeed, in this case, writing
$C$ (analogously to the proof of
Propositions~\ref{prop:micro-out}-\ref{prop:micro-in} here)
instead of $Q$ for the commutant of
\cite[Lemma~5.1]{Datchev-Vasy:Gluing-prop} to avoid confusion, and
denoting its (real) principal symbol by $c$, and
letting $\tilde
P_\sigma=P_\sigma+\frac{1}{2\imath}(Q_\sigma-Q_\sigma^*)$, so
$P_\sigma-\imath Q_\sigma=\tilde P_\sigma-\imath\tilde Q_\sigma$, and
the principal symbol of the formally self-adjoint operator $\tilde P_\sigma$
is $p$, we have
\begin{equation}\label{eq:expand-cx-pair}
\langle u,-\imath [C^*C,\tilde P_\sigma]u\rangle=-2\re\langle
u,\imath C^*C(P_\sigma-\imath Q_\sigma)u\rangle-2\re\langle
u,C^*C\tilde Q_\sigma u\rangle.
\end{equation}
The operator on the left hand side has principal symbol $\sH_pc^2$,
and will preserve its signs under sufficiently small perturbations of $p$ using the same
construction of $c$ as in \cite[Lemma~5.1]{Datchev-Vasy:Gluing-prop}
(which is just a real-principal type construction), much as in the
radial point setting discussed in the previous subsection. On the
other hand, the second term on the right hand side can be rewritten as
$$
2\re\langle u,C^*C\tilde Q_\sigma u\rangle=2\re\langle u,C^*\tilde
Q_\sigma C u\rangle+2\re\langle u,C^*[C,\tilde Q_\sigma]u\rangle,
$$
where the second term is $\langle u,[C,[C,\tilde Q_\sigma]]u\rangle$
plus similar pairings involving $(C^*-C)[C,\tilde Q_\sigma$, etc.,
which are all lower order than the operator on the left hand side of
\eqref{eq:expand-cx-pair} due to the real principal symbol of $C$ and
the presence of a commutator, or to the presence of
the double commutator. The first term, on the other hand, is non-negative
modulo terms that can be absorbed into the left hand side of
\eqref{eq:expand-cx-pair}, since by the sharp G{\aa}rding
inequality\footnote{If one assumes that $q$ is microlocally the square
  of a symbol, one need not use the sharp G{\aa}rding inequality.}, 
$\langle u,C^*\tilde Q_\sigma C u\rangle\geq -\langle u, C^* R_\sigma
Cu\rangle$ where $R_\sigma$ is one order lower than $Q_\sigma$, i.e.\
is in $\Psi^{\difford-1}(X)$, and as the principal symbol of
$C^*R_\sigma C$ does not contain derivatives of $c$, an appropriate
choice of $C$ lets one use the $\sH_p c^2$ term,
i.e.\ the principal symbol of the left hand side of
\eqref{eq:expand-cx-pair},  to dominate this,
as
usual in real principal type estimates when subprincipal terms are
dominated (see also the treatment of the $\im\lambda$ term in
\cite[Lemma~5.1]{Datchev-Vasy:Gluing-prop}).

\begin{rem}\label{rem:add-bdy}
As mentioned in the introduction,
these complex absorption methods could be replaced in
specific cases, including all the specific examples we discuss here, by
adding a boundary $\tilde Y$ instead, provided that the Hamilton flow
is well-behaved relative to the base space, namely inside the characteristic
set $\sH_p$ is not tangent to $T^*_{\tilde Y}X$ with orbits crossing $T^*_{\tilde Y}X$
in the opposite directions in $\Sigma_\pm$ in the following way. If $\tilde Y$ is defined
by $\tilde y$ which is positive on `our side' $U$ with $U$ as discussed
at the beginning of Subsection~\ref{subsec:micro-gen}, we need
$\pm\sH_p\tilde y|_{\tilde Y}>0$ on
$\Sigma_\pm$. Then the functional analysis described in
\cite[Proof of Theorem~23.2.2]{Hor}, see also \cite[Proof of Lemma~4.14]{Vasy:AdS},
can be used to prove analogues of the
results we give below on $X_+=\{\tilde y\geq 0\}$. For instance, if one
has a Lorentzian
metric on $X$ near $\tilde Y$, and
$\tilde Y$ is space-like, then  (up to the sign)
this statement holds with $\Sigma_\pm$ being the two components of
the characteristic set.
However, in the author's opinion, this detracts from the clarity of the microlocal
analysis by introducing projection to physical space in an essential way.
\end{rem}

\subsection{Global estimates}
Recall now that $q\geq 0$ near $\Sigma_-$, and $q\leq 0$
on $\Sigma_+$, and recall our non-trapping assumptions, i.e.\ \eqref{eq:non-trapping-impl}.
Thus, we can piece together the estimates described earlier (elliptic,
real principal type, radial points, complex absorption) to propagate
estimates forward in $\Sigma_-$ and backward in $\Sigma_+$, thus away
from $\Lambda_\pm$ (as well as from one end of a bicharacteristic which
intersects the elliptic set of $q$ in both directions). This yields
that for any $N$, and for any $s\geq m>(\difford-1-\beta\im\sigma)/2$, and for any $A\in\Psi^0(X)$ elliptic
at $\Lambda_+\cup\Lambda_-$,
$$
\|u\|_{H^s}\leq C(\|(P_\sigma-\imath Q_\sigma)u\|_{H^{s-\difford+1}}+\|Au\|_{H^m}+\|u\|_{H^{-N}}).
$$ 
This implies that
for any $s> m>(\difford-1-\beta\im\sigma)/2$,
\begin{equation}\label{eq:unique-est}
\|u\|_{H^s}\leq C(\|(P_\sigma-\imath Q_\sigma)u\|_{H^{s-\difford+1}}+\|u\|_{H^m}).
\end{equation}
On the other hand, recalling
that the adjoint switches the sign of the imaginary part of the
principal symbol and also that of the subprincipal
symbol at the radial sets, propagating the estimates in the other direction,
i.e.\ backward in $\Sigma_-$ and forward in $\Sigma_+$, thus
towards $\Lambda_\pm$, from the elliptic set of $q$,
we deduce that for any $N$ (which we take to satisfy $s'>-N$)
and for any $s'<(\difford-1+\beta\im\sigma)/2$,
\begin{equation}\label{eq:exist-est}
\|u\|_{H^{s'}}\leq C(\|(P_\sigma^*+\imath
Q_\sigma^*)u\|_{H^{s'-\difford+1}}+\|u\|_{H^{-N}}).
\end{equation}
Note that the dual of $H^s$, $s>(\difford-1-\beta\im\sigma)/2$, is $H^{-s}=H^{s'-\difford+1}$, $s'=\difford-1-s$, so $s'<(\difford-1+\beta\im\sigma)/2$,
while the dual of $H^{s-\difford+1}$, $s>(\difford-1-\beta\im\sigma)/2$, is $H^{\difford-1-s}=H^{s'}$, with $s'=\difford-1-s<(\difford-1+\beta\im\sigma)/2$ again.
Thus, the spaces (apart from the residual spaces, into which the inclusion is
compact) in the left, resp.\ right, side of \eqref{eq:exist-est}, are exactly
the duals of those on the right, resp.\ left, side of \eqref{eq:unique-est}.
Thus, by a standard functional analytic argument,
see e.g.\ \cite[Proof of Theorem~26.1.7]{Hor} or indeed \cite[Section~4.3]{Vasy:Microlocal-AH}
in the present context,
namely dualization
and using the compactness
of the inclusion $H^{s'}\to H^{-N}$ for $s'>-N$, this gives the solvability of
$$
(P_\sigma-\imath Q_\sigma)u=f,\ s>(\difford-1-\beta\im\sigma)/2,
$$
for $f$ in the annihilator in $H^{s-\difford+1}$ (via the duality between $H^{s-\difford+1}$
and $H^{-s+\difford-1}$ induced by the $L^2$-pairing)
of the finite dimensional subspace $\Ker(P_\sigma^*+\imath Q_\sigma^*)$
of $H^{-s+\difford-1}=H^{s'}$,
and
indeed elements of this finite dimensional subspace have wave front
set\footnote{Since the original version of this paper, the work of
  Haber and Vasy \cite{Haber-Vasy:Radial} showed that elements of this kernel are in fact
  Lagrangian distributions, i.e.\ they possess iterative regularity
  under the module of first order pseudodifferential operators with
  principal symbol vanishing on the Lagrangian.}
in $\Lambda_+\cup\Lambda_-$ and lie in $\cap_{s'<(\difford-1+\beta\im\sigma)/2}H^{s'}$.
Thus, there is the usual real principal type loss of one derivative relative
to the elliptic problem, and in addition,
there are restrictions on the orders for which is valid.

In addition, one also has almost uniqueness by a standard compactness
argument
(using the compactness of the inclusion of $H^s$ into $H^m$ for $s>m$), by
\eqref{eq:unique-est}, namely not only is
the space of $f$ in the space as above is finite codimensional, but the nullspace of $P_\sigma-\imath Q_\sigma$ on $H^s$, $s>(\difford-1-\beta\im\sigma)/2$,
is also finite dimensional, and its elements are in $\CI(X)$; again,
see \cite[Section~4.3]{Vasy:Microlocal-AH} for details in this setup.

In order to analyze the $\sigma$-dependence of solvability of the PDE,
we reformulate our problem as a more conventional Fredholm problem.
Thus, let $\tilde P$ be any
operator with principal symbol $p-\imath q$; e.g.\ $\tilde P$ is
$P_{\sigma_0}-\imath Q_{\sigma_0}$ for some
$\sigma_0$. Then consider
\begin{equation}\label{eq:XY-def}
\cX^s=\{u\in H^s:\ \tilde Pu\in H^{s-\difford+1}\},\ \cY^s=H^{s-\difford+1},
\end{equation}
with
$$
\|u\|_{\cX^s}^2=\|u\|_{H^s}^2+\|\tilde Pu\|^2_{H^{s-\difford+1}}.
$$
Note that $\cX^s$ only depends on the principal symbol of $\tilde P$. Moreover,
$\CI(X)$ is dense in $\cX^s$; this follows by considering $R_\ep\in\Psi^{-\infty}(X)$,
$\ep>0$,
such that $R_\ep\to\Id$ in $\Psi^\delta(X)$ for $\delta>0$, $R_\ep$ uniformly
bounded in $\Psi^0(X)$; thus $R_\ep\to\Id$ strongly (but not in the operator norm
topology) on $H^s$ and $H^{s-\difford+1}$.
Then for $u\in \cX^s$, $R_\ep u\in\CI(X)$ for $\ep>0$, $R_\ep u\to u$
in $H^s$ and $\tilde P R_\ep u=R_\ep\tilde Pu+[\tilde P,R_\ep]u$, so the first
term on the right converges to $\tilde Pu$ in $H^{s-\difford+1}$, while $[\tilde P,R_{\ep}]$
is uniformly bounded in $\Psi^{\difford-1}(X)$, converging to $0$ in $\Psi^{\difford-1+\delta}(X)$ for
$\delta>0$, so converging to $0$ strongly as a map $H^s\to H^{s-\difford+1}$. Thus,
$[\tilde P,R_\ep]u\to 0$ in $H^{s-\difford+1}$, and
we conclude that $R_\ep u\to u$ in $\cX^s$.
(In fact, $\cX^s$ is a first-order coisotropic space, more general
function spaces of this nature are discussed
by Melrose, Vasy and Wunsch in \cite[Appendix~A]{Melrose-Vasy-Wunsch:Corners}.)

With these preliminaries,
$$
P_\sigma-\imath Q_\sigma:\cX^s\to \cY^s
$$
is Fredholm for each $\sigma$ with $s\geq m>(\difford-1-\beta\im\sigma)/2$,
and is an analytic family of bounded operators in this half-plane of $\sigma$'s.

\begin{thm}\label{thm:classical-absorb}
Let $P_\sigma$, $Q_\sigma$ be as above,
and $\cX^s$, $\cY^s$ as in \eqref{eq:XY-def}.
If $\difford-1-2s>0$, let $\beta=\beta_{\inf}$, if
$\difford-1-2s<0$, let $\beta=\beta_{\sup}$.
Then
$$
P_\sigma-\imath Q_\sigma:\cX^s\to\cY^s
$$
is an analytic family of Fredholm operators on $\Cx_s\cap\Omega$, where
\begin{equation}\label{eq:Cx-s-def}
\Cx_s=\{\sigma\in\Cx:\ \im\sigma>\beta^{-1}(\difford-1-2s)\}.
\end{equation}
\end{thm}

Thus, analytic Fredholm theory applies, giving meromorphy of the inverse
provided the inverse exists for a particular value of $\sigma$.

\begin{rem}\label{rem:dual-Fredholm}
Note that the Fredholm property means that $P_\sigma^*+\imath Q_\sigma^*$ is
also Fredholm on the dual spaces; this can also be seen directly from
the estimates; rather than being a holomorphic family, it is an
anti-holomorphic family.
The analogue of this remark also applies to the semiclassical discussion below.
\end{rem}

\begin{rem}\label{rem:inv-indep-of-s}
Note that if $s'>s\geq m>(\difford-1-\beta\im\sigma)/2$ and if
$P_\sigma-\imath Q_\sigma: \cX^s\to \cY^s$ and $P_\sigma-\imath
Q_\sigma: \cX^{s'}\to \cY^{s'}$ are both invertible, then, as
$\cX^{s'}\subset\cX^s$ and $\cY^{s'}\subset\cY^s$,
$(P_\sigma-\imath Q_\sigma)^{-1}|_{\cY^{s'}}$ agrees with $(P_\sigma-\imath
Q_\sigma)^{-1}:\cY^{s'}\to\cX^{s'}$. Moreover, as $\cY^{s'}=H^{s'-\difford+1}$ is dense in
$\cY^s$, $(P_\sigma-\imath Q_\sigma)^{-1}:\cY^{s'}\to\cX^{s'}$ determines $(P_\sigma-\imath
Q_\sigma)^{-1}:\cY^s\to\cX^s$, i.e.\ if $A:\cY^s\to\cX^s$ is
continuous and
$A|_{\cY^{s'}}$ is $(P_\sigma-\imath
Q_\sigma)^{-1}:\cY^{s'}\to\cX^{s'}$ then $A$ is $(P_\sigma-\imath
Q_\sigma)^{-1}:\cY^{s}\to\cX^{s}$.
Thus, in this sense, $(P_\sigma-\imath
Q_\sigma)^{-1}$ is independent of $s$ (satisfying $s\geq m>(\difford-1-\beta\im\sigma)/2$).
\end{rem}

\subsection{Stability}\label{subsec:stability}
We also want to understand the behavior of $P_\sigma-\imath Q_\sigma$
under perturbations. To do so, assume that $P_\sigma=P_\sigma(w)$, $Q_\sigma=Q_\sigma(w)$
depend continuously on a parameter $w\in\RR^l$, with values in
(analytic functions of $\sigma$ with values in)
$\Psi^\difford(X)$ and the principal symbols of $P_\sigma(w)$ and
$Q_\sigma(w)$ are real and independent of $\sigma$ with that of
$Q_\sigma(w)$ being non-negative. We {\em do not} assume that the principal
symbols are independent of $w$, in fact, fixing some $w_0$, we do not
even assume that for $w\neq w_0$ the other assumptions on
$P_\sigma(w)-\imath Q_\sigma(w)$ are satisfied for $w\neq w_0$. (So,
for instance, as already mentioned in Remark~\ref{rem:stable}, the structure of the radial set at
$w_0$ may drastically change for $w\neq w_0$.) However, see
Remark~\ref{rem:stable} for the most delicate part, our estimates
at $w_0$ are stable just under the assumption of continuous dependence
with values $\Psi^\difford(X)$, thus there exists $\delta_0>0$ such
that for $|w-w_0|<\delta_0$, we have uniform versions of the
estimates
\eqref{eq:unique-est}-\eqref{eq:exist-est}, i.e.\ the constant $C$ and
the orders $m$ and $N$ can be taken to be uniform in these
(independent of $w$), so e.g.
\begin{equation}\label{eq:unique-est-unif}
\|u\|_{H^s}\leq C(\|(P_\sigma(w)-\imath Q_\sigma(w))u\|_{H^{s-\difford+1}}+\|u\|_{H^m}).
\end{equation}
Thus,
$P_\sigma(w)-\imath
Q_\sigma(w):\cX^s(w)\to\cY^s(w)$ is Fredholm, depending analytically
on $\sigma$, for each $w$ with
$|w-w_0|<\delta_0$,
$\cY^s(w)=\cY^s=H^{s-\difford+1}$ is independent of $w$ (and of $\sigma$), but
$\cX^s(w)=\{u\in H^s:\ P_\sigma(w)u\in H^{s-\difford+1}\}\subset H^s$
does depend on $w$ (but not on $\sigma$).
We claim, however, that,
assuming that $(P_\sigma(w_0)-\imath Q_\sigma(w_0))^{-1}$ is
meromorphic in $\sigma$
(i.e.\ the inverse exists at
least at one point $\sigma$),
$(P_\sigma(w)-\imath Q_\sigma(w))^{-1}$ is also meromorphic in
$\sigma$ for $w$ close to $w_0$, and it depends continuously
on $w$ in the weak operator topology of $\cL(\cY^s,H^s)$, and thus in
the norm topology of $\cL(H^{s-\difford+1+\ep},H^{s-\ep})$ for $\ep>0$.

To see this, note first that if $P_{\sigma_0}(w_0)-\imath Q_{\sigma_0}(w_0):\cX^s(w_0)\to\cY^s$
is invertible, then so is $P_{\sigma}(w)-\imath
Q_{\sigma}(w):\cX^s(w)\to\cY^s$
for $w$ near $w_0$ and $\sigma$
near $\sigma_0$. Once this is shown, the meromorphy of
$(P_\sigma(w)-\imath Q_\sigma(w))^{-1}$ follows when $w$ is close to
$w_0$, with this operator
being the inverse of an analytic Fredholm family which is invertible
at a point. To see the invertibility of $P_\sigma(w)-\imath
Q_\sigma(w)$ for $w$ near $w_0$ and $\sigma$ near $\sigma_0$, first suppose
there exist sequences $w_j\to w_0$ and $\sigma_j\to\sigma_0$
such that $P_{\sigma_j}(w_j)-\imath
Q_{\sigma_j}(w_j)$ is not invertible, so either $\Ker(P_{\sigma_j}(w_j)-\imath
Q_{\sigma_j}(w_j))$ on $H^s$ or $\Ker(P_{\sigma_j}(w_j)^*+\imath
Q_{\sigma_j}(w_j)^*)$ on $(H^{s-\difford+1})^*$ is non-trivial in view of
the preceding Fredholm discussion. By passing to a subsequence, we may
assume that the same one of these two possibilities holds for all $j$, and
as the case of the adjoint is completely analogous, we may also assume
that $\Ker(P_{\sigma_j}(w_j)-\imath Q_{\sigma_j}(w_j))$ on $H^s$ is
non-trivial for all $j$. Now, if $u_j\in H^s$, $\|u_j\|_{H^s}=1$, and
$(P_{\sigma_j}(w_j)-\imath Q_{\sigma_j}(w_j))u_j=0$ then
\eqref{eq:unique-est-unif} gives $1\leq C\|u_j\|_{H^m}$. Now, $u_j$
has a weakly convergent subsequence in $H^s$ to some $u_0\in H^s$, which is thus
norm-convergent in $H^m$; so $(P_{\sigma_0}(w_0)-\imath
Q_{\sigma_0}(w_0))u_0=0$. Since $1\leq C\|u_j\|_{H^m}$, and the
subsequence is norm-convergent in $H^m$, $u_0\neq 0$, and thus
$\Ker(P_{\sigma_0}(w_0)-\imath Q_{\sigma_0}(w_0))$ on $H^s$ is
non-trivial, so $P_{\sigma_0}(w_0)-\imath Q_{\sigma_0}(w_0)$ is not
invertible, proving our claim.

So suppose now that $f_j\in H^{s-\difford+1}$ and
$\|f_j\|_{H^{s-\difford+1}}\leq 1$. Let $w_j\to w_0$,
$\sigma_j\to\sigma_0$ (with $w_j$ sufficiently close to $w_0$,
$\sigma_j$ sufficiently close to $\sigma_0$ for invertibility), and let $u_j=(P_{\sigma_j}(w_j)-\imath
Q_{\sigma_j}(w_j))^{-1}f_j$. Suppose first that $u_j$ is not bounded
in $H^s$, and let $v_j=\frac{u_j}{\|u_j\|_{H^s}}$. Then by \eqref{eq:unique-est-unif},
$1\leq C(\|u_j\|_{H^s}^{-1}+\|v_j\|_{H^m})$, so for $j$ sufficiently
large, $\|v_j\|_{H^m}\geq\frac{1}{2C}$. On the other hand, a
subsequence $v_{j_r}$ of $v_j$ converges weakly to some $v_0$ in $H^s$, and
$\frac{f_{j_r}}{\|u_{j_r}\|_{H^s}}=(P_{\sigma_{j_r}}(w_{j_r})-\imath Q_{\sigma_{j_r}}(w_{j_r}))v_{j_r}\to
(P_{\sigma_0}(w_0)-\imath Q_{\sigma_0}(w_0)) v_0$ weakly in $H^{s-\difford}$,
so as the left hand side converges to $0$ in $H^{s-\difford+1}$,
$(P_{\sigma_0}(w_0)-\imath Q_{\sigma_0}(w_0)) v_0=0$. As $v_{j_r}\to
v_0$ in norm in $H^m$, we deduce that $v_0\neq 0$, contradicting the
invertibility of $P_{\sigma_0}(w_0)-\imath Q_{\sigma_0}(w_0)$. Thus,
$u_j$ is uniformly bounded in $H^s$.

Next, suppose that $f_j\to f$ in
$H^{s-\difford+1}$, so $u_j$ is bounded in $H^s$ by what we just showed. Then any subsequence
of $u_j$ has a weakly convergent subsequence $u_{j_r}$ with some limit $u_0\in
H^s$. Then $f_{j_r}=(P_{\sigma_{j_r}}(w_{j_r})-\imath Q_{\sigma_{j_r}}(w_{j_r}))u_{j_r}\to
(P_{\sigma_0}(w_0)-\imath Q_{\sigma_0}(w_0)) u_0$ weakly in $H^{s-\difford}$,
so $(P_{\sigma_0}(w_0)-\imath Q_{\sigma_0}(w_0)) u_0=f$. By the
injectivity of $P_{\sigma_0}(w_0)-\imath Q_{\sigma_0}(w_0)$, $u_0$ is
thus independent of the subsequence of $u_j$, i.e.\ every subsequence
of $u_j$ has a subsequence converging weakly to $u_0$, and thus $u_j$
converges weakly to $u_0$ in $H^s$. This gives the convergence of 
$(P_{\sigma}(w)-\imath Q_{\sigma}(w))^{-1}$ to
$(P_{\sigma_0}(w_0)-\imath Q_{\sigma_0}(w_0))^{-1}$ in the weak
operator topology on $\cL(\cY^s,H^s)$ as $\sigma\to\sigma_0$ and $w\to
w_0$, and thus in the norm topology
on $\cL(\cY^{s+\ep},H^{s-\ep})$ for $\ep>0$.

\subsection{Semiclassical estimates}\label{subsec:semi-abstract}
For reasons of showing meromorphy of the inverse, and also for wave propagation,
we also want to know the $|\sigma|\to\infty$ asymptotics
of $P_\sigma-\imath\sigma$ and $P_\sigma^*+\imath Q_\sigma^*$; here
$P_\sigma,Q_\sigma$ are operators with a large parameter.
As discussed earlier, this can be translated into a
semiclassical problem, i.e.\ one obtains families of operators $P_{h,z}$,
with $h=|\sigma|^{-1}$, and $z$ corresponding to $\sigma/|\sigma|$ in the unit
circle in $\Cx$. As usual, we multiply through by $h^\difford$ for convenient notation
when we define $P_{h,z}$:
$$
P_{h,z}=h^\difford P_{h^{-1}z}\in\Psihcl^\difford(X).
$$

From now on, we merely require $P_{h,z},Q_{h,z}\in\Psihcl^\difford(X)$.
Then the semiclassical principal symbol $p_{\semi,z}$, $z\in O\subset\Cx$,
$0\notin\overline{O}$ compact, which is a function
on $T^*X$, has limit $p$ at infinity in the fibers of the cotangent bundle, so is
in particular real in the limit. More precisely,
as in the classical setting, but now $\rhot$ made smooth at the zero section as well
(so is not homogeneous there), we consider
$$
\rhot^\difford p_{\semi,z}
\in\CI(\overline{T}^*X\times O);
$$
then $\rhot^\difford p_{\semi,z}|_{S^*X\times O}=\rhot^\difford p$, where
$S^*X=\pa \overline{T}^*X$.
We assume that $p_{\semi,z}$ and $q_{\semi,z}$ are real when $z$ is
real.
We shall be interested in $\im z\geq -Ch$, which
corresponds to $\im\sigma\geq -C$ (recall that $\im\sigma\gg 0$ is where
we expect holomorphy). Note that when $\im z=\cO(h)$, $\im p_{\semi,z}$ still
vanishes, as the contribution of $\im z$ is semiclassically subprincipal in view
of the order $h$ vanishing.

We write the semiclassical characteristic set of $p_{\semi,z}$ as
$\Sigma_{\semi,z}$, and sometimes drop the $z$ dependence and write
$\Sigma_{\semi}$ simply; assume
that
$$
\Sigma_{\semi}=\Sigma_{\semi,+}\cup\Sigma_{\semi,-},
\ \Sigma_{\semi,+}\cap\Sigma_{\semi,-}=\emptyset,
$$
$\Sigma_{\semi,\pm}$ are relatively open
in $\Sigma_{\semi}$, and
$$
\pm\im p_{\semi,z}\geq 0\Mand\mp q_{\semi,z}\geq 0\ \text{near}\ \Sigma_{\semi,\pm}.
$$

\begin{figure}[ht]
\begin{center}
\mbox{\epsfig{file=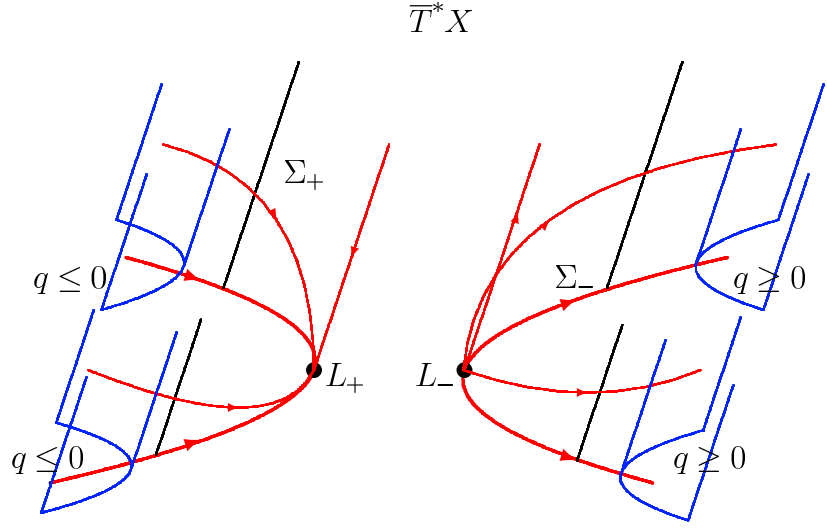}}
\end{center}
\caption{The components $\Sigma_{\semi,\pm}$ of the semiclassical characteristic
set in $\overline{T}^*X$, which are now two-dimensional in the figure.
The cosphere bundle is the horizontal plane at the bottom
of the picture; the intersection of this figure with the cosphere bundle is what
is shown on Figure~\ref{fig:microloc-damp1}.
The submanifolds $L_\pm$ are still points,
with $L_-$ a source, $L_+$ a sink. The red lines are bicharacteristics, with the
thick ones inside $S^*X=\pa\overline{T}^*X$. The blue regions near the edges
show the absorbing region, i.e.\ the support of $q$. For $P_{h,z}-\imath Q_{h,z}$,
the estimates are
always propagated away from $L_\pm$ towards the support of $q$, so in the
direction of the Hamilton flow in $\Sigma_{\semi,-}$, and in the direction
opposite of the Hamilton flow in $\Sigma_{\semi,+}$; for $P_{h,z}^*+\imath Q_{h,z}^*$,
the directions
are reversed.}
\label{fig:microloc-damp-sc1}
\end{figure}

Microlocal results analogous to the classical results also
exist in the semiclassical setting. In the interior of $\overline{T}^*X$,
i.e.\ in $T^*X$, only the microlocal elliptic, real principal type and
complex absorption estimates are relevant. At $L_\pm\subset S^*X$ we in addition need
the analogue of Propositions~\ref{prop:micro-out}-\ref{prop:micro-in}.
As these are the only non-standard estimates, though they are
very similar to estimates of \cite{Vasy-Zworski:Semiclassical}, where, however,
only global estimates were stated, we explicitly state these here and indicate
the very minor changes needed in the proof compared to
Propositions~\ref{prop:micro-out}-\ref{prop:micro-in}.

\begin{prop}
For all $N$, for $s\geq m>(\difford-1-\beta\im\sigma)/2$, $\sigma=h^{-1}z$,
and for all $A,B,G\in\Psih^0(X)$ such that
$\WFh'(G)\cap\WFh'(Q_\sigma)=\emptyset$, $A$ elliptic at $L_\pm$,
and forward (or backward) bicharacteristics
from $\WFh'(B)$ tend to $L_\pm$,
with closure in the elliptic set of $G$, one has estimates
\begin{equation}\label{eq:radial-out-h}
Au\in H^m_h\Rightarrow
\|Bu\|_{H^s_h}\leq C(h^{-1}\|GP_{h,z} u\|_{H^{s-\difford+1}_h}+h\|u\|_{H^{-N}_h}),
\end{equation}
where, as usual, $GP_{h,z}u\in H^{s-\difford+1}_h$ and $u\in H^{-N}_h$ are assumptions implied
by the right hand side.
\end{prop}

\begin{prop}
For $s<(\difford-1-\beta\im\sigma)/2$, for all $N$, $\sigma=h^{-1}z$,
and for all $A,B,G\in\Psih^0(X)$ such that
$\WFh'(G)\cap\WFh'(Q_\sigma)=\emptyset$, $B,G$ elliptic at $L_\pm$,
and forward (or backward) bicharacteristics
from $\WFh'(B)\setminus L_\pm$ reach $\WFh'(A)$,
while remaining in the elliptic set of $G$, one has estimates
\begin{equation}\label{eq:radial-in-h}
\|Bu\|_{H^s_h}\leq C(h^{-1}\|GP_{h,z}^* u\|_{H^{s-\difford+1}_h}+\|Au\|_{H^{s}_h}+h\|u\|_{H^{-N}_h}).
\end{equation}
\end{prop}

\begin{proof}
We just need to localize in $\tilde\rho$ in addition to $\rho_0$;
such a localization in the
classical setting is implied by working on $S^*X$ or with homogeneous
symbols.
We achieve this by modifying the localizer $\phi$ in the commutant
constructed in the proof of Propositions~\ref{prop:micro-out}-\ref{prop:micro-in}.
As already remarked, the proof is
much like at radial points in semiclassical scattering on asymptotically
Euclidean spaces, studied by Vasy and Zworski \cite{Vasy-Zworski:Semiclassical}, but
we need to be more careful about localization in $\rho_0$ and $\tilde\rho$ as we
are assuming less about the structure.

First, note that $L_\pm$ is defined by $\tilde\rho=0$, $\rho_0=0$, so $\tilde\rho^2+
\rho_0$ is a quadratic defining function of $L_\pm$.
Thus, let $\phi\in\CI_c(\RR)$ be identically $1$ near $0$,
$\phi'\leq 0$ and
$\phi$ supported sufficiently close to $0$ so that
\begin{equation*}\begin{split}
&\tilde\rho^2+\rho_0\in\supp d\phi\Rightarrow
\mp \tilde\rho^{\difford-1}(\sH_p\rho_0+2\tilde\rho\sH_p\tilde\rho)>0
\end{split}\end{equation*}
and
\begin{equation*}\begin{split}
&\tilde\rho^2+\rho_0\in\supp \phi\Rightarrow
\mp \tilde\rho^{\difford-2}\sH_p\tilde\rho>0.
\end{split}\end{equation*}
Such $\phi$ exists by
\eqref{eq:weight-definite} and \eqref{eq:rho-0-property} as
$$
\mp\tilde\rho(\sH_p\rho_0+2\tilde\rho\sH_p\tilde\rho)\geq \beta_1\rho_0+
2\beta_0\tilde\rho^2-\cO((\tilde\rho^2+
\rho_0)^{3/2}).
$$
Then let $c$ be given by
$$
c=\phi(\rho_0+\tilde\rho^2)\tilde\rho^{-s+(\difford-1)/2},\qquad
c_\ep=c(1+\ep \tilde\rho^{-1})^{-\delta}.
$$
The rest of the proof proceeds exactly as for
Propositions~\ref{prop:micro-out}-\ref{prop:micro-in}.
\end{proof}

We first show that under extra assumptions, giving semiclassical ellipticity for
$\im z$ bounded away from $0$, we have non-trapping estimates.
So assume that for $|\im z|>\ep>0$, $p_{\semi,z}$ is semiclassically elliptic on $T^*X$
(but not necessarily at $S^*X=\pa\overline{T}^*X$, where the standard principal
symbol $p$ already describes the behavior). Also assume that $\pm\im p_{\semi,z}\geq 0$
near the classical characteristic set $\Sigma_{\semi,\pm}\subset
S^*X$. Assume also that
$p_{\semi,z}-\imath q_{\semi,z}$ is elliptic for $|\im z|>\ep>0$,
$h^{-1}z\in\Omega$, and $\mp\re q_{\semi,z}\geq 0$
near the classical characteristic set $\Sigma_{\semi,\pm}\subset
S^*X$.
Then the semiclassical version of the classical results (with ellipticity
in $T^*X$ making these trivial except at $S^*X$) apply. Let $H^s_h$
denote the usual semiclassical function spaces. Then,
on the one hand,
for any $s\geq m>(\difford-1-\beta\im z/h)/2$, $h<h_0$,
\begin{equation}\label{eq:unique-est-h-sector}
\|u\|_{H^s_h}\leq Ch^{-1}(\|(P_{h,z}-\imath Q_{h,z})u\|_{H^{s-\difford+1}_h}+h^2\|u\|_{H^m_h}),
\end{equation}
and on the other hand, for any $N$ and for any $s<(\difford-1+\beta\im z/h)/2$, $h<h_0$,
\begin{equation}\label{eq:exist-est-h-sector}
\|u\|_{H^s_h}\leq Ch^{-1}(\|(P_{h,z}^*+\imath Q_{h,z}^*)u\|_{H^{s-\difford+1}_h}+h^2\|u\|_{H^{-N}_h}).
\end{equation}
The $h^2$ term can be absorbed in the left hand side for sufficiently small $h$,
so we automatically obtain
invertibility of $P_{h,z}-\imath Q_{h,z}$.

In particular, $P_{h,z}-\imath Q_{h,z}$ is invertible for $h^{-1}z\in\Omega$
with $\im z>\ep>0$ and $h$ small,
i.e.\ $P_\sigma-\imath Q_\sigma$ is such for $\sigma\in\Omega$ in a
cone bounded away from the real axis with $\im\sigma$ sufficiently large, proving the
meromorphy of $P_\sigma-\imath Q_\sigma$ under these extra
assumptions. Note also that for instance
$$
\|u\|^2_{H_{|\sigma|^{-1}}^1}=\|u\|_{L^2}^2+|\sigma|^{-2}\|du\|_{L^2}^2,\qquad
\|u\|_{H_{|\sigma|^{-1}}^0}=\|u\|_{L^2},
$$
(with the norms with respect to any positive definite inner product).

\begin{thm}\label{thm:classical-absorb-sector}
Let $P_\sigma$, $Q_\sigma$, $\beta$, $\Cx_s$ be as above,
and $\cX^s$, $\cY^s$ as in \eqref{eq:XY-def}.
Then, for $\sigma\in\Cx_s\cap \Omega$,
$$
P_\sigma-\imath Q_\sigma:\cX^s\to\cY^s
$$
has a meromorphic inverse
$$
R(\sigma):\cY^s\to\cX^s.
$$
Moreover, for all $\ep>0$ there is $C>0$
such that it is invertible in $\im\sigma>C+\ep|\re \sigma|$, $\sigma\in\Omega$, and
{\em non-trapping estimates} hold:
$$
\|R(\sigma)f\|_{H^s_{|\sigma|^{-1}}}
\leq C'|\sigma|^{-\difford+1}\|f\|_{H^{s-1}_{|\sigma|^{-1}}}.
$$
\end{thm}

\begin{rem}\label{rem:off-real-not-imp}
In fact, the large $\im\sigma$ behavior of $P_\sigma-\imath Q_\sigma$
does not matter for our main results, except the support conclusion of
the existence part of Lemma~\ref{lemma:Mellin-expand}, and the
analogous statement in its
consequences, Proposition~\ref{prop:Mellin-expand} and
Corollary~\ref{cor:Mellin-expand}. In particular, when the solution is
known to exist in a weighted b-Sobolev space, the large $\im\sigma$
behavior is not used at all; for existence the only loss would be that
the solution would not have the stated support property (which is desirable to have in the wave equation setting).
The behavior when $\im\sigma$ is bounded, but
$\re\sigma$ goes to infinity, which we now discuss, is, on the other
hand, more crucial, and depends on the more delicate Hamiltonian dynamics (but not on the
ellipticity for non-real $z$ which we just discussed).
\end{rem}

To deal with estimates for $z$ (almost) real, we need additional assumptions.
We make
the non-trapping assumption into a definition:

\begin{Def}\label{Def:non-trap}
We say that $p_{\semi,z}-\imath q_{\semi,z}$ is {\em semiclassically non-trapping} if
the bicharacteristics from any point in $\Sigma_{\semi}\setminus (L_+\cup L_-)$
flow to $\elliptic(q_{\semi,z})\cup L_+$ (i.e.\ either enter $\elliptic(q_{\semi,z})$ at some finite
time, or tend to $L_+$) in the forward direction, and to $\elliptic(q_{\semi,z})\cup L_-$
in the backward direction.
\end{Def}

\begin{rem}
The part of the semiclassically non-trapping property on $S^*X$ is just
the classical non-trapping property; thus, the point is its extension into
to the interior $T^*X$ of $\overline{T}^*X$. Since the classical principal
symbol is assumed real, there did not need to be any additional
restrictions on $\im p_{\semi,z}$ there.
\end{rem}

The semiclassical
version of all of the above
estimates are then applicable for $\im z\geq -Ch$, and one obtains on the one hand that
for any $s\geq m>(\difford-1-\beta\im z/h)/2$, $h<h_0$,
\begin{equation}\label{eq:unique-est-h}
\|u\|_{H^s_h}\leq Ch^{-1}(\|(P_{h,z}-\imath Q_{h,z})u\|_{H^{s-\difford+1}_h}+h^2\|u\|_{H^m_h}),
\end{equation}
On the other hand, for any $N$ and for any $s<(\difford-1+\beta\im z/h)/2$, $h<h_0$,
\begin{equation}\label{eq:exist-est-h}
\|u\|_{H^s_h}\leq Ch^{-1}(\|(P_{h,z}^*+\imath Q_{h,z}^*)u\|_{H^{s-\difford+1}_h}+h^2\|u\|_{H^{-N}_h}),
\end{equation}
The $h^2$ term can again be absorbed in the left hand side for sufficiently
small $h$, so we automatically obtain
invertibility of $P_{h,z}-\imath Q_{h,z}$.

Translated into the classical setting this gives

\begin{thm}\label{thm:classical-absorb-strip}
Let $P_\sigma$, $Q_\sigma$, $\Cx_s$, $\beta$ be as above, in particular
semiclassically non-trapping,
and $\cX^s$, $\cY^s$ as in \eqref{eq:XY-def}. Let $C>0$. Then there exists $\sigma_0$
such that
$$
R(\sigma):\cY^s\to\cX^s,
$$
is holomorphic
in $\{\sigma\in\Omega:\ \im\sigma>-C,\ |\re\sigma|>\sigma_0\}$, assumed to be a subset
of $\Cx_s$,
and {\em non-trapping estimates}
$$
\|R(\sigma)f\|_{H^s_{|\sigma|^{-1}}}
\leq C'|\sigma|^{-\difford+1}\|f\|_{H^{s-\difford+1}_{|\sigma|^{-1}}}
$$
hold.
For $s=1$, $\difford=2$ this states that for $|\re\sigma|>\sigma_0$, $\im\sigma>-C$,
$$
\|R(\sigma)f\|_{L^2}^2+|\sigma|^{-2}\|dR(\sigma)\|_{L^2}^2
\leq C''|\sigma|^{-2}\|f\|_{L^2}^2.
$$
\end{thm}

Analogous results work for other Sobolev spaces; $H^1_h$ was chosen above
for simplicity.

\begin{rem}\label{rem:off-spectrum-est}
We emphasize that if semiclassical non-trapping assumptions
are made, but not ellipticity for $z$ non-real,
meromorphy
still follows by taking $h$ small and $z>0$, say, to get a point of invertibility. This
is useful because one can eliminate the need to conjugate by a factor to
induce such ellipticity when the resulting estimate is irrelevant. (Mostly estimates
in strips for $h^{-1}z$, i.e.\ $\cO(h)$ estimates for $z$, matter.) However, there
is a cost: while only finitely many poles can lie in any strip $|\im\sigma|<C$,
there is no need for this statement to hold if we allow $\im \sigma>-C$.
Since,
for the application to
the wave equation, $\im\sigma$ depends
on the a priori growth rate of the solution $u$ which we are Mellin transforming,
this would mean that depending on the a priori growth rate one could get
more (faster growing) terms in the expansion of $u$ if one relaxes the growth condition on $u$.
\end{rem}

While we stated just the global results here, of course one has microlocal estimates
for the solution. In particular we have the following, stated in the semiclassical
language, as immediate from the estimates used to derive from the Fredholm
property:

\begin{thm}\label{thm:semicl-outg}
Let $P_\sigma$, $Q_\sigma$, $\beta$ be as above, in particular
semiclassically non-trapping,
and $\cX^s$, $\cY^s$ as in \eqref{eq:XY-def}. 

For $\re z>0$ and $s'>s$, the resolvent $R_{h,z}$ is
{\em semiclassically outgoing with a loss of $h^{-1}$} in
the sense that if $\alpha\in \overline{T}^*X\cap\Sigma_{\semi,\pm}$, and if for the forward ($+$),
resp.\ backward ($-$), bicharacteristic $\gamma_\pm$, from $\alpha$,
$\WFh^{s'-\difford+1}(f)\cap\overline{\gamma_\pm}=\emptyset$ then $\alpha\notin\WFh^{s'}(hR_{h,z}f)$.

In fact, for any $s'\in\RR$, the resolvent $R_{h,z}$ extends to $f\in H^{s'}_h(X)$,
with non-trapping bounds, provided that $\WFh^s(f)\cap(L_+\cup L_-)=\emptyset$.
The semiclassically outgoing with a loss of $h^{-1}$ result holds for such $f$ and
$s'$ as well.
\end{thm}

\begin{proof}
The only part that is not immediate by what has been discussed is the last
claim. This follows immediately, however, by microlocal solvability in arbitrary
ordered Sobolev spaces away from the radial points (i.e.\ solvability modulo
$\CI$, with semiclassical estimates), combined with our preceding
results to deal
with this smooth remainder plus the contribution near $L_+\cup L_-$, which
are assumed to be in $H^s_h(X)$.
\end{proof}

This result is needed for gluing constructions as in \cite{Datchev-Vasy:Gluing-prop},
namely polynomially bounded trapping with appropriate microlocal geometry
can be glued to our resolvent. Furthermore,
it gives non-trapping estimates microlocally away from the trapped set provided
the overall (trapped) resolvent is polynomially bounded as shown by Datchev
and Vasy \cite{Datchev-Vasy:Trapped}.

\begin{Def}\label{Def:mild-trap}
Suppose $K_\pm\subset T^*X$ is compact, and
$O_\pm$ is a neighborhood of $K$ with compact
closure and $O_\pm\cap\Sigma_{\semi}\subset\Sigma_{\semi,\pm}$.
We say that $p_{\semi,z}$ is {\em semiclassically locally
mildly trapping of order $\varkappa$ in a $C_0$-strip}
if
\begin{enumerate}
\item
there is a function\footnote{\label{footnote:convex-escape} For $\ep>0$,
such a function $F$ provides an escape function,
$\tilde F=e^{-C F}\sH_{p_{\semi,z}}F$ on the set where $1+\ep\leq F\leq 2-\ep$.
Namely,
by taking $C>0$ sufficiently large,
$\sH_{p_{\semi,z}}\tilde F<0$ there; thus, every bicharacteristic must leave the
compact set $F^{-1}([1+\ep,2-\ep])$ in finite time. However, the existence
of such an $F$ is a stronger statement than that of an escape function: a
bicharacteristic segment cannot leave $F^{-1}([1+\ep,2-\ep])$ via the boundary
$F=2-\ep$ in both directions since $F$ cannot have a local minimum.
This is exactly the way this condition is used in \cite{Datchev-Vasy:Gluing-prop}.} $F\in\CI(T^*X)$, $F\geq 2$ on $K_\pm$,
$F\leq 1$ on $T^*X\setminus O_\pm$, and for $\alpha\in (O_\pm\setminus K_\pm)\cap
\Sigma_{\semi,\pm}$,
$(\sH_{p_{\semi,z}} F)(\alpha)=0$ implies $(\sH_{p_{\semi,z}}^2 F)(\alpha)<0$; and
\item
there exists $\tilde Q_{h,z}\in\Psih(X)$ with $\WFh'(\tilde Q_{h,z})\cap K_\pm=\emptyset$,
$\mp\tilde q_{\semi,z}\geq 0$ near $\Sigma_{\semi,\pm}$,
$\tilde q_{\semi,z}$ elliptic on $\Sigma_{\semi}\setminus (O_+\cup O_-)$ and $h_0>0$
such that if $\im z>-C_0 h$ and $h<h_0$ then
\begin{equation}\label{eq:poly-trapping-loss}
\|(P_{h,z}-\imath \tilde Q_{h,z})^{-1}f\|_{H^s_h}
\leq C h^{-\varkappa-1}\|f\|_{H^{s-\difford+1}_h},\ f\in H^{s-\difford+1}_h.
\end{equation}
\end{enumerate}

We say that $p_{\semi,z}-\imath q_{\semi,z}$ is  {\em semiclassically
mildly trapping of order $\varkappa$ in a $C_0$-strip} if it is
semiclassically locally
mildly trapping of order $\varkappa$ in a $C_0$-strip and if
the bicharacteristics from any point in
$\Sigma_{\semi,+}\setminus(L_+\cup K_+)$
flow to $\{q_{\semi,z}<0\}\cup O_+$ in the backward direction and
to $\{q_{\semi,z}<0\}\cup O_+\cup L_+$ in the forward direction, while
the bicharacteristics from any point in
$\Sigma_{\semi,-}\setminus(L_-\cup K_-)$
flow to $\{q_{\semi,z}>0\}\cup O_-\cup L_-$ in the backward direction and
to $\{q_{\semi,z}>0\}\cup O_-$ in the forward direction.
\end{Def}

An example\footnote{Condition (i) follows by letting $\tilde F=\varphi_+^{2\kappa}
+\varphi_-^{2\kappa}$
with the notation of \cite[Lemma~4.1]{Wunsch-Zworski:Resolvent}; so
$$
\sH_p^2\tilde F=4\kappa^2\big((c_+^4-\kappa^{-1}c_+\sH_p c_+)\varphi_+^{2\kappa}
+4\kappa^2\big((c_-^4+\kappa^{-1}c_-\sH_p c_-)\varphi_-^{2\kappa}
$$
near the trapped set, $\varphi_+=0=\varphi_-$.
Thus, for sufficiently large $\kappa$, $\sH_p\tilde F>0$ outside $\tilde F=0$.
Since $\tilde F=0$
defines the trapped set, in order to satisfy Definition~\ref{Def:mild-trap}, writing
$K$ and $O$ instead of $K_\pm$ and $O_\pm$,
one lets $K=\{\tilde F\leq\alpha\}$, $O=\{\tilde F<\beta\}$ for suitable
(small) $\alpha$ and $\beta$, $\alpha<\beta$, and
takes $F=G\circ\tilde F$ with $G$ strictly decreasing, $G|_{[0,\alpha]}>2$,
$G|_{[\beta,\infty)}<1$.}
of locally mild trapping
is hyperbolic trapping, studied by Wunsch and Zworski
\cite{Wunsch-Zworski:Resolvent}, which is of order $\varkappa$
for some $\varkappa>0$. Note that (i) states that the sets $K_c=\{F\geq c\}$, $1<c<2$,
are bicharacteristically convex in $O_\pm$, for by (i) any critical points of $F$ along
a bicharacteristic are strict local maxima.

As a corollary, we
have:

\begin{thm}\label{thm:classical-absorb-glued}
Let $P_\sigma$, $Q_\sigma$, $\Cx_s$, $\beta$ be as above, satisfying mild
trapping assumptions with order $\varkappa$ estimates in a $C_0$-strip,
and $\cX^s$, $\cY^s$ as in \eqref{eq:XY-def}.
Then there exists $\sigma_0$
such that
$$
R(\sigma):\cY^s\to\cX^s,
$$
is holomorphic
in $\{\sigma\in\Omega:\im\sigma>-C_0,\ |\re\sigma|>\sigma_0\}$, assumed to be
a subset of $\Cx_s$,
and
\begin{equation}\label{eq:high-energy-glued}
\|R(\sigma)f\|_{H^s_{|\sigma|^{-1}}}
\leq C'|\sigma|^{\varkappa-\difford+1}\|f\|_{H^{s-\difford+1}_{|\sigma|^{-1}}}.
\end{equation}
Further, if one has logarithmic loss in \eqref{eq:poly-trapping-loss}, i.e.\ if $h^{-\varkappa}$ can be replaced by $\log (h^{-1})$, for $\sigma\in\RR$, \eqref{eq:high-energy-glued}
also holds with a logarithmic loss, i.e.\ $|\sigma|^{\varkappa}$ can be replaced
by $\log|\sigma|$ for $\sigma$ real.
\end{thm}

\begin{proof}
This is an almost immediate consequence of \cite{Datchev-Vasy:Gluing-prop}. To get
into that setting, we replace $Q_{h,z}$ by $Q'_{h,z}$ with $\WFh'(Q_{h,z}-Q'_{h,z})
\subset O_+\cup O_-$ and $Q'_{h,z}$ elliptic
on $K_+\cup K_-$, with $\mp q'_{\semi,z}\geq 0$ on $\Sigma_{\semi,\pm}$.
Then $P_{h,z}-\imath Q'_{h,z}$ is semiclassically
non-trapping in the sense discussed earlier, so all of our estimates apply. With
the polynomial resolvent bound assumption on $P_{h,z}-\imath\tilde Q_{h,z}$,
and the function $F$ in place of $x$ used in \cite{Datchev-Vasy:Gluing-prop},
the results of \cite{Datchev-Vasy:Gluing-prop} apply, taking into account
Theorem~\ref{thm:semicl-outg} and \cite[Lemma~5.1]{Datchev-Vasy:Gluing-prop}.
Note that the results of \cite{Datchev-Vasy:Gluing-prop} are stated in a
slightly different context for convenience, namely the function $x$ is defined
on the manifold $X$ and not on $T^*X$, but this is a minor issue:
the results and proofs apply verbatim
in our setting.
\end{proof}

\section{Mellin transform and Lorentzian b-metrics}\label{sec:Mellin-Lorentz}

\subsection{The Mellin transform}\label{subsec:Mellin}
In this section we discuss the basics of Melrose's b-analysis
on an $n$-dimensional manifold with boundary $\bM$, where the boundary
is denoted by $X$.
We refer
to \cite{Melrose:Atiyah} as a general reference. In the main cases of interest here,
the b-geometry is trivial, and $\bM=X\times[0,\infty)_\tau$ with respect
to some (almost) canonical (to the problem) product decomposition. Thus,
the reader should feel comfortable in trivializing all the statements below
with respect to this decomposition. In this trivial case, the main result
on the Mellin transform, Lemma~\ref{lemma:Mellin-expand} is fairly standard,
with possibly different notation of the function spaces;
we include it here for completeness.

First, recall that the Lie algebra of b-vector fields, $\Vb(\bM)$ consists of
$\CI$ vector fields on $\bM$ tangent to the boundary. In local coordinates
$(\tau,y)$, such that $\tau$ is a boundary defining function, they are
of the form $a_n\tau\pa_\tau+\sum_{j=1}^{n-1}a_j \pa_{y_j}$, with $a_j$
arbitrary $\CI$ functions. Correspondingly, they are the set of all smooth
sections of a $\CI$ vector bundle, $\Tb\bM$, with local basis
$\tau\pa_\tau,\pa_{y_1},\ldots,\pa_{y_{n-1}}$. The dual bundle, $\Tb^*\bM$, thus
has a local basis given by $\frac{d\tau}{\tau},dy_1,\ldots,dy_{n-1}$. All tensorial
constructions, such as form and density bundles, go through as usual.

The natural bundles related to the boundary are
reversed in the b-setting. Thus, the b-normal bundle of the boundary $X$
is well-defined as
the span of $\tau\pa_\tau$ defined using any coordinates, or better yet, as
the kernel of the natural map $\iota:\Tb_m\bM\to T_m\bM$, $m\in X$, induced
via the inclusion $\Vb(\bM)\to\Vf(\bM)$, so
$$
a_n\tau\pa_\tau+\sum_{j=1}^{n-1}a_j \pa_{y_j}\mapsto \sum_{j=1}^{n-1}a_j \pa_{y_j},
\qquad a_j\in\RR.
$$
Its annihilator
in $\Tb^*_m\bM$ is called
the b-cotangent bundle of the boundary; in local coordinates $(\tau,y)$
it is spanned by $dy_1,\ldots,dy_{n-1}$. Invariantly, it is the image of $T^*_m\bM$
in $\Tb^*_m\bM$ under the adjoint of the tangent bundle map $\iota$; as this
has kernel $N^*_mX$, $\Tb^*_m\bM$ is naturally
identified with $T^*_mX=T^*_m\bM/N^*_m X$.

The algebra of differential operators generated by $\Vb(\bM)$ over $\CI(\bM)$
is denoted $\Diffb(\bM)$; in local coordinates as above, elements of
$\Diffb^{\difford}(\bM)$ are of the form
$$
\cP=\sum_{j+|\alpha|\leq\difford} a_{j\alpha} (\tau D_\tau)^j D_y^\alpha
$$
in the usual multiindex notation, $\alpha\in\Nat^{n-1}$, with $a_{j\alpha}\in\CI(\bM)$.
Writing b-covectors as
$$
\sigma\,\frac{d\tau}{\tau}+\sum_{j=1}^{n-1} \eta_j\,dy_j,
$$
we obtain canonically dual coordinates to $(\tau,y)$, namely $(\tau,y,\sigma,\eta)$
are local coordinates on $\Tb^*\bM$. The principal symbol of $\cP$ is
\begin{equation}\label{eq:b-symbol}
\tilde p=\sigma_{\bl,\difford}(\cP)=\sum_{j+|\alpha|=\difford} a_{j\alpha} \sigma^j \eta^\alpha;
\end{equation}
it is a $\CI$ function, which is a homogeneous polynomial of degree $\difford$ in
the fibers, on $\Tb^*\bM$. Its Hamilton vector field, $\sH_{\tilde p}$,
is a $\CI$ vector field, which is just the extension of the standard Hamilton vector
field from $\bM^\circ$, is
homogeneous of degree $\difford-1$,
on $\Tb^*\bM$, and it is tangent to $\Tb^*_X\bM$. Explicitly, as a change of variables
shows, in local coordinates,
\begin{equation}\label{eq:b-Ham}
\sH_{\tilde p}=(\pa_\sigma\tilde p)(\tau\pa_\tau)+\sum_j(\pa_{\eta_j}\tilde p)\pa_{y_j}
-(\tau\pa_\tau\tilde p)\pa_\sigma-\sum_j(\pa_{y_j}\tilde p)\pa_{\eta_j},
\end{equation}
so the restriction of $\sH_{\tilde p}$ to $\tau=0$ is
\begin{equation}\label{eq:b-Ham-rest}
\sH_{\tilde p}|_{\Tb^*_X\bM}
=\sum_j(\pa_{\eta_j}\tilde p)\pa_{y_j}-\sum_j(\pa_{y_j}\tilde p)\pa_{\eta_j},
\end{equation}
and is thus tangent to the fibers (identified with $T^*X$)
of $\Tb^*_XM$ over $\Tb^*_X M/T^*X$ (identified with $\RR_\sigma$).

We next want to define normal operator\footnote{In fact, $\cP\in\Psib^\difford(\bM)$
works similarly.} of $\cP\in\Diffb^\difford(\bM)$,
obtained by freezing coefficients at $X=\pa\bM$. To do this naturally,
we want to extend the `frozen operator' to one invariant under dilations in the fibers of
the inward pointing normal bundle ${}_+N(X)$ of $X$; see
\cite[Equation~(4.91)]{Melrose:Atiyah}. The latter can always be trivialized by the
choice of an inward-pointing vector field $V$, which in turn fixes the
differential of a boundary defining function $\tau$ at $X$ by $V\tau|_X=1$;
given such a choice we can
identify ${}_+N(X)$ with a product
$$
\bM_\infty=X\times[0,\infty)_\tau,
$$
with the normal operator being invariant
under dilations in $\tau$. Then for $m=(x,\tau)$,
$\Tb_m \bM_\infty$ is identified with $\Tb_{(x,0)}\bM$.

On $\bM_\infty$ operators of the form
$$
\sum_{j+|\alpha|\leq\difford} a_{j\alpha}(y) (\tau D_\tau)^j D_y^\alpha,
$$
i.e.\ $a_{j\alpha}\in\CI(X)$, are invariant under the $\RR^+$-action on $[0,\infty)_\tau$;
its elements are denoted by $\DiffbI(\bM_\infty)$. The {\em normal operator}
of $\cP\in\Diffb^{\difford}(\bM)$ is given by freezing the coefficients
at $X$:
$$
N(\cP)=
\sum_{j+|\alpha|\leq\difford} a_{j\alpha}(0,y) (\tau D_\tau)^j D_y^\alpha\in\DiffbI^{\difford}(\bM_\infty).
$$ 
The {\em normal operator family} is then defined as
$$
\hat N(\cP)(\sigma)=P_\sigma=
\sum_{j+|\alpha|\leq\difford} a_{j\alpha}(0,y) \sigma^j D_y^\alpha\in\DiffbI^{\difford}(\bM_\infty).
$$
Note also that we can identify a neighborhood of $X$ in $\bM$ with a neighborhood
of $X\times\{0\}$ in $\bM_\infty$ (this depends on choices), and then transfer
$\cP$ to an operator (still denoted by $\cP$)
on $\bM_\infty$, extended in an arbitrary smooth manner; then
$\cP-N(\cP)\in\tau\Diffb^{\difford}(\bM_\infty)$.

The principal symbol $p$ of the normal operator family, including in the
high energy (or, after rescaling, semiclassical)
sense, is given by $\sigma_{\bl,\difford}(\cP)|_{\Tb^*_X \bM}$.
Correspondingly, the Hamilton vector field, including in the high-energy sense,
of $p$ is given by $\sH_{\sigma_{\bl,\difford}(\cP)}|_{\Tb^*_X\bM}$; see
\eqref{eq:b-Ham-rest}. It is useful to note that via this restriction we drop information
about $\sH_{\sigma_{\bl,\difford}(\cP)}$ as a b-vector field, namely the $\tau\pa_\tau$
component is neglected. Correspondingly, the dynamics (including at high
energies) for the normal operator family is the same at radial points
of the Hamilton flow regardless of the behavior of the $\tau\pa_\tau$ component,
thus whether on $\Sb^*\bM=(\Tb^*\bM\setminus o)/\RR^+$, with the $\tau$
variable included, we have a source/sink, or a saddle point, with the other
(stable/unstable) direction being transversal to the boundary. This is reflected
by the same normal operator family showing up in both de Sitter space and
in Minkowski space, even though in de Sitter space (and also in Kerr-de Sitter space)
in the full b-sense the radial points are saddle points, while in Minkowski space they
are sources/sinks (with a neutral direction along the conormal bundle
of the event horizon/light cone inside the boundary in both cases).

We now translate our results to solutions of $(\cP-\imath\cQ) u=f$
when $P_\sigma-\imath Q_\sigma$ is the
normal operator family of the b-operator $\cP-\imath\cQ$.
A typical application is when $\cP=\Box_g$ is the d'Alembertian of a Lorentzian
b-metric on $\bM$, discussed in Subsection~\ref{subsec:Lorentz}.

Thus, consider the Mellin transform in $\tau$, i.e.\ consider the map
\begin{equation}\label{eq:Mellin}
\cM:u\mapsto \hat u(\sigma,.)=\int_0^\infty
\tau^{-\imath\sigma} u(\tau,.)\,\frac{d\tau}{\tau},
\end{equation} 
with inverse transform
\begin{equation}\label{eq:Mellin-inv}
\cM^{-1}:v\mapsto \check v(\tau,.)=\frac{1}{2\pi}\int_{\RR+\imath\alpha}
\tau^{\imath\sigma} v(\sigma,.)\,d\sigma,
\end{equation}
with $\alpha$ chosen in the region of holomorphy.
Note that for polynomially bounded (in $\tau$)
$u$ (with values in a space, such as $\CI(X)$, $L^2(X)$, $\dist(X)$),
for $u$ supported near $\tau=0$,
$\cM u$ is holomorphic in $\im\sigma>C$, $C>0$ sufficiently large,
with values in the same space (such as $\CI(X)$, etc). We discuss more
precise statements below.
The Mellin transform is described in detail in \cite[Section~5]{Melrose:Atiyah},
but it is also merely a renormalized Fourier transform, so the results below
are simply those for the Fourier transform (often of Paley-Wiener type)
after suitable renormalization.

First, Plancherel's theorem is that
if $\nu$ is a smooth non-degenerate density on $X$ and $r_c$ denotes
restriction to the line $\im\sigma=c$,
then
\begin{equation}\label{eq:Mellin-weight-L2}
r_{-\alpha}\circ\cM:\tau^\alpha L^2(X\times[0,\infty);\frac{|d\tau|}{\tau}\nu)\to
L^2(\RR;L^2(X;\nu))
\end{equation}
is an isomorphism. We are interested in functions $u$ supported near $\tau=0$,
in which case, with $r_{(c_1,c_2)}$ denoting restriction to the strip $c_1<\im\sigma<c_2$,
for $N>0$,
\begin{equation}\begin{split}\label{eq:Mellin-range-L2}
&r_{-\alpha,-\alpha+N}\circ\cM:\tau^\alpha(1+\tau)^{-N}
L^2(X\times[0,\infty);\frac{|d\tau|}{\tau}\nu)\\
&\qquad\to
\Big\{v:\RR\times\imath(-\alpha,-\alpha+N)\ni\sigma
\to v(\sigma)\in L^2(X;\nu);\\
&\qquad\qquad v\ \text{is holomorphic in}\ \sigma\Mand
\sup_{-\alpha<r<-\alpha+N}\|v(.+\imath r,.)\|_{L^2(\RR;L^2(X;\nu))}<\infty\Big\},
\end{split}\end{equation}
see \cite[Lemma~5.18]{Melrose:Atiyah}. Note that in accordance
with \eqref{eq:Mellin-weight-L2}, $v$ in \eqref{eq:Mellin-range-L2} extends
continuously to the boundary values, $r=-\alpha$ and $r=-\alpha-N$, with
values in the same space as for holomorphy. Moreover, for functions supported
in, say, $\tau<1$, one can take $N$ arbitrary.

Analogous results also hold for the b-Sobolev spaces $\Hb^s(X\times[0,\infty))$.
For $s\geq 0$, these can be defined as in \cite[Equation~(5.41)]{Melrose:Atiyah}:
\begin{equation}\begin{split}\label{eq:b-Sobolev-iso}
&r_{-\alpha}\circ\cM:\tau^\alpha \Hb^s(X\times[0,\infty);\frac{|d\tau|}{\tau}\nu)\\
&\to
\Big\{v\in L^2(\RR;H^s(X;\nu)):\ v\in (1+|\sigma|^2)^{s/2}v\in L^2(\RR;L^2(X;\nu))\Big\},
\end{split}\end{equation}
with the analogue of \eqref{eq:Mellin-range-L2} also holding; for $s<0$ one
needs to use the appropriate dual statements.
See also \cite[Equations~(5.41)-(5.42)]{Melrose:Atiyah} for differential versions
for integer order spaces. Note that the right hand side of
\eqref{eq:b-Sobolev-iso} is equivalent to
\begin{equation}\label{eq:b-Sobolev-iso-mod}
\langle|\sigma|\rangle^{s/2}v\in L^2(\RR;H^s_{\langle|\sigma|\rangle^{-1}}(X;\nu)),
\end{equation}
where the space on the right hand side is the standard semiclassical
Sobolev space and $\langle|\sigma|\rangle=(1+|\sigma|^2)^{1/2}$;
indeed, for $s\geq 0$ integer both are equivalent to the statement
that for all $\alpha$ with $|\alpha|\leq s$,
$\langle|\sigma|\rangle^{s-|\alpha|} D_y^\alpha v\in
L^2(\RR;L^2(X;\nu))$. Here by equivalence we mean not only the
membership of a set, but also that of the standard
norms\footnote{Standard up to equivalence, such as
$\Big(\sum_{|\alpha|\leq s}\int_{\im\sigma=-\alpha}\langle|\sigma|\rangle^{2(s-|\alpha|)}
\|D_y^\alpha v\|^2_{L^2(X;\nu)}\,d\sigma\Big)^{1/2}$.}
corresponding
to these spaces. Note that by dualization,
\eqref{eq:b-Sobolev-iso-mod} holds for all $s\in\RR$.

If $\cP-\imath\cQ$
is invariant under dilations in $\tau$ on $\bM_\infty=X\times[0,\infty)$ then
$N(\cP-\imath \cQ)$ can be identified with $\cP-\imath\cQ$ and
we have the following simple lemma:

\begin{lemma}\label{lemma:Mellin-expand}
Suppose $\cP-\imath\cQ$
is invariant under dilations in $\tau$ for functions supported near
$\tau=0$,
and the normal operator family $\hat N(\cP-\imath\cQ)$ is of the form
$P_\sigma-\imath Q_\sigma$ satisfying the conditions of Section~\ref{sec:microlocal},
including
semiclassical non-trapping. Let $\sigma_j$ be the poles of the meromorphic
family $(P_\sigma-\imath Q_\sigma)^{-1}$. Then for $\ell<\beta^{-1}(2s-\difford+1)$,
$\ell\neq-\im\sigma_j$ for any $j$,
$(\cP-\imath\cQ) u=f$, $u$ tempered, supported near $\tau=0$,
$f\in \tau^\ell \Hb^{s-\difford+1}(\bM_\infty)$,
$u$ has an asymptotic expansion
\begin{equation}\label{eq:Mellin-simple-non-trap}
u=\sum_j\sum_{\kappa\leq m_j} \tau^{\imath\sigma_j}(\log |\tau|)^\kappa a_{j\kappa}+u'
\end{equation}
with $a_{j\kappa}\in\CI(X)$ and $u'\in \tau^\ell \Hb^{s}(\bM_\infty)$.

If instead $N(\cP-\imath\cQ)$ is semiclassically mildly trapping
of order $\varkappa$ in a $C_0$-strip then for $\ell<C_0$ (still with
$\ell<\beta^{-1}(2s-\difford+1)$,
$\ell\neq-\im\sigma_j$ for any $j$) and
$f\in \tau^\ell \Hb^{s-\difford+1+\varkappa}(\bM_\infty)$ one has
\begin{equation}\label{eq:Mellin-simple-mild-trap}
u=\sum_j\sum_{\kappa\leq m_j} \tau^{\imath\sigma_j}(\log |\tau|)^\kappa a_{j\kappa}+u'
\end{equation}
with $a_{j\kappa}\in\CI(X)$ and $u'\in \tau^\ell \Hb^{s}(\bM_\infty)$.

Conversely, given $f$ in the indicated spaces, with $f$
supported near $\tau=0$, a solution $u$ of $(\cP-\imath\cQ) u=f$ of the form
\eqref{eq:Mellin-simple-non-trap}, resp.\ \eqref{eq:Mellin-simple-mild-trap},
supported near $\tau=0$ exists.

In either case,
the coefficients $a_{j\kappa}$ are given by the Laurent coefficients
of $(\cP-\imath\cQ)^{-1}$ at the poles $\sigma_j$ applied to $f$,
with simple poles corresponding
to $m_j=0$.

If $f=\sum_j\sum_{\kappa\leq m'_j} \tau^{\alpha_j} (\log |\tau|)^\kappa b_{j\kappa}+f'$, with
$f'$ in the spaces indicated above for $f$, and $b_{j\kappa}\in H^{s-\difford+1}(X)$,
analogous results hold when the expansion of $f$ is added to the form of
\eqref{eq:Mellin-simple-non-trap} and \eqref{eq:Mellin-simple-mild-trap},
in the sense of the extended union of index sets, see
\cite[Section~5.18]{Melrose:Atiyah}.

Further, the result is stable under sufficiently small dilation-invariant
perturbations in the b-sense, i.e.\ if $\cP'$ and $\cQ'$ are sufficiently close to
$\cP$ and $\cQ$ in $\Psib^\difford(\bM_\infty)$ with $P'_\sigma$ and
$Q'_\sigma$ possessing real principal symbols, and that of $Q'_\sigma$
is non-negative, then there is a
similar expansion for solutions of $(\cP'-\imath\cQ')u=f$.

For $\cP^*+\imath\cQ^*$ in place of $\cP-\imath\cQ$, analogous results apply, but
we need $\ell<-\beta^{-1}(2s-\difford+1)$, and the $a_{j\kappa}$ are
not smooth, rather have wave front set\footnote{See the discussion
  after \eqref{eq:exist-est}.} in the Lagrangians
$\Lambda_\pm$.
\end{lemma}

\begin{rem}\label{rem:expand-reg}
Thus, for $\cP-\imath\cQ$,
the more terms we wish to obtain in an expansion, the better Sobolev space
we need to work in. For $\cP^*+\imath\cQ^*$, dually, we need to be in a weaker
Sobolev space under the same circumstances. However, these spaces only need
to be worse at the radial points, so under better regularity assumptions on $f$
we still get the expansion in better Sobolev spaces away from the radial points
--- in particular in elliptic regions. This is relevant in our description of
Minkowski space.
\end{rem}

\begin{rem}\label{rem:off-real-not-imp-for-exp}
If the large $\im\sigma$, i.e.\ $\im z\neq 0$, assumptions in
Subsection~\ref{subsec:semi-abstract} are not satisfied (but the
assumptions corresponding to a strip, i.e.\ roughly real $z$, still are), the proof of this lemma
still goes through apart from the support conclusion in the existence
part. For the existence argument, one can then pick any $\alpha<\ell$
with $(P_\sigma-\imath Q_\sigma)^{-1}$ having no
poles on the line $\im\sigma = -\alpha$; different choices of $\alpha$
result in different solutions.
See also Remark~\ref{rem:off-real-not-imp}.
\end{rem}

\begin{proof}
First consider the expansion.
Suppose $\alpha,r\in\RR$ are such that
$u\in\tau^\alpha \Hb^r(\bM_\infty)$ and $(P_\sigma-\imath Q_\sigma)^{-1}$ has no
poles on the line $\im\sigma = -\alpha$; note
that the vanishing of $u$ for $\tau>1$ and the absence of poles of
$(P_\sigma-\imath Q_\sigma)^{-1}$ near infinity inside
strips (by the semiclassical non-trapping/mildly trapping assumptions)
means that this can be arranged, and then also
$u\in\tau^\alpha (1+\tau)^{-N}\Hb^r(\bM_\infty)$ for all $N$.
The Mellin transform of the PDE, a priori on $\im\sigma = -\alpha$, is
$(P_\sigma-\imath Q_\sigma)\cM u=\cM f$. Thus,
\begin{equation}\label{eq:Mellin-PDE}
\cM u=(P_\sigma-\imath Q_\sigma)^{-1}\cM f
\end{equation}
there. If $f\in \tau^\ell \Hb^{s-\difford+1}(\bM_\infty)$, then
shifting the contour of integration to $\im\sigma=-\ell$, we
obtain contributions from the poles of $(P_\sigma-\imath Q_\sigma)^{-1}$, giving
the expansion in \eqref{eq:Mellin-simple-non-trap} and
\eqref{eq:Mellin-simple-mild-trap} by Cauchy's theorem.
The error term $u'$ is what one obtains
by integrating along the new contour in view of the high energy bounds
on $(P_\sigma-\imath Q_\sigma)^{-1}$ (which differ as one changes one's
assumption from non-trapping to mild trapping), and the assumptions on $f$.

Conversely, to obtain existence, let $\alpha<\min(\ell,-\sup\im\sigma_j)$ and
define $u\in \tau^\alpha \Hb^{s}(\bM_\infty)$ by
\eqref{eq:Mellin-PDE} using the inverse Mellin transform with $\im\sigma=-\alpha$.
Then $u$ solves the PDE, hence the expansion follows by the first part of the
argument. The support property of $u$ follows from Paley-Wiener, taking
into account holomorphy in $\im\sigma>-\alpha$, and the
estimates on
$\cM f$ and $(P_\sigma-\imath Q_\sigma)^{-1}$ there.

Finally, stability of the expansion follows from
Subsection~\ref{subsec:stability} since the meromorphy and the large $\sigma$
estimates are stable under such a perturbation. Note that the
condition on the principal symbol of $P'_\sigma$ and $Q'_\sigma$ to be
independent of $\sigma$ is automatically satisfied, for this is just
the principal symbol in $\Psib^\difford(\bM_\infty)$ (which stands for
{\em one-step}, or {\em classical} b-pseudodifferential operators) of $\cP'$ and
$\cQ'$ evaluated at $\sigma=0$ (or any other finite constant), cf.\
the large parameter discussion at the end of Subsection~\ref{subsec:notation}.
\end{proof}

\begin{rem}\label{rem:estimates-for-expansion}
One actually gets estimates for the coefficients $a_{j\kappa}$ and
$u'$ in
\eqref{eq:Mellin-simple-non-trap}-\eqref{eq:Mellin-simple-mild-trap}.
Indeed,
in view of the isomorphism \eqref{eq:b-Sobolev-iso-mod} and the contour
deformation after \eqref{eq:Mellin-PDE}, $u'$ is bounded in $\tau^\ell
\Hb^{s}(\bM_\infty)$ by $f$ in $\tau^\ell
\Hb^{s-\difford+1}(\bM_\infty)$ in the non-trapping case, with the
$\varkappa$ shift in the mildly trapping case. Then, in the
nontrapping case, the norm of $f$ in $\tau^\ell
\Hb^{s-\difford+1}(\bM_\infty)$ gives a bound for the norm of $\cM f$
on the line $\im\sigma=-\ell$
in $\langle|\sigma|\rangle^{-(s-\difford+1)}
L^2(\RR;H_{\langle|\sigma|\rangle^{-1}}^{s-\difford+1}(X))$;
now the non-trapping bounds for $(P_\sigma-\imath
Q_\sigma)^{-1}$ imply that $(P_\sigma-\imath
Q_\sigma)^{-1}\cM f$ is bounded in
$\langle|\sigma|\rangle^{-s}
L^2(\RR;H_{\langle|\sigma|\rangle^{-1}}^{s}(X))$, and thus $u'$ is bounded in $\tau^{\ell}\Hb^{s}(\bM_\infty)$.
One gets similar bounds for the $a_{j\kappa}$; indeed only the norm of
$f$ in slightly
less weighted spaces (i.e.\ with a weight $\ell'<\ell$) corresponding
to the location of the poles of $(P_\sigma-\imath
Q_\sigma)^{-1}$ is needed to estimate these. In the case of mild
trapping, one needs stronger norms on $f$ corresponding to $\varkappa$
in view of the bounds on $(P_\sigma-\imath
Q_\sigma)^{-1}$ then.
\end{rem}

One can iterate this to obtain a full expansion even when
$\cP-\imath\cQ$ is not dilation invariant. Note that in most cases considered below,
Lemma~\ref{lemma:Mellin-expand} suffices; the exception is if we allow
general, non-stationary, b-perturbations of Kerr-de Sitter or
Minkowski metrics.

\begin{prop}\label{prop:Mellin-expand}
Suppose $(\cP-\imath\cQ) u=f$,
and the normal operator family $\hat N(\cP-\imath\cQ)$ is of the form
$P_\sigma-\imath Q_\sigma$ satisfying the conditions of Section~\ref{sec:microlocal},
including
semiclassical non-trapping. Then for $\ell<\beta^{-1}(2(s-|\ell-\alpha|)-\difford+1)$,
$\ell\notin-\im\sigma_j+\Nat$ for any $j$, $u\in\tau^\alpha\Hb^r(\bM_\infty)$
supported near $0$, $(\cP-\imath\cQ) u=f$,
$f\in \tau^\ell \Hb^{s-\difford+1}(\bM_\infty)$, $u$ has an asymptotic expansion
\begin{equation}\label{eq:Mellin-gen-non-trap}
u=\sum_j\sum_l\sum_{\kappa\leq m_{jl}} \tau^{\imath\sigma_j+l}(\log |\tau|)^\kappa
a_{j\kappa l}+u'
\end{equation}
with $a_{j\kappa}\in\CI(X)$ and $u'\in \tau^\ell \Hb^{s-[\ell-\alpha]}(\bM_\infty)$, $[\ell-\alpha]$
being the integer part of $\ell-\alpha$.

If instead $N(\cP-\imath\cQ)$ is semiclassically mildly trapping
of order $\varkappa$ in a $C_0$-strip then for $\ell<C_0$ and
$f\in \tau^\ell \Hb^{s-\difford+1+\varkappa}(\bM_\infty)$ one has
\begin{equation}\label{eq:Mellin-gen-mild-trap}
u=\sum_j\sum_l\sum_{\kappa\leq m_{jl}} \tau^{\imath\sigma_j+l}(\log |\tau|)^\kappa a_{j\kappa l}+u'
\end{equation}
with $a_{j\kappa l}\in\CI(X)$ and $u'\in \tau^\ell \Hb^{s-[\ell-\alpha]}(\bM_\infty)$.

If $f=\sum_j\sum_{\kappa\leq m'_j} \tau^{\alpha_j} (\log |\tau|)^\kappa b_{j\kappa}+f'$,
with
$f'$ in the spaces indicated above for $f$, and $b_{j\kappa}\in H^{s-\difford+1}(X)$,
analogous results hold when the expansion of $f$ is added to the form of
\eqref{eq:Mellin-gen-non-trap} and \eqref{eq:Mellin-gen-mild-trap}
in the sense of the extended union of index sets, see
\cite[Section~5.18]{Melrose:Atiyah}.

If $\sigma_{\bl,\difford}(\cP-\imath\cQ)$ vanishes on the characteristic set of
$N(\cP-\imath\cQ)$ to infinite order in Taylor series at $\tau=0$, then there
are no losses in the order of $u'$, i.e.\ one can replace $u'\in \tau^\ell \Hb^{s-[\ell-\alpha]}(\bM_\infty)$ by $u'\in \tau^\ell \Hb^{s}(\bM_\infty)$, and
$\ell<\beta^{-1}(2(s-|\ell-\alpha|)-\difford+1)$ by
$\ell<\beta^{-1}(2s-\difford+1)$, giving the same form as in Lemma~\ref{lemma:Mellin-expand}.

Conversely, under the characteristic assumption in the previous paragraph,
given $f$ in the indicated spaces, with $f$
supported near $\tau=0$, a solution $u$ of $(\cP-\imath\cQ) u=f+f^\sharp$
of the form
\eqref{eq:Mellin-simple-non-trap}, resp.\ \eqref{eq:Mellin-simple-mild-trap},
$f^\sharp\in\tau^\infty \Hb^{s-\difford+1}(\bM_\infty)$,
resp.\ $\Hb^{s-\difford+1+\varkappa}(\bM_\infty)$,
supported near $\tau=0$, exists.

Again, the result is stable under sufficiently small perturbations, in
the b-sense, of $\cP$ and $\cQ$, in the same sense, apart from
dilation invariance, as stated in Lemma~\ref{lemma:Mellin-expand}.
\end{prop}

\begin{rem}
The losses in the regularity of $u'$ without further assumptions are natural due to the
lack of ellipticity. Specifically, if, for instance,
$u$ is conormal to a hypersurface $S$ transversal to $X$, as is the case in many
interesting examples, the orbits of the $\RR^+$-action on $\bM_\infty$ must
be tangent to $S$ to avoid losses of regularity in the Taylor series expansion.

In particular, there are no losses if
$(\cP-\imath \cQ)- N(\cP-\imath \cQ)\in\tau\Diffb^{\difford-1}(\bM_\infty)$,
rather than merely in $\tau\Diffb^{\difford}(\bM_\infty)$.

We only stated the converse result under the extra characteristic assumption to
avoid complications with the Sobolev orders. Global solvability depends on more
than the normal operator, which is why we do not state such a result here.
\end{rem}

\begin{proof}
One proceeds as in Lemma~\ref{lemma:Mellin-expand}, Mellin transforming the
problem, but replacing $\cP-\imath \cQ$ by $N(\cP-\imath\cQ)$. Note
that $(\cP-\imath \cQ)- N(\cP-\imath \cQ)\in\tau\Diffb^{\difford}(\bM_\infty)$.
We treat
$$
\tilde f=((\cP-\imath \cQ)- N(\cP-\imath \cQ))u
$$
as part of the right hand side, subtracting it from $f$, so
$$
N(\cP-\imath\cQ)u=f-\tilde f.
$$
If $u\in\tau^\alpha \Hb^{r}(\bM_\infty)$ is supported near $0$, then
$\tilde f\in \tau^{\alpha+1} \Hb^{r}(\bM_\infty)$, so Lemma~\ref{lemma:Mellin-expand}
is applicable with $\ell$ replaced by $\min(\ell,\alpha+1)$.
If $\ell\leq\alpha+1$, we are done, otherwise we repeat
the argument. Indeed, we now know that
$u$ is given by an expansion giving rise to poles of $\cM u$
in $\im\sigma>\alpha+1$ plus an element of $\tau^{\alpha+1} \Hb^{s}(\bM_\infty)$, so
we also have better information on $\tilde f$, namely it is also given by a partial
expansion, plus an element of $\tau^{\alpha+2} \Hb^{s-\difford}(\bM_\infty)$, or indeed
$\tau^{\alpha+2}\Hb^{s-\difford+1}(\bM_\infty)$ under the characteristic assumption on
$\cP-\imath\cQ$.
Using the $f$ with a partial expansion part of Lemma~\ref{lemma:Mellin-expand}
to absorb the $\tilde u$ terms,
we can work with $\ell$ replaced by $\min(\ell,\alpha+2)$. It is this step
that starts generating the sum over $l$ in \eqref{eq:Mellin-gen-non-trap}
and \eqref{eq:Mellin-gen-mild-trap}.
The iteration stops in
a finite number of steps, completing the proof.

For the existence, define a zeroth approximation $u_0$ to $u$ using
$N(\cP-\imath\cQ)$ in
Lemma~\ref{lemma:Mellin-expand}, and iterate away the error
$\tilde f=((\cP-\imath\cQ)u-N(\cP-\imath\cQ))u_0-f$ in Taylor series.
\end{proof}

\subsection{Lorentzian metrics}\label{subsec:Lorentz}
We now review common properties of Lorentzian b-metrics $g$
on $\bM$.
Lorentzian b-metrics are symmetric non-degenerate bilinear forms on $\Tb_m \bM$,
$m\in\bM$, of signature
$(1,n-1)$, i.e.\ the maximal dimension of a subspace on which $g$ is positive
definite is {\em exactly} 1, which depend smoothly on $m$. In other
words, they are symmetric sections of $\Tb^*\bM\otimes\Tb^*\bM$
which are in addition non-degenerate of Lorentzian signature.
Usually it is more convenient to work with the dual metric $G$, which is
then a symmetric section of $\Tb\bM\otimes\Tb\bM$
which is in addition non-degenerate of Lorentzian signature.

By non-degeneracy there is a nowhere vanishing b-density associated to the metric,
$|dg|$, which in local coordinates $(\tau,y)$ is given by $\sqrt{|\det g|}\frac{|d\tau|}{\tau}
\,|dy|$, and which gives rise to a Hermitian (positive definite!) inner
product on functions. There is also a non-degenerate, but not positive definite, inner
product on the fibers of the b-form bundle, $\Lambdab \bM$,
and thus, when combined
with the aforementioned Hermitian inner product on functions, an inner product
on differential forms which is not positive definite only due to the lack of
definiteness of the fiber inner product.
Thus, $A^*$ is defined, as a formal adjoint,
for any differential operator $A\in\Diffb^{\difford}(\bM;\Lambdab \bM)$
acting on sections of the
b-form bundle, such as the exterior derivative, $d$. Thus,
$g$ gives rise to the d'Alembertian,
$$
\Box_g=d^*d+dd^*\in\Diffb^2(\bM;\Lambda \bM),
$$
which preserves form degrees.
The d'Alembertian on functions is also denoted by $\Box_g$.
The principal symbol of $\Box_g$ is
$$
\sigma_{\bl,2}(\Box_g)=G.
$$

As discussed above,
the normal operator of $\Box_g$ on $\bM$ is $N(\Box_g)\in\DiffbI(\bM_\infty)$,
$\bM_\infty=X\times[0,\infty)_\tau$.
If $\bM=\bM_\infty$ (i.e.\ it is a product space to start with) and
if $\Box_g$ already has this invariance property under a product
decomposition, then the
normal operator can be identified with $\Box_g$ itself.
Taking the Mellin transform in $\tau$, we obtain a
family of operators, $P_\sigma$, on $X$, depending analytically
on $\sigma$, the b-dual variable of $\tau$.
The semiclassical principal symbol of $P_{h,z}=h^2P_{h^{-1}z}$
is just the dual metric $G$ on
the complexified cotangent bundle
$\TbC^*_m \bM$, $m=(x,\tau)$, evaluated on covectors
$\varpi+z\,\frac{d\tau}{\tau}$, where $\varpi$ is in the (real) span $\Pi$
of the `spatial variables'
$T^*_xX$; thus
$\Pi$ and $\frac{d\tau}{\tau}$ are
linearly independent. In general,
\begin{equation}\begin{split}\label{eq:Lorentz-metric-img}
&\langle \varpi+z\,\frac{d\tau}{\tau},\varpi+z\,\frac{d\tau}{\tau}\rangle_G\\
&=\langle \varpi+\re z\,\frac{d\tau}{\tau},\varpi+\re z\,\frac{d\tau}{\tau}\rangle_G
-(\im z)^2\langle \frac{d\tau}{\tau},\frac{d\tau}{\tau}\rangle_G\\
&\qquad
+2\imath\im z\langle \varpi+\re z\frac{d\tau}{\tau},\frac{d\tau}{\tau}\rangle_G.
\end{split}\end{equation}
For $\im z\neq 0$,
the vanishing of the imaginary part states that
$\langle \varpi+\re z\,\frac{d\tau}{\tau},\frac{d\tau}{\tau}\rangle_G=0$;
the real part is the first
two terms on the right hand side of \eqref{eq:Lorentz-metric-img}.

In the setting of Subsection~\ref{subsec:semi-abstract}
we want that when $\im z\neq 0$ and $\im p_{\semi,z}$ vanishes then $\re p_{\semi,z}$
does not vanish, i.e.\ that on the orthocomplement of the span
of $\frac{d\tau}{\tau}$ the metric should have the opposite sign as that of
$\langle \frac{d\tau}{\tau},\frac{d\tau}{\tau}\rangle_G$. For a Lorentzian
metric this is only possible if $\frac{d\tau}{\tau}$ is time-like (note that
$\varpi+\re z\,\frac{d\tau}{\tau}$ spans the whole fiber of the b-cotangent
bundle as $\re z$ and $\varpi\in\Pi$ vary), when, however,
this is automatically the case, namely the metric is negative definite
on this orthocomplement. {\em From now on we always assume that
$\frac{d\tau}{\tau}$ is time-like for $G$.}

Furthermore, for $z$ real, non-zero, the characteristic set of $p_{\semi,z}$
cannot intersect the
hypersurface
$\langle \varpi+\re z\frac{d\tau}{\tau},\frac{d\tau}{\tau}\rangle_G=0$, for
$G$ is negative definite on covectors satisfying this equality, so if the intersection
were non-empty,
$\varpi+\re z\frac{d\tau}{\tau}$ would vanish there, which cannot happen for
$\varpi\in\Pi$ since $\re z\neq 0$ by assumption.
Correspondingly, we can divide the semiclassical characteristic set
in two parts by
\begin{equation}\label{eq:Lorentz-semicl-char}
\Sigma_{\semi,\pm}\cap T^*X=\{\varpi\in \Sigma_{\semi}\cap T^*X:
\pm\langle \varpi+\re z\frac{d\tau}{\tau},\frac{d\tau}{\tau}\rangle_G>0\};
\end{equation}
note that by the definiteness of the quadratic form on this hypersurface, in
fact this separation
holds on the fiber-compactified bundle, $\overline{T}^*X$. In general,
one of the `components' $\Sigma_{\semi,\pm}$ may be empty.
However, in any case, when
$\im z\geq 0$, the sign
of the imaginary part of $p_{\semi,z}$ on $\Sigma_{\semi,\pm}$
is given by $\pm\im p_{\semi,z}\geq 0$,
as needed for the propagation of estimates: in $\Sigma_{\semi,+}$ we can propagate
estimates backwards, in $\Sigma_{\semi,-}$ we can propagate estimates forward.
For $\im z\leq 0$, the direction of propagation is reversed.

Moreover, for $m\in \bM$,
and with $\Pi$ denoting the `spatial' hyperplane in the real cotangent bundle,
$\Tb^*_m \bM$,
the Lorentzian nature of $G$
means that for $z$ real and non-zero, the intersection of
$\Pi+z\frac{d\tau}{\tau}$ with the zero-set of $G$ in $\Tb^*_q \bM$,
i.e.\ the characteristic set, has two components
if $G|_{\Pi}$ is Lorentzian, and one component if it is negative definite (i.e.\ Riemannian,
up to the sign). Further,
in the second case, on the only component
$\langle \varpi+\re z\frac{d\tau}{\tau},\frac{d\tau}{\tau}\rangle_G$
and $\langle \re z\frac{d\tau}{\tau},\frac{d\tau}{\tau}\rangle_G$ have the
same sign, so only $\Sigma_{\semi,\sgn(\re z)}$ can enter the elliptic region.

We also need information about $p_{\semi,z}-\imath q_{\semi,z}$, i.e.\
when the complex absorption has been added, with $q_{\semi,z}$ defined
for $z$ in an open set $\tilde\Omega\subset \Cx$. Here we need to choose
$q_{\semi,z}$ in such a way as to ensure
that $p_{\semi,z}-\imath q_{\semi,z}$ does not vanish when $\im z>0$,
but for real $z\neq 0$, $q_{\semi,z}$ is real and for $z$ sufficiently
close to $\RR$ with $\im z\geq 0$, $\mp \re q_{\semi,z}\geq
0$ on $\Sigma_{\semi,\pm}$ (as well as an ellipticity condition that
we discuss below \eqref{eq:Lorentz-q-form}). In order to arrange this, we
take
\begin{equation}\begin{split}\label{eq:Lorentz-q-form}
q_{\semi,z}=-\chi f_z\langle
\varpi+z\frac{d\tau}{\tau},\frac{d\tau}{\tau}\rangle_G,\qquad&\re
f_z\geq 0,\ z\in\RR\Rightarrow f_z\ \text{is real},\\
&\chi\geq 0,\ \text{independent of}\ z;
\end{split}\end{equation}
note that if in addition $f_z$ is bounded away from $0$ when $z$ is
bounded away from $0$ in $\RR$, then
the above conditions for real $z$ are then automatically satisfied
in view of \eqref{eq:Lorentz-semicl-char}. In addition, at points
where $\chi>0$, $p_{\semi,z}-\imath q_{\semi,z}$ does not vanish for
$z$ real, since the imaginary part, $q_{\semi,z}$, is non-zero except
when $\langle
\varpi+z\frac{d\tau}{\tau},\frac{d\tau}{\tau}\rangle_G=0$, but in that
case the real part satisfies $p_{\semi,z}<0$ by
\eqref{eq:Lorentz-semicl-char}.

Note that
$\langle
\varpi+z\frac{d\tau}{\tau},\frac{d\tau}{\tau}\rangle_G$ rather than the function
$\langle
\varpi+\re z\frac{d\tau}{\tau},\frac{d\tau}{\tau}\rangle_G$ appearing
above is used in the definition of $q_{\semi,z}$
to make sure that $q_{\semi,z}$ is holomorphic in $z$ (if $f_z$
is such). Thus, if we ensure that for $\im z>0$,
\begin{equation}\label{eq:p-iq-elliptic}
\im\left(\langle \varpi+z\frac{d\tau}{\tau},\frac{d\tau}{\tau}\rangle_G^{-1} p_{\semi,z}\right)>0,
\end{equation}
then $\im (\imath f)\geq 0$ shows that $p_{\semi,z}-\imath q_{\semi,z}\neq 0$ as desired.
We note that the imaginary part
of $\langle
\varpi+z\frac{d\tau}{\tau},\frac{d\tau}{\tau}\rangle_G$ is $\im z \langle
\frac{d\tau}{\tau},\frac{d\tau}{\tau}\rangle_G$, and thus is non-zero
as $\frac{d\tau}{\tau}$ is time-like and $\im z>0$. But the expression inside the
imaginary part on the
left hand side of \eqref{eq:p-iq-elliptic} is a positive
multiple of
$\langle
\varpi+\overline{z}\frac{d\tau}{\tau},\frac{d\tau}{\tau}\rangle_G
p_{\semi,z}$, so it suffices to consider the latter, whose imaginary
part is, with $\beta=\varpi+\re z\frac{d\tau}{\tau}$,
\begin{equation}\begin{split}\label{eq:p-iq-rotate}
&-\im z \langle\frac{d\tau}{\tau},\frac{d\tau}{\tau}\rangle_G
\left(\langle \beta,\beta\rangle_G
-(\im z)^2\langle
\frac{d\tau}{\tau},\frac{d\tau}{\tau}\rangle_G\right)
+\langle\beta,\frac{d\tau}{\tau}\rangle_G (2\im
z)\langle \beta,\frac{d\tau}{\tau}\rangle_G\\
&=\im z
\left(-\langle \beta,\beta\rangle_G \langle\frac{d\tau}{\tau},\frac{d\tau}{\tau}\rangle_G
+2\langle\beta,\frac{d\tau}{\tau}\rangle_G^2 +(\im z)^2\langle
\frac{d\tau}{\tau},\frac{d\tau}{\tau}\rangle_G^2
\right).
\end{split}\end{equation}
Now, as $\frac{d\tau}{\tau}$ time-like, the first two terms
inside the parentheses on
the right hand side give twice the positive definite stress-energy
tensor; the positive definite character is easily checked by writing
$\beta=\gamma+\lambda\frac{d\tau}{\tau}$ with\footnote{This is possible for
$\frac{d\tau}{\tau}$ time-like. Note further that typically $\gamma$
is {\em not} in the `spatial' slice $T^*_x X$; the latter need even not
be space-like.}
$\langle\gamma,\frac{d\tau}{\tau}\rangle_G=0$,
for then these terms give
\begin{equation}\label{eq:real-stress-energy}
\langle\frac{d\tau}{\tau},\frac{d\tau}{\tau}\rangle_G\left(-\langle\gamma,\gamma\rangle_G
+\lambda^2 \langle\frac{d\tau}{\tau},\frac{d\tau}{\tau}\rangle_G\right),
\end{equation}
and the Lorentzian character of $G$ then implies that $-G$ is positive
definite on the orthocomplement of the span of
$\frac{d\tau}{\tau}$. In view of the third term, $(\im z)^2\langle
\frac{d\tau}{\tau},\frac{d\tau}{\tau}\rangle_G^2$, on the right hand
side of \eqref{eq:p-iq-rotate}, we actually conclude that \eqref{eq:p-iq-elliptic} holds
(i.e. the inequality is strict) when $\im z>0$.

In fact, typically $p_{\semi,z}$ itself is not globally defined, so we
need to extend it beyond the domain where it {\em is}
defined. Typically one has a function $\mu$ on $X$ with $d\mu$
time-like\footnote{This actually does not matter for the discussion
  below, but due to Subsection~\ref{subsec:local-wave} it ensures that the
  choice of the extension is irrelevant.} for $\mu$ near $\mu_1\in\RR$, and $p_{\semi,z}$ is given by a Lorentzian b-metric in
$\mu\geq\mu_1$, but we need to extend
$p_{\semi,z}$ to $\mu<\mu_1$. For this purpose, we first let $\tilde
H$ to be a Riemannian metric on $X$, and then for some $j\geq 1$
integer we let
$$
\hat p_{\semi,z}=\left(\|\varpi\|^{2j}_{\tilde H}+z^{2j}\right)^{1/j};
$$
thus the case of $j=1$ actually corresponds to a Riemannian
b-metric in which $\frac{d\tau}{\tau}$ and $\varpi\in T^*_x X$ are
orthogonal, unlike in the case of $G$. Here we choose the branch of the $j$th root function which
is positive along the positive reals and has a branch cut along the
negative reals, and take as the domain of $\hat p_{\semi,z}$ the
values of $\varpi$ and $z$ for which $\|\varpi\|^{2j}_{\tilde
  H}+z^{2j}$ lies outside the branch cut. Thus, the complement $\cD_j$ of the
rays $z=re^{i\pi(2k+1)/(2j)}$, $k$ an integer, $r\geq 0$, is always in
the domain of $\hat p_{\semi,z}$, and thus as $j$ varies, these domains
cover any compact set of $\Cx$ disjoint from $0$.
Now let $\mu_0>\mu_0'>\mu_1'>\mu_1$, and choose a partition of unity
$\chi_1+\chi_2=1$, $\chi_j\geq 0$, with $\supp\chi_1\subset
(\mu_1',+\infty)$, $\supp\chi_2\subset (-\infty,\mu_0')$. Further,
with $\digamma>0$ a constant to be chosen, let
$\chi\geq 0$ be identically $\digamma$ on $[\mu_1',\mu_0']$ and supported in
$(\mu_1,\mu_0)$, let
$$
f_z=\left(\|\varpi\|^{2j}_{\tilde H}+z^{2j}\right)^{1/2j},
$$
where again the branch cut for the $2j$th root is along the negative
reals, so $\hat p_{\semi,z}=f_z^2$. In particular, note that $\re
f_z>0$ with this choice, and even $\re f_z^2>0$ if $j\geq 2$. Now, let
\begin{equation}\label{eq:extend-p-def}
\tilde p_{\semi,z}=\chi_1 p_{\semi,z}-\chi_2\hat p_{\semi,z},
\end{equation}
and let $q_{\semi,z}$ be defined by \eqref{eq:Lorentz-q-form}. We
already know from the above discussion that where $\chi_1=1$,
$\tilde p_{\semi,z}-\imath q_{\semi,z}$ satisfies our requirements. We claim
that where either $\chi_2>0$ or $\chi>0$, $\tilde p_{\semi,z}-\imath
q_{\semi,z}$ is actually elliptic; where $\chi_2=0$ but $\chi>0$, this
was checked above. Note that as $\|\varpi\|_{\tilde H}\to\infty$ and
$z$ in a compact set, semiclassical ellipticity becomes a statement
(namely, that of ellipticity) for the standard
principal symbol, $\tilde p-\imath q$, which is easy to check as
$$
\tilde p-\imath q=\chi_1 \langle\varpi,\varpi\rangle_G-\chi_2\|\varpi\|^{2}_{\tilde H}
-\imath \chi \|\varpi\|_{\tilde H}\langle\varpi,\frac{d\tau}{\tau}\rangle,
$$
homogeneous of degree $2$ in $\varpi$.
Indeed, the imaginary part of $\tilde p-\imath q$ only vanishes (for
$\varpi\neq 0$) when
either $\chi$ or $\langle\varpi,\frac{d\tau}{\tau}\rangle$ does. In
the former case, $\chi_2=1$ and $\chi_1=0$, so $\tilde p-\imath
q<0$. In the latter case, $\langle\varpi,\varpi\rangle_G$ is
negative definite, so the real part does not vanish as $\chi_1+\chi_2=1$. Hence, for
sufficiently large $\varpi$, $\tilde p_{\semi,z}-\imath q_{\semi,z}$
is elliptic in this region; we only need to consider whether it
vanishes for finite $\varpi$.
Next, if $z$ is real, $z\neq 0$, then $f_z>0$,
$$
\tilde p_{\semi,z}-\imath q_{\semi,z}=\chi_1
\langle\varpi+z\frac{d\tau}{\tau},\varpi+z\frac{d\tau}{\tau}\rangle_G-\chi_2 f_z^2
-\imath \chi f_z\langle\varpi+z\frac{d\tau}{\tau},\frac{d\tau}{\tau}\rangle,
$$
so again the imaginary part only  vanishes if either $\chi$ or
$\langle\varpi+z\frac{d\tau}{\tau},\frac{d\tau}{\tau}\rangle$ does. In
the former case, $\chi_2=1$ and $\chi_1=0$ so $\tilde
p_{\semi,z}-\imath q_{\semi,z}<0$. In the latter case,
$\langle\varpi+z\frac{d\tau}{\tau},\varpi+z\frac{d\tau}{\tau}\rangle_G$
is negative definite, so the real part of $\tilde p_{\semi,z}-\imath
q_{\semi,z}$ is negative. Thus, it remains to consider finite $\varpi$
and non-real $z$, and show that $\tilde p_{\semi,z}-\imath
q_{\semi,z}$ does not vanish then; we have a priori that for
$K\subset\cD_j$ compact (thus disjoint from the branch cuts), there exist
$C>0$ and $\delta>0$ such that for $z\in K$ either one of
$\|\varpi\|_{\tilde H}\geq C$, resp.\ $|\im z|\leq\delta$, implies
non-vanishing of $\tilde p_{\semi,z}-\imath
q_{\semi,z}$.

To see the remaining cases, i.e.\ when both $\im z>\delta$ and
$\|\varpi\|_{\tilde H}<C$, first assume that $\chi_2=1$, so $\chi_1=0$. Then
$$
\tilde p_{\semi,z}-\imath q_{\semi,z}=f_z\left(-\chi_2 f_z+\imath \chi\langle\varpi+z\frac{d\tau}{\tau},\frac{d\tau}{\tau}\rangle_G\right),
$$
and the real part of the second factor on the right hand side is
$$
-\chi_2 \re f_z-(\im z) \chi\langle\frac{d\tau}{\tau},\frac{d\tau}{\tau}\rangle_G,
$$
which is $<0$ as $\chi_2>0$ and $\im z\geq 0$, showing ellipticity.

Now, if neither $\chi_1$ nor $\chi_2$ vanish, then $\chi>0$.
First, suppose that,
with $\beta=\varpi+\re z\frac{d\tau}{\tau}$ as above,
$\langle\beta,\frac{d\tau}{\tau}\rangle_G=0$, and thus
$\langle\beta,\beta\rangle\leq 0$, with the inequality strict if
$\beta\neq 0$. Then 
$$
\tilde p_{\semi,z}-\imath
q_{\semi,z}=\chi_1\langle\beta,\beta\rangle_G
-\chi_1(\im z)^2 \langle\frac{d\tau}{\tau},\frac{d\tau}{\tau}\rangle_G-\chi_2
f_z^2- f_z\chi(\im z) \langle\frac{d\tau}{\tau},\frac{d\tau}{\tau}\rangle_G,
$$
so if $\re f_z^2>0$, the real part is negative, thus showing ellipticity.
On the other hand, if $\langle\beta,\frac{d\tau}{\tau}\rangle_G\neq 0$, then
we compute
\begin{equation*}
\langle\varpi+\overline{z}\frac{d\tau}{\tau},\frac{d\tau}{\tau}\rangle_G (\tilde p_{\semi,z}-\imath q_{\semi,z}).
\end{equation*}
For the $p_{\semi,z}$ part this is the computation performed in
\eqref{eq:p-iq-rotate} (with an extra factor, $\chi_1$, now); on the
other hand,
\begin{equation*}\begin{split}
\langle\varpi+\overline{z}\frac{d\tau}{\tau},\frac{d\tau}{\tau}\rangle_G
(-\chi_2\hat p_{\semi,z}-\imath q_{\semi,z})=&-\chi_2 f_z^2
\left(\langle\beta,\frac{d\tau}{\tau}\rangle_G-\imath\im z\langle \frac{d\tau}{\tau},\frac{d\tau}{\tau}\rangle_G\right)\\
&\qquad+\imath \chi f_z\left(\langle\beta,\frac{d\tau}{\tau}\rangle_G^2+(\im
z)^2\langle \frac{d\tau}{\tau},\frac{d\tau}{\tau}\rangle_G^2\right).
\end{split}\end{equation*}
Thus,
\begin{equation*}\begin{split}
&\im\left(\langle\varpi+\overline{z}\frac{d\tau}{\tau},\frac{d\tau}{\tau}\rangle_G
(-\chi_2\hat p_{\semi,z}-\imath q_{\semi,z})\right)\\
&\qquad=\chi_2 \re (f_z^2)\im z\langle
\frac{d\tau}{\tau},\frac{d\tau}{\tau}\rangle_G
-\chi_2 \im(f_z^2) \langle\beta,\frac{d\tau}{\tau}\rangle_G\\
&\qquad\qquad+\chi \re f_z\left(\langle\beta,\frac{d\tau}{\tau}\rangle_G^2+(\im
z)^2\langle \frac{d\tau}{\tau},\frac{d\tau}{\tau}\rangle_G^2\right)
\end{split}\end{equation*}
Combining this with \eqref{eq:p-iq-rotate} we note that the only term in
\begin{equation*}\begin{split}
\im\Big(
  \langle\varpi+\overline{z}\frac{d\tau}{\tau},\frac{d\tau}{\tau}\rangle_G
  (\tilde p_{\semi,z}-\imath q_{\semi,z})\Big)
\end{split}\end{equation*}
that is not automatically non-negative provided $\im z\geq 0$,
$\re f_z>0$, $\re (f_z^2)>0$ is $-\chi_2 \im(f_z^2)
\langle\beta,\frac{d\tau}{\tau}\rangle_G$.
Thus, if we arrange that
$$
2|\chi_2|\,|\im
f_z|\,|\langle\beta,\frac{d\tau}{\tau}\rangle_G|<\frac{\digamma}{2}
(\im
z)^2\langle \frac{d\tau}{\tau},\frac{d\tau}{\tau}\rangle_G^2,
$$
which is guaranteed by choosing $\digamma>0$ sufficiently large as
$\im z>\delta$ and $\beta$ is in a compact set,
then in view of the actual positivity of $\chi \re f_z (\im
z)^2\langle \frac{d\tau}{\tau},\frac{d\tau}{\tau}\rangle_G^2$, this
imaginary part does not vanish, completing the proof of ellipticity.

In summary, we have shown that {\em if
$\frac{d\tau}{\tau}$ is time-like} then
the assumptions on imaginary part of $p_{\semi,z}$, as well as on the
ellipticity of
$p_{\semi,z}-\imath q_{\semi,z}$ for non-real $z$,
in Section~\ref{sec:microlocal}
are automatically satisfied in the Lorentzian setting if $q_{\semi,z}$
is given by \eqref{eq:Lorentz-q-form}. Further, if we extend
$p_{\semi,z}$ to a new symbol, $\tilde p_{\semi,z}$ across a
hypersurface, $\mu=\mu_1$, in the manner
\eqref{eq:extend-p-def}, then with $\chi$, $\chi_1$ and $\chi_2$ as
discussed there, $\tilde p_{\semi,z}-\imath q_{\semi,z}$ satisfies the
requirements for $p_{\semi,z}-\imath q_{\semi,z}$, and in addition it
is elliptic in the extended part of the domain. We usually write
$p_{\semi,z}-\imath q_{\semi,z}$ for this extension.
Thus, these properties need not be checked
individually in specific cases.

We remark that $f_z$ as above arises from the standard quantization $F_\sigma$ of
$$
\left(\|\varpi\|^{2j}_{\tilde H}+\sigma^{2j}+C^{2j}\right)^{1/2j},
$$
for $C>0$ arbitrarily chosen; the large-parameter rescaling $h F_{h^{-1}z}$ of this has
the semiclassical principal symbol $f_z$. Then for the induced
operators $P_\sigma-\imath Q_\sigma$, the operators are holomorphic in
the domain $\Omega_{j,C}$ given with $\Cx$ with the half-lines
$e^{i\pi(2k+1)/(2j)}[C,+\infty)$, $k$ an integer, removed, and thus as
$j$ and $C$ vary, these domains
cover $\Cx$.

It is useful to note the following explicit calculation regarding the time-like
character of $\frac{d\tau}{\tau}$ if we are given a Lorentzian b-metric $g$
that, with respect to some local boundary defining function $\taut$ and local product
decomposition $U\times [0,\delta)_\taut$ of $\bM$ near $U\subset X$ open, is of the form
$G=(\taut\pa_{\taut})^2-\tilde G$ on $U\times[0,\delta)_\taut$,
$\tilde G$ a Riemannian metric on $U$. In this case, if we define
$\tau=\taut e^\phi$, $\phi$ a function on $X$, so
$\frac{d\tau}{\tau}=\frac{d\taut}{\taut}+d\phi$, then
$$
\langle \frac{d\tau}{\tau},\frac{d\tau}{\tau}\rangle_G=
\langle \frac{d\taut}{\taut},\frac{d\taut}{\taut}\rangle_G
-\langle d\phi,d\phi\rangle_G=1-\langle d\phi,d\phi\rangle_{\tilde G},
$$
so $\frac{d\tau}{\tau}$ is time-like if $|d\phi|_{\tilde G}<1$. Note that the effect
of such a coordinate change on the Mellin transform of the normal
operator of $\Box_g$ is conjugation by $e^{-\imath\sigma\phi}$ since
$\taut^{-\imath\sigma}=\tau^{-\imath\sigma}e^{\imath\sigma\phi}$. Such a coordinate
change is useful when $G$ has a product structure on $U\times[0,\delta)_\taut$,
but $\taut$ is only a local boundary defining function on $U\times[0,\delta)_\taut$
(the product
structure might not extend smoothly beyond $U$), in which case it is useful to
see if one can conserve the time-like nature of $\frac{d\taut}{\taut}$ for
a global boundary defining function. This is directly relevant for the study
of conformally compact spaces in Subsection~\ref{subsec:conf-comp-results}.

Finally, we remark that if $\frac{d\tau}{\tau}$ is time-like and $f$ is supported
in $\tau<\tau_0$, the forward problem
for the wave equation,
$$
\Box_g u=f,\ u|_{\tau>\tau_0}=0,
$$
is uniquely solvable on $\bM\setminus X$ near $X$ by standard energy estimates
(a priori, without more structure, with no estimates on growth at $X$),
see e.g.\ \cite[Chapter~23]{Hor}. Applying Lemma~\ref{lemma:Mellin-expand}
in this case, assuming that the normal operator family has the structure
stated there, gives another solution, which must be equal to $u$ by uniqueness.
Thus, $u$ has the expansion stated in the lemma.

\subsection{Wave equation localization}\label{subsec:local-wave}
In this section we recall energy estimates and their consequences
when $P_\sigma\in\Diff^2(X)$ is the wave
operator to leading order in a region $\waveopen\subset X$, i.e.\ when
for a Lorentzian metric $h$ on $\waveopen$,
$P_\sigma-\Box_h\in\Diff^1(\waveopen)$, and indeed when an analogous
statement holds in the `large parameter' sense as well, with the
latter naturally arising by Mellin transforming b-wave equations.
These results are not needed for the Fredholm properties, but are very useful in
describing the asymptotics of the solutions of the wave equation on
Kerr-de Sitter space as they show that certain terms arising from cutoffs do not affect the
solution. These are also useful for giving an alternative
explanation why the choice of the extension of the
modified Laplacian across the boundary is unimportant on an
(asymptotically) hyperbolic space.

So assume that one is also given a function $t:\waveopen\to(t_0,t_1)$ with $dt$ time-like and
with $t$ proper on $\waveopen$. Then one has the standard energy
estimate; see e.g.\ \cite[Section~2.8]{Taylor:Partial-I} for a version
of these estimates (in a slightly different setup):

\begin{prop}
Assume $t_0<T_0<T_0'<T_1$. Then
$$
\|u\|_{H^1(t^{-1}([T_0',T_1]))}\leq C
\left(\|u\|_{H^1(t^{-1}([T_0,T_0']))}+\|P_\sigma u\|_{L^2(t^{-1}([T_0,T_1]))}\right).
$$
\end{prop}

We also note that in the various estimates in
Subsections~\ref{subsec:ellip-hyp}-\ref{subsec:complex-absorb} though
the error terms were stated globally as $\|u\|_{H^{-N}}$, in fact they
can be localized to any neighborhood $U$ of the projection to the base $X$ of $\WF'(G)$.

For us, even more important is a semiclassical version of this estimate. The setup is more
conveniently formulated in the large parameter setting, where the
large parameter is interpreted as the dual variable of an extra
variable of which the operator is independent. So with $\waveopen$ as above, consider the family
$P_\sigma\in\Diff^2(\waveopen)$ with large parameter dependence, and assume
that the large parameter principal symbol of $P_\sigma$, $p_\sigma$, is
the dual metric function $G$ on $T^*\waveopen\times \Tb^*\overline{\RR^+_\tau}$ of
an $\RR^+$-invariant (acting as dilations in the second factor on
$\waveopen\times\RR^+$) Lorentzian
b-metric $g$, where
$\sigma$ is the b-dual variable\footnote{We could also work with $T^*\RR$ and standard dual variables via a
logarithmic change of variables, changing dilations to translations,
but in view of the previous section, the b-setup is particularly
convenient.}
on $\Tb^*\overline{\RR^+}$. Suppose moreover that, as above, we are also given a function
$t:\waveopen\to(t_0,t_1)$
with $dt$ time-like and with $t$ proper on $\waveopen$. Then

\begin{prop}\label{prop:semicl-wave-prop}
Assume $t_0<T_0<T_0'<T_1<t_1$. Then, with $P_{h,z}=h^2 P_{h^{-1}z}$,
$$
\|u\|_{H^1_\semi(t^{-1}([T_0',T_1]))}\leq C
\left(\|u\|_{H^1_\semi(t^{-1}([T_0,T_0']))}+h^{-1}\|P_{h,z} u\|_{L^2(t^{-1}([T_0,T_1]))}\right).
$$
\end{prop}

\begin{proof} We start by remarking that the $L^2$ norm on the right
  hand side is just the $H^0_\semi$ norm, so this is a non-microlocal real
  principal type estimate except that there is no error term
  $h\|u\|_{H^N_\semi(t^{-1}([T_0,T_1]))}$ norm inside the parentheses
  on the right hand side (cf.\ the right hand side of
  \eqref{eq:radial-in-h}; although this is stated at a radial point,
  the estimate has a similar form, without constraints on $s$, away
  from these), and except that one would usually expect that both
$h^{-1}\|P_{h,z} u\|_{L^2(t^{-1}([T_0,T_1]))}$ and
$h\|u\|_{H^{-N}_\semi(t^{-1}([T_0,T_1]))}$ should in fact be taken on a
larger set, such as $t^{-1}([T_0,T'_1])$, $T'_1>T_1$. The point is
thus to gain these improvements; this is done by a version of the classical energy
estimates. We note that of these observations the only truly important
part is the absence of a term
$h\|u\|_{H^{-N}_\semi(t^{-1}([T_0,T_1']))}$, which is thus on a larger
set -- this would prevent the argument leading to
Proposition~\ref{prop:wave-local} below.

The following is essentially the standard proof of energy estimates,
see e.g.\ \cite[Section~2.8]{Taylor:Partial-I},
but in a different context. Here we phrase it as done in
\cite[Sections~3 and 4]{Vasy:AdS}. So consider $V_\sigma=-\imath Z_\sigma$,
$Z_\sigma=\chi(t)W_\sigma$, and let $W_\sigma$ be given by $W=G(dt,.)$ considered as a first order
differential operator on $\waveopen$. That is, on $\waveopen\times\RR^+$, $G(dt,.)$
gives a vector field of the form $W=W'+a\,\tau\pa_\tau$, with $W'$ a
vector field on $\waveopen$ and $a$ a function on $\waveopen$ (both independent of
$\tau$), and via the Mellin transform one can consider $W$ as
$W_\sigma=W'+a\sigma$, a $\sigma$-dependent first order differential operator
on $\waveopen$, with standard large-parameter dependence. Let $\Box$ be
the d'Alembertian of $g$, and let $\Box_\sigma$ be its Mellin
conjugate, so $P_\sigma-\Box_\sigma$ is first order, even in the large
parameter sense, on $\waveopen$. As usual in
energy estimates, we want to consider the `commutator'
\begin{equation}\label{eq:sigma-dep-comm}
-\imath(V^*_\sigma\Box_\sigma-\Box_\sigma V_\sigma).
\end{equation}
While this can easily be computed directly, in order to connect it to
the wave equation, we first recall the computation of
$-\imath(V^*\Box-\Box V)$ on $\waveopen\times\RR^+$, rephrase this in
terms of $\Diffb(\waveopen\times[0,\infty))$; note that
$\Box,V\in\Diffb(\waveopen\times[0,\infty))$ so
$$
-\imath(V^*\Box-\Box V)\in\Diffb^2(\waveopen\times[0,\infty)).
$$
Since all operators here are $\RR^+$-invariant, the b-expressions are
mostly a matter of notation.  We then Mellin transform to compute\footnote{
Operating with $-\log\tau$ in place of
$\tau$ one would have translation invariance, no changes required into the
b-notation, and one would use the Fourier transform.}
\eqref{eq:sigma-dep-comm}.

Then, the usual
computation, see \cite[Section~3]{Vasy:AdS} for it written down in
this form, using the standard summation convention, yields
\begin{equation}\begin{split}\label{eq:B_ij-formula}
&-\imath(V^*\Box-\Box V)=d^*Cd,\ C_{ij}=g_{i\ell}B_{\ell j}\\
&B_{ij}=-J^{-1}\pa_k(JZ^kG^{ij})
+G^{ik}(\pa_k Z^j)+G^{jk}(\pa_k Z^i),
\end{split}\end{equation}
where $C_{ij}$ are the matrix entries of $C$ relative to the basis $\{dz_\ell\}$
of the fibers of the cotangent bundle (rather than the b-cotangent
bundle), $z_\ell=y_\ell$ for $\ell\leq n-1$, the $y_\ell$ local
coordinates on a chart in $\waveopen$, $z_n=\tau$.
Expanding $B$ using $Z=\chi W$, and separating the terms
with $\chi$ derivatives, gives
\begin{equation}\begin{split}\label{eq:B_ij-exp-gen}
B_{ij}&=
G^{ik}(\pa_k Z^j)+G^{jk}(\pa_k Z^i)-J^{-1}\pa_k(JZ^kG^{ij})\\
&=(\pa_k \chi) \big(G^{ik}W^j+G^{jk}W^i-G^{ij}W^k\big)\\
&\qquad\qquad\qquad+\chi
\big(G^{ik}(\pa_k W^j)+G^{jk}(\pa_k W^i)-J^{-1}\pa_k(JW^kG^{ij})\big).
\end{split}\end{equation}
Multiplying the first term on the right hand side by
$\alpha_i\,\overline{\alpha_j}$ (and summing over $i,j$; here
$\alpha\in\Cx^n\simeq\Cx T^*_q(\waveopen\times\RR^+)$, $q\in\waveopen\times\RR^+$)
gives
\begin{equation}\begin{split}\label{eq:stress-energy}
E_{W,d\chi}(\alpha)&=
(\pa_k \chi) (G^{ik}W^j+G^{jk}W^i-G^{ij}W^k)\alpha_i\,\overline{\alpha_j}\\
&=( \alpha,d\chi)_G \,\overline{\alpha(W)}
+\alpha(W)\,( d\chi,\alpha)_G-d\chi(W) (\alpha,\alpha)_G=\chi'(t)\,E_{W,dt},\\
E_{W,dt}&=( \alpha,dt)_G \,\overline{\alpha(W)}
+\alpha(W)\,( dt,\alpha)_G-dt(W) (\alpha,\alpha)_G.
\end{split}\end{equation}
Now, $E_{W,dt}$ is twice the sesquilinear stress-energy tensor
associated to $\alpha$, $W$ and $dt$. This is well-known to be positive definite in
$\alpha$, i.e.\ $E_{W,d\chi}(\alpha)\geq 0$, with vanishing if and only if $\alpha=0$,
when
$W$ and $dt$ are both forward time-like for smooth Lorentz metrics,
see e.g.\ \cite[Section~2.7]{Taylor:Partial-I} or
\cite[Lemma~24.1.2]{Hor}; \eqref{eq:real-stress-energy} provides the
computation when $\alpha$ is real.

We change to a local basis of the b-cotangent bundle and use the
b-differential $\bdiff=(d_X,\tau\pa_\tau)$ and the local basis
$\{dy_1,\ldots,dy_{n-1},\frac{d\tau}{\tau}\}$ of the fibers of the b-cotangent
bundle, $\hat\pa_j=\delta_j\pa_j$, $\delta_j=1$ for $j\leq n-1$, $\delta_n=\tau$,
for the local basis of the fibers of the b-tangent bundle,
$\hat G^{ij}$, $\hat g_{ij}$ the corresponding metric entries, $\hat
Z^i$ the vector field components, yields
\begin{equation}\begin{split}\label{eq:b-B_ij-formula}
&-\imath(V^*\Box-\Box V)=\bdiff^*\,\hat C\,\bdiff,\ \hat C_{ij}=\hat g_{i\ell}\hat B_{\ell j}\\
&\hat B_{ij}=-J^{-1}\delta_k^{-1}\delta_i^{-1}\delta_j^{-1}\hat\pa_k(J\delta_k\hat Z^k\delta_i\delta_j\hat G^{ij})
+\hat G^{ik}(\delta_j^{-1}\hat\pa_k \delta_j\hat Z^j)+\hat G^{jk}(\delta_i^{-1}\hat\pa_k \delta_i\hat Z^i).
\end{split}\end{equation}
While the $\delta$-factors may have non-vanishing derivatives in the
above expresson for $\hat B$, if these are differentiated, $\chi$ in
$\hat Z^i=\chi\hat W^i$ is not, so we conclude that
\begin{equation*}\begin{split}
\hat B_{ij}=&(\hat\pa_k \chi) \big(\hat G^{ik}\hat W^j+\hat G^{jk}\hat
W^i-\hat G^{ij}\hat W^k\big)\\
&\qquad+\chi
\big(\hat G^{ik}\delta_j^{-1}(\hat \pa_k \delta_jW^j)+G^{jk}\delta_i^{-1}(\hat\pa_k
\delta_i W^i)\\
&\qquad\qquad\qquad-J^{-1}\delta_k^{-1}\delta_i^{-1}\delta_j^{-1}\hat\pa_k(J\delta_k\delta_i\delta_jW^kG^{ij})\big),
\end{split}\end{equation*}
and so
\begin{equation*}
\hat C=\chi' A+\chi R,
\end{equation*}
with $A$ positive definite.

The Mellin transformed version of \eqref{eq:b-B_ij-formula} finally
computes \eqref{eq:sigma-dep-comm}; it reads as
\begin{equation*}
-\imath(V^*_\sigma\Box_\sigma-\Box_\sigma V_\sigma)=d^*_\sigma
\hat C d_\sigma,\ \hat C_{ij}=g_{i\ell}B_{\ell j},
\ \hat C=\chi'A+\chi R
\end{equation*}
where $d_\sigma=(d_X,\sigma)$, with the last component being
multiplication by $\sigma$, and $A$ is positive definite.
Correspondingly,
\begin{equation}\begin{split}\label{eq:perturbed-wave-sigma-comm}
-\imath(V^*_\sigma P_\sigma-P_\sigma^* V_\sigma)&=-\imath(V^*_\sigma
\Box_\sigma-\Box_\sigma V_\sigma) -\imath V^*_\sigma
(P_\sigma-\Box_\sigma)+\imath (P_\sigma-\Box_\sigma)^* V_\sigma\\
&=d^*_\sigma
\tilde C d_\sigma+\hat E \chi d_\sigma+d_\sigma^* \chi\hat E^*,
\end{split}\end{equation}
where $\hat E=(\hat E_X,\hat E')$, $\hat E_X\in\CI(\waveopen;TX)$, $\hat
E'\in\CI(\waveopen)$,
and
$$
\tilde C=\chi'A+\chi \tilde R,
$$
since the contribution of $P_\sigma-\Box_\sigma$ to second order terms
in the large parameter sense is only via terms not differentiated in $\chi$.
A standard argument, given below, making $\chi'$ large relative to $\chi$, can be
used to complete the proof.

Indeed, let
$\chi_0(s)=e^{-1/s}$ for $s>0$, $\chi_0(s)=0$ for $s\leq 0$, $\chi_1\in\CI(\RR)$
identically 1 on $[1,\infty)$, vanishing on $(-\infty,0]$,
Thus,
$s^2\chi_0'(s)=\chi_0(s)$ for $s\in\RR$.
Now let $T_1'\in(T_1,t_1)$, and
consider
$$
\chi(s)=\chi_0(-\digamma^{-1}(s-T'_1))\chi_1((s-T_0)/(T_0'-T_0)),
$$ 
so
$$
\supp\chi\subset [T_0,T'_1]
$$ 
and
$$
s\in [T'_0,T'_1]\Rightarrow \chi'=-\digamma^{-1}\chi_0'(-\digamma^{-1}(s-T'_1)),
$$
so
$$
s\in [T'_0,T'_1]\Rightarrow \chi=-\digamma^{-1} (s-T'_1)^2\chi',
$$
so for $\digamma>0$ sufficiently large, this is bounded by a small
multiple of $\chi'$, namely
\begin{equation}\label{eq:chip-gamma-est}
s\in [T'_0,T'_1]\Rightarrow \chi\leq -\gamma\chi',
\ \gamma=(T'_1-T'_0)^2\digamma^{-1}.
\end{equation}
In particular, for sufficiently large
$\digamma$,
$$
-(\chi' A+\chi R)\geq -\chi' A/2
$$
on $[T'_0,T'_1]$. Thus,
$$
\langle d^*_\sigma \tilde Cd_\sigma u,u\rangle\geq -\frac{1}{2}\langle \chi'
A d_\sigma u,d_\sigma u\rangle-C'\|d_\sigma u\|^2_{L^2(t^{-1}([T_0,T'_0]))}.
$$
So for some $c_0>0$,  by \eqref{eq:perturbed-wave-sigma-comm},
\begin{equation}\begin{split}\label{eq:Box-V-est}
c_0&\|(-\chi')^{1/2}d_\sigma u\|^2\leq-\frac{1}{2}\langle \chi'
A d_\sigma u,d_\sigma u\rangle\\
&\leq C'\|d_\sigma
u\|^2_{L^2(t^{-1}([T_0,T'_0]))}+C'\|\chi^{1/2}P_\sigma
u\|\|\chi^{1/2}d_\sigma u\|+C'\|\chi^{1/2} u\|\|\chi^{1/2}d_\sigma u\|\\
&\leq C'\|d_\sigma
u\|^2_{L^2(t^{-1}([T_0,T'_0]))}+C'\|\chi^{1/2}P_\sigma
u\|^2+2C'\gamma\|(-\chi')^{1/2}d_\sigma u\|^2\\
&\qquad\qquad\qquad \qquad\qquad\qquad+C'\gamma^{-1}|\sigma|^{-1}\|(-\chi')^{1/2}d_\sigma u\|^2,
\end{split}\end{equation}
where we used $\|\chi^{1/2}u\|\leq |\sigma|^{-1}\|\chi^{1/2}d_\sigma
u\|$ (which holds in view of the last component of $d_\sigma$).
Thus, choosing first $\digamma>0$ sufficiently large (thus $\gamma>0$
is sufficiently small), and then $|\sigma|$
sufficiently
large, the last two terms on the right hand
side can be absorbed into the left hand side. Rewriting in the
semiclassical notation gives the desired result, except that
$\|P_{h,z} u\|_{L^2_\semi(t^{-1}([T_0,T_1]))}$ is replaced by
$\|P_{h,z} u\|_{L^2_\semi(t^{-1}([T_0,T'_1]))}$
(or $\|\chi^{1/2}P_{h,z}
u\|_{L^2_\semi(t^{-1}([T_0,T'_1]))}$). This however is easily
remedied by replacing $\chi$ by
$$
\chi(s)=H(T_1-s)\chi_0(-\digamma^{-1}(s-T'_1))\chi_1((s-T_0)/(T_0'-T_0)),
$$
where $H$ is the Heaviside step function (the characteristic function
of $[0,\infty)$) so $\supp\chi\subset [T_0,T_1]$. Now $\chi$ is not
smooth, but either approximating $H$ by smooth bump functions and
taking a limit, or indeed directly performing the calculation,
integtating on the domain with boundary $t\leq T_1$, the contribution
of the derivative of $H$ to $\chi'$ is a delta distribution at
$t=T_1$, corresponding to a boundary term on the domain, which has the
same sign as the derivative of $\chi_0$. Thus, with $\cS_{T_1}$ the
hypersurface $\{t=T_1\}$, \eqref{eq:Box-V-est}
holds in the form
\begin{equation*}\begin{split}
c_0&\|d_\sigma u\|^2_{L^2(t^{-1}([T'_0,T_1]))}+c_0\|d_\sigma u\|^2_{L^2(\cS_{T_1})}\\
&\leq C'\|d_\sigma
u\|^2_{L^2(t^{-1}([T_0,T'_0]))}+C'\|\chi^{1/2}P_\sigma
u\|^2+2C'\gamma\|(-\chi')^{1/2}d_\sigma u\|^2\\
&\qquad\qquad\qquad \qquad\qquad\qquad+C'\gamma^{-1}|\sigma|^{-1}\|(-\chi')^{1/2}d_\sigma u\|^2.
\end{split}\end{equation*}
Now one can simply drop the second term from the left hand side and
proceed as above; the
semiclassical rewriting now proves the claimed result.
\end{proof}

Suppose now that $t$ is a globally defined function on $X$, with
$t|_\waveopen$ having the properties discussed above, and such that
$p_{\semi,z}$ is {\em semiclassically non-trapping, resp.\ mildly
trapping}, in $t^{-1}((-\infty,T_1])$, in the sense that in
Definitions~\ref{Def:non-trap}, resp.\ Definition~\ref{Def:mild-trap}, $\elliptic(q_{\semi,z})$
is replaced by $\overline{T}^*_{t=T_1}X$ (and $X$ itself is replaced by $t^{-1}((-\infty,T_1])$).
Proposition~\ref{prop:semicl-wave-prop}
can be used to show directly that when $\re\sigma$ is
large, $\sigma$ is in a strip, if $\supp P_\sigma u\subset
t^{-1}([T_1,+\infty))$, then $\supp u\subset
t^{-1}([T_1,+\infty))$. Indeed, by the discussion preceding
Theorem~\ref{thm:classical-absorb-strip}, if $P_\sigma-\imath
Q_\sigma$ is semiclassically non-trapping, we have, with $T_0'<T_0''<T_1$, and
suitably large $s$ (which we may take to satisfy $s\geq 1$),
$$
\|u\|_{H^s_\semi(t^{-1}(-\infty,T_0'])}\leq
C\big(h^{-1}\|P_{h,z} u\|_{H^{s-1}_\semi(t^{-1}(-\infty,T_0''])}
+h\|u\|_{H^{-N}_\semi(t^{-1}(-\infty,T_0''))}\big).
$$
If instead $P_\sigma-\imath Q_\sigma$ is mildly trapping of
order $\varkappa$ then
$$
\|u\|_{H^s_\semi(t^{-1}(-\infty,T_0'])}\leq
C\big(h^{-\varkappa-1}\|P_{h,z} u\|_{H^{s-1}_\semi(t^{-1}(-\infty,T_0''])}
+h\|u\|_{H^{-N}_\semi(t^{-1}(-\infty,T_0''))}\big).
$$
Using this in combination with
Proposition~\ref{prop:semicl-wave-prop} yields
$$
\|u\|_{H^1_\semi(t^{-1}(-\infty,T_1])}\leq
C\big(h^{-\varkappa-1}\|P_{h,z}u\|_{H^{s-1}_\semi(t^{-1}(-\infty,T_1])}
+h\|u\|_{H^{-N}_\semi(t^{-1}(-\infty,T_0''])}\big).
$$
Now, for sufficiently small $h$, the second term on the right hand
side can be absorbed into the left hand side to yield
\begin{equation}\label{eq:semicl-wave-loc}
\|u\|_{H^1_\semi(t^{-1}(-\infty,T_1])}\leq
Ch^{-\varkappa-1}\|P_{h,z}u\|_{H^{s-1}_\semi(t^{-1}(-\infty,T_1])}.
\end{equation}
This shows that for $h$ sufficiently small, i.e.\ $\re\sigma$
sufficiently large, the vanishing of $P_\sigma u$ in $\{t<T_1\}$ gives
that of $u$ in the same region.

Moving now to the operator $P_\sigma-\imath Q_\sigma$, we thus conclude:

\begin{prop}\label{prop:wave-local}
Suppose $\waveopen$, $t$, $P_\sigma$ are as discussed before the
statement of Proposition~\ref{prop:semicl-wave-prop}, with $t$
globally defined on $X$, and $P_\sigma$, $Q_\sigma$ as in
Theorem~\ref{thm:classical-absorb-strip} or as in Theorem~\ref{thm:classical-absorb-glued}.
Suppose also that $p_{\semi,z}$ is semiclassically mildly trapping
in $t^{-1}((-\infty,T_1])$ in the sense discussed above.
Finally, suppose that the Schwartz kernel of $Q_\sigma$ is supported in
$t^{-1}((T_1,+\infty))\times t^{-1}((T_1,+\infty))$.
If $\supp f\subset
t^{-1}([T_1,+\infty))$, then $\supp (P_\sigma-\imath Q_\sigma)^{-1}f\subset
t^{-1}([T_1,+\infty))$. 
\end{prop}

\begin{proof}
Note that $u=(P_\sigma-\imath Q_\sigma)^{-1}f$ satisfies
$(P_\sigma-\imath Q_\sigma) u=f$, so in view of the support condition
on $Q_\sigma$, $\supp P_\sigma u\subset t^{-1}([T_1,+\infty))$.
For $\sigma$ in a strip $|\im\sigma|<C'$ and with $\re\sigma$
sufficiently large, the proposition then follows from
\eqref{eq:semicl-wave-loc}. Thus, with $\psi$ supported in
$t^{-1}((-\infty,T_1))$, $\phi$ supported in $t^{-1}((T_1,+\infty))$, and $\sigma$ in
this region, $\psi (P_\sigma^{-1}-\imath Q_\sigma)\phi=0$. By the meromorphicity of
$\psi (P_\sigma-\imath Q_\sigma)^{-1}\phi$,
$\psi (P_\sigma-\imath Q_\sigma^{-1})\phi=0$ follows for all $\sigma$.
\end{proof}

One reason this proposition is convenient is that it shows that for
$\psi$ supported in $t^{-1}((-\infty,T_1))$, $\psi  (P_\sigma-\imath
Q_\sigma^{-1})$ is independent of the choice of $Q_\sigma$ (satisfying
the general conditions); analogous results hold for modifying
$P_\sigma$ in $t^{-1}((T_1,+\infty))$. Indeed, let $Q'_\sigma$ be another operator satisfying
conditions analogous to those on $Q_\sigma$ for $\sigma$ in some open set
$\Omega'\subset\Cx$, and let $(\cX')^s\subset H^s$ be the corresponding function space in
place of $\cX^s$ (note that $\cY^s=H^{s-1}$ is independent of $Q_\sigma$);
thus $(P_\sigma-\imath Q_\sigma)^{-1}:\cY^s\to\cX^s$ and
$(P_\sigma-\imath Q'_\sigma)^{-1}\cY^s\to(\cX')^s$ are meromorphic on $\Omega\cap\Omega'$.
Then for $\sigma\in\Omega\cap\Omega'$ which is not
a pole of either $(P_\sigma-\imath Q_\sigma)^{-1}$ or
$(P_\sigma-\imath Q'_\sigma)^{-1}$, and for $f\in\cY^s$,
let $u=(P_\sigma-\imath Q_\sigma)^{-1}f$,
$u'=(P_\sigma-\imath Q'_\sigma)^{-1}f$. Then $(P_\sigma-\imath
Q_\sigma)u'=f+\imath (Q_\sigma-Q'_\sigma)u'\in H^{s-2}$, and thus,
provided $s-1>(1-\beta\im\sigma)/2$ (rather than this inequality holding merely
for $s$),
$$
u'=(P_\sigma-\imath Q_\sigma)^{-1}f+(P_\sigma-\imath
Q_\sigma)^{-1}\imath (Q_\sigma-Q'_\sigma)u',
$$
so
$$
\psi (P_\sigma-\imath Q'_\sigma)^{-1}f=\psi(P_\sigma-\imath
Q_\sigma)^{-1}f
$$
since $\psi(P_\sigma-\imath
Q_\sigma)^{-1}\imath (Q_\sigma-Q'_\sigma)=0$ in view of the support
properties of $Q_\sigma$ and $Q'_\sigma$. In particular, one may drop
the particular choice of $Q_\sigma$ from the notation; note also that
this also establishes the equality for $s>(1-\beta\im\sigma)/2$ since $(P_\sigma-\imath
Q_\sigma)^{-1}$ is independent of $s$ in this range in the sense of
Remark~\ref{rem:inv-indep-of-s}. This is particularly helpful
if for $\sigma$ in various open subsets $\Omega_j$ of $\Cx$ we
construct different operators $Q^{(j)}_\sigma$; if for instance for
each $\Omega_j$, $(P_\sigma-\imath Q^{(j)}_\sigma)^{-1}$ is meromorphic,
resp.\ holomorphic, the same follows for the single operator family
(independent of $j$) $\psi (P_\sigma-\imath Q_\sigma)^{-1}$ where we
now we write $Q_\sigma$ for any of the valid choices (i.e.\
$Q_\sigma=Q^{(j)}_\sigma$ for any one of the $j$'s such that
$\sigma\in\Omega_j$).
We then have the
following extension of Lemma~\ref{lemma:Mellin-expand}.

\begin{cor}\label{cor:Mellin-expand}
Suppose $\cP$
is invariant under dilations in $\tau$ for functions supported near
$\tau=0$,
and the normal operator family $\hat N(\cP)$ is of the form
$P_\sigma$ satisfying the conditions\footnote{Again, as discussed in
Remark~\ref{rem:off-real-not-imp}, the large $\im\sigma$ assumptions
only affect the existence part below, and do so relatively mildly.} of
Section~\ref{sec:microlocal},
and such that there are is an open cover of $\Cx$ by sets $\Omega_j$,
and for each $j$ there is an operator $Q^{(j)}_\sigma$ satisfying
the conditions of Section~\ref{sec:microlocal},
including
semiclassical non-trapping. Let $t$ be as in
Proposition~\ref{prop:wave-local}, $\psi$ supported in
$t^{-1}((-\infty,T_1))$, identically $1$ on $t<T_1'$.
Let $\sigma_j$ be the poles of the meromorphic
family\footnote{As remarked above, these are independent of the choice
 of $j$ for $\sigma\in\Cx$ as long as $\sigma\in\Omega_j$.}
$\psi (P_\sigma-\imath Q_\sigma)^{-1}$. Then for $\ell<\beta^{-1}(2s-\difford+1)$,
$\ell\neq-\im\sigma_j$ for any $j$,
$\cP u=f$ in $t<T_1$, $u$ tempered, supported near $\tau=0$,
$f\in \tau^\ell \Hb^{s-\difford+1}(\bM_\infty)$, in $t<T_1'$,
$u$ has an asymptotic expansion
\begin{equation}\label{eq:Mellin-simple-non-trap-mod}
u=\sum_j\sum_{\kappa\leq m_j} \tau^{\imath\sigma_j}(\log |\tau|)^\kappa a_{j\kappa}+u'
\end{equation}
with $a_{j\kappa}\in\CI(X)$ and $u'\in \tau^\ell \Hb^{s}(\bM_\infty)$.

If instead the family $P_\sigma-\imath Q_\sigma$ is semiclassically mildly trapping
of order $\varkappa$ in a $C_0$-strip then for $\ell<C_0$ and
$f\in \tau^\ell \Hb^{s-\difford+1+\varkappa}(\bM_\infty)$ one has, in $t<T_1'$,
\begin{equation}\label{eq:Mellin-simple-mild-trap-mod}
u=\sum_j\sum_{\kappa\leq m_j} \tau^{\imath\sigma_j}(\log |\tau|)^\kappa a_{j\kappa}+u'
\end{equation}
with $a_{j\kappa}\in\CI(X)$ and $u'\in \tau^\ell \Hb^{s}(\bM_\infty)$.

Conversely, given $f$ in the indicated spaces, with $f$
supported near $\tau=0$, a solution $u$ of $\cP u=f$ of the form
\eqref{eq:Mellin-simple-non-trap-mod}, resp.\ \eqref{eq:Mellin-simple-mild-trap-mod},
supported near $\tau=0$ exists in $t<T_1'$.

In either case,
the coefficients $a_{j\kappa}$ are given by the Laurent coefficients
of $\psi(\cP-\imath\cQ)^{-1}$ at the poles $\sigma_j$ applied to $f$,
with simple poles corresponding
to $m_j=0$.

If $f=\sum_j\sum_{\kappa\leq m'_j} \tau^{\alpha_j} (\log |\tau|)^\kappa b_{j\kappa}+f'$, with
$f'$ in the spaces indicated above for $f$, and $b_{j\kappa}\in H^{s-\difford+1}(X)$,
analogous results hold when the expansion of $f$ is added to the form of
\eqref{eq:Mellin-simple-non-trap-mod} and \eqref{eq:Mellin-simple-mild-trap-mod},
in the sense of the extended union of index sets, see
\cite[Section~5.18]{Melrose:Atiyah}.

Further, the result is stable under sufficiently small dilation-invariant
perturbations in the b-sense, i.e.\ if $\cP'$ is sufficiently close to
$\cP$ in $\Psib^\difford(\bM_\infty)$ with real principal symbol, then there is a
similar expansion for solutions of $\cP' u=f$ in $t<T_1'$.

For $\cP^*$ in place of $\cP$, analogous results apply, but
we need $\ell<-\beta^{-1}(2s-\difford+1)$, and the $a_{j\kappa}$ are
not smooth, but have wave front set in the Lagrangians $\Lambda_\pm$.
\end{cor}

\begin{rem}\label{rem:Mellin-expand}
Proposition~\ref{prop:Mellin-expand} has an analogous extension, but we do
not state it here explicitly.

Further, estimates analogous to
Remark~\ref{rem:estimates-for-expansion} are applicable, with the
norms of restrictions to $t<T'_1$ bounded in terms of the norms of
restrictions to $t<T_1$.
\end{rem}

\begin{proof}
We follow the proof of Lemma~\ref{lemma:Mellin-expand} closely.
Again, we first consider the expansion, and
let $\alpha,r\in\RR$ be such that
$u\in\tau^\alpha \Hb^r(\bM_\infty)$ and $\psi(P_\sigma-\imath
Q_\sigma)^{-1}$ has\footnote{Recall that this operator, when
  considered as a product, refers to $\psi(P_\sigma-\imath
Q^{(j)}_\sigma)^{-1}$, with $j$ appropriately chosen.} no
poles on $\im\sigma= -\alpha$. These $\alpha,r$ exist due to
the vanishing of $u$ for $\tau>1$ and the absence of poles of $\psi(P_\sigma-\imath
Q^{(j)}_\sigma)^{-1}$ for $\re\sigma$ large, $\sigma$ in a strip; then also
$u\in\tau^\alpha (1+\tau)^{-N}\Hb^r(\bM_\infty)$ for all $N$.
The Mellin transform of the PDE, a priori on $\im\sigma=-\alpha$, is
$P_\sigma\cM u=\cM f$, and thus $(P_\sigma-\imath Q_\sigma)\cM
u=f-\imath Q_\sigma \cM u$. Thus,
\begin{equation}\label{eq:Mellin-PDE-mod}
\cM u=(P_\sigma-\imath Q_\sigma)^{-1}\cM f-(P_\sigma-\imath Q_\sigma)^{-1}\imath Q_\sigma \cM u
\end{equation}
there. Restricting to $t<T_1$, the last term vanishes by
Proposition~\ref{prop:wave-local}, so
$$
\cM u|_{t<T'_1}=(P_\sigma-\imath Q_\sigma)^{-1}\cM f|_{t<T'_1}
$$
If $f\in \tau^\ell \Hb^{s-\difford+1}(\bM_\infty)$, then
shifting the contour of integration to $\im\sigma=-\ell$, we
obtain contributions from the poles of $(P_\sigma-\imath Q_\sigma)^{-1}$, giving
the expansion in \eqref{eq:Mellin-simple-non-trap-mod} and
\eqref{eq:Mellin-simple-mild-trap-mod} by Cauchy's theorem.
The error term $u'$ is what one obtains
by integrating along the new contour in view of the high energy bounds
on $(P_\sigma-\imath Q_\sigma)^{-1}$ (which differ as one changes one's
assumption from non-trapping to mild trapping), and the assumptions on $f$.

Conversely, to obtain existence, let $\alpha<\min(\ell,-\sup\im\sigma_j)$ and
define $u\in \tau^\alpha \Hb^{s}(\bM_\infty)$ by
$$
\cM u=(P_\sigma-\imath Q_\sigma)^{-1}\cM f,
$$
using the inverse Mellin transform with $\im\sigma=-\alpha$.
Then
$$
P_\sigma\cM u=\cM f+\imath Q_\sigma\cM u,
$$
and so $P_\sigma\cM u|_{t<T_1}=\cM f|_{t<T_1}$. Thus,
the expansion follows by the first part of the
argument. The support property of $u$ follows from Paley-Wiener, taking
into account holomorphy in $\im\sigma>-\alpha$, and the
estimates on
$\cM f$ and $\psi(P_\sigma-\imath Q_\sigma)^{-1}$ there.

Finally, stability under perturbations follows for the same reasons as
those stated
in Lemma~\ref{lemma:Mellin-expand} once one remarks that the
$Q_\sigma^{(j)}$ used for $\cP$ can actually be used for $\cP'$ as
well provided these two are sufficiently close, since the relationship
between $P_\sigma$ and $Q_\sigma^{(j)}$ is via ellipticity considerations,
and these are preserved under small perturbations of $P_\sigma$.
\end{proof}

\section{De Sitter space and conformally compact spaces}\label{sec:dS}

\subsection{De Sitter space as a symmetric space}
Rather than starting with the static picture of de Sitter space, we consider it
as a Lorentzian symmetric space. We follow the treatment of
\cite{Vasy:De-Sitter} and \cite{Melrose-SaBarreto-Vasy:Asymptotics}.
De Sitter space $M$ is given by the hyperboloid
$$
z_1^2+\ldots+z_n^2=z_{n+1}^2+1\ \text{in}\ \Real^{n+1}
$$
equipped with the pull-back $g$ of the Minkowski metric
$$
dz_{n+1}^2-dz_1^2-\ldots-dz_n^2.
$$
Introducing polar coordinates $(R,\theta)$ in $(z_1,\ldots,z_n)$,
so
$$
R=\sqrt{z_1^2+\ldots+z_n^2}=\sqrt{1+z_{n+1}^2},
\ \theta=R^{-1}(z_1,\ldots,z_n)\in\sphere^{n-1}
,\ \taut=z_{n+1},
$$
the hyperboloid
can be identified with $\Real_\taut\times\sphere^{n-1}_\theta$ with the
Lorentzian metric
\begin{equation*}
g=\frac{d\taut^2}{\taut^2+1}-(\taut^2+1)\,d\theta^2,
\end{equation*}
where $d\theta^2$ is the standard Riemannian metric on the sphere.
For $\taut>1,$ set $x=\taut^{-1}$, so the metric becomes
$$
g=\frac{(1+x^2)^{-1}\,dx^2-(1+x^2)\,d\theta^2}{x^2}.
$$
An analogous formula holds for $\taut<-1$, so compactifying
the real line to an interval $[0,1]_T$, with $T=x=\taut^{-1}$
for $x<\frac{1}{4}$ (i.e.\ $\taut>4$), say, and $T=1-|\taut|^{-1}$, $\taut<-4$,
gives a compactification, $\hM,$ of
de Sitter space on which the metric is conformal to a non-degenerate
Lorentz metric. There is natural generalization, to 
{\em asymptotically de Sitter-like spaces} $\hM$,
which are diffeomorphic to compactifications $[0,1]_T\times Y$
of $\RR_\taut\times Y$, where $Y$ is a compact manifold
without boundary, and $\hM$ is equipped
with a Lorentz metric on its interior which is conformal
to a Lorentz metric smooth up to the boundary. These
space-times are Lorentzian analogues of the much-studied conformally
compact (Riemannian) spaces. On this class of space-times the solutions of
the Klein-Gordon equation were analyzed by Vasy
in \cite{Vasy:De-Sitter}, and were shown to have simple asymptotics
analogous to those for generalized eigenfunctions on conformally compact manifolds.

\begin{thm*}(\cite[Theorem~1.1.]{Vasy:De-Sitter})
Set $s_\pm(\lambda)=\frac{n-1}{2}\pm\sqrt{\frac{(n-1)^2}{4}-\lambda}.$
If $s_+(\lambda)-s_-(\lambda)\notin\Nat,$ any solution $u$ of the Cauchy
problem for $\Box-\lambda$ with $\CI$ initial data at $\taut=0$ is of the form
\begin{equation*}
u=x^{s_+(\lambda)}v_++x^{s_-(\lambda)}v_-,\ v_\pm\in\CI(\hM).
\end{equation*}
If $s_+(\lambda)-s_-(\lambda)$ is an integer, the same conclusion holds if
$v_-\in\CI(\hM)$ is replaced by $v_-=\CI(\hM)
+x^{s_+(\lambda)-s_-(\lambda)}\log x\,\CI(\hM)$.
\end{thm*}

\begin{figure}[ht]
\begin{center}
\mbox{\epsfig{file=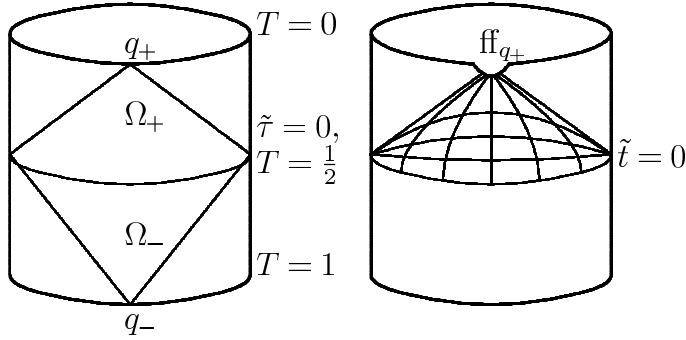}}
\end{center}
\caption{On the left, the compactification of de Sitter space with the
backward light cone from $q_+=(1,0,0,0)$ and forward light cone from $q_-
=(-1,0,0,0)$ are
shown. $\Omega_+$, resp.\ $\Omega_-$, denotes the intersection of these
light cones with $\tilde\tau>0$, resp.\ $\tilde\tau<0$.
On the right, the blow up of de Sitter space at $q_+$ is shown. The
interior of the light cone inside the front face $\ff_{q_+}$ can
be identified with the spatial part of the static model of de Sitter
space. The spatial and temporal coordinate lines for the static model
are also shown.}
\label{fig:deSitter-space}
\end{figure}

One important feature of asymptotically de Sitter spaces is the following:
a conformal factor, such as $x^{-2}$ above, does not change the image
of null-geodesics, only reparameterizes them. More precisely, recall that
null-geodesics
are merely projections to $M$ of null-bicharacteristics of the metric function
in $T^* M$. Since $p\mapsto\sH_p$ is a derivation, $ap\mapsto a\sH_p+p\sH_a$,
which is $a\sH_p$ on the characteristic set of $p$. Thus, the null-geodesics of
de Sitter space are the same (up to reparameterization) as those of the metric
$$
(1+x^2)^{-1}\,dx^2-(1+x^2)\,d\theta^2
$$
which is smooth on the compact space $\hM$.

\subsection{The static model of a part of de Sitter space}
The simple structure of de Sitter metric (and to some extent of
the asymptotically de Sitter-like metrics) can be hidden by
blowing up certain submanifolds of the boundary of $\hM$.
In particular, the {\em static model} of de Sitter space arises by
singling out a point on
$\sphere^{n-1}_{\theta}$,
e.g.\ $q_0=(1,0,\ldots,0)\in\sphere^{n-1}\subset\RR^n$.
Note that $(\theta_2,\ldots,\theta_n)\in\RR^{n-1}$ are local
coordinates on $\sphere^{n-1}$ near $q_0$.
Now consider
the intersection of the backward light cone
from $q_0$ considered as a point $q_+$ at future infinity, i.e.\ 
where $T=0$, and the forward
light cone from $q_0$ considered as a point $q_-$ at past infinity,
i.e.\ where $T=1$.
These intersect the equator $T=1/2$ (here $\taut=0$)
in the same set, and together form a `diamond', $\hat\Omega$, with a conic
singularity at $q_+$ and $q_-.$ Explicitly $\hat\Omega$
is given by $z_2^2+\ldots+z_n^2\leq 1$ inside the
hyperboloid. If $q_+$, $q_-$ are blown up, as well as
the corner $\pa\Omega\cap\{\taut=0\}$, i.e.\ where
the light cones intersect $\taut=0$ in $\hat\Omega$,
we obtain a manifold
$\bM$,
which can be blown down to (i.e.\ is a blow up of) the space-time
product $[0,1]\times\BB^{n-1}$, with
$\BB^{n-1}=\{Z\in\RR^{n-1}:\ |Z|\leq 1\}$ on which the Lorentz metric has
a time-translation invariant warped product form. Namely, first considering
the interior $\Omega$ of $\hat\Omega$ we introduce
the global (in $\Omega$) standard static
coordinates $(\tilde t,Z)$, given by (with the expressions
involving $x$ valid near $T=0$)
\begin{equation*}\begin{split}
(\BB^{n-1})^\circ\ni Z&=(z_2,\ldots,z_n)=x^{-1}\sqrt{1+x^2}(\theta_2,\ldots,\theta_n),\\
\sinh \tilde t&=\frac{z_{n+1}}{\sqrt{z_1^2-z_{n+1}^2}}=(x^2-(1+x^2)(\theta_2^2+
\ldots+\theta_n^2))^{-1/2}.
\end{split}\end{equation*}
It is convenient to rewrite these as well in terms of polar coordinates
in $Z$ (valid away from $Z=0$):
\begin{equation*}\begin{split}
r&=\sqrt{z_2^2+\ldots+z_n^2}=\sqrt{1+z_{n+1}^2-z_1^2}
=x^{-1}\sqrt{1+x^2}\sqrt{\theta_2^2+\ldots
+\theta_n^2},\\
\sinh \tilde t&=\frac{z_{n+1}}{\sqrt{z_1^2-z_{n+1}^2}}=(x^2-(1+x^2)(\theta_2^2+
\ldots+\theta_n^2))^{-1/2}=x^{-1}(1-r^2)^{-1/2},\\
&\ \omega=r^{-1}(z_2,\ldots,z_n)=(\theta_2^2+\ldots+\theta_n^2)^{-1/2}
(\theta_2,\ldots,\theta_n)\in\sphere^{n-2}.
\end{split}\end{equation*}
In these coordinates the
metric becomes
\begin{equation}\label{eq:static-polar-orig}
(1-r^2)\,d\tilde t^2-(1-r^2)^{-1}dr^2-r^2\,d\omega^2,
\end{equation}
which is a special case of the de Sitter-Schwarzschild metrics with vanishing
mass, $M=0$,
and cosmological constant $\Lambda=3$, see Section~\ref{sec:Kerr}.
Correspondingly, the dual metric is
\begin{equation}\label{eq:static-polar}
(1-r^2)^{-1}\pa_{\tilde t}^2-(1-r^2)\pa_r^2-r^{-2}\pa_{\omega}^2.
\end{equation}
We also rewrite this in terms of coordinates valid at the origin, namely
$Y=r\omega$:
\begin{equation}\label{eq:static-Eucl}
(1-|Y|^2)^{-1}\pa_{\tilde t}^2+(\sum_{j=1}^{n-1} Y_j\pa_{Y_j})^2-\sum_{j=1}^{n-1}\pa_{Y_j}^2.
\end{equation}

\subsection{Blow-up of the static model}
We have already seen that de Sitter space has a smooth conformal
compactification; the singularities in the metric of the form
\eqref{eq:static-polar-orig} at $r=1$ must thus be
artificial. On the other hand, the metric is already well-behaved for $r<1$ bounded
away from $1$, so we want the coordinate change to be smooth there --- this
means smoothness in valid coordinates ($Y$ above) at the origin
as well. This singularity can be removed by a blow-up on an appropriate
compactification.
We phrase this at first in a way that is closely related to our treatment of
Kerr-de Sitter space, and the Kerr-star coordinates. So let
$$
t=\tilde t+h(r),\ h(r)=-\frac{1}{2}\log\mu,\ \mu=1-r^2.
$$
Note that $h$ is smooth at the origin,
A key feature of this change of coordinates is
$$
h'(r)=-\frac{r}{\mu}=-\frac{1}{\mu}+\frac{1}{1+r},
$$
which is $-\mu^{-1}$ near $r=1$ modulo terms smooth at $r=1$. Other coordinate
changes with this property would also work.
Let
$$
\tau=e^{-t}=\frac{e^{-\tilde t}}{\mu^{1/2}}.
$$
Thus, if we compactify static space-time as $\BB^{n-1}_{r\omega}\times[0,1]_T$,
with $T=\tau$ for say $t>4$, then this procedure amounts to blowing up
$T=0$, $\mu=0$ parabolically. (If we used $\tau=e^{-2t}$, everything would
go through, except there would be many additional factors of $2$; then
the blow-up would be homogeneous, i.e.\ spherical.)
Then the dual metric becomes
$$
-\mu\pa_r^2-2r\pa_r\tau\pa_\tau+\tau^2\pa_{\tau}^2-r^{-2}\pa_{\omega}^2,
$$ 
or
$$
-4r^2\mu\pa_\mu^2+4r^2\tau\pa_{\tau}\pa_\mu+\tau^2\pa_{\tau}^2-r^{-2}\pa_{\omega}^2,
$$
which is a non-degenerate Lorentzian b-metric\footnote{See Section~\ref{sec:Mellin-Lorentz} for a quick introduction to b-geometry and further references.} on $\RR^{n-1}_{r\omega}\times[0,1)_\tau$,
i.e.\ it extends smoothly and non-degenerately across the `event horizon', $r=1$.
Note that in coordinates valid near $r=0$ this becomes
$$
(\sum_j Y_j\pa_{Y_j})^2-2(\sum_j Y_j\pa_{Y_j})\tau\pa_\tau+\tau^2\pa_{\tau}^2-\sum_j\pa_{Y_j}^2=(\tau\pa_\tau-\sum_j Y_j\pa_{Y_j})^2-\sum_j\pa_{Y_j}^2.
$$
In slightly different notation, this agrees with the symbol of
\cite[Equation~(7.3)]{Vasy:De-Sitter}.

We could have used other equivalent local coordinates;
for instance
replaced $e^{-\tilde t}$ by $(\sinh \tilde t)^{-1}$, in which case the coordinates
$(r,\tau,\omega)$
we obtained are replaced by
\begin{equation}\label{eq:equiv-coords-dS}
r,\ \rho= (\sinh \tilde t)^{-1}/(1-r^2)^{1/2}=x,\ \omega.
\end{equation}
As expected, in these coordinates the metric would still be a smooth and
non-degenerate
b-metric.
These coordinates
also show that Kerr-star-type coordinates are smooth in the interior of
the front face on
the blow-up of our conformal compactification of de Sitter space at $q_+$.
\footnote{If we had worked with $e^{-2t}$ instead of $e^{-t}$ above,
we would obtain $x^2$ as the
defining function of the temporal face, rather than $x$.}
In summary we have reproved (modulo a few
details):

\begin{lemma}(See \cite[Lemma~2.1]{Melrose-SaBarreto-Vasy:Asymptotics} for
a complete version.)
The lift of $\hat\Omega$ to the blow up
$[\hM;q_+,q_-]$ is a $\CI$ manifold
with corners, $\bar\Omega$. Moreover, near the front faces $\ff_{q_\pm}$,
$\bar\Omega$ is naturally
diffeomorphic to a neighborhood of the temporal faces $\tf_\pm$ in
the $\CI$ manifold with corners obtained from
$[0,1]_T\times\BB^{n-1}$ by blowing up $\{0\}\times\pa\BB^{n-1}$ and
$\{1\}\times\pa\BB^{n-1}$ in the parabolic manner indicated
in \eqref{eq:equiv-coords-dS}; here $\tf_\pm$ are the lifts of
$\{0\}\times\pa\BB^{n-1}$ and
$\{1\}\times\pa\BB^{n-1}$.
\end{lemma}

It is worthwhile comparing the de Sitter space wave asymptotics
of \cite{Vasy:De-Sitter},
\begin{equation}\label{eq:dS-asymp}
u=x^{n-1}v_++v_-,\ v_+\in\CI(\hM),\ v_-\in\CI(\hM)+x^{n-1}(\log x)\CI(\hM),
\end{equation}
with our main result, Theorem~\ref{thm:exp-decay}.
The fact that the coefficients in the de Sitter expansion
are $\CI$ on $\hM$ means that on $\bM$, the leading terms are
constant. Thus, \eqref{eq:dS-asymp} implies (and is much stronger than)
the statement that $u$ decays to a constant on $\bM$ at an exponential
rate.

\subsection{D'Alembertian and its Mellin transform}\label{subsec:dS-Mellin}
Consider the d'Alembertian, $\Box_g$, whose principal symbol, including
subprincipal terms, is given by the metric function. Thus, writing
b-covectors as
$$
\xi\,d\mu+\sigma\,\frac{d\tau}{\tau}+\eta\,d\omega,
$$ 
we have
\begin{equation}\label{eq:Box-symbol-dS}
G=\sigma_{\bl,2}(\Box)=
-4r^2\mu\xi^2+4r^2\sigman\xi+\sigman^2-r^{-2}|\eta|^2,
\end{equation}
with $|\eta|^2_{\omega}$ denoting the dual metric function on the sphere.
Note that there is a polar coordinate singularity at $r=0$; this is resolved
by using actually valid coordinates $Y=r\omega$ on $\RR^{n-1}$ near the origin;
writing b-covectors as 
$$
\sigma\,\frac{d\tau}{\tau}+\zeta\,dY,
$$ 
we have
\begin{equation}\begin{split}\label{eq:Box-symbol-dS-orig}
G=\sigma_{\bl,2}(\Box)&=
(Y\cdot\zeta)^2-2(Y\cdot\zeta)\sigman+\sigman^2-|\zeta|^2=(Y\cdot\zeta-\sigman)^2-|\zeta|^2,\\
&\qquad Y\cdot\zeta=\sum_j Y_j\cdot\zeta_j,\ |\zeta|^2=\sum_j\zeta_j^2.
\end{split}\end{equation}
Since there are no interesting phenomena at the origin, we may ignore this point below.

Via conjugation by the (inverse) Mellin transform, see Subsection~\ref{subsec:Mellin},
we obtain a family of operators $P_\sigma$
depending on $\sigma$ on $\RR^{n-1}_{r\omega}$
with both principal and high energy ($|\sigma|\to\infty$)
symbol given by \eqref{eq:Box-symbol-dS}.
Thus, the principal symbol of $P_\sigma\in\Diff^2(\RR^{n-1})$, including in the
high energy sense ($\sigma\to\infty$), is
\begin{equation}\begin{split}\label{eq:p-sig-symbol-dS}
p_{\full}&=-4r^2\mu\xi^2+4r^2\sigman\xi+\sigman^2-r^{-2}|\eta|_{\omega}^2\\
&=(Y\cdot\zeta)^2-2(Y\cdot\zeta)\sigman+\sigman^2-|\zeta|^2
=(Y\cdot\zeta-\sigman)^2-|\zeta|^2.
\end{split}\end{equation}
The Hamilton vector field is
\begin{equation}\begin{split}\label{eq:Ham-vf-p-sig-dS}
\sH_{p_\full}&=4r^2(-2\mu\xi+\sigman)\pa_\mu-r^{-2}\sH_{|\eta|^2_{\omega}}
-(4(1-2r^2)\xi^2-4\sigman\xi-r^{-4}|\eta|^2_{\omega})\pa_\xi\\
&=2 (Y\cdot\zeta-\sigman)(Y\cdot\pa_Y-\zeta\cdot\pa_{\zeta})-2\zeta\cdot\pa_Y,
\end{split}\end{equation}
with $\zeta\cdot\pa_Y=\sum\zeta_j\pa_{Y_j}$, etc.
Thus, in the standard `classical' sense, which effectively means
letting $\sigma=0$, the principal symbol is
\begin{equation}\begin{split}\label{eq:p-sig-symbol-dS-standard}
p=\sigma_2(P_\sigma)
&=-4r^2\mu\xi^2-r^{-2}|\eta|_{\omega}^2\\
&=(Y\cdot\zeta)^2-|\zeta|^2,
\end{split}\end{equation}
while the Hamilton vector field is
\begin{equation}\begin{split}\label{eq:Ham-vf-p-sig-dS-standard}
\sH_p&=-8r^2\mu\xi\pa_\mu-r^{-2}\sH_{|\eta|^2_{\omega}}
-(4(1-2r^2)\xi^2-r^{-4}|\eta|^2_{\omega})\pa_\xi\\
&=2 (Y\cdot\zeta)(Y\cdot\pa_Y-\zeta\cdot\pa_{\zeta})-2\zeta\cdot\pa_Y,
\end{split}\end{equation}
Moreover, the imaginary part of the subprincipal symbol, given by the principal
symbol of $\frac{1}{2\imath}(P_\sigma-P_\sigma^*)$, is
$$
\sigma_1(\frac{1}{2\imath}(P_\sigma-P_\sigma^*))=4r^2(\im\sigman)\xi
=-2(Y\cdot\zeta)\im\sigman.
$$

When comparing these with \cite[Section~7]{Vasy:De-Sitter}, it is important
to keep in mind that what is denoted by $\sigma$ there (which we
refer to as $\tilde\sigma$ here to avoid confusion) is $\imath\sigman$
here corresponding to the
Mellin transform, which is a decomposition in terms of $\tau^{\imath\sigma}\sim
x^{\imath\sigman}$,
being replaced by weights $x^{\tilde\sigma}$ in \cite[Equation~(7.4)]{Vasy:De-Sitter}.

One important feature of
this operator is that
$$
N^*\{\mu=0\}=\{(\mu,\omega,\xi,\eta):\ \mu=0,\ \eta=0\}
$$
is invariant under the classical flow (i.e.\ effectively letting $\sigma=0$).
Let
$$
N^*S\setminus o=\Lambda_+\cup\Lambda_-,
\qquad \Lambda_\pm=N^*S\cap\{\pm\xi>0\},\qquad S=\{\mu=0\}.
$$
Let $L_\pm$ be the image of $\Lambda_\pm$ in $S^*\RR^{n-1}$.
Next we analyze the flow at $\Lambda_\pm$. First,
\begin{equation}\label{eq:eta-loc}
\sH_p|\eta|^2_{\omega}=0
\end{equation}
and
\begin{equation}\label{eq:mu-exact-ev}
\sH_p\mu=-8r^2\mu\xi=-8\xi\mu+a\mu^2\xi
\end{equation}
with $a$ being $\CI$ in $T^*X$, and homogeneous of
degree $0$. 
While, in the spirit of linearizations, we used
an expression in \eqref{eq:mu-exact-ev} that is linear in the coordinates whose
vanishing defines $N^*S$, the key point is that
$\mu$ is an elliptic multiple of $p$
in a linearization sense,
so one can simply use $\hat p=p/|\xi|^2$ (which is homogeneous
of degree $0$, like $\mu$), in its place.

\begin{figure}[ht]
\begin{center}
\mbox{\epsfig{file=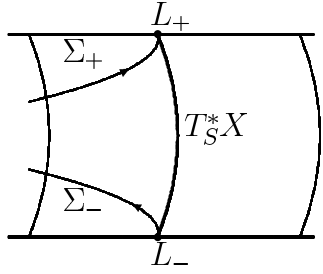}}
\end{center}
\caption{The cotangent bundle near the event horizon $S=\{\mu=0\}$.
It is drawn in a fiber-radially compactified view. The boundary of the fiber
compactificaton is the cosphere
bundle $S^*\RR^{n-1}$; it is the surface of the cylinder shown.
$\Sigma_\pm$ are the components of the (classical) characteristic set containing
$L_\pm$. They lie in $\mu\leq 0$, only meeting $S^*_S\RR^{n-1}$ at $L_\pm$.
Semiclassically, i.e.\ in the interior of $\overline{T}^*\RR^{n-1}$, for $z=h^{-1}\sigma>0$,
only the component of the semiclassical characteristic set containing $L_+$ can
enter $\mu>0$. This is reversed for $z<0$.}
\label{fig:event-horizon-bundle1}
\end{figure}

It is convenient to rehomogenize \eqref{eq:eta-loc} in terms
of $\hat\eta=\eta/|\xi|$. To phrase this
more invariantly, consider the fiber-compactification
$\overline{T}^*\RR^{n-1}$ of $T^*\RR^{n-1}$, see Subsection~\ref{subsec:micro-gen}.
On this space, the classical principal symbol, $p$, is (essentially) a function
on $\pa\overline{T}^*\RR^{n-1}=S^*\RR^{n-1}$.
Then at fiber infinity near $N^*S$, we can take $(|\xi|^{-1},\etah)$
as coordinates on the fibers of the cotangent bundle, with $\rhot=|\xi|^{-1}$
defining $S^*X$ in $\overline{T}^*X$.
Then $|\xi|^{-1}\sH_p$ is a $\CI$ vector field in this region and
\begin{equation}\label{eq:eta-hat-loc}
|\xi|^{-1}\sH_p|\hat\eta|^2=|\hat\eta|^2\sH_p|\xi|^{-1}
=-4(\sgn\xi) |\hat\eta|^2
+\tilde a,
\end{equation}
where $\tilde a$ vanishes cubically at $N^*S$, i.e.\ \eqref{eq:weight-definite} holds.
In similar notation we have
\begin{equation}\begin{split}\label{eq:dS-Ham-2}
&\sH_p|\xi|^{-1}=-4\sgn(\xi)+\tilde a',\\
&|\xi|^{-1}\sH_p\mu=-8(\sgn\xi)\mu.
\end{split}\end{equation}
with $\tilde a'$ smooth (indeed, homogeneous degree zero without the
compactification) vanishing at $N^*S$.
As the vanishing of $\hat\eta,|\xi|^{-1}$ and $\mu$ defines
$\pa N^*S$, we conclude that
$L_-=\pa\Lambda_-$ is a source, while $L_+=\pa\Lambda_+$ is a sink,
in the sense that all nearby
bicharacteristics (in fact, including semiclassical (null)bicharacteristics, since
$\sH_p|\xi|^{-1}$ contains the additional information needed) converge
to $L_\pm$ as the parameter along the bicharacteristic goes to $\pm\infty$.
In particular, the quadratic defining function of $L_\pm$ given by 
$$
\rho_0=\widehat{\tilde p}+\hat p^2,\ \text{where}\ \hat p=|\xi|^{-2}p,
\ \widehat{\tilde p}=|\etah|^2,
$$
satisfies \eqref{eq:rho-0-property}.

The imaginary part of the subprincipal
symbol at $L_\pm$ is given by
$$
(4\sgn(\xi))\im\sigman|\xi|;
$$
here $(4\sgn(\xi))$ is pulled out due to \eqref{eq:dS-Ham-2}, namely its
size relative to $\sH_p|\xi|^{-1}$ matters, with a change of sign,
see Subsection~\ref{subsec:micro-gen}, thus
\eqref{eq:subpr-form}-\eqref{eq:tilde-beta-form} hold.
This corresponds to the fact\footnote{This needs the analogous statement
for full subprincipal symbol,
not only its imaginary part.}
that $(\mu\pm \imath 0)^{\imath\sigman}$, which are Lagrangian
distributions
associated to $\Lambda_\pm$, solve the PDE
modulo an error that is two orders lower than what one might a priori expect,
i.e.\ $P_\sigma (\mu\pm \imath 0)^{\imath\sigman}\in (\mu\pm i0)^{\imath\sigman}\CI(\RR^{n-1})$.
Note that $P_\sigma$ is second order, so one should lose two orders a priori;
the characteristic nature of $\Lambda_\pm$ reduces the loss to $1$, and the
particular choice of exponent eliminates the loss. This has much in common
with $e^{\imath\lambda/x}x^{(n-1)/2}$ being an approximate solution in asymptotically
Euclidean scattering. The precise situation for Kerr-de Sitter space
is more delicate as the Hamilton vector
field does not vanish at $L_\pm$, but this\footnote{This would be relevant for a full Lagrangian analysis, as done
e.g.\ in \cite{RBMSpec}, or in a somewhat different, and more complicated,
context by Hassell, Melrose and Vasy
in \cite{Hassell-Melrose-Vasy:Spectral, Hassell-Melrose-Vasy:Microlocal}.}
is irrelevant for our estimates: only a quantitative version of the source/sink
statements and the imaginary part of the subprincipal symbol are
relevant.

While $(\mu\pm \imath 0)^{\imath\sigman}$ is singular regardless of $\sigma$
apart from integer coincidences (when this should be corrected anyway), it
is interesting to note that for $\im\sigma>0$ this is not bounded at $\mu=0$,
while for $\im\sigma<0$ it vanishes there. This is interesting because if
one reformulates the problem as one in $\mu\geq 0$, as was done
for instance by S\'a Barreto and Zworski \cite{Sa-Barreto-Zworski:Distribution}, and
later by Melrose, S\'a Barreto and Vasy \cite{Melrose-SaBarreto-Vasy:Asymptotics} 
for de Sitter-Schwarzschild space
then one obtains an operator that is essentially (up to
a conjugation and a weight, see below)
the Laplacian on an asymptotically hyperbolic space at energy
$\sigman^2+\frac{(n-2)^2}{4}$ --- more
precisely its normal operator (which encodes its behavior near $\mu=0$) is a
multiple of that of the hyperbolic Laplacian. Then the growth/decay
behavior corresponds to the usual scattering theory phenomena, but
in our approach smooth extendability across $\mu=0$ is the
distinguishing feature of the solutions we want, not growth/decay.
See Remark~\ref{rem:ah-decay} for
more details.

\subsection{Global behavior of the characteristic set}
First remark that $\langle\frac{d\tau}{\tau},\frac{d\tau}{\tau}\rangle_G=1>0$,
so $\frac{d\tau}{\tau}$ is time-like. Correspondingly all the results of
Subsection~\ref{subsec:Lorentz} apply. In particular, \eqref{eq:Lorentz-semicl-char}
gives that
the characteristic set is divided into two
components with $\Lambda_\pm$ in different components. It is easy to make
this explicit:
points with $\xi=0$, or equivalently $Y\cdot\zeta=0$,
cannot lie in the characteristic set.
Thus,
$$
\Sigma_\pm=\Sigma\cap\{\pm\xi>0\}=\Sigma\cap\{\mp (Y\cdot\zeta)>0\}.
$$

While it is not important here since the characteristic set in $\mu\geq 0$
is localized at $N^*S$, hence one has a similar
localization for nearby $\mu$, for global purposes (which we do not need here),
we point out
that
$\sH_p\mu=-8r^2\mu\xi$.
Since $\xi\neq 0$ on $\Sigma$, and in $\Sigma$, $r=1$ can only happen
at $N^*S$, i.e.\ only at the radial set, the $\CI$ function
$\mu$ provides a negative global escape function which is increasing
on $\Sigma_+$, decreasing on $\Sigma_-$. Correspondingly, bicharacteristics
in $\Sigma_+$ travel from infinity to $L_+$, while in $\Sigma_-$ they travel from
$L_-$ to infinity.

\subsection{High energy, or semiclassical, asymptotics}
We are also interested in the high energy behavior, as $|\sigma|\to\infty$.
For the associated semiclassical problem one obtains a family of operators
$$
P_{h,z}=h^2P_{h^{-1}z},
$$
with $h=|\sigma|^{-1}$, and $z$ corresponding to $\sigma/|\sigma|$ in the unit
circle in $\Cx$.
Then the semiclassical principal symbol $p_{\semi,z}$ of $P_{h,z}$ is a function
on $T^*\RR^{n-1}$.
As in Section~\ref{sec:microlocal},
we are interested in $\im z\geq -Ch$, which
corresponds to $\im\sigma\geq -C$.
It is sometimes convenient to think of $p_{\semi,z}$, and its rescaled Hamilton vector field,
as objects on $\overline{T}^*\RR^{n-1}$.
Thus,
\begin{equation}\begin{split}\label{eq:p-h-symbol-dS}
p_{\semi,z}=\sigma_{2,h}(P_{h,z})
&=-4r^2\mu\xi^2+4r^2\zn\xi+\zn^2-r^{-2}|\eta|_{\omega}^2\\
&=(Y\cdot\zeta)^2-2(Y\cdot\zeta)\zn+\zn^2-|\zeta|^2
=(Y\cdot\zeta-\zn)^2-|\zeta|^2.
\end{split}\end{equation}
We make the general discussion of Subsection~\ref{subsec:Lorentz}
explicit. First,
\begin{equation}\label{eq:im-p-h-symbol-dS}
\im p_{\semi,z}=2\im \zn(2r^2\xi+\re \zn)=-2\im \zn(Y\cdot \zeta-\re \zn).
\end{equation}
In particular, for $z$ non-real, $\im p_{\semi,z}=0$ implies
$2r^2\xi+\re \zn=0$, i.e.\ $Y\cdot\zeta-\re \zn=0$,
which means that $\re p_{\semi,z}$ is
\begin{equation}\label{eq:re-p-h-symbol-dS}
-r^{-2}(\re \zn)^2-(\im \zn)^2-r^{-2}|\eta|^2_\omega=-(\im \zn)^2-|\zeta|^2<0,
\end{equation}
i.e.\ $p_{\semi,z}$ is semiclassically elliptic on $T^*\RR^{n-1}$, but {\em not}
at fiber infinity, i.e.\ at $S^*\RR^{n-1}$ (standard ellipticity is lost only
in $r\geq 1$, of course). Explicitly, if we introduce
for instance
$$
(\mu,\omega,\nu,\hat\eta),\qquad \nu=|\xi|^{-1},\ \hat\eta=\eta/|\xi|,
$$
as valid projective coordinates
in a (large!) neighborhood of $L_\pm$ in $\overline {T}^* \RR^{n-1}$, then
\begin{equation*}\begin{split}
\nu^{2}p_{\semi,z}=-4r^2\mu+4r^2(\sgn\xi)\zn\nu+\zn^2\nu^2-r^{-2}|\hat\eta|_{\omega}^2
\end{split}\end{equation*}
so
$$
\nu^{2}\im p_{\semi,z}=4r^2(\sgn\xi)\nu\im \zn+2\nu^2\re \zn\im \zn
$$
which automatically vanishes at $\nu=0$, i.e.\ at $S^*\RR^{n-1}$. Thus,
for $\sigma$ large and pure imaginary, the semiclassical problem adds no
complexity to the `classical' quantum problem, but of course it does not
simplify it. In fact, we need somewhat more information at the characteristic
set, which is thus at $\nu=0$ when $\im z$ is bounded away from $0$:
\begin{equation*}\begin{split}
&\nu\ \text{small},\ \im z\geq 0\Rightarrow (\sgn\xi)\im p_{\semi,z}\geq 0
\Rightarrow \pm\im p_{\semi,z}\geq 0\ \text{near}\ \Sigma_{\semi,\pm},\\
&\nu\ \text{small},\ \im z\leq 0\Rightarrow (\sgn\xi)\im p_{\semi,z}\leq 0
\Rightarrow \pm\im p_{\semi,z}\geq 0\ \text{near}\ \Sigma_{\semi,\pm},\\
\end{split}
\end{equation*}
which, as we have seen, means that for $P_{h,z}$ with $\im z>0$ one can
propagate estimates forwards along the bicharacteristics where $\xi<0$ (in
particular, away from $L_-$, as the latter is a source) and
backwards where $\xi>0$ (in particular, away from $L_+$,
as the latter is a sink), while for $P^*_{h,z}$ the directions are reversed. The
directions are also reversed if $\im z$ switches sign. This
is important because it gives invertibility for $z=\imath$ (corresponding
to $\im\sigma$ large positive, i.e.\ the physical halfplane), but does not give
invertibility for $z=-\imath$ negative.

We now return to the claim that even semiclassically, for $z$ almost
real\footnote{So the operator is not semiclassically elliptic on $T^*\RR^{n-1}$;
as mentioned above, for $\im z$ uniformly bounded away from
$\RR$, we have ellipticity in $T^*\RR^{n-1}$.},
the characteristic
set can be divided into two components $\Sigma_{\semi,\pm}$,
with $L_\pm$ in different
components.
As explained in Subsection~\ref{subsec:Lorentz}
the vanishing of the factor following $\im z$ in
\eqref{eq:im-p-h-symbol-dS} gives a hypersurface that separates $\Sigma_{\semi}$
into two parts; this can be easily checked also by a direct computation.
Concretely, this is the hypersurface given by
\begin{equation}\label{eq:dS-sep-hyp}
0=2r^2\xi+\re \zn=-(Y\cdot\zeta-\re \zn),
\end{equation}
and so
$$
\Sigma_{\semi,\pm}=\Sigma_{\semi}\cap\{\mp (Y\cdot\zeta-\re \zn)>0\}.
$$

We finally need more information about the global semiclassical dynamics. Here
all null-bicharacteristics go to either $L_+$ in the forward direction
or to $L_-$ in the backward direction, and escape to infinity in the other direction.
Rather than proving this at once, which depends on the global non-trapping structure
on $\RR^{n-1}$,
we first give an argument that is local near the event horizon, and suffices for
the extension discussed below for asymptotically hyperbolic spaces.

As stated above, first, we are only concerned about semiclassical dynamics in $\mu>\mu_0$,
where $\mu_0<0$ might be close to $0$. To analyze this, we observe that
the semiclassical Hamilton vector field is
\begin{equation}\begin{split}\label{eq:Ham-vf-p-h-dS}
\sH_{p_{\semi,z}}&=4r^2(-2\mu\xi+\zn)\pa_\mu-r^{-2}\sH_{|\eta|^2_{\omega}}
-(4(1-2r^2)\xi^2-4\zn\xi-r^{-4}|\eta|^2_{\omega})\pa_\xi\\
&=2(Y\cdot\zeta-\zn)(Y\cdot\pa_Y-\zeta\cdot\pa_\zeta)
-2\zeta\cdot\pa_Y;
\end{split}\end{equation}
here we are concerned about $z$ real.
Thus,
$$
\sH_{p_{\semi,z}}(Y\cdot\zeta)=-2|\zeta|^2,
$$
and $\zeta=0$ implies $p_{\semi,z}=\zn^2$, so $\sH_{p_{\semi,z}}(Y\cdot\zeta)$ has a negative
upper bound on the characteristic set in compact subsets of $T^*\{r<1\}$;
note that the characteristic set is compact in $T^*\{r\leq r_0\}$ if $r_0<1$
by standard ellipticity. Thus, bicharacteristics have to leave $\{r\leq r_0\}$ for
$r_0<1$ in both the forward and backward direction (as $Y\cdot\zeta$ is
bounded over this region on the characteristic set). We already know the dynamics
near $L_\pm$, which is the only place where the characteristic set intersects
$S^*_S\RR^{n-1}$, namely $L_+$ is a sink and $L_-$ is a source.
Now, at $\mu=0$, $\sH_{p_{\semi,z}}\mu=z$, which is positive when $z>0$, so bicharacteristics
can only cross $\mu=0$ in the inward direction. In view of our preceding observations,
thus, once a bicharacteristic crossed $\mu=0$, it has to tend to $L_+$. As
bicharacteristics in a neighborhood of $L_+$ (even in $\mu<0$) tend
to $L_+$ since $L_+$ is a sink, it follows that in $\Sigma_{\semi,+}$
the same is true in $\mu>\mu_0$ for
some $\mu_0<0$. On the other hand, in a neighborhood of $L_-$ all
bicharacteristics emanate from $L_-$ (but cannot cross into $\mu>0$ by our
observations), so leave $\mu>\mu_0$ in the forward direction. These
are all the relevant features of the bicharacteristic flow for our purposes
as we shall place a complex absorbing potential near $\mu=\mu_0$ in the next
subsection.

However, it is easy to see the global claim by noting that
$\sH_{p_{\semi,z}}\mu=4r^2(-2\mu\xi+\zn)$,
and this cannot vanish on $\Sigma_{\semi}$ in $\mu<0$,
since where it vanishes, a simple calculation gives $p_{\semi,z}=4\mu\xi^2-r^{-2}|\eta|^2$.
Thus, $\sH_{p_{\semi,z}}\mu$ has a constant sign on $\Sigma_{\semi,\pm}$
in $\mu<0$, so combined with the observation above that all bicharacteristics escape
to $\mu=\mu_0$ in the appropriate direction, it shows that all bicharacteristics
in fact escape to infinity in that direction\footnote{There is in fact a not too complicated
global escape function, e.g.
$$
f=\frac{2Y\cdot\zeta-\re \zn}{2\sqrt{1+|Y|^2}(Y\cdot\zeta-\re \zn)}
=\frac{2Y\cdot\hat\zeta-\re \zn|\zeta|^{-1}}
{2\sqrt{1+|Y|^2}(Y\cdot\hat\zeta-\re \zn|\zeta|^{-1})},
$$
which is a smooth function on the characteristic set in $T^*\RR^{n-1}$
as $Y\cdot\zeta\neq \re \zn$ there; further, it extends smoothly
to the characteristic set
in $\overline{T}^*\RR^{n-1}$ away from $L_\pm$ since $\sqrt{1+|Y|^2}(Y\cdot\hat\zeta-\re \zn|\zeta|^{-1})$ vanishes only there near $S^*\RR^{n-1}$ (where these
are valid coordinates), at which it has conic points.
This function
arises in a straightforward manner when one reduces Minkowski space,
$\RR^n=\RR^{n-1}_{z'}\times\RR_t$ with metric $g_0$,
to the boundary of its radial compactification,
as described in Section~\ref{sec:Minkowski}, and uses the natural
escape function
\begin{equation}\label{eq:Minkowski-escape}
\tilde f=\frac{t t^*-z' (z')^*}{t^*\sqrt{t^2+|z'|^2}}
\end{equation}
there; here
$t^*$ is the dual variable of $t$ and $(z')^*$ of $z'$, outside the origin.}.

In fact, for applications, it is also useful to remark that for $\alpha\in T^*X$,
\begin{equation}\label{eq:mu-convex-ah}
0<\mu(\alpha)<1,\ p_{\semi,z}(\alpha)=0\Mand
(\sH_{p_{\semi,z}}\mu)(\alpha)=0\Rightarrow (\sH_{p_{\semi,z}}^2\mu)(\alpha)<0.
\end{equation}
Indeed,
as $\sH_{p_{\semi,z}}\mu=4r^2(-2\mu\xi+\zn)$, the hypotheses imply
$\zn=2\mu\xi$
and $\sH_{p_{\semi,z}}^2\mu=-8r^2\mu\sH_{p_{\semi,z}}\xi$, so we only
need to show that $\sH_{p_{\semi,z}}\xi>0$ at these points. Since
$$
\sH_{p_{\semi,z}}\xi=-4(1-2r^2)\xi^2+4\zn\xi+r^{-4}|\eta|^2_{\omega}
=4\xi^2+r^{-4}|\eta|^2_{\omega}=4r^{-2}\xi^2,
$$
where the second equality uses $\sH_{p_{\semi,z}}\mu=0$ and the third
uses that in addition $p_{\semi,z}=0$, this follows from
$2\mu\xi=\zn\neq 0$, so $\xi\neq 0$.
Thus, $\mu$ can be used for gluing
constructions as in \cite{Datchev-Vasy:Gluing-prop}.

\subsection{Complex absorption}\label{subsec:complex-absorb-dS}
The final step of fitting $P_\sigma$ into our general microlocal framework is
moving the problem to a compact manifold, and adding a complex absorbing
second order operator. We thus consider a compact manifold without boundary
$X$ for which $X_{\mu_0}=\{\mu>\mu_0\}$, $\mu_0<0$,
say, is identified as an open subset with smooth boundary;
it is convenient to take $X$ to be the double\footnote{In fact, in the de Sitter context,
this essentially means moving to the boundary of $n$-dimensional
Minkowski space, where our $(n-1)$-dimensional model is the `upper hemisphere',
see Section~\ref{sec:Minkowski}. Thus,
doubling over means working with the whole boundary, but putting an absorbing
operator near the equator, corresponding to the usual Cauchy hypersurface
in Minkowski space,
and solving from the radial points at
both the future and past light cones towards the equator --- this would
be impossible without the complex absorption.} of $X_{\mu_0}$.

It is convenient to separate the `classical' (i.e.\ quantum!) and `semiclassical'
problems, for in the former setting trapping does not matter, while in the
latter it does.

\begin{figure}[ht]
\begin{center}
\mbox{\epsfig{file=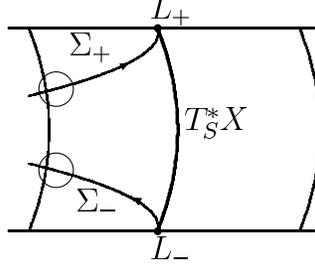}}
\end{center}
\caption{The cotangent bundle near the event horizon $S=\{\mu=0\}$.
It is drawn in a fiber-radially compactified view, as in Figure~\ref{fig:event-horizon-bundle1}. The circles on the left show the support of $q$; it has opposite signs on
the two disks corresponding to the opposite directions of propagation relative
to the Hamilton vector field.}
\label{fig:event-horizon-damp1}
\end{figure}

Ultimately, we want to extend $P_\sigma$ to $X$ (as currently it is
only defined near $X_{\mu_0}$), and introduce a complex
absorbing operator $Q_\sigma\in\Psi_{\cl}^2(X)$
with principal symbol $q$, such that $h^2Q_{h^{-1}z}\in\Psihcl^2(X)$ with
semiclassical principal symbol $q_{\semi,z}$, and such that $p\pm\imath q$
is elliptic near $\pa X_{\mu_0}$, i.e.\ near $\mu=\mu_0$, $Q_\sigma$ is supported
there (say, in $\mu<\mu_0/2$)
and which satisfies that the $\mp q\geq 0$ on $\Sigma_{\pm}$.
In fact, it is convenient
to also arrange that $p\pm\imath q$ are elliptic near $X\setminus X_{\mu_0}$; the region
we added is thus irrelevant in the sense that it does not complicate
the analysis.
In view of Subsection~\ref{subsec:local-wave},
the solution in, say, $\mu>\mu_0/2$
is {\em unaffected} by thus modifying $P_\sigma$, i.e.\ working with $P_\sigma$
and $P_\sigma-\imath Q_\sigma$ is equivalent for this purpose, so the
region we added is irrelevant in this sense as well.

First, for the `classical' problem (i.e.\ completely ignoring the
issue of uniform bounds as $\sigma\to\infty$), we
can make $Q_\sigma$ actually independent of $\sigma$. Indeed, it is
straightforward to write down $q$ with the required properties
(so $q$ is independent of $\sigma$), as we now do; quantizing it in a standard
$\sigma$-independent manner gives a desired $Q_\sigma$; now $Q_\sigma$
depends holomorphically on $\sigma$ (since there is no
$\sigma$-dependence at all).
Concretely, as discussed in Subsection~\ref{subsec:Lorentz},
this can be achieved by defining $\tilde
p=\chi_1 p-\chi_2\hat p$, where $\hat p$ is a Riemannian metric
function (or is simply a homogeneous degree 2, positive function),
$\chi_1+\chi_2=1$, $\chi_1=1$ on $X_{\mu_0}$, is supported nearby,
$\chi_j\geq 0$,
letting $q=-2r^2\xi \tilde p^{1/2}\chi(\mu)=(Y\cdot\zeta) \tilde
p^{1/2}\chi(\mu)$,
$\chi\geq 0$ supported near $\mu_0$, identically equal to a positive constant where neither
$\chi_1$ nor $\chi_2$ vanishes. As
$p<0$ when $\xi=0$, on the set where $\chi=1$, $\tilde p\pm\imath q$ does not
vanish (for the vanishing of $q$ there implies that $p<0$, and
$-\hat p$ is always negative), while if $\chi_2=1$, $\tilde p<0$ so
$\tilde p\pm\imath q$ does not
vanish,
and if $\chi_1=1$, $\tilde p\pm\imath q=p\pm\imath q$ behave the required
way as $\mp q\geq 0$ on $\Sigma_\pm$. Thus, renaming $\tilde p$ as $p$
(since we consider it an extension of $p$) $p\pm\imath q$ satisfy
our requirements.

An alternative to this extension would be simply adding a boundary at $\mu=\mu_0$;
this is easy to do since this is a space-like hypersurface, but this is
slightly unpleasant from the point of view of microlocal analysis as one has
to work on a manifold with boundary (though as mentioned this is easily done,
see Remark~\ref{rem:add-bdy}).

For the semiclassical problem
we need to increase the requirements on $Q_\sigma$.
As explained in Remarks~\ref{rem:off-real-not-imp}
and \ref{rem:off-real-not-imp-for-exp},
the main point is to ensure that the requirements for almost real $z$,
i.e.\ $\im z=\cO(h)$ (corresponding to $\sigma$ in a
strip) are satisfied.
For this we need in addition, in the semiclassical notation, semiclassical ellipticity
near $\mu=\mu_0$, i.e.\ that
$p_{\semi,z}\pm\imath q_{\semi,z}$ are elliptic near $\pa X_{\mu_0}$, i.e.\ near $\mu=\mu_0$,
and which satisfies that the $\mp q_{\semi,z}\geq 0$ on
$\Sigma_{\semi,\pm}$. While this is extremely easy to arrange
if we ignore holomorphy in $z$, a bit of care is required to ensure
the latter. Following \eqref{eq:Lorentz-q-form}, and taking into
account
\eqref{eq:im-p-h-symbol-dS},
we take
\begin{equation*}
q_{\semi,z}=-(2r^2\xi+z)f_z\chi(\mu)=(Y\cdot\zeta-z)f_z\chi(\mu),
\end{equation*}
with
\begin{equation*}
f_z=(|\zeta|^{2j}+z^{2j})^{1/2j},
\end{equation*}
$\chi\in\CI_c(\RR)$, $\chi\geq 0$ is supported near $\mu_0$, and is identically $1$ in a
smaller neighborhood of $\mu_0$, $j\geq 1$ integer and the branch of
the $2j$th root function is chosen so that it is defined on
$\Cx\setminus(-\infty,0)$ and is non-negative when the argument is
non-negative, thus the real part of this root is $\geq 0$ on
$\Cx\setminus(-\infty,0)$ with vanishing only at the origin.
In fact, we can make $Q_\sigma$ the (standard)
quantization of
$$
(Y\cdot\zeta-\sigma) (|\zeta|^{2j}+\sigma^{2j}+C^{2j})^{1/2j}\chi(\mu),
$$
where $C>0$ is chosen suitably large; then $Q_\sigma$ is holomorphic
away from the inverse images of the branch cuts, and in particular
when $|\im \sigma|<C\im e^{\imath\pi/j}$.
Here we can take even $j=1$ and then choose $C$ greater than the width
of the strip we want to study. However, by allowing $j$ to vary, we
obtain an open cover of $\Cx$ by domains $\Omega_j$ of holomorphy for
$Q^{(j)}_\sigma$ as discussed in Subsection~\ref{subsec:local-wave},
and as $\re f\geq 0$, Subsection~\ref{subsec:Lorentz} shows that
$Q^{(j)}_\sigma$ in fact satisfies the required semiclassical properties
in $\Omega_j$.
Again, we extend $P_\sigma$ to $X$ ($Q_\sigma$ can already be
considered as being defined on $X$ in view of $\supp\chi$), for
instance in a manner analogous to the `classical' one discussed above,
i.e.\ replacing $p_{\semi,z}$ by $\chi_1 p_{\semi,z}-\chi_2\hat
p_{\semi,z}$, with
$\hat p=(\|\zeta\|^{2j}+z^{2j})^{1/j}$, with $\|.\|^2$ denoting a
Riemannian metric function on $X$. Then
$p\pm\imath q$
and $p_{\semi,z}\pm\imath q_{\semi,z}$ are elliptic near $X\setminus
X_{\mu_0}$, as desired, as discussed in Subsection~\ref{subsec:Lorentz}.

\subsection{More general metrics}\label{subsec:more-gen}
If the operator is replaced by one on a neighborhood of
$Y_y\times(-\delta,\delta)_\mu$ with full principal symbol (including high energy terms)
\begin{equation}\begin{split}\label{eq:p-sig-symbol-ah}
&-4(1+a_1)\mu\xi^2+4(1+a_2)\sigman\xi+(1+a_3)\sigman^2-|\eta|_{h}^2,
\end{split}\end{equation}
and $h$ a family of Riemannian metrics on $Y$ depending smoothly
on $\mu$, $a_j$ vanishing at $\mu=0$,
then the local behavior of this operator
$P_\sigma$ near the `event horizon' $Y\times\{0\}$
is exactly as in the de Sitter setting. If we start with a compact manifold
$X_0$ with boundary $Y$ and a neighborhood of the boundary identified
with $Y\times[0,\delta)_\mu$ with the operator of the form above, and which
is elliptic in $X_0$ (we only need to assume this away from $Y\times[0,\delta/2)$, say),
including in the non-real high energy sense (i.e.\ for $z$ away from $\RR$ when
$\sigma=h^{-1} z$)
then we can extend the operator smoothly to one on $X_{\mu_0}$, $\mu_0=-\delta$,
which enjoys all the properties above, except semiclassical non-trapping. If
we assume that $X_0^\circ$ is non-trapping in the usual sense, the semiclassical
non-trapping property also follows. In addition, for $\mu>0$ sufficiently small,
\eqref{eq:mu-convex-ah} also holds since $\eta$ is small when
$\sH_{p_{\semi,z}}\mu=0$ and $p_{\semi,z}=0$, for the former gives
$\zn=2(1+a_2)^{-1}(1+a_1)\mu\xi$, and then the latter gives
$$
4(1+a_1)\left(1+\frac{(1+a_1)(1+a_3)}{(1+a_2)^2}\mu\right)\mu\xi^2=|\eta|^2_h,
$$
so the contribution of $|\eta|_{h}^2$
to $\sH_{p_{\semi,z}}\xi$, which can be large elsewhere even at $\mu=0$,
is actually small.

\subsection{Results}\label{subsec:conf-comp-results}
The preceding subsections show that for the Mellin transform of $\Box_g$
on $n$-dimensional de Sitter space, all the hypotheses needed in
Section~\ref{sec:microlocal} are satisfied, thus analogues of the results stated
for Kerr-de Sitter space in the introduction,
Theorems~\ref{thm:complete-absorb}-\ref{thm:exp-decay}, hold. It is important
to keep in mind, however, that there is no trapping to remove, so
Theorem~\ref{thm:complete-absorb} applies with $Q_\sigma$ supported
outside the event horizon, and one does not need gluing or the result of
Wunsch and Zworski \cite{Wunsch-Zworski:Resolvent}. In particular,
Theorem~\ref{thm:glued} holds with arbitrary $C'$,
without the logarithmic or polynomial loss.
As already mentioned
when discussing \cite[Theorem~1.1]{Vasy:De-Sitter} at the beginning of
this section, this
is weaker than the result of
\cite[Theorem~1.1]{Vasy:De-Sitter}, since there one has smooth asymptotics
without a blow-up of a boundary point\footnote{Note that our methods work equally
well for asymptotically de Sitter spaces in the sense of \cite{Vasy:De-Sitter};
after the blow up, the boundary metric is `frozen' at the point that is
blown up, hence the induced problem at the front face is the same
as for the de Sitter metric with asymptotics given by this `frozen' metric.}.

We now reinterpret our results on the Mellin transform side in terms of
$(n-1)$-dimensional hyperbolic space. Let $\BB^{n-1}_{1/2}$
be $\BB^{n-1}=\{r\leq 1\}$ with $\nu=\sqrt{\mu}$ added to the smooth
structure.
For the purposes of the discussion below,
we identify the interior $\{r<1\}$ of $\BB^{n-1}_{1/2}$
with the Poincar\'e ball model of hyperbolic $(n-1)$-space $(\HH^{n-1},g_{\HH^{n-1}})$.
Using polar coordinates around the origin, let $\cosh\rho=\nu^{-1}$,
$\rho$ is the distance from the origin. The Laplacian on $\HH^{n-1}$ in
these coordinates is
\begin{equation*}
\Delta_{\HH^{n-1}}=D_{\rho}^2-\imath(n-2)\coth\rho\,D_{\rho}+(\sinh\rho)^{-2}
\Delta_{\omega}.
\end{equation*}

It is shown in \cite[Lemma~7.10]{Vasy:De-Sitter} that in $r<1$, 
and with $s$ be such that $2s=\imath\sigman-\frac{n}{2}$,
\begin{equation}\begin{split}\label{eq:conj-formula-dS}
&(1-r^2)^{-s}P_\sigma(1-r^2)^s
=\nu^{\frac{n}{2}-\imath\sigman}P_\sigma\nu^{\imath\sigman-\frac{n}{2}}\\
&=-\nu^{-1}\left(\Delta_{\HH^{n-1}}-\sigman^2-\left(\frac{n-2}{2}\right)^2
-\nu^2\frac{n(n-2)}{4}\right)\nu^{-1}\\
&=-\cosh\rho\left(\Delta_{\HH^{n-1}}-\sigman^2-\left(\frac{n-2}{2}\right)^2
-(\cosh\rho)^{-2}\frac{n(n-2)}{4}\right)\cosh\rho.
\end{split}
\end{equation}

We thus deduce:

\begin{prop}\label{prop:hyp-resolvent-construct}
The inverse $\cR(\sigma)$ of
$$
\Delta_{\HH^{n-1}}-\sigma^2-\left(\frac{n-2}{2}\right)^2
-(\cosh\rho)^{-2}\frac{n(n-2)}{4}
$$
has a meromorphic continuation from
$\im\sigma>0$ to $\Cx$ with poles with finite rank residues as a map
$\cR(\sigma):\dCI(\BB^{n-1})\to\dist(\BB^{n-1})$, and with
non-trapping estimates in every strip $|\im\sigma|<C$, $|\re\sigma|\gg 0$:
$s>\frac{1}{2}+C$,
\begin{equation}\label{eq:non-trap-hyp-model}
\|(\cosh\rho)^{(n-2)/2-\imath\sigma} \cR(\sigma)f\|_{H^s_{|\sigma|^{-1}}(\BB^{n-1})}
\leq C|\sigma|^{-1}\|(\cosh\rho)^{(n+2)/2-\imath\sigma}f\|_{H^{s-1}_{|\sigma|^{-1}}(\BB^{n-1})},
\end{equation}
where the Sobolev spaces are those on $\BB^{n-1}$ (rather than $\BB^{n-1}_{1/2}$).
If $\supp f\subset (\BB^{n-1})^\circ$,
the $s-1$ norm on $f$ can be replaced by the $s-2$ norm.

The same conclusion holds for small {\em even} $\CI$ perturbations, vanishing
at $\pa\BB^{n-1}_{1/2}$, of $g_{\HH^{n-1}}$
in the class of conformally compact metrics,
or the addition of (not necessarily small) $V\in \mu\CI(\BB^{n-1})$.
\end{prop}

\begin{proof}
By self-adjointness and positivity of $\Delta_{\HH^{n-1}}$,
$$
\left(\Delta_{\HH^{n-1}}-\sigma^2-\left(\frac{n-2}{2}\right)^2
-\nu^2\frac{n(n-2)}{4}\right) u=f\in \dCI(\BB^{n-1})
$$
has a unique solution $u=\cR(\sigma)f\in L^2(\BB^{n-1}_{1/2},|dg_{\HH^{n-1}}|)$ when
when $\im\sigma\gg0$.
On the other hand, let $\tilde f_0=\nu^{\imath\sigma-n/2}\nu^{-1} f$
in $r\leq 1$, and $\tilde f_0$ still
vanishes to infinite order at $r=1$. Let $\tilde f$ be an arbitrary smooth
extension of $\tilde f_0$ to the compact manifold $X$ on which
$P_{\sigman}-\imath Q_{\sigman}$ is defined.
Let $\tilde u=(P_{\sigman}-\imath Q_{\sigman})^{-1}\tilde f$,
with $(P_{\sigman}-\imath Q_{\sigman})^{-1}$ given by our results in
Section~\ref{sec:microlocal}; this satisfies
$(P_{\sigman}-\imath Q_{\sigman})\tilde u=\tilde f$ and $\tilde u\in\CI(X)$.
Thus, $u'=\nu^{-\imath\sigma+n/2}\nu^{-1} \tilde u|_{r<1}$ satisfies
$u'\in\nu^{(n-2)/2-\imath\sigma}\CI(\BB^{n-1})$, and
$$
\left(\Delta_{\HH^{n-1}}-\sigma^2-\left(\frac{n-2}{2}\right)^2
-\nu^2\frac{n(n-2)}{4}\right)u'=f
$$
by \eqref{eq:conj-formula-dS} (as $Q_\sigma$ is supported in $r>1$).
Since $u'\in L^2(\BB^{n-1},|dg_{\HH^{n-1}}|)$ for $\im\sigma>0$,
by the aforementioned
uniqueness, $u=u'$.

To make the extension from $\BB^{n-1}$ to $X$ more systematic,
let $E_s:H^s(\BB^{n-1})\to H^s(X)$ be a continuous extension operator,
$R_s:H^s(X)\to H^s(\BB^{n-1})$ the restriction map. Then, as we have just
seen, for $f\in\dCI(\BB^{n-1})$,
\begin{equation}\label{eq:hyp-res-repr}
\cR(\sigma)f=\nu^{-\imath\sigma+n/2}\nu^{-1} R_s(P_{\sigman}-\imath Q_{\sigman})^{-1}E_{s-1}\nu^{\imath\sigma-n/2}\nu^{-1} f.
\end{equation}
Thus, the first half of the proposition (including the non-trapping estimate)
follows immediately from the
results
of Section~\ref{sec:microlocal}. Note also that this proves that every pole of
$\cR(\sigma)$ is a pole of $(P_\sigma-\imath Q_\sigma)^{-1}$ (for otherwise
\eqref{eq:hyp-res-repr}
would show $\cR(\sigma)$ does not have a pole either),
but it is possible for $(P_\sigma-\imath Q_\sigma)^{-1}$ to have poles which are not poles of
$\cR(\sigma)$. However, in the latter case, the Laurent coefficients of
$(P_\sigma-\imath Q_\sigma)^{-1}$ would be annihilated by multiplication by $R_s$ from the
left, i.e.\ the resonant states (which are smooth) would be supported in $\mu\leq 0$,
in particular vanish to infinite order at $\mu=0$.

In fact, a stronger statement can be made: by a calculation completely analogous
to what we just performed, we can easily see that in $\mu<0$, $P_\sigma$
is a conjugate (times a power of $\mu$) of a
Klein-Gordon-type operator on $(n-1)$-dimensional de Sitter space
with $\mu=0$ being the boundary (i.e.\ where time goes to infinity). Thus,
if $\sigma$ is not a pole of $\cR(\sigma)$ and
$(P_\sigma-\imath Q_\sigma)\tilde u=0$ then one would have a
solution $u$ of this Klein-Gordon-type equation near $\mu=0$, i.e.\ infinity,
that rapidly vanishes at infinity. It is shown
in \cite[Proposition~5.3]{Vasy:De-Sitter} by a Carleman-type estimate that this
cannot happen; although there $\sigma^2\in\RR$ is assumed, the argument given
there goes through almost verbatim in general. Thus, if $Q_\sigma$ is supported
in $\mu<c$, $c<0$, then $\tilde u$ is also supported in $\mu<c$. This argument
can be iterated for Laurent coefficients of higher order poles; their range (which
is finite dimensional) contains only functions supported
in $\mu<c$.

We now turn to the perturbation.
After the conjugation, division by $\mu^{1/2}$ from both sides, elements of
$V\in\mu\CI(\BB^{n-1})$ can be extended to become elements of $\CI(\RR^{n-1})$,
and they do not affect any of the structures discussed in Section~\ref{sec:microlocal},
so the results automatically go through. Operators of the form $x^2L$,
$L\in\Diffbeven(\BB^{n-1}_{1/2})$, i.e.\ with even coefficients with respect to
the local product structure, become elements of $\Diffb(\BB^{n-1})$ after conjugation
and division by $\mu^{1/2}$ from both sides. Hence, they can be
smoothly extended across $\pa\BB^{n-1}$, and they do not affect either the principal
or the subprincipal symbol at $L_\pm$
in the classical sense. They do, however, affect the classical symbol elsewhere and the
semiclassical symbol everywhere, thus the semiclassical
Hamilton flow, but under the smallness assumption
the required properties are preserved, since the dynamics is non-degenerate
(the rescaled Hamilton vector field on $\overline{T}^*\RR^{n-1}$ does not vanish)
away from the radial points.
\end{proof}

Without the non-trapping estimate,
this is a special case of a result of Mazzeo and Melrose
\cite{Mazzeo-Melrose:Meromorphic}, with improvements by
Guillarmou \cite{Guillarmou:Meromorphic}. The point is that first, we do not
need the machinery of the zero calculus here, and second, the analogous result
holds true on arbitary asymptotically hyperbolic spaces, with the non-trapping
estimates holding under dynamical assumptions (namely, no trapping). The poles
were actually computed in \cite[Section~7]{Vasy:De-Sitter} using special
algebraic properties, within the
Mazzeo-Melrose framework; however, given the Fredholm properties our
methods here give, the rest of the algebraic computation in \cite{Vasy:De-Sitter}
go through. Indeed,
the results are stable under perturbations\footnote{Though of course the resonances
vary with the perturbation, in the same manner as they would vary when perturbing
any other Fredholm problem.}, provided they fit into the framework after
conjugation and the weights.
In the context
of the perturbations (so that the asymptotically hyperbolic structure is preserved)
though with evenness conditions relaxed, the non-trapping estimate
is almost the same as in \cite{Melrose-SaBarreto-Vasy:Semiclassical}, where it
is shown by a parametrix construction; here the estimates are slightly stronger.

In fact, by the discussion of Subsection~\ref{subsec:more-gen}, we deduce
a more general result, which in particular, for even metrics,
generalizes the results
of Mazzeo and Melrose \cite{Mazzeo-Melrose:Meromorphic},
Guillarmou \cite{Guillarmou:Meromorphic}, and adds high-energy non-trapping
estimates under non-degeneracy assumptions. It also adds the semiclassically
outgoing property which is useful for resolvent gluing, including
for proving non-trapping bounds microlocally away from trapping, provided
the latter is mild, as shown by Datchev and Vasy \cite{Datchev-Vasy:Gluing-prop,
Datchev-Vasy:Trapped}.

\begin{thm}\label{thm:conf-compact-high}
Suppose that $(X_0,g_0)$ is an $(n-1)$-dimensional manifold with boundary with
an even conformally compact metric and boundary defining function $x$. Let
$X_{0,\even}$ denote the even version of $X_0$, i.e.\ with the boundary defining
function replaced by its square with respect to a decomposition in which
$g_0$ is even. Then
the inverse of
$$
\Delta_{g_0}-\left(\frac{n-2}{2}\right)^2-\sigma^2,
$$
written as $\cR(\sigma):L^2\to L^2$,
has a meromorphic continuation from
$\im\sigma\gg0$ to $\Cx$,
$$
\cR(\sigma):\dCI(X_0)\to\dist(X_0),
$$
with poles with finite rank residues. If in addition
$(X_0,g_0)$ is non-trapping, then
non-trapping estimates hold in every strip $|\im\sigma|<C$, $|\re\sigma|\gg 0$: for
$s>\frac{1}{2}+C$,
\begin{equation}\label{eq:non-trap-conf-comp}
\|x^{-(n-2)/2+\imath\sigma} \cR(\sigma)f\|_{H^s_{|\sigma|^{-1}}(X_{0,\even})}
\leq \tilde C|\sigma|^{-1}\|x^{-(n+2)/2+\imath\sigma}f\|_{H^{s-1}_{|\sigma|^{-1}}(X_{0,\even})}.
\end{equation}
If $f$ is supported in $X_0^\circ$,
the $s-1$ norm on $f$ can be replaced by the $s-2$ norm.

If instead $\Delta_{g_0}-\sigma^2$ satisfies mild
trapping assumptions with order $\varkappa$ estimates in a $C_0$-strip,
see Definition~\ref{Def:mild-trap}, then the mild trapping estimates
hold, with $|\sigma|^{\varkappa-1}$ replacing $|\sigma|^{-1}$ on the right hand side
of \eqref{eq:non-trap-conf-comp}, as long as $C\leq C_0$.

Furthermore, for $\re z>0$, $\im z=\cO(h)$,
the resolvent $\cR(h^{-1}z)$ is {\em semiclassically outgoing with a loss
of $h^{-1}$} in the sense that
if $f$ has compact support in $X_0^\circ$, $\alpha\in T^*X$ is in
the semiclassical characteristic set and if $\WFh^{s-1}(f)$ is disjoint from
the backward bicharacteristic from $\alpha$, then
$\alpha\notin\WFh^s(h^{-1}\cR(h^{-1}z)f)$.
\end{thm}

We remark that although in order
to go through without changes, our methods require the
evenness property, it is not hard to deduce more restricted results without
this. Essentially one would have operators with coefficients that have a conormal
singularity at the event horizon; as long as this is sufficiently mild relative
to what is required for the analysis, it does not affect the results. The problems
arise for the analytic continuation, when one needs strong function spaces
($H^s$ with $s$ large); these are not preserved when one multiplies by the
singular coefficients.

\begin{proof}
Suppose that $g_0$ is an even asymptotically hyperbolic metric. Then
we may choose a product decomposition near the boundary such that
\begin{equation}\label{eq:ah-g-0-prod}
g_0=\frac{dx^2+h}{x^2}
\end{equation}
there,
where $h$ is an even family of metrics; it is convenient to take $x$ to be a
globally defined boundary defining function. Then
\begin{equation}\label{eq:conf-comp-Lap-form}
\Delta_{g_0}=(xD_x)^2+\imath(n-2+x^2\gamma) (xD_x)+x^2\Delta_{h},
\end{equation}
with $\gamma$ even. Changing to coordinates $(\mu,y)$, $\mu=x^2$, we
obtain
\begin{equation}\label{eq:Lap-in-mu}
\Delta_{g_0}=4(\mu D_{\mu})^2+2\imath(n-2+\mu\gamma) (\mu D_{\mu})+\mu\Delta_{h},
\end{equation}
Now we conjugate by $\mu^{-\imath\sigman/2+n/4}$ to obtain
\begin{equation*}\begin{split}
&\mu^{\imath\sigman/2-n/4}(\Delta_{g_0}-\frac{(n-2)^2}{4}-\sigman^2)\mu^{-\imath\sigman/2+n/4}\\
&=4(\mu D_{\mu}-\sigman/2-\imath n/4)^2+2\imath(n-2+\mu\gamma) (\mu D_{\mu}-\sigman/2-\imath n/4)\\
&\qquad\qquad\qquad\qquad+\mu\Delta_{h}-\frac{(n-2)^2}{4}-\sigman^2\\
&=4(\mu D_\mu)^2-4\sigman(\mu D_\mu)+\mu\Delta_h
-4\imath (\mu D_\mu) +2\imath\sigman-1
+2\imath\mu\gamma (\mu D_{\mu}-\sigman/2-\imath n/4).
\end{split}\end{equation*}
Next we multiply by $\mu^{-1/2}$ from both sides to obtain
\begin{equation}\begin{split}\label{eq:ah-prefinal-conj-form}
&\mu^{-1/2}\mu^{\imath\sigman/2-n/4}(\Delta_{g_0}-\frac{(n-2)^2}{4}-\sigman^2)
\mu^{-\imath\sigman/2+n/4}\mu^{-1/2}\\
&=4\mu D_\mu^2-\mu^{-1}-4\sigman D_\mu-2\imath\sigman\mu^{-1}+\Delta_h
-4\imath D_\mu+2\mu^{-1}+2\imath\sigman\mu^{-1}-\mu^{-1}\\
&\qquad\qquad+2\imath\gamma (\mu D_{\mu}-\sigman/2-\imath (n-2)/4)\\
&=4\mu D_\mu^2-4\sigman D_\mu+\Delta_h
-4\imath D_\mu+2\imath\gamma (\mu D_{\mu}-\sigman/2-\imath (n-2)/4).
\end{split}\end{equation}
This is certainly in $\Diff^2(X)$, and for $\sigma$ (almost) real, is equivalent
to the form we want via conjugation by a smooth function, with exponent
depending on $\sigma$. The latter would
make no difference even semiclassically in the real regime as it is conjugation
by an elliptic semiclassical FIO. However, in the
non-real regime (where we would like ellipticity) it does; the present operator is
not semiclassically elliptic at the zero section.
So finally we conjugate by $(1+\mu)^{\imath\sigman/4}$ to obtain
\begin{equation}\begin{split}\label{eq:ah-final-conj-form}
&4\mu D_\mu^2-4\sigman D_\mu-\sigman^2+\Delta_h
-4\imath D_\mu+2\imath\gamma (\mu D_{\mu}-\sigman/2-\imath (n-2)/4)
\end{split}\end{equation}
modulo terms that can be absorbed into the error terms in
the {\em negative} of operators in the class \eqref{eq:p-sig-symbol-ah}.

We still need to check that $\mu$ can be appropriately chosen in the interior
away from the region of validity of the product decomposition \eqref{eq:ah-g-0-prod}
(where we had no requirements so far on $\mu$). This only matters for semiclassical
purposes, and (being smooth and non-zero in the interior)
the factor $\mu^{-1/2}$ multiplying from both sides does not affect
any of the relevant properties (semiclassical ellipticity and possible non-trapping
properties), so can be ignored --- the same is true for $\sigma$ independent
powers of $\mu$.

To do so, it is useful to think of $(\taut \pa_\taut)^2-G_0$, $G_0$ the dual
metric of $g_0$, as a Lorentzian b-metric on $X_0^\circ\times[0,\infty)_{\taut}$.
From this perspective, we want to
introduce a new boundary defining function $\tau=\taut e^\phi$, with our
$\sigma$ the b-dual variable of $\tau$ and $\phi$ a function on $X_0$,
i.e.\ with our $\tau$ already
given, at least near $\mu=0$, i.e.\ $\phi$ already fixed there, namely
$e^{\phi}=\mu^{1/2}(1+\mu)^{-1/4}$. Recall from the end of
Subsection~\ref{subsec:Lorentz} that such a change of variables amounts
to a conjugation on the Mellin transform side by $e^{-\imath\sigma\phi}$. Further,
properties of the Mellin transform are preserved provided
$\frac{d\tau}{\tau}$ is globally time-like, which, as noted at
the end of Subsection~\ref{subsec:Lorentz}, is satisfied if $|d\phi|_{G_0}<1$.
But, reading off the dual metric from the principal symbol of \eqref{eq:Lap-in-mu},
$$
\frac{1}{4}
\left|d(\log\mu-\frac{1}{2}\log(1+\mu))\right|^2_{G_0}=\left(1-\frac{\mu}{2(1+\mu)}\right)^2<1
$$
for $\mu>0$, with a strict bound as long as $\mu$ is bounded away from
$0$. Correspondingly, $\mu^{1/2}(1+\mu)^{-1/4}$ can be extended to a function
$e^\phi$ on all of $X_0$ so that $\frac{d\tau}{\tau}$ is time-like,
and we may even require that $\phi$ is constant on
a fixed (but arbitrarily large) compact subset of $X_0^\circ$. Then, after conjugation by $e^{-\imath\sigma\phi}$ all of the
semiclassical requirements of Section~\ref{sec:microlocal} are satisfied.
Naturally, the semiclassical properties could be easily checked directly for the conjugate
of $\Delta_{g_0}-\sigma^2$ by the so-extended $\mu$.

Thus, all of the results of Section~\ref{sec:microlocal} apply. The only
part that needs some explanation is the direction of propagation for the
semiclassically outgoing condition. For $\re z>0$, as in the de Sitter case,
null-bicharacteristics in $X_0^\circ$ must go to $L_+$, hence lie in $\Sigma_{\semi,+}$.
Theorem~\ref{thm:semicl-outg} states {\em backward} propagation of regularity
for the operator considered there. However, the operator we just constructed is
the negative of the class considered in \eqref{eq:p-sig-symbol-ah}, and under
changing the sign of the operator, the Hamilton vector field also changes
direction, so semiclassical estimates (or $\WFh$)
indeed propagate in the forward direction.
\end{proof}

\begin{rem}
We note that if the dual metric $G_1$ on $X_0$ is of the form $\kappa^2 G_0$,
$G_0$ the dual of $g_0$ as in \eqref{eq:ah-g-0-prod}, then
$$
\Delta_{G_1}-\kappa^2\frac{(n-2)^2}{4}-\sigman^2
=\kappa^2(\Delta_{G_0}-\frac{(n-2)^2}{4}-(\sigman/\kappa)^2).
$$
Thus, with $\mu$ as above, and with $\tilde P_\sigma$ the conjugate
of $\Delta_{G_0}-\frac{(n-2)^2}{4}-(\sigman/\kappa)^2$, of the form
\eqref{eq:ah-final-conj-form} (modulo error terms as described there) then
with $e^\phi=\mu^{1/(2\kappa)}(1+\mu)^{-1/(4\kappa)}$ extended into the interior of $X_0$
as above,
we have
\begin{equation*}\begin{split}
\mu^{-1/2}\mu^{n/4}e^{\imath\sigma\phi}(\Delta_{g_1}-\kappa^2\frac{(n-2)^2}{4}-\sigman^2)
e^{-\imath\sigma\phi}\mu^{n/4}\mu^{-1/2}=\kappa^2\tilde P_{\sigma/\kappa}.
\end{split}\end{equation*}
Now, $P_\sigma=\kappa^2\tilde P_{\sigma/\kappa}$ still satisfies all the assumptions
of Section~\ref{sec:microlocal}, thus directly conjugation by $e^{-\imath\sigma\phi}$ and multiplication
from both sides by $\mu^{-1/2}$ gives an operator to which the results of
Section~\ref{sec:microlocal} apply. This is relevant because if we have an asymptotically
hyperbolic manifold with ends of different sectional curvature, the manifold
fits into the general framework directly, including the semiclassical estimates\footnote{For `classical' results, the interior is automatically irrelevant.}.
A particular example is de Sitter-Schwarzschild space, on which resonances
and wave propagation were analyzed from this asymptotically hyperbolic perspective in
\cite{Sa-Barreto-Zworski:Distribution, Bony-Haefner:Decay, Melrose-SaBarreto-Vasy:Asymptotics}; this is a special case of the Kerr-de Sitter family discussed in
Section~\ref{sec:Kerr}. The stability of estimates for operators such as $P_\sigma$
under small smooth, in the b-sense, perturbations of the coefficients
of the
associated d'Alembertian means that all the properties of de Sitter-Schwarzschild
obtained by this method are also valid for Kerr-de Sitter with sufficiently small
angular momentum. However, working directly with Kerr-de Sitter space, and showing
that it satisfies the assumptions of Section~\ref{sec:microlocal} on its own, gives
a better result; we accomplish this in Section~\ref{sec:Kerr}.
\end{rem}

\begin{rem}\label{rem:ah-decay}
We now return to our previous remarks regarding the fact that our solution disallows
the conormal singularities $(\mu\pm i0)^{\imath\sigman}$ from the perspective
of conformally compact spaces of dimension $n-1$. The
two indicial roots on these spaces\footnote{Note that $\mu=x^2$.}
correspond to the asymptotics $\mu^{\pm\imath\sigman/2+(n-2)/4}$ in $\mu>0$.
Thus for the operator
$$
\mu^{-1/2}\mu^{\imath\sigman/2-n/4}
(\Delta_{g_0}-\frac{(n-2)^2}{4}-\sigman^2)
\mu^{-\imath\sigman/2+n/4}\mu^{-1/2},
$$
or indeed $P_\sigma$, they correspond to
$$
\left(\mu^{-\imath\sigman/2+n/4}\mu^{-1/2}\right)^{-1}\mu^{\pm\imath\sigman/2+(n-2)/4}=\mu^{\imath\sigman/2\pm\imath\sigman/2}.
$$
Here the indicial root $\mu^0=1$ corresponds to the smooth solutions
we construct for $P_\sigma$, while $\mu^{\imath\sigman}$ corresponds
to the conormal behavior we rule out. Back to the original Laplacian, thus,
$\mu^{-\imath\sigman/2+(n-2)/4}$ is the allowed asymptotics
and $\mu^{\imath\sigman/2+(n-2)/4}$ is the disallowed one. Notice that $\re\imath\sigma
=-\im\sigma$, so the disallowed solution is growing at $\mu=0$
relative to the allowed one, as expected in the physical half plane, and the
behavior reverses when $\im\sigma<0$. Thus, in the original asymptotically
hyperbolic picture one has to distinguish two different rates of growths, whose
relative size changes.
On the other hand, in our approach, we rule out the singular solution and allow
the non-singular (smooth one), so there is no change in behavior at all
for the analytic continuation.
\end{rem}

\begin{rem}\label{rem:asymp-dS}
For {\em even} asymptotically de Sitter metrics on an $(n-1)$-dimensional
manifold $X'_0$ with boundary, the methods
for asymptotically hyperbolic spaces work, except $P_\sigma-\imath Q_\sigma$
and $P_\sigma^*+\imath Q_\sigma^*$ switch roles, which does not affect
Fredholm properties, see Remark~\ref{rem:dual-Fredholm}.
Again, evenness means that
we may choose a product decomposition near the boundary such that
\begin{equation}\label{eq:dS-g-0-prod}
g_0=\frac{dx^2-h}{x^2}
\end{equation}
there,
where $h$ is an even family of Riemannian metrics; as above, we take $x$ to be a
globally defined boundary defining function. Then with $\mut=x^2$, so $\mut>0$
is the Lorentzian region, $\overline{\sigma}$ in place of $\sigma$ (recalling that
our aim is to get to $P_\sigma^*+\imath Q_\sigma^*$) the above calculations for
$\Box_{g_0}-\frac{(n-2)^2}{4}-\overline{\sigma}^2$ in place of
$\Delta_{g_0}-\frac{(n-2)^2}{4}-\sigma^2$ leading to
\eqref{eq:ah-prefinal-conj-form} all go through with $\mu$ replaced by $\mut$,
$\sigma$ replaced by $\overline{\sigma}$ and
$\Delta_h$ replaced by $-\Delta_h$. Letting $\mu=-\mut$, and conjugating by
$(1+\mu)^{\imath\overline{\sigma}/4}$ as above, yields
\begin{equation}\begin{split}\label{eq:dS-final-conj-form}
&-4\mu D_\mu^2+4\overline{\sigman} D_\mu+\overline{\sigman}^2-\Delta_h
+4\imath D_\mu+2\imath\gamma (\mu D_{\mu}-\overline{\sigman}/2-\imath (n-2)/4),
\end{split}\end{equation}
modulo terms that can be absorbed into the error terms in
operators in the class \eqref{eq:p-sig-symbol-ah}, i.e.\ this is indeed
of the form $P_\sigma^*+\imath Q_\sigma^*$ in the framework of
Subsection~\ref{subsec:more-gen}, at least near $\mut=0$. If now $X'_0$
is extended to a manifold without boundary in such a way that in $\mut<0$,
i.e.\ $\mu>0$, one has a classically elliptic, semiclassically either
non-trapping or mildly trapping problem, then all the results of
Section~\ref{sec:microlocal} are applicable.
\end{rem}

\section{Minkowski space}\label{sec:Minkowski}
Perhaps our simplest example is Minkowski space $M=\RR^n$ with the metric
$$
g_0=dz_{n}^2-dz_1^2-\ldots-dz_{n-1}^2.
$$
Also, let $\hM=\overline{\RR^n}$ be the radial (or geodesic)
compactification of space-time,
see \cite[Section~1]{RBMSpec};
thus $\hM$ is the $n$-ball, with boundary $X=\sphere^{n-1}$. Writing
$z'=(z_1,\ldots,z_{n-1})=r\omega$ in terms of Euclidean product coordinates,
and $t=z_n$,
local coordinates on $\hM$ in  $|z'|>\ep|z_n|$, $\ep>0$, are given by
\begin{equation}\label{eq:near-spatial-inf}
s=\frac{t}{r},\ \rho=r^{-1},\ \omega,
\end{equation}
while in $|z_n|>\ep|z'|$, by
\begin{equation}\label{eq:near-temporal-inf}
\tilde\rho=|t|^{-1},\ Z=\frac{z'}{|t|}.
\end{equation}
Note that in the overlap, the curves given by $Z$ constant are the same as
those given by $s,\omega$ constant, but the actual defining function of
the boundary we used, namely $\tilde\rho$ vs.\ $\rho$,  differs, and does so
by a factor which is constant on each fiber. For some
purposes it is useful to fix a global boundary defining function, such
as $\hat\rho=(r^2+t^2)^{-1/2}$. We remark that if one takes
a Mellin transform of functions supported near infinity along these curves,
and uses conjugation by the Mellin transform to obtain families of operators
on $X=\pa\hM$, the effect of changing the boundary defining function in this
manner is conjugation
by a non-vanishing factor which does not affect most relevant
properties of the induced
operator on the boundary, so one can use local defining functions when convenient.

The metric $g_0$ is
a Lorentzian scattering metric in the sense of Melrose \cite{RBMSpec} (where,
however, only the Riemannian case was discussed) in that
it is a symmetric non-degenerate bilinear form on the scattering tangent bundle of
$\hM$ of Lorentzian signature. This would be the appropriate locus of analysis
of the Klein-Gordon operator, $\Box_{g_0}-\lambda$ for $\lambda>0$, but
for $\lambda=0$ the scattering problem becomes degenerate at the zero section
of the scattering cotangent bundle at infinity. However, one can convert $\Box_{g_0}$
to a non-degenerate b-operator on $\hat M$: it is of the form $\hat\rho^2\tilde P$,
$\tilde P\in\Diffb^2(X)$,
where $\hat\rho$ is a defining function of the boundary. In fact, following
Wang \cite{Wang:Thesis}, we consider (taking into account the different
notation for dimension)
\begin{equation}\begin{split}
&\rho^{-2}\rho^{-(n-2)/2}\Box_{g_0}\rho^{(n-2)/2}=\Box_{\tilde g_0}+\frac{(n-2)(n-4)}{4};\\
&\tilde G_0=(1-s^2)\pa_s^2-2(s\pa_s)(\rho\pa_\rho)-(\rho\pa_\rho)^2-\pa_\omega^2,
\end{split}\end{equation}
with $\tilde G_0$ being the dual metric of $\tilde g_0$.
Again, this $\rho$ is not a globally valid defining function, but changing
to another one does not change the properties we need\footnote{Only when
$\im\sigma\to\infty$ can such a change matter.} where this is a valid
defining function.
It is then a straightforward calculation that the induced operator on the boundary
is
$$
P'_\sigma=D_s(1-s^2)D_s-\sigma (sD_s+D_s s)-\sigma^2-\Delta_\omega +\frac{(n-2)(n-4)}{4},
$$
In the other coordinate region, where $\rhot$ is a valid defining function, and $t>0$,
it is even easier to compute
\begin{equation}\label{eq:Box-g0-spatial-form}
\Box_{g_0}=\rhot^2\left((\rhot D_{\rhot})^2+2(\rhot D_\rhot)ZD_Z+(ZD_Z)^2-\Delta_Z
-\imath(\rhot D_\rhot)-\imath ZD_Z\right),
\end{equation}
so after Mellin transforming $\rhot^{-2}\Box$, we obtain
$$
\tilde P_\sigma=(\sigma-\imath/2)^2+\frac{1}{4}
+2(\sigma-\imath/2) Z D_Z+(ZD_Z)^2-\Delta_Z.
$$
Conjugation by $\rhot^{(n-2)/2}$ simply replaces $\sigma$ by
$\sigma-\imath\frac{n-2}{2}$, yielding that the Mellin transform $P_\sigma$
of $\rhot^{-(n-2)/2}\rhot^{-2}\Box_{g_0}\rhot^{(n-2)/2}$
is
\begin{equation}\begin{split}\label{eq:Minkowski-temporal-inf-op}
P_\sigma&=
(\sigma-\imath(n-1)/2)^2++\frac{1}{4}+
2(\sigma-\imath(n-1)/2) Z D_Z+(ZD_Z)^2-\Delta_Z\\
&=(ZD_Z+\sigma-\imath(n-1)/2)^2+\frac{1}{4}-\Delta_Z.
\end{split}\end{equation}
Note that $P_\sigma$ and $P'_\sigma$ are not the same operator in different coordinates;
they are related by a $\sigma$-dependent conjugation.
The operator $P_\sigma$ in \eqref{eq:Minkowski-temporal-inf-op}
is {\em almost} exactly the operator arising from de Sitter space on the front face,
see the displayed equation after \cite[Equation~7.4]{Vasy:De-Sitter} (the $\sigma$
in \cite[Equation~7.4]{Vasy:De-Sitter} is $\imath\sigma$ in our notation as already
remarked in Section~\ref{sec:dS}),
with the only change that our $\sigma$ would need to be replaced by $-\sigma$,
and we need to add $\frac{(n-1)^2}{4}-\frac{1}{4}$ to our operator. (Since replacing
$t>0$ by $t<0$ in the region we consider reverses the sign when relating $D_\rho$
and $D_t$, the signs would agree with those from the discussion after
\cite[Equation~7.4]{Vasy:De-Sitter}
at the backward light cone.) However,
we need to think of this as the {\em adjoint} of an operator of
the type we considered in Section~\ref{sec:dS} up to
Remark~\ref{rem:asymp-dS}, or after
\cite[Equation~7.4]{Vasy:De-Sitter} due to the way we need to propagate estimates.
(This is explained below.) Thus, we think of $P_\sigma$ as the adjoint (with respect
to $|dZ|$) of
\begin{equation*}\begin{split}
P_\sigma^*&=(ZD_Z+\overline{\sigma}-\imath(n-1)/2)^2 +\frac{1}{4}-\Delta_Z\\
&=
(\overline{\sigma}-\imath(n-1)/2)^2+2(\overline{\sigma}-\imath(n-1)/2) Z D_Z
+(ZD_Z)^2 +\frac{1}{4}-\Delta_Z\\
&=(ZD_Z+\overline{\sigma}-\imath (n-1))(ZD_Z+\overline{\sigma})-\Delta_Z+\frac{1}{4}-\frac{(n-1)^2}{4},
\end{split}\end{equation*}
which is like the de Sitter operator after \cite[Equation~7.4]{Vasy:De-Sitter}, except,
denoting $\sigma$ of that paper by $\check\sigma$,
$\imath\check\sigma=\overline{\sigma}$, and we need to take
$\lambda=\frac{(n-1)^2}{4}-\frac{1}{4}$ in \cite[Equation~7.4]{Vasy:De-Sitter}.
Thus,
all of the analysis of Section~\ref{sec:dS} applies.

In particular, note that $P_\sigma$ is elliptic inside the light cone, where $s>1$,
and hyperbolic outside the light cone, where $s<1$.
It follows from Subsection~\ref{subsec:conf-comp-results}
that $P_\sigma$ is a conjugate of
the hyperbolic Laplacian plus a potential (decaying quadratically in the usual
conformally compact sense) inside the light cones\footnote{As pointed out
to the author by Gunther Uhlmann, this means that the Klein model
of hyperbolic space is the one induced by the Minkowski boundary reduction.}, and of
the Klein-Gordon operator plus a potential
on de Sitter space outside the light cones: with
$\nu=(1-|Z|^2)^{1/2}=(\cosh\rho_{\HH^{n-1}})^{-1}$,
\begin{equation*}\begin{split}
&\nu^{\frac{n}{2}-\imath\sigman}P_\sigma\nu^{\imath\sigman-\frac{n}{2}}\\
&=-\nu^{-1}\left(\Delta_{\HH^{n-1}}-\sigman^2
-\frac{(n-2)^2}{4}
-\nu^2\frac{n(n-2)-(n-1)^2}{4}\right)\nu^{-1}\\
&=-\cosh\rho_{\HH^{n-1}}\left(\Delta_{\HH^{n-1}}-\sigman^2
-\frac{(n-2)^2}{4}
+\frac{1}{4}(\cosh\rho_{\HH^{n-1}})^{-2}\right)\cosh\rho_{\HH^{n-1}}.
\end{split}
\end{equation*}
We remark that in terms of dynamics on $\Sb^*\hM$, as discussed in
Subsection~\ref{subsec:Mellin}, there is a sign difference in the
normal to the boundary component of the Hamilton vector field (normal in the
b-sense, only), so in terms of the full b-dynamics (rather than normal family
dynamics) the radial points here are sources/sinks, unlike the saddle points
in the de Sitter case. This is closely related to the appearance of adjoints in the
Minkowski problem (as compared to the de Sitter one).

This immediately assures that not only the wave equation on Minkowski
space fits into our framework, wave propagation on it
is stable under small smooth perturbation in $\Diffb^2(X)$ of
$\hat\rho^2\Box_{g_0}$ which have real principal symbol.

Further, it is shown in \cite[Corollary~7.18]{Vasy:De-Sitter} that the problem
for $P_\sigma^*$ is invertible in the interior of hyperbolic space,
but with the behavior that
corresponds to our more global point of view at the boundary, unless
$$
-\imath \overline{\sigma}=
\check\sigma\in-\frac{n-1}{2}\pm\sqrt{\frac{(n-1)^2}{4}-\lambda}-\Nat
=-\frac{n-1}{2}\pm\frac{1}{2}-\Nat=-\frac{n-2}{2}-\Nat,
$$
i.e.
\begin{equation}\label{eq:Minkowski-res}
\sigma\in -\imath(\frac{n-2}{2}+\Nat).
\end{equation}
Recall also from the proof of Proposition~\ref{prop:hyp-resolvent-construct}
that $(P_\sigma-\imath Q_\sigma)^{-1}$ may have additional poles as compared
to the resolvent of the asymptotically hyperbolic model, but the resonant states would
vanish in a neighborhood of the event horizon and the elliptic region --- with
the vanishing valid in a large region, denoted by $\mu>c$, $c<0$, there,
depending on the support of $Q_\sigma$. For
$(P_\sigma^*+\imath Q_\sigma^*)^{-1}$ the resonant states corresponding to these
additional poles may have
support in the elliptic region, but their coefficients are given by pairing
with resonant states of $(P_\sigma-\imath Q_\sigma)^{-1}$. Thus, if $f$ vanishes
in $\mu<c$ then $(P_\sigma^*+\imath Q_\sigma^*)^{-1}f$ only has the poles
given by the asymptotically hyperbolic model.

To recapitulate, $P_\sigma$ is
of the form described in Section~\ref{sec:microlocal},
at least if we restrict away from the backward
light cone\footnote{The latter is only done to avoid combining for the same operator
the estimates we state below for an operator $P_\sigma$ and its adjoint; as follows
from the remark above regarding the sign of $\sigma$,
for the operator here, the microlocal picture near the backward light cone
is like that for the $P_\sigma$
considered in Section~\ref{sec:microlocal}, and near the forward
light cone like that for $P_\sigma^*$.
It is thus fine to include both the backward and the forward light cones; we
just end up with a combination of the problem we study here and its adjoint, and
with function spaces much like in \cite{RBMSpec,Vasy-Zworski:Semiclassical}.}.
To be more precise, for the forward problem for the wave equation,
the {\em adjoint} of the
operator $P_\sigma$
we need to study satisfies the properties in Section~\ref{sec:microlocal},
i.e.\ singularities are propagated towards the radial points at the forward light cone,
which means that our solution lies in the `bad' dual spaces -- of course, these
are just the singularities corresponding to the radiation field of Friedlander
\cite{Friedlander:Radiation}, see also \cite{Sa-Barreto-Wunsch:Radiation},
which is singular on the radial compactification of Minkowski space. However,
by elliptic regularity or microlocal propagation of singularities,
we of course automatically have estimates in better spaces away from the boundary
of the light cone. We also need complex absorbtion supported, say,
near $s=-1/2$ in the coordinates \eqref{eq:near-spatial-inf}, as in
Subsection~\ref{subsec:complex-absorb-dS} (i.e.\ the role of $\mu$
there is played by $s$). If we wanted
to, we could instead add a boundary
at $s=-1/2$, or indeed at $s=0$ (which would give the standard Cauchy problem),
see Remark~\ref{rem:add-bdy}. By standard uniqueness results based on energy
estimates, this does not affect the solution in $s>0$, say, when the forcing $f$
vanishes in $s<0$ and we want the solution $u$ to vanish there as well.

We thus deduce from Lemma~\ref{lemma:Mellin-expand} and the analysis
of Section~\ref{sec:microlocal}:

\begin{thm}
Let $K$ be a compact subset of the interior of the light cone at infinity on $\hM$.
Suppose that $g$ is a Lorentzian scattering metric and $\hat\rho^2\Box_g$
is sufficiently close to $\hat\rho^2\Box_{g_0}$ in $\Diffb^2(\hM)$, with $n$ the
dimension of $\hM$. Then
solutions of the wave equation $\Box_g u=f$ vanishing in $t<0$
and $f\in\dCI(\hM)=\cS(\RR^n)$
have a polyhomogeneous asymptotic expansion in the sense of
\cite{Melrose:Atiyah}
in $K$
of the form $\sim\sum_j\sum_{k\leq m_j} a_{jk}\hat\rho^{\delta_j}(\log|\hat\rho|)^k$,
with $a_{jk}$ in $\CI$, and with
$$
\delta_j=\imath\sigma_j+\frac{n-2}{2},
$$
with $\sigma_j$ being
a point of non-invertibility of $P_\sigma$ on the appropriate function
spaces. On Minkowski space, the exponents are given by
$$
\delta_j=\imath(-\imath\frac{n-2}{2}-\imath j)+\frac{n-2}{2}=n-2+j,\ j\in\Nat,
$$
and they depend continuously on the perturbation if one perturbs the metric.
A distributional version holds globally.

For polyhomogeneous $f$ the analogous conclusion holds, except that one has
to add to the set of exponents (index set) the index set of $f$, increased by $2$
(corresponding factoring out $\hat\rho^2$ in \eqref{eq:Box-g0-spatial-form}),
in the sense of extended unions \cite[Section~5.18]{Melrose:Atiyah}.
\end{thm}

\begin{rem}
Here a compact $K$ is required since we allow drastic perturbations that may
change where the light cone hits infinity. If one imposes more structure,
so that the light cone at infinity is preserved, one can get more
precise results.

As usual, the smallness of the perturbation is only relevant to the extent
that rough properties of the global dynamics and the local dynamics
at the radial points are preserved (so the analysis is only impacted via dynamics).
There are no size restrictions on perturbations if one keeps the relevant features
of the dynamics.
\end{rem}

In a different class of spaces, namely asymptotically conic Riemannian spaces,
analogous and
more precise results exist for the induced product wave equation,
see especially the work of Guillarmou, Hassell
and Sikora \cite{Gui-Hass-Sik:Resolvent-III}; the decay rate in their work is
the same in {\em odd} dimensional space-time (i.e.\ even dimensional space).
In terms of space-time, these
spaces look like a blow-up of the `north and south poles' $Z=0$ of Minkowski
space, with product type structure in terms of space time, but general smooth
dependence on $\omega$ (with the sphere in $\omega$ replaceable by
another compact manifold). In that paper a parametrix is constructed for $\Delta_g$
at all energies by combining a series of preceding papers. Their conclusion in even
dimensional space-time is one order better; this is presumably the result
of a global (as opposed to local, via complex absorbing potentials near, say,
$t/r=-1/2$) cancellation. It is a very interesting question whether our
analysis can be extended to non-product versions of their setting.

Note that for the Mellin transform of
$\Box_{g_0}$ one can perform a more detailed analysis, giving Lagrangian regularity
at the light cone, with high energy control. This would be preserved for
other metrics that preserve the light cone at infinity to sufficiently high order.
The result is an expansion on the $\hM$ blown up at the boundary
of the light cone, with the singularities corresponding to the Friedlander
radiation field. However, in this relatively basic paper we do not pursue this further.

\section{The Kerr-de Sitter metric}\label{sec:Kerr}

\subsection{The basic geometry}\label{subsec:Kerr-geo}
We now give a brief description of the Kerr-de Sitter metric on
\begin{equation*}\begin{split}
&M_\delta=X_\delta\times [0,\infty)_\tau,
\ X_\delta=(r_--\delta,r_++\delta)_r\times\sphere^2,\\
&X_+=(r_-,r_+)_r\times\sphere^2,
\ X_-=\big((r_--\delta,r_++\delta)_r \setminus [r_-,r_+]_r\big)\times\sphere^2,
\end{split}\end{equation*}
where $r_\pm$ are specified later.
We refer the reader to the excellent treatments of the geometry by
Dafermos and Rodnianski \cite{Dafermos-Rodnianski:Black,
Dafermos-Rodnianski:Axi} and Tataru and Tohaneanu
\cite{Tataru-Tohaneanu:Local, Tataru:Local} for details, and Dyatlov's paper
\cite{Dyatlov:Quasi-normal} for the set-up and most of the notation we adopt.

Away from the north and south poles $q_\pm$
we use spherical coordinates
$(\theta,\phi)$ on $\sphere^2$:
$$
\sphere^2\setminus\{q_+,q_-\}=(0,\pi)_\theta\times
\sphere^1_{\phi}.
$$
Thus, away from $[0,\infty)_\mut\times[0,\infty)_\tau\times\{q_+,q_-\}$,
the Kerr-de Sitter space-time is
$$
(r_--\delta,r_++\delta)_r\times[0,\infty)_\tau\times(0,\pi)_\theta\times
\sphere^1_{\phi}
$$ 
with the metric we specify momentarily. 

The Kerr-de Sitter metric has a very similar microlocal
structure at the event horizon to de Sitter
space.
Rather than specifying the metric $g$, we specify the dual metric; it is
\begin{equation}\begin{split}\label{eq:metric-Kerr}
G=-\rho^{-2}\Big(&\mut\pa_r^2
+\frac{(1+\gamma)^2}{\kappa\sin^2\theta}(\alpha\sin^2\theta\pa_{\tilde t}+\pa_{\tilde \phi})^2+\kappa\pa_\theta^2\\
&\qquad-\frac{(1+\gamma)^2}{\mut}
((r^2+\alpha^2)\pa_{\tilde t}+\alpha\pa_{\tilde \phi})^2\Big)
\end{split}\end{equation}
with $r_s,\Lambda,\alpha$ constants, $r_s,\Lambda\geq 0$,
\begin{equation*}\begin{split}
&\rho^2=r^2+\alpha^2\cos^2\theta,\\
&\mut=(r^2+\alpha^2)(1-\frac{\Lambda r^2}{3})-r_s r,\\
&\kappa=1+\gamma\cos^2\theta,\\
&\gamma=\frac{\Lambda\alpha^2}{3}.
\end{split}\end{equation*}
While $G$ is defined for all values of the parameters $r_s,\Lambda,\alpha$,
with $r_s,\Lambda\geq 0$, we make further restrictions.
Note that under the rescaling
$$
r'=\sqrt{\Lambda}r,\ \tilde t'=\sqrt{\Lambda} \tilde t,
\ r_s'=\sqrt{\Lambda} r_s,\ \alpha'=\sqrt{\Lambda}\alpha,\ \Lambda'=1,
$$
$\Lambda^{-1}G$ would have the same form, but with all the unprimed variables replaced
by the primed ones. Thus, effectively, the general
case $\Lambda>0$ is reduced to $\Lambda=1$.

Our first assumption is that $\mut(r)=0$ has two positive roots $r=r_\pm$,
$r_+>r_-$, with
\begin{equation}\label{eq:digamma-def}
\digamma_\pm=\mp\frac{\pa\mut}{\pa r}|_{r=r_\pm}>0;
\end{equation}
$r_+$ is the de Sitter end, $r_-$ is the Kerr end. Since $\mut$ is a quartic
polynomial, is $>0$ at $r=0$ if $|\alpha|>0$, and goes to $-\infty$
at $\pm\infty$, it can have at most 3 positive roots; the derivative
requirements imply that these three positive roots
exist, and $r_\pm$ are the larger two of these. If $\alpha=0$,
\eqref{eq:digamma-def} is satisfied if and only if
$0<\frac{9}{4}r_s^2\Lambda<1$. Indeed, if \eqref{eq:digamma-def} is satisfied,
$\frac{\pa}{\pa r}(r^{-4}\mut)$ must have a zero between $r_-$ and $r_+$,
where $\mut$ must be positive; $\frac{\pa}{\pa r}(r^{-4}\mut)=0$ gives
$r=\frac{3}{2}r_s$, and then $\mut(r)>0$ gives $1>\frac{9}{4}r_s^2\Lambda$.
Conversely, if $0<\frac{9}{4}r_s^2\Lambda<1$, then the cubic polynomial
$r^{-1}\mut=r-\frac{\Lambda}{3}r^3-r_s$ is negative at $0$ and at $+\infty$, and
thus will have exactly two positive roots if it is positive at one point, which is the
case at $r=\frac{3}{2} r_s$. Indeed, note that
$r^{-4}\mut=r^{-2}-\frac{\Lambda}{3}-r_sr^{-3}$
is a cubic polynomial in $r^{-1}$, and
$\pa_r(r^{-4}\mut)=-2r^{-3}\left(1-\frac{3r_s}{2r}\right)$,
so $r^{-4}\mut$ has a non-degenerate critical point at
$r=\frac{3}{2}r_s$, and if $0\leq \frac{9}{4}r_s^2\Lambda<1$, then the
value of $\mut$ at this critical point is positive.
Correspondingly, for small $\alpha$ (depending on $\frac{9}{4}r_s^2\Lambda$,
but with uniform estimates in compact subintervals of $(0,1)$),
$r_\pm$ satisfying \eqref{eq:digamma-def} still exist.

We next note that for $\alpha$ not necessarily zero, if \eqref{eq:digamma-def} is
satisfied then
$\frac{d^2\mut}{dr^2}=2-\frac{2}{3}\Lambda\alpha^2-4\Lambda r^2$
must have a positive zero, so
we need
\begin{equation}\label{eq:gamma-size}
0\leq \gamma=\frac{\Lambda\alpha^2}{3}<1,
\end{equation}
i.e.\ \eqref{eq:digamma-def} implies \eqref{eq:gamma-size}.

Physically, $\Lambda$ is the
cosmological constant, $r_s=2M$ the Schwarzschild radius, with $M$ being
the mass of the black hole, $\alpha$ the angular momentum. Thus,
de Sitter-Schwarzschild space is the particular case with $\alpha=0$, while
further de Sitter space is the case when $r_s=0$ in which limit $r_-$ goes to the origin
and simply `disappears', and Schwarzschild space is the case when $\Lambda=0$, in
which case $r_+$ goes to infinity, and `disappears', creating an asymptotically
Euclidean end. On the other hand,
Kerr is the special case $\Lambda=0$, with again $r_+\to\infty$,
so the structure near the event horizon is unaffected, but the de Sitter end
is replaced by a different, asymptotically Euclidean, end. One should note, however,
that of the limits $\Lambda\to 0$, $\alpha\to 0$ and $r_s\to 0$, the only
non-degenerate one is $\alpha\to 0$; in both other cases the geometry changes
drastically corresponding to the disappearance of the de Sitter, resp.\ the black hole,
ends. Thus, arguably, from a purely mathematical point of view,
de Sitter-Schwarzschild space-time is the most natural
limiting case.
Perhaps the best way to follow this section then is to keep
de Sitter-Schwarzschild space in mind. Since our
methods are stable, this automatically gives the case of small $\alpha$;
of course working directly with $\alpha$ gives better results.

In fact, from the point of view of our setup, all the relevant features are symbolic,
including dependence on the Hamiltonian dynamics. Thus, the only not
completely straightforward part in showing that our abstract hypotheses are
satisfied is the semi-global study of dynamics. The dynamics of the
rescaled Hamilton flow depends
smoothly on $\alpha$, so it is automatically well-behaved for finite
times for small $\alpha$ if it is such for $\alpha=0$;
here rescaling is understood
on the fiber-radially compactified cotangent bundle $\overline{T}^*X_\delta$
(so that one has a smooth
dynamical system whose only non-compactness comes from that of the
base variables). The only place where dynamics matters for unbounded times are
critical points or trapped orbits of the Hamilton vector field.
In $S^*X_\delta=\pa\overline{T}^*X_\delta$, one can analyze the structure
easily for all $\alpha$, and show that for a specific range of
$\alpha$,
given below implicitly by \eqref{eq:alpha-restrict}, the only critical points/trapping is at
fiber-infinity $SN^*Y$ of the conormal bundle of the event
horizon $Y$.
We also analyze the semiclassical dynamics
(away from $S^*X_\delta=
\pa\overline{T}^*X_\delta$) directly for
$\alpha$ satisfying \eqref{eq:semicl-alpha-limit}, which allows
$\alpha$ to be comparable to $r_s$. We show that in this range of
$\alpha$ (subject to \eqref{eq:digamma-def} and \eqref{eq:alpha-restrict}),
the only trapping is hyperbolic trapping,
which was analyzed by Wunsch and Zworski
\cite{Wunsch-Zworski:Resolvent}; further, we also show that
the trapping is normally hyperbolic for small $\alpha$, and is thus
structurally stable then.

In summary, apart from the full analysis of semiclassical dynamics, we work with
arbitrary $\alpha$ for which \eqref{eq:digamma-def}  and
\eqref{eq:alpha-restrict} holds, which are both natural constraints, since it is
straightforward to check the requirements of Section~\ref{sec:microlocal} in this
generality. Even in the semiclassical setting, we work under the
relatively large $\alpha$ bound, \eqref{eq:semicl-alpha-limit}, to
show hyperbolicity of the trapping, and it is only for normal
hyperbolicity that we deal with (unspecified) small $\alpha$.

We now put the metric \eqref{eq:metric-Kerr} into a form needed for the analysis.
Since the metric
is not smooth b-type in terms of $r,\theta,\tilde\phi,e^{-\tilde t}$,
in order to eliminate the $\mut^{-1}$ terms we let
$$
t=\tilde t+h(r),\ \phi=\tilde\phi+P(r)
$$
with
\begin{equation}\label{eq:h-P-form}
h'(r)=\mp\frac{1+\gamma}{\mut}(r^2+\alpha^2)\mp c,
\qquad P'(r)=\mp\frac{1+\gamma}{\mut}\alpha
\end{equation}
near $r_\pm$. Here $c=c(r)$ is a smooth function of $r$ (unlike $\mut^{-1}$!), that is
to be specified.
One also needs to specify the behavior in $\mut>0$ bounded away from $0$, much
like we did so in the asymptotically hyperbolic setting; this affects semiclassical
ellipticity for $\sigma$ away from the reals as well as semiclassical propagation
there. We at first focus on the `classical' problem, however, for which the
choice of $c$ is irrelevant.
Then the dual metric becomes
\begin{equation*}\begin{split}
G=-\rho^{-2}\Big(&\mut\big(\pa_r\mp c\pa_t\big)^2
\mp2(1+\gamma)(r^2+\alpha^2)\big(\pa_r \mp c\pa_t\big)
\pa_t\\
&\mp 2(1+\gamma)\alpha\big(\pa_r\mp  c\pa_t\big)\pa_\phi
+\kappa\pa_\theta^2+\frac{(1+\gamma)^2}{\kappa\sin^2\theta}
(\alpha\sin^2\theta\pa_{t}+\pa_{\phi})^2\Big).
\end{split}\end{equation*}

We write
$\tau=e^{-t}$, so $-\tau\pa_\tau=\pa_t$, and b-covectors as
$$
\xi\,dr+\sigma\,\frac{d\tau}{\tau}+\eta\,d\theta+\zeta\,d\phi,
$$
so
\begin{equation*}\begin{split}
\rho^2 G
&=-\mut\big(\xi\pm  c\sigma\big)^2
\mp2(1+\gamma)(r^2+\alpha^2)\big(\xi\pm  c\sigma\big)
\sigma\\
&\qquad\pm 2(1+\gamma)\alpha\big(\xi\pm  c\sigma\big)\zeta
-\kappa\eta^2
-\frac{(1+\gamma)^2}{\kappa\sin^2\theta}(-\alpha\sin^2\theta\sigma+\zeta)^2.
\end{split}\end{equation*}
Note that the sign of $\xi$ here is the {\em opposite} of the sign in our de Sitter
discussion in Section~\ref{sec:dS} where it was the dual variable (thus
the symbol of $D_\mu$) of $\mu$,
which is $r^{-2}\mut$ in the present notation,
since $\frac{d\mut}{dr}<0$ at the de Sitter end, $r=r_+$.

A straightforward calculation shows
$\det g=(\det G)^{-1}=-(1+\gamma)^4\rho^4\sin^2\theta$, so apart from the
usual polar coordinate singularity at $\theta=0,\pi$, which is an artifact of
the spherical coordinates and is discussed below, we see at once that $g$ is
a smooth Lorentzian b-metric. In particular, it is non-degenerate, so the
d'Alembertian $\Box_g=d^*d$ is a well-defined b-operator, and
$$
\sigma_{\bl,2}(\rho^2 \Box_g)=\rho^2 G.
$$
Factoring out $\rho^2$ does not affect any of the statements below but simplifies
some formulae, see Footnote~\ref{footnote:factor} and Footnote~\ref{footnote:subpr}
for general
statements; one could also work with $G$ directly.

\subsection{The `spatial' problem: the Mellin transform}
The Mellin transform, $P_\sigma$, of $\rho^2\Box_g$ has the same principal
symbol, including in the high energy sense,
\begin{equation}\begin{split}\label{eq:p-symb-off-poles}
p_{\full}=\sigma_{\full}(P_\sigma)=&-\mut\big(\xi\pm  c\sigma\big)^2 \mp2(1+\gamma)(r^2+\alpha^2)\big(\xi\pm  c\sigma\big)\sigma\\
&\pm 2(1+\gamma)\alpha\big(\xi\pm  c\sigma\big)\zeta-\tilde p_{\full}
\end{split}\end{equation}
with
$$
\tilde p_{\full}=\kappa\eta^2+\frac{(1+\gamma)^2}{\kappa\sin^2\theta}(-\alpha\sin^2\theta\sigma+\zeta)^2,
$$
so $\tilde p_{\full}\geq 0$ for real $\sigma$.
Thus,
\begin{equation*}\begin{split}
&\sH_{p_\full}=\Big(-2\mut\big(\xi\pm  c\sigma\big) \mp2(1+\gamma)(r^2+\alpha^2)\sigma\pm 2(1+\gamma)\alpha\zeta\Big)\pa_r\\
&\qquad\qquad
-\Big(-\frac{\pa\mut}{\pa r}\big(\xi\pm  c\sigma\big)^2
\mp 4r(1+\gamma)\sigma\big(\xi\pm  c\sigma\big)
\pm\frac{\pa c}{\pa r}\tilde c\sigma\Big)\pa_\xi\\
&\qquad\qquad\pm 2(1+\gamma)\alpha\big(\xi\pm  c\sigma\big)\pa_\phi
-\sH_{\tilde p_{\full}},\\
&\tilde c=
-2\mut\big(\xi\pm  c\sigma\big) \mp2(1+\gamma)(r^2+\alpha^2)\sigma\pm 2(1+\gamma)\alpha\zeta.
\end{split}\end{equation*}

To deal with $q_+$ given by $\theta=0$ ($q_-$ being similar), let
$$
y=\sin\theta\sin\phi,\ z=\sin\theta\cos\phi,\ \text{so}\ \cos^2\theta=1-(y^2+z^2).
$$
We can then perform a similar calculation yielding that if $\lambda$ is the
dual variable to $y$ and $\nu$ is the dual variable to $z$ then
$$
\zeta=z\lambda-y\nu
$$
and
\begin{equation*}\begin{split}
\tilde p_\full&=(1+\gamma\cos^2\theta)^{-1}\Big((1+\gamma)^2(\lambda^2+\nu^2)
+\tilde p''\Big)+\tilde p^\sharp_\full,\\
\tilde p^\sharp_\full&=(1+\gamma\cos^2\theta)^{-1}
(1+\gamma)^2(2\alpha\sin^2\theta\sigma-\zeta)\alpha\sigma,
\end{split}\end{equation*}
with $\tilde p''$ smooth and vanishing quadratically at the origin.
Correspondingly, by \eqref{eq:p-symb-off-poles}, $P_\sigma$ is indeed
smooth at $q_\pm$. Thus, one can perform all symbol calculations away from
$q_\pm$, since the results will extend smoothly to $q_\pm$, and correspondingly
from now on we do not emphasize these two poles.

In the sense of `classical' microlocal analysis, we thus have:
\begin{equation}\begin{split}\label{eq:p-symb-off-poles-classical}
&p=\sigma_2(P_\sigma)
=-\mut\xi^2 \pm 2(1+\gamma)\alpha\xi\zeta-\tilde p,
\qquad \tilde p=\kappa\eta^2+\frac{(1+\gamma)^2}{\kappa\sin^2\theta}\zeta^2\geq 0,\\
&\sH_p=\Big(-2\mut\xi \pm 2(1+\gamma)\alpha\zeta\Big)\pa_r
\pm 2(1+\gamma)\alpha\xi\pa_\phi+\frac{\pa\mut}{\pa r}\xi^2\pa_\xi
-\sH_{\tilde p}.
\end{split}\end{equation}

\subsection{Microlocal geometry of Kerr-de Sitter space-time}
As already stated in Section~\ref{sec:microlocal},
it is often convenient to consider the fiber-radial compactification
$\overline{T}^*X_\delta$ of the
cotangent bundle $T^*X_\delta$, 
with $S^*X_{\delta}$ considered as the boundary at fiber-infinity
of $\overline{T}^*X_\delta$.

We let
$$
\Lambda_+=N^*\{\mut=0\}\cap\{\mp\xi>0\},
\ \Lambda_-=N^*\{\mut=0\}\cap\{\pm\xi>0\},
$$
with the sign inside the braces corresponding to that of $r_\pm$.
This is
consistent with our definition of $\Lambda_\pm$ in the de Sitter case.
We let $L_\pm=\pa\Lambda_\pm\subset S^*X_{\delta}$.
Since $\Lambda_+\cup\Lambda_-$ is given by $\eta=\zeta=0$, $\mut=0$,
$\Lambda_\pm$ are
preserved by the {\em classical} dynamics (i.e.\ with $\sigma=0$). Note that the
special structure of $\tilde p$ is irrelevant for the purposes of this observation;
only the quadratic vanishing at $L_\pm$ matters. Even for other local aspects
of analysis, considered below,
the only relevant part\footnote{This could be relaxed: quadratic behavior with small
leading term would be fine as well; quadratic
behavior follows from $\sH_p$ being tangent to $\Lambda_\pm$; smallness
is needed so that $\sH_p|\xi|^{-1}$ can be used to dominate this in terms
of homogeneous dynamics, so that the dynamical character of $L_\pm$ (sink/source)
is as desired.}, is that $\sH_p\tilde p$
vanishes cubically at $L_\pm$,
which in some sense reflects the behavior
of the linearization of $\tilde p$.

\begin{figure}[ht]
\begin{center}
\mbox{\epsfig{file=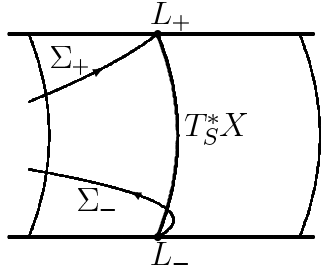}}
\end{center}
\caption{The cotangent bundle near the event horizon $S=\{\mut=0\}$.
It is drawn in a fiber-radially compactified view.
$\Sigma_\pm$ are the components of the (classical) characteristic set containing
$L_\pm$. The characteristic set crosses the event horizon on both components; here
the part near $L_+$ is hidden from view. The projection of this region to the
base space is the ergoregion. Semiclassically, i.e.\ the interior of
$\overline{T}^*X_\delta$, for $z=h^{-1}\sigma>0$,
only $\Sigma_{\semi,+}$ can
enter $\mut>\alpha^2$, see the paragraph after \eqref{eq:c-est}.}
\label{fig:kerr-horizon-bundle1}
\end{figure}

To analyze the dynamics near $L_\pm$ on the characteristic set,
starting with the classical dynamics,
we note that
$$
\sH_{\tilde p}r=0,\ \sH_{\tilde p}\xi=0,\ \sH_p\zeta=0,
\ \sH_{\tilde p}\tilde p=0,\ \sH_p\tilde p=0;
$$
note that $\sH_p\tilde p=0$ and $\sH_p\zeta=0$ correspond to the integrability
of the Hamiltonian dynamical system; these were observed by
Carter~\cite{Carter:Global}
in the Kerr setting.
Furthermore,
\begin{equation}\label{eq:Kerr-weight-est}
\sH_p|\xi|^{-1}|_{S^*X_{\delta}}=-(\sgn\xi ) \frac{\pa\mut}{\pa r},
\end{equation}
so at $\pa N^*\{\mut=0\}$ it is given by $\pm(\sgn\xi )\digamma_\pm$,
which is bounded
away from $0$. We note that
\begin{equation}\label{eq:Kerr-hom-tilde-p-est}
|\xi|^{-1}\sH_p(|\xi|^{-2}\tilde p) |_{S^*X_{\delta}}
=-2 (\sgn\xi )\frac{\pa\mut}{\pa r}\tilde p|\xi|^{-2}.
\end{equation}
Since $\tilde p=0$ and $\mut\neq 0$
implies $\xi=0$ on the classical characteristic set (i.e.\ when we take
$\sigma=0$), which cannot happen on $S^*X$ (we are away from the zero section!),
this shows that the Hamilton vector field
is non-radial except possibly at $\Lambda_\pm$.
Moreover,
\begin{equation*}\begin{split}
\sH_p\left(\mut\mp 2(1+\gamma)\alpha\frac{\zeta}{\xi}\right) |_{S^*X_{\delta}}
&=-2|\xi|\left(\mut\mp 2(1+\gamma)\alpha\frac{\zeta}{\xi}\right) (\sgn\xi)
\frac{\pa\mut}{\pa r};
\end{split}\end{equation*}
as usual, this corresponds to $\hat p=|\xi|^{-2}p$ at $L_\pm$.
Finally, the imaginary part of the subprincipal symbol at $L_\pm$ is
\begin{equation}\begin{split}\label{eq:subpr-Kerr}
&((\sgn\xi ) \frac{\pa\mut}{\pa r})(\beta_\pm\im\sigma)|\xi|,\ \text{where}\\
&\beta_\pm=\mp 2\left(\frac{d\mut}{dr}\right)^{-1}(1+\gamma)(r^2+\alpha^2)|_{r=r_\pm}
=2\digamma_\pm^{-1}(1+\gamma)(r_\pm^2+\alpha^2)>0;
\end{split}\end{equation}
here $(\sgn\xi ) \frac{\pa\mut}{\pa r}$ was factored out in view of
\eqref{eq:Kerr-weight-est}, \eqref{eq:subpr-form} and \eqref{eq:weight-definite}.

Thus, $L_+$ is a sink,
$L_-$ a source. Furthermore, in the classical sense, $\xi=0$ is disjoint from
the characteristic set in the region of validity of the form \eqref{eq:p-symb-off-poles}
of the operator,
as well as at the poles of the sphere (i.e.\ the only issue is when
$r$ is farther from $r_\pm$), so the
characteristic set has two components there with $L_\pm$ lying
in different components. We note that that
as $\gamma<1$,
$\kappa\sin^2\theta=\sin^2\theta(1+\gamma)-\gamma\sin^4\theta$ 
has its maximum at $\theta=\pi/2$, where it is $1$.
Since on the characteristic set
\begin{equation}\label{eq:kerr-cs}
\alpha^2\xi^2+(1+\gamma)^2\zeta^2\geq
\pm2(1+\gamma)\alpha\xi\zeta=\tilde p+\mut\xi^2\geq\eta^2
+(1+\gamma)^2\zeta^2+\mut\xi^2
\end{equation}
and $\xi\neq 0$, we conclude that
\begin{equation}\label{eq:char-set-est}
\mut\leq \alpha^2
\end{equation}
there, so this form of the operator remains valid, and the characteristic set can indeed
be divided into two
components, separating $L_\pm$.

Next, we note that if $\alpha$ is so large that at $r=r_0$ with
$\frac{d\mut}{dr}(r_0)=0$, one has $\mut(r_0)=\alpha^2$, then letting $\eta_0=0$,
$\theta_0=\frac{\pi}{2}$, $\xi_0\neq 0$,
$\zeta_0=\pm\frac{\alpha}{1+\gamma}\xi_0$, the bicharacteristics through
$(r_0,\theta_0,\phi_0,\xi_0,\eta_0,\zeta_0)$ are stationary for any $\phi_0$,
so the operator is classically
trapping in the strong sense that not only is the Hamilton vector
field radial, but it vanishes. Since such vanishing means that weights
cannot give positivity in positive commutator estimates, see
Section~\ref{sec:microlocal},
it is natural to impose the restriction on $\alpha$ that
\begin{equation}\label{eq:alpha-restrict}
r_0\in(r_+,r_-),\ \frac{d\mut}{dr}(r_0)=0\Rightarrow \alpha^2<\mut(r_0).
\end{equation}
Under this assumption, by \eqref{eq:char-set-est}, the ergoregions from
the two ends do not intersect.

Finally, we show that bicharacteristics leave the region $\mut>\mut_0$,
where $\mut_0<0$ is such that $\frac{d\mut}{dr}$ is bounded away from $0$ on
$[\mut_0,(1+\ep)\alpha^2]_\mut$ for some $\ep>0$, which completes
checking the hypotheses in the classical sense. Note that by \eqref{eq:digamma-def}
and \eqref{eq:alpha-restrict}
such $\mut_0$ and $\ep$ exists.
To see this, we use $\tilde p$
to measure the size of the characteristic set over points in the base. Using
$ab\leq (1+\ep)a^2+b^2/(1+\ep)$ and $\kappa\sin^2\theta\leq 1$,
we note that on the characteristic set
$$
(1+\ep)\alpha^2\xi^2+\frac{(1+\gamma)^2}{1+\ep}\zeta^2\geq
\tilde p+\mut\xi^2\geq\frac{\ep}{1+\ep}\tilde p+\frac{(1+\gamma)^2}{1+\ep}\zeta^2+\mut\xi^2,
$$
so
$$
((1+\ep)\alpha^2-\mut)\geq\frac{\ep}{1+\ep}|\xi|^{-2}\tilde p,
$$
where now both sides are homogeneous of degree zero, or equivalently functions
on $S^*X_{\delta}$. Note that $\tilde p=0$ implies that $\xi\neq 0$ on $S^*X_\delta$,
so our formulae make sense.
By \eqref{eq:Kerr-hom-tilde-p-est}, using that $\frac{\pa\mut}{\pa r}$
is bounded away from $0$, $|\xi|^{-2}\tilde p$ is growing exponentially in the
forward/backward direction along the flow as long as the flow remains in a region
$\mut\geq\mut_0$,
where the form of the operator is valid (which is automatic in this region, as
farther on `our side' of the event horizon, $X_+$,
where the form of the operator
is not valid, it is elliptic), which shows
that the bicharacteristics have to leave this region. As noted already, this proves
that the operator fits into our framework in the classical sense.

\subsection{Semiclassical behavior}\label{subsec:Kerr-semicl}
The semiclassical principal symbol is
\begin{equation}\begin{split}\label{eq:p-semicl-Kerr}
p_{\semi,z}=&-\mut\big(\xi\pm  cz\big)^2
\mp2(1+\gamma)(r^2+\alpha^2)\big(\xi\pm  cz\big) z
\pm 2(1+\gamma)\alpha\big(\xi\pm  cz\big)
\zeta-\tilde p_{\semi,z}
\end{split}\end{equation}
with
$$
\tilde p_{\semi,z}=\kappa\eta^2+\frac{(1+\gamma)^2}{\kappa\sin^2\theta}(-\alpha\sin^2\theta z+\zeta)^2.
$$
Recall now that $M_\delta=X_\delta\times[0,\infty)_\tau$, and that, due
to Section~\ref{subsec:Lorentz}, we need
to choose $c$ in our definition of $\tau$
so that $\frac{d\tau}{\tau}$ is time-like with respect to $G$. But
$$
\langle \frac{d\tau}{\tau},\frac{d\tau}{\tau}\rangle_G=-\mut c^2-2c(1+\gamma)(r^2+\alpha^2)-\frac{\alpha^2(1+\gamma)^2\sin^2\theta}{\kappa},
$$
and as this must be positive for all $\theta$, we need to arrange that
\begin{equation}\label{eq:c-est}
\mut c^2+2c(1+\gamma)(r^2+\alpha^2)+\alpha^2(1+\gamma)^2<0,
\end{equation}
and this in turn suffices.
Note that $c=-\mut^{-1}(1+\gamma)(r^2+\alpha^2)$ automatically satisfies
this in $\mut>0$; this would correspond to undoing our change of coordinates
in \eqref{eq:h-P-form} (which is harmless away from $\mut=0$, but of course
$c$ needs to be smooth at $\mut=0$).
At $\mut=0$, \eqref{eq:c-est} gives a (negative)
upper bound for $c$; for $\mut>0$ we have an interval of possible values of
$c$; for $\mut<0$ large negative values of $c$ always work. Thus, we may choose
a smooth function $c$ such that \eqref{eq:c-est} is satisfied everywhere, and we may further arrange that $c=-\mut^{-1}(1+\gamma)(r^2+\alpha^2)$ for $\mut>\mut_1$
where $\mut_1$ is an arbitrary positive constant;
in this
case, as already discussed, $p_{\semi,z}$ is semiclassically elliptic when $\im z\neq 0$.

Note also that, as discussed in Subsection~\ref{subsec:Lorentz}, there
is only one component of the characteristic set in $\mut>\alpha^2$ by
\eqref{eq:char-set-est}, namely $\Sigma_{\semi,\sgn(\re z)}$.

It remains to discuss trapping. Note that the dynamics depends continuously
on $\alpha$, with $\alpha=0$ being the de Sitter-Schwarzschild case, when
there is no trapping near the event horizon, so the same holds for Kerr-de Sitter
with slow rotation.
Below we first describe the dynamics in de Sitter-Schwarzschild space
explicitly,
and then, in \eqref{eq:semicl-alpha-limit}, give an explicit range of
$\alpha$
in which the non-trapping dynamical assumption, apart from hyperbolic
trapping, is satisfied.

First, on de Sitter-Schwarzschild space, recalling that $c$ is irrelevant
for the dynamics for real
$z$, we may take $c=0$ (i.e.\ otherwise we would simply
change this calculation by the effect of a symplectomorphism, corresponding
to a conjugation, which we note does not affect the `base' variables
on the cotangent bundle). Then
\begin{equation*}\begin{split}
p_{\semi,z}=-\mut\xi^2\mp2r^2\xi z-\tilde p_{\semi,z},
\qquad \tilde p_{\semi,z}=\eta^2+\frac{\zeta^2}{\sin^2\theta},
\end{split}\end{equation*}
so
$$
\sH_{p_{\semi,z}}=-2(\mut\xi\pm r^2 z)\pa_r+\left(\frac{\pa\mut}{\pa r}\xi^2\pm
4r z\xi\right)\pa_\xi-\sH_{\tilde p_{\semi,z}},
$$
hence $\sH_{p_{\semi,z}}r=-2(\mut\xi\pm r^2 z)$, and so
$\sH_{p_{\semi,z}}r=0$ implies $\mp z=r^{-2}\mut\xi$. We first note that
$\sH_{p_{\semi,z}}r$ cannot vanish in $T^*X_\delta$ in $\mut\leq 0$
(though it can vanish at fiber infinity at $L_\pm$) since (for $z\neq 0$)
\begin{equation}\label{eq:no-stationary-pt-dS-Sch}
\mut\leq 0\Mand
\sH_{p_{\semi,z}}r=0\Rightarrow \mut\xi\neq 0\Mand p_{\semi,z}=\mut\xi^2-\tilde p_{\semi,z}<0.
\end{equation}
It remains to consider $\sH_{p_{\semi,z}}r=0$ in $\mut>0$. At such a point
$$
\sH_{p_{\semi,z}}^2r=-2\mut\sH_{p_{\semi,z}}\xi
=-2\mut\xi^2\left(\frac{\pa\mut}{\pa r}-4r^{-1}\mut\right)
=-2\mut\xi^2 r^4\frac{\pa(r^{-4}\mut)}{\pa r},
$$
so as $\mp z=r^{-2}\mut\xi$, so $\xi\neq 0$, by the discussion after
\eqref{eq:digamma-def},
$$
\mut>0,\ \pm(r-\frac{3}{2} r_s)>0,
\ \sH_{p_{\semi,z}}r=0\Rightarrow \pm \sH_{p_{\semi,z}}^2r>0,
$$ 
and thus the gluing hypotheses of \cite{Datchev-Vasy:Gluing-prop} are
satisfied arbitrarily close to\footnote{Or far from, in $\mut>0$.} $r=\frac{3}{2} r_s$.
Furthermore, as $p_{\semi,z}=-\mut^{-1}(\mut\xi\pm r^2 z)^2+r^4z^2-\tilde p_{\semi,z}$,
if $r=\frac{3}{2} r_s$, $\sH_{p_{\semi,z}}r=0$ and $p_{\semi,z}=0$
then $\tilde p_{\semi,z}=r^4z^2$, so
with
$$
\Gamma_z=\{r=\frac{3}{2} r_s,\ \mut\xi\pm r^2 z=0,\ \tilde p_{\semi,z}=r^4z^2\},
$$
we have
\begin{equation}\label{eq:dS-Sch-convex}
p_{\semi,z}(\varpi)=0,\ \mut(\varpi)>0,\ \varpi\notin\Gamma_z,
\ (\sH_{p_{\semi,z}}r)(\varpi)=0\Rightarrow (\pm \sH_{p_{\semi,z}}^2r)(\varpi)>0,
\end{equation}
with $\pm$ corresponding to whether $r>\frac{3}{2} r_s$ or $r<\frac{3}{2} r_s$.
In particular, taking into account
\eqref{eq:no-stationary-pt-dS-Sch},
$r$ gives rise to an escape function in $T^*X_\delta\setminus\Gamma_z$
as discussed in
Footnote~\ref{footnote:convex-escape}, and $\Gamma_z$ is the only
possible
trapping. (In this statement we ignore fiber infinity.)
Correspondingly, if one regards
a compact interval $I$ in $(r_-,\frac{3}{2} r_s)$, or $(\frac{3}{2} r_s,r_+)$ as 
the gluing region, for sufficiently small $\alpha$, for $r\in I$,
$\sH_{p_{\semi,z}}r=0$ still implies $\pm \sH_{p_{\semi,z}}^2r>0$, and
\cite{Datchev-Vasy:Gluing-prop}  is applicable. If instead one works with
compact subsets of $\{\mut>0\}\setminus\Gamma_z$, one has
non-trapping dynamics for $\alpha$ small.

Since in \cite[Section~2]{Wunsch-Zworski:Resolvent}
Wunsch and Zworski only check normal hyperbolicity in Kerr space-times
with sufficiently small angular momentum, in order to use their
general results for normally hyperbolic trapped sets,
we need to check that Kerr-de Sitter space-times are
still normally hyperbolic. For this, with small $\alpha$, we follow
\cite[Section~2]{Wunsch-Zworski:Resolvent}, and note that for $\alpha=0$
the linearization of the flow at $\Gamma$ in the normal variables
$r-\frac{3}{2}r_s$ and $\mut\xi\pm r^2 z$ is
$$
\begin{bmatrix}r-\frac{3}{2}r_s\\ \mut\xi\pm r^2 z\end{bmatrix}'
=\begin{bmatrix}0&-2(\frac{3}{2}r_s)^4 z^2\mut|_{r=\frac{3}{2}r_s}^{-1}\\-2&0\end{bmatrix}
\begin{bmatrix}r-\frac{3}{2}r_s\\\mut\xi\pm r^2 z\end{bmatrix}
+\cO((r-\frac{3}{2}r_s)^2+(\mut\xi\pm r^2 z)^2),
$$
so the eigenvalues of the linearization are $\lambda=\pm 3\sqrt{3}r_s z
\big(1-\frac{9}{4}\Lambda r_s^2\big)^{-1/2}$, in agreement with the result
of \cite{Wunsch-Zworski:Resolvent} when $\Lambda=0$. The rest of the
arguments concerning the flow in \cite[Section~2]{Wunsch-Zworski:Resolvent}
go through. In particular, when analyzing the flow {\em within}
$\Gamma=\cup_{z>0}\Gamma_z$, the pull backs of both $dp$ and $d\zeta$
are {\em exactly} as in the Schwarzschild setting (unlike the normal dynamics, which
has different eigenvalues), so the arguments of
\cite[Proof of Proposition~2.1]{Wunsch-Zworski:Resolvent} go
through unchanged, giving normal hyperbolicity for small $\alpha$ by
the structural stability.

We now check the hyperbolic nature of trapping for larger values of
$\alpha$. With $c=0$, as above,
\begin{equation*}
p_{\semi,z}=-\mut\xi^2\mp
2(1+\gamma)\big((r^2+\alpha^2)z-\alpha\zeta\big)\xi-\tilde p_{\semi.z},
\end{equation*}
and in the region $\mut>0$ this can be rewritten as
\begin{equation*}
p_{\semi,z}=-\mut\left(\xi\pm
  \frac{1+\gamma}{\mut}\big((r^2+\alpha^2)z-\alpha\zeta\big)
\right)^2+\frac{(1+\gamma)^2}{\mut}\big((r^2+\alpha^2)z-\alpha\zeta\big)^2-\tilde p_{\semi.z};
\end{equation*}
note that
the first term would be just $-\mut\xi^2$ in the
original coordinates \eqref{eq:metric-Kerr} which are valid in $\mut>0$. Thus,
\begin{equation}\label{eq:semicl-r-deriv-Kerr}
\sH_{p_{\semi.z}}r =-2\left(\mut\xi\pm  (1+\gamma)\big((r^2+\alpha^2)z-\alpha\zeta\big)\right).
\end{equation}
Correspondingly,
$$
\mut\leq 0,\ \sH_{p_{\semi,z}}r=0\Rightarrow
p_{\semi,z}=\mut\xi^2-\tilde p_{\semi.z}\leq 0,
$$
and equality on the right hand side implies $\zeta=\alpha\sin^2\theta
z$, so $\sH_{p_{\semi,z}}r=(r^2+\alpha^2\cos^2\theta)z>0$, a
contradiction, showing that
in $\mut\leq 0$, $\sH_{p_{\semi,z}}r$ cannot vanish on the
characteristic set.

We now turn to $\mut>0$, where
\begin{equation*}
\sH_{p_{\semi,z}}r=0\Rightarrow\sH^2_{p_{\semi,z}}r=-2\mut\sH_{p_{\semi,z}}\xi
=2\mut (1+\gamma)^2\frac{\pa}{\pa
  r}\left(\mut^{-1}\big((r^2+\alpha^2)z-\alpha\zeta\big)^2\right).
\end{equation*}
Thus, we are interested in critical points of
$$
F=\mut^{-1}\big((r^2+\alpha^2)z-\alpha\zeta\big)^2,
$$
and whether these are non-degenerate.
We remark that
\begin{equation*}\begin{split}
\sH_{p_{\semi,z}}r=0\Mand (r^2+\alpha^2)z-\alpha\zeta=0\Mand p_{\semi,z}=0&\Rightarrow
\xi=0\Mand \tilde p_{\semi,z}=0;\\
\tilde p_{\semi,z}=0\Mand (r^2+\alpha^2)z-\alpha\zeta=0\Mand p_{\semi,z}=0&\Rightarrow (r^2+\alpha^2\cos^2\theta)z=0;
\end{split}\end{equation*}
which is a contradiction, so $(r^2+\alpha^2)z-\alpha\zeta$ does not
vanish when $p_{\semi,z}$ and $\sH_{p_{\semi,z}}r$ do.
Note that
\begin{equation}\label{eq:F-prime-Kerr}
\frac{\pa F}{\pa r}=-\big((r^2+\alpha^2)z-\alpha\zeta\big)\mut^{-2} f,
\qquad f=\big((r^2+\alpha^2)z-\alpha\zeta\big)\frac{\pa\mut}{\pa r}-4r\mut z,
\end{equation}
so
\begin{equation}\label{eq:zeta-at-F-crit}
\frac{\pa F}{\pa r}=0\Mand (r^2+\alpha^2)z-\alpha\zeta\neq 0\Rightarrow \frac{\pa^2 F}{\pa r^2}
=-\big((r^2+\alpha^2)z-\alpha\zeta\big)\mut^{-2}\frac{\pa f}{\pa r}.
\end{equation}
Also, from \eqref{eq:F-prime-Kerr},
$$
\mut>0,\ f=0\Rightarrow \frac{\pa\mut}{\pa r}\neq 0.
$$
Now
\begin{equation}\label{eq:f-prime-Kerr}
\frac{\pa f}{\pa r}=\big((r^2+\alpha^2)z-\alpha\zeta\big)
\frac{\pa^2\mut}{\pa r^2}
-4\mut z-2rz\frac{\pa\mut}{\pa r},
\end{equation}
and
$$
\frac{\pa F}{\pa r}=0\Rightarrow (r^2+\alpha^2)z-\alpha\zeta=\frac{4r\mut z}{\frac{\pa\mut}{\pa r}},
$$
so substituting into \eqref{eq:f-prime-Kerr},
\begin{equation}\label{eq:f-prime-Kerr-no-zeta}
\frac{\pa\mut}{\pa r}\frac{\pa f}{\pa r}=4r\mut z
\frac{\pa^2\mut}{\pa r^2}
-4z \mut  \frac{\pa\mut}{\pa r}-2rz\left(\frac{\pa\mut}{\pa r}\right)^2.
\end{equation}
Thus,
\begin{equation*}\begin{split}
\frac{\pa\mut}{\pa r}\frac{\pa f}{\pa r}
=2z\left(2\mut\left(r\frac{\pa^2\mut}{\pa r^2}-3\frac{\pa\mut}{\pa r}\right)
-(r\frac{\pa\mut}{\pa r}-4\mut)\frac{\pa\mut}{\pa r}\right),
\end{split}\end{equation*}
so taking into account
\begin{equation*}\begin{split}
r\frac{\pa\mut}{\pa
  r}-4\mut&=-2\left(1-\frac{\Lambda\alpha^2}{3}\right)r^2+3r_s
r-4\alpha^2,\\
r\frac{\pa^2\mut}{\pa r^2}-3\frac{\pa\mut}{\pa r}&=
-4\left(1-\frac{\Lambda\alpha^2}{3}\right)r+3r_s,\\
r\frac{\pa^2\mut}{\pa r^2}-3\frac{\pa\mut}{\pa r}&=
\frac{2}{r}\left(r\frac{\pa\mut}{\pa r}-4\mut\right)-3r_s+\frac{8\alpha^2}{r},
\end{split}\end{equation*}
we obtain
\begin{equation*}\begin{split}
\frac{\pa\mut}{\pa r}\frac{\pa f}{\pa r}
=2z\left(-\frac{1}{r}\left(r\frac{\pa\mut}{\pa r}-4\mut\right)^2
-\frac{2\mut}{r}(3r_s r-8\alpha^2)\right).
\end{split}\end{equation*}

We claim that if $|\alpha|<r_s/2$ then $r_->r_s/2$. To see this, note that
for $r=r_s/2$,
$$
\mut(r)=\left(\frac{r_s^2}{4}+\alpha^2\right)
\left(1-\frac{\Lambda r_s^2}{12}\right) -\frac{r_s^2}{2}<0;
$$
since at $\alpha=0$, $r_->r_s/2$, we deduce that $r_->r_s/2$ for
$|\alpha|<r_s/2$. Making the slightly stronger assumption,
\begin{equation}\label{eq:semicl-alpha-limit}
|\alpha|<\frac{\sqrt{3}}{4} r_s,
\end{equation}
we obtain that for $\mut>0$, $r>r_-$,
$3r_s r-8\alpha^2>\frac{3}{2}r_s^2-8\alpha^2>0$, so
$$
z\frac{\pa\mut}{\pa r}\frac{\pa f}{\pa r}<0.
$$
Thus, when $\frac{\pa F}{\pa r}=0$, using \eqref{eq:zeta-at-F-crit},
\begin{equation}\label{eq:non-deg-crit-pts-Kerr}
\frac{\pa^2 F}{\pa r^2}
=-\big((r^2+\alpha^2)z-\alpha\zeta\big)\mut^{-2}\frac{\pa f}{\pa r}
=-\frac{4r}{\mut\left(\frac{\pa\mut}{\pa r}\right)^2} z\frac{\pa\mut}{\pa r}\frac{\pa f}{\pa r}>0,
\end{equation}
so critical points of $F$ are all non-degenerate and are
minima. Correspondingly, as $F\to+\infty$ as $\mut\to 0$ in $\mut>0$,
the critical point $r_c$ of $F$ exists and is unique in $(r_-,r_+)$
(when $\zeta$ is fixed),
depends smoothly on $\zeta$,
and $\frac{\pa F}{\pa r}>0$ if $r>r_c$, and $\frac{\pa F}{\pa r}<0$ if
$r<r_c$. Thus,
$$
\mut>0,\ \pm(r-r_c)>0,
\ \sH_{p_{\semi,z}}r=0\Rightarrow \pm \sH_{p_{\semi,z}}^2r>0,
$$
giving the natural generalization of \eqref{eq:dS-Sch-convex},
allowing the application of the results of
\cite{Datchev-Vasy:Gluing-prop}.
Since $\sH_{p_{\semi,z}}r$ cannot vanish in $\mut\leq 0$ (apart from
fiber infinity, which is understood already), we conclude that $r$
gives rise to an escape function, as in Footnote~\ref{footnote:convex-escape}, away from
$$
\Gamma_z=\{\varpi:\ \frac{\pa F}{\pa r}(\varpi)=0,\ (\sH_{p_{\semi,z}}
r)(\varpi)=0,\ p_{\semi,z}(\varpi)=0\},
$$
which is a smooth submanifold as the differentials of the defining
functions
are linearly independent on it in view of
\eqref{eq:non-deg-crit-pts-Kerr}, \eqref{eq:semicl-r-deriv-Kerr}, and
the
definition of $\tilde p_{\semi,z}$ (as the latter is independent of $r$ and $\xi$).

The linearization of the Hamilton flow at $\Gamma_z$ is
\begin{equation*}\begin{split}
&\begin{bmatrix}r-r_c\\ \mut\xi\pm (1+\gamma)\big((r^2+\alpha^2)z-\alpha\zeta\big)
\end{bmatrix}'\\
&\qquad=\begin{bmatrix}0&-\mut(1+\gamma)^2\frac{\pa^2 F}{\pa r^2}\\-2&0\end{bmatrix}
\begin{bmatrix}r-r_c\\ \mut\xi\pm
  (1+\gamma)\big((r^2+\alpha^2)z-\alpha\zeta\big)\end{bmatrix}\\
&\qquad\qquad\qquad
+\cO\Big((r-r_c)^2+\big(\mut\xi\pm (1+\gamma)\big((r^2+\alpha^2)z-\alpha\zeta\big) \big)^2\Big),
\end{split}\end{equation*}
so by \eqref{eq:non-deg-crit-pts-Kerr}, the linearization is
non-degenerate, and is indeed hyperbolic. This suffices for the
resolvent estimates of \cite{Wunsch-Zworski:Resolvent} for exact
Kerr-de Sitter, but for
stability one also needs to check normal hyperbolicity.
While it is quite straightforward to check that the only degenerate
location is $\eta=0$, $\theta=\frac{\pi}{2}$, the computation of the
Morse-Bott non-degeneracy in the spirit of \cite[Proof of
Proposition~2.1]{Wunsch-Zworski:Resolvent}, where it is done for Kerr
spaces with small angular momentum, is rather involved, so we do not
pursue this here (for small angular momentum in Kerr-de Sitter space, the de
Sitter-Schwarzschild calculation above implies normal hyperbolicity already).

In addition, in view of an overall sign difference between our convention and
that of \cite{Wunsch-Zworski:Resolvent} for the operator we are considering,
\cite{Wunsch-Zworski:Resolvent} requires the positivity
of $z\frac{\pa}{\pa z} p_{\semi,z}$ for $z\neq 0$. (Note that the notation for $z$
is also different; our $z$ is $1+z$ in the notation of
\cite{Wunsch-Zworski:Resolvent},
so our $z$ being near $1$
corresponds to the $z$ of \cite{Wunsch-Zworski:Resolvent} being near $0$.)
Unlike the flow, whose behavior is independent of
$c$ when $z$ is real, this fact does depend on the choice of $c$. Note that
in the high energy version, this corresponds to the positivity
of $\sigma\frac{\pa}{\pa \sigma} p_\full$.
Now, $p_\full=\langle \sigma\,\frac{d\tau}{\tau}+\varpi,\sigma\,\frac{d\tau}{\tau}+\varpi\rangle_G$, with $\varpi\in\Pi$, the `spatial' hyperplane, identified
with $T^*X$ in $\Tb^*\bM$, so
\begin{equation*}\begin{split}
\sigma\pa_\sigma p_{\full}&=2\langle \sigma\,\frac{d\tau}{\tau},\sigma\,\frac{d\tau}{\tau}\rangle_G+2 \langle \sigma\,\frac{d\tau}{\tau},\varpi\rangle_G\\
&=\sigma^2\langle \frac{d\tau}{\tau},\frac{d\tau}{\tau}\rangle_G
+\langle \sigma\,\frac{d\tau}{\tau}+\varpi,\sigma\,\frac{d\tau}{\tau}+\varpi\rangle_G
-\langle\varpi,\varpi\rangle_G.
\end{split}\end{equation*}
Thus, if non-zero elements of $\Pi$ are space-like and $\frac{d\tau}{\tau}$ is time-like,
$\sigma\pa_\sigma p_{\full}>0$ for $\sigma\neq 0$ on the characteristic set
of $p_\full$. If $c$ is such that $c=-\mut^{-1}(1+\gamma)(r^2+\alpha^2)$
near $r=\frac{3}{2}r_s$, which as we mentioned can be arranged, and
which corresponds to undoing our change of coordinates
in \eqref{eq:h-P-form}, then directly from \eqref{eq:metric-Kerr} both the time-like
and space-like statements hold, completing our checking of the hypotheses
of \cite{Wunsch-Zworski:Resolvent}, and thus their result is applicable for
de Sitter-Schwarzschild space-times. As these results are structurally
stable, see the proof of \cite[Proposition~2.1]{Wunsch-Zworski:Resolvent}, the result
follows for Kerr-de Sitter spaces with angular momenta satisfying \eqref{eq:semicl-alpha-limit}.

\subsection{Complex absorption}\label{subsec:Kerr-absorb}
The final step of fitting $P_\sigma$ into a general microlocal framework is
moving the problem to a compact manifold, and adding a complex absorbing
second order operator.
This section is {\em almost completely parallel} to
Subsection~\ref{subsec:complex-absorb-dS}
in the de
Sitter case; the only change is that absorption needs to be added at the
trapped set as well.

We thus consider a compact manifold without boundary
$X$ for which $X_{\delta}$ is identified as an open subset with smooth boundary;
we can again take $X$ to be the double of $X_{\delta}$.
As in the de Sitter case, we discuss the
`classical' and `semiclassical' cases separately,
for in the former setting trapping does not matter, while in the
latter it does.

We then introduce a complex
absorbing operator $Q_\sigma\in\Psi_{\cl}^2(X)$ with principal symbol $q$,
such that $h^2Q_{h^{-1}z}\in\Psihcl^2(X)$ with
semiclassical principal symbol $q_{\semi,z}$, and such that $p\pm\imath q$
is elliptic near $\pa X_{\delta}$, i.e.\ near $\mut=\mut_0$, the
Schwartz kernel of $Q_\sigma$ is supported in $\mut<\mut_0+\ep'$ for
some sufficiently small $\ep'>0$,
and which satisfies that the $\pm q\geq 0$
on $\Sigma_\mp$.
Having done this, we extend $P_\sigma$ and $Q_\sigma$ to $X$ in such a way that
$p\pm\imath q$ are elliptic near $X\setminus X_{\delta}$; the region
we added is thus irrelevant.  In particular, as the event horizon is characteristic
for the wave equation, the solution in the exterior of the event horizons
is {\em unaffected} by thus modifying $P_\sigma$, i.e.\ working with $P_\sigma$
and $P_\sigma-\imath Q_\sigma$ is equivalent for this purpose.

Again, as in de Sitter space,
an alternative to this extension would be simply adding a boundary at $\mut=\mut_0$;
this is easy to do since this is a space-like hypersurface, see
Remark~\ref{rem:add-bdy}.

For the semiclassical problem, when $z$ is almost real
we need to increase the requirements on $Q_\sigma$.
As in the de Sitter setting, discussed in Subsection~\ref{subsec:complex-absorb-dS},
we need in addition, in the semiclassical notation, semiclassical ellipticity
near $\mut=\mut_0$, i.e.\ that
$p_{\semi,z}\pm\imath q_{\semi,z}$ are elliptic near $\pa X_{\delta}$, i.e.\ near $\mut=\mut_0$,
and which satisfies that $\pm q_{\semi,z}\geq 0$ on
$\Sigma_{\semi,\mp}$.
Following the general prescription of Subsection~\ref{subsec:Lorentz},
as well as the discussion of
Subsection~\ref{subsec:complex-absorb-dS},
this can be achieved by taking $Q_\sigma$ the (standard)
quantization of
\begin{equation}\label{eq:Kerr-absorb}
-\langle \xi\,dr+\sigma\frac{d\tau}{\tau}+\eta\,d\theta+\zeta\,d\phi,
\frac{d\tau}{\tau}\rangle_G \,
(\|\xi\,dr+\eta\,d\theta+\zeta\,d\phi\|_{\tilde
  H}^{2j}+\sigma^{2j}+C^{2j})^{1/2j}\chi(\mu),
\end{equation}
where $\tilde H$ is a {\em Riemannian} dual metric on $X$, $\chi\geq 0$ as in
Subsection~\ref{subsec:complex-absorb-dS} supported near
$\mut=\mut_0$,
$C>0$ is chosen suitably large, and the branch of the $2j$th
root as in Subsection~\ref{subsec:complex-absorb-dS}. One can again
combine $p$ with a Riemannian metric function $\|.\|^2_{\tilde H}$, to replace
$p$ by $\chi_1 p-\chi_2\hat p_{\semi,z}$, $\hat
p_{\semi,z}=(\|.\|^{2j}_{\tilde H}+z^{2j})^{1/j}$,
as in Subsection~\ref{subsec:complex-absorb-dS}.

We recall that in the proof of
Theorem~\ref{thm:classical-absorb-glued} one also need to arrange
semiclassical ellipticity (i.e.\ define an appropriate $Q'_\sigma$) for an appropriate
perturbation of $p_{\semi,z}$ at the trapped set, which
is in $X_+$; we now make this more explicit. This is only an issue if $z$ is real; otherwise the
Lorentzian nature of the metric means that $p_{\semi,z}$ itself is
elliptic, as discussed in Subsection~\ref{subsec:Lorentz}.
To achieve this, we want $q'_{\semi,z}$ elliptic
on the trapped set; since this is in $\Sigma_{\semi,\sgn{\re z}}$, we
need $q'_{\semi,z}\leq 0$ there. To do so, we simply add a microlocal absorbing term
$Q'_\sigma$ supported microlocally near the trapping with
$h^2Q_{h^{-1}z}$ having
semiclassical principal symbol $q'_{\semi,z}$. We {\em do not} need to
arrange that $Q'_\sigma$ is holomorphic in $\sigma$; thus simply
quantizing a $q'_{\semi,z}$ of compact support on $T^*X$ and with
smooth dependence on $z\in\RR\setminus \{0\}$ suffices.
Then with $Q_\sigma$ as above (so we do not change $Q_\sigma$)
$P_\sigma-\imath Q_\sigma$ is holomorphic, and its inverse is
meromorphic, with non-trapping large energy estimates in closed cones disjoint from $\RR$
in the upper half plane, corresponding to the semiclassical estimates
for non-real $z$ given in Theorem~\ref{thm:classical-absorb-sector}. To see that large energy estimates also hold for
$(P_\sigma-\imath Q_\sigma)^{-1}$ near
$\RR$, namely in $\im\sigma>-C$, one considers
$(P_\sigma-\imath(Q_\sigma+Q'_\sigma))^{-1}$, which enjoys such
estimates but is not holomorphic, and then use the semiclassical
resolvent gluing of Datchev-Vasy \cite{Datchev-Vasy:Gluing-prop}
together with the semiclassical normally hyperbolic trapping estimates
of Wunsch-Zworski \cite{Wunsch-Zworski:Resolvent}, to conclude the
same estimates for $(P_\sigma-\imath Q_\sigma)^{-1}$.
For $\alpha$ as in \eqref{eq:semicl-alpha-limit}, the dynamics
(away from the
radial points) has only the hyperbolic trapping (and for small
$\alpha$, it is normally hyperbolic);
however, our results apply more generally, as long as the dynamics has
the same non-trapping character (so $\alpha$ might be even larger as
\eqref{eq:semicl-alpha-limit} may not be optimal).
Note also that since the trapping is in a compact subset of $X_+=\{\mut>0\}$,
we arranged that the complex absorption $Q_\sigma+ Q'_\sigma$
is the sum of two terms: one supported
near the trapping in $X_+$, the other in $\mut<0$; this is useful for relating
our construction to that of Dyatlov \cite{Dyatlov:Quasi-normal} in the appendix.

This completes the setup. Now all of the results of Section~\ref{sec:microlocal}
are applicable, proving all the theorems stated in the introduction on Kerr-de Sitter
spaces, Theorems~\ref{thm:complete-absorb}-\ref{thm:exp-decay}. Namely,
Theorem~\ref{thm:complete-absorb} follows from
Theorem~\ref{thm:classical-absorb-strip},
Theorem~\ref{thm:spatial-absorb}
follows from Theorem~\ref{thm:classical-absorb-sector},
Theorem~\ref{thm:glued} follows from Theorem~\ref{thm:classical-absorb-glued}.
Finally Theorem~\ref{thm:exp-decay} is an immediate consequence of
Theorem~\ref{thm:glued}, the Mellin transform result,
Corollary~\ref{cor:Mellin-expand}, in the Kerr-de Sitter setting,
or the appropriately modified, as indicated in Remark~\ref{rem:Mellin-expand}, version of
Proposition~\ref{prop:Mellin-expand} for general b-perturbations
(so $\pa_{\tilde t}$ may no longer be Killing, and the space-time may no
longer be stationary), together with the fact that $d\mut$ is time-like
in $\mut<0$ (since $p_{\full}$, considered as a quadratic form,
evaluated at $\xi=1$, $\sigma=0$, $\zeta=0$, $\eta=0$, is positive
then) and
Proposition~\ref{prop:wave-local}.


\appendix

\section{Comparison with cutoff resolvent constructions}

\begin{center}
{\sc By Semyon Dyatlov}\footnote{S.D.'s address is
Department of Mathematics, University of California, Berkeley, CA 94720-3840, USA,
and e-mail address is \texttt{dyatlov@math.berkeley.edu}.}

\end{center}
\vspace*{3mm}

In this appendix, we will first examine the relation of the resolvent
considered in the present paper to the cutoff resolvent for slowly
rotating Kerr--de Sitter metric constructed in~\cite{Dyatlov:Quasi-normal} using
separation of variables and complex contour deformation near the event
horizons. Then, we will show how to extract information on the
resolvent beyond event horizons from information about the cutoff
resolvent.

First of all, let us list some notation of~\cite{Dyatlov:Quasi-normal} along with its
analogues in the present paper:
\begin{center}
\begin{tabular}{| r | l | r | l |}
\hline
Present paper & \cite{Dyatlov:Quasi-normal} &
Present paper & \cite{Dyatlov:Quasi-normal}\\
\hline
$\alpha$ & $a$ & $r_s$ & $2M_0$\\
$\gamma$ & $\alpha$ & $\tilde\mu$ & $\Delta_r$\\
$\kappa$ & $\Delta_\theta$ & $F_\pm$ & $A_\pm$\\
$\tilde t,\tilde\phi$ & $t,\varphi$ &
$t,\phi$ & $t^*,\varphi^*$\\
$\omega$ & $\sigma$ & $e^{-i\sigma h(r)}P_\sigma e^{i\sigma h(r)}$ & $-P_g(\sigma)$\\
$X_+$ & $M$ & $K_\delta$ & $M_K$ \\
\hline
\end{tabular}
\end{center}
\medskip
The difference between $P_g(\omega)$ and $P_\sigma$ is due to the fact
that $P_g(\omega)$ was defined using Fourier transform in the $\tilde t$
variable and $P_\sigma$ is defined using Fourier transform in the
variable $t=\tilde t+h(r)$. We will henceforth use the notation of the
present paper.

We assume that $\delta>0$ is small and fixed, and $\alpha$ is small
depending on $\delta$. Define
$$
K_\delta=(r_-+\delta,r_+-\delta)_r\times \mathbb S^2.
$$
Then~\cite[Theorem~2]{Dyatlov:Quasi-normal} gives a family of operators
$$
R_g(\sigma):L^2(K_\delta)\to H^2(K_\delta)
$$
meromorphic in $\sigma\in \mathbb C$ and such that
$P_g(\sigma)R_g(\sigma)f=f$ on $K_\delta$ for each $f\in
L^2(K_\delta)$.

\begin{prop}
Assume that the complex absorbing operator $Q_\sigma$ satisfies the
assumptions of Section~\ref{subsec:Kerr-absorb} in the `classical' case and furthermore,
its Schwartz kernel is supported in $(X \setminus X_+)^2$. Let
$R_g(\sigma)$ be the operator constructed in~\cite{Dyatlov:Quasi-normal} and
$R(\sigma)=(P_\sigma-iQ_\sigma)^{-1}$ be the operator defined in
Theorem~1.2 of the present paper. Then for each $f\in
C_0^\infty(K_\delta)$,
\begin{equation}\label{ae:r-g-omega-eq}
-e^{i\sigma h(r)}R_g(\sigma)e^{-i\sigma h(r)}f=R(\sigma)f|_{K_\delta}.
\end{equation}
\end{prop}
\begin{proof} 
The proof follows~\cite[Proposition~1.2]{Dyatlov:Quasi-normal}. Denote by $u_1$ the
left-hand side of~\eqref{ae:r-g-omega-eq} and by $u_2$ the right-hand
side.  Without loss of generality, we may assume that $f$ lies in the
kernel $\mathcal D'_k$ of the operator $D_\phi-k$, for some $k\in
\mathbb Z$; in this case, by~\cite[Theorem~1]{Dyatlov:Quasi-normal}, $u_1$ can be
extended to the whole $X_+$ and solves the equation $P_\sigma u_1=f$
there. Moreover, by~\cite[Theorem~3]{Dyatlov:Quasi-normal}, $u_1$ is smooth up to the
event horizons $\{r=r_\pm\}$. Same is true for $u_2$; therefore, the
difference $u=u_1-u_2$ solves the equation $P_\sigma(u)=0$ and is
smooth up to the event horizons.

Since both sides of~\eqref{ae:r-g-omega-eq} are meromorphic, we may
further assume that $\Imag\sigma>C_e$, where $C_e$ is a large
constant.  Now, the function $\tilde u(t,\cdot)=e^{-it\sigma}u(\cdot)$
solves the wave equation $\Box_g \tilde u=0$ and is smooth up to the
event horizons in the coordinate system $(t,r,\theta,\phi)$;
therefore, if $C_e$ is large enough, by~\cite[Proposition~1.1]{Dyatlov:Quasi-normal}
$\tilde u$ cannot grow faster than $\exp(C_e t)$.  Therefore, $u=0$ as
required.
\end{proof}
Now, we show how to express the resolvent $R(\sigma)$ on the whole
space in terms of the cutoff resolvent $R_g(\sigma)$ and the
nontrapping parametrix constructed in the present paper.  Let
$Q_\sigma$ be as above, but with the additional assumption of
semiclassical ellipticity near $\partial X_\delta$, and $Q'_\sigma$ be
an operator satisfying the assumptions of Section~\ref{subsec:Kerr-absorb} in the
`semiclassical' case on the trapped set. Moreover, we require that the semiclassical
wavefront set of $|\sigma|^{-2}Q'_\sigma$ be compact and
$Q'_\sigma=\chi Q'_\sigma=Q'_\sigma\chi$, where $\chi\in
C_0^\infty(K_\delta)$.  Such operators exist for $\alpha$ small
enough, as the trapped set is compact and located $O(\alpha)$ close to
the photon sphere $\{r=3r_s/2\}$ and thus is far from the event
horizons.  Denote $R'(\sigma)=(P_\sigma-iQ_\sigma-iQ'_\sigma)^{-1}$;
by Theorem~\ref{thm:classical-absorb-strip} applied in the case of Section~\ref{subsec:Kerr-absorb}, for each $C_0$
there exists a constant $\sigma_0$ such that for $s$ large enough,
$\Imag\sigma>-C_0$, and $|\Realp\sigma|>\sigma_0$,
$$
\|R'(\sigma)\|_{H_{|\sigma|^{-1}}^{s-1}\to H_{|\sigma|^{-1}}^s}\leq C|\sigma|^{-1}.
$$
We now use the identity
\begin{equation}\label{ae:gluing-identity}
R(\sigma)=R'(\sigma)-R'(\sigma)(iQ'_\sigma+Q'_\sigma(\chi R(\sigma)\chi)Q'_\sigma) R'(\sigma).
\end{equation}
(To verify it, multiply both sides of the equation by
$P_\sigma-iQ_\sigma-iQ'_\sigma$ on the left and on the right.) 
Combining~\eqref{ae:gluing-identity} with the fact that for each $N$,
$Q'_\sigma$ is bounded $H^{-N}_{|\sigma|^{-1}}\to H^N_{|\sigma|^{-1}}$
with norm $O(|\sigma|^2)$, we get for $\sigma$ not a pole of $\chi
R(\sigma)\chi$,
\begin{equation}\label{ae:gluing-estimate}
\|R(\sigma)\|_{H_{|\sigma|^{-1}}^{s-1}\to H_{|\sigma|^{-1}}^s}\leq C(1+|\sigma|^2\|\chi R(\sigma)\chi\|_{L^2(K_\delta)\to L^2(K_\delta)}).
\end{equation}
Also, if $\sigma_0$ is a pole of $R(\sigma)$ of algebraic multiplicity
$j$, then we can multiply the identity~\eqref{ae:gluing-identity} by
$(\sigma-\sigma_0)^j$ to get an estimate similar
to~\eqref{ae:gluing-estimate} on the function $(\sigma-\sigma_0)^j
R(\sigma)$, holomorphic at $\sigma=\sigma_0$.

The discussion above in particular implies that the cutoff resolvent
estimates of~\cite{Bony-Haefner:Decay} also hold for the resolvent $R(\sigma)$. Using
the Mellin transform, we see that the resonance expansion
of~\cite{Bony-Haefner:Decay} is valid for any solution $u$ to the forward time Cauchy
problem for the wave equation on the whole $M_\delta$, with initial
data in a high enough Sobolev class; the terms of the expansion are
defined and the remainder is estimated on the whole $M_\delta$ as
well.

\section*{Acknowledgments}
A.V.\ is very grateful to Maciej Zworski, Richard Melrose, Semyon
Dyatlov, Mihalis Dafermos,
Gunther Uhlmann, Jared
Wunsch, Rafe Mazzeo,
Kiril Datchev, Colin Guillarmou and Dean Baskin for very helpful discussions,
for their enthusiasm for this project and for carefully reading parts
of this manuscript. Special thanks are due to Semyon
Dyatlov in this regard who noticed an incomplete argument in an
earlier version of this paper in holomorphicity considerations, and to
Mihalis Dafermos, who urged the author to supply details to the
argument at the end of Section~\ref{sec:Kerr}, which resulted in the
addition of Subsection~\ref{subsec:local-wave}, as well as the part of
Subsection~\ref{subsec:Lorentz} covering the complex absorption, to the main body of
the argument, and that of Subsection~\ref{subsec:stability} for
stability considerations.

A.V.\ gratefully
  acknowledges partial support from the National Science Foundation under grants number  
  DMS-0801226 and DMS-1068742 and from a Chambers Fellowship at Stanford University, as well
as the hospitality of Mathematical Sciences Research Institute in
Berkeley. S.D.\ is grateful
for partial support from
the National Science Foundation under grant number DMS-0654436.

\def\cprime{$'$} \def\cprime{$'$}

\end{document}